\documentclass[numbers=withendperiod]{scrartcl}
\usepackage{appendix}
\usepackage{libertinus}
\usepackage[T1]{fontenc}

\usepackage[hypertexnames=false]{hyperref}
\usepackage[utf8]{inputenc}
\usepackage[dvipsnames,table]{xcolor}
\usepackage{multirow}
\usepackage{float}
\usepackage{quiver}

\usepackage[backend=bibtex,style=alphabetic, maxbibnames=4, maxalphanames=4]{biblatex} 

\usepackage{bm}

\usepackage{caption}
\usepackage{subcaption}
\usepackage{import}
\usepackage{xifthen}
\usepackage{pdfpages}
\usepackage{transparent}
\usepackage{float}

\usetikzlibrary{math}

\usepackage{amsmath}
\usepackage{amsthm}
\usepackage{amssymb}
\usepackage{mathrsfs}
\usepackage{dsfont}
\usepackage{tikz-cd}
\usepackage{mathtools}
\usepackage{hyperref}

\newtheorem{thm}{Theorem}[subsection]

\newtheorem*{thm*}{Theorem}
\newtheorem*{theorem*}{Theorem}
\newtheorem*{prop*}{Proposition}
\newtheorem*{lem*}{Lemma}
\newtheorem*{question*}{Question}

\newtheorem{proposition}[thm]{Proposition}
\newtheorem{prop}[thm]{Proposition}
\newtheorem{corollary}[thm]{Corollary}
\newtheorem{cor}[thm]{Corollary}

\newtheorem{lemma}[thm]{Lemma}
\newtheorem{lem}[thm]{Lemma}

\theoremstyle{definition}
\newtheorem{definition}[thm]{Definition}
\newtheorem{defn}[thm]{Definition}

\theoremstyle{remark}
\newtheorem{remark}[thm]{\textsc{Remark}}

\newtheorem{rmk}[thm]{\textsc{Remark}}

\newtheorem{example}[thm]{\textsc{Example}}
\newtheorem{ex}[thm]{\textsc{Example}}
\newtheorem{notation}[thm]{\textsc{Notation}}

\makeatletter
\@addtoreset{thm}{section}
\@addtoreset{thm}{chapter}
\makeatother

\newcommand{\operator}[1]{\operatorname{#1}}
\newcommand{\rk}{\operatorname{rk}}
\newcommand{\ONE}{\mathds{1}}
\newcommand{\id}{\operator{id}}
\newcommand{\Id}{\operator{Id}}
\renewcommand{\d}{\mathrm{d}}

\newcommand{\mscr}[1]{\mathscr{#1}}

\newcommand{\eps}{\varepsilon}

\newcommand{\tensor}{\otimes}

\newcommand{\dirSum}{\oplus}
\newcommand{\DirSum}{\bigoplus}

\newcommand{\comp}{\circ}

\newcommand{\iso}{\cong}

\newcommand{\define}{\coloneqq}

\newcommand{\into}{\hookrightarrow}
\newcommand{\onto}{\twoheadrightarrow}
\newcommand{\xto}{\xrightarrow}

\newcommand{\lowxto}[1]{\mathrel{\raisebox{-2pt}{$\xrightarrow{#1}$}}}
\newcommand{\lowxot}[1]{\mathrel{\raisebox{-2pt}{$\xleftarrow{#1}$}}}
\newcommand{\isoto}{\lowxto{\cong}}
\newcommand{\isoot}{\lowxot{\cong}}
\newcommand{\eqto}{\lowxto{\simeq}}
\newcommand{\eqot}{\lowxot{\simeq}}
\newcommand{\ot}{\leftarrow}

\newcommand{\IN}{\mathds{N}}
\newcommand{\IZ}{\mathds{Z}}

\newcommand{\IR}{\mathds{R}}
\newcommand{\IC}{\mathds{C}}
\newcommand{\RP}{\mathds{R}\mathds{P}}

\newcommand{\categoryname}{\mathbf}
\newcommand{\cat}[1]{\mbf{#1}} % categories
\newcommand{\sh}[1]{\mc{#1}}   % sheaves
\newcommand{\st}[1]{\sh{#1}}   % stacks
\newcommand{\ger}[1]{\sh{#1}}  % gerbes
\newcommand{\orb}[1]{\sh{#1}}  % orbifolds
\newcommand{\f}[1]{\mathcal{#1}} % family notation

\newcommand{\twocat}{\mscr}
\newcommand{\oneObjBicat}{\mathbf{1}}
\newcommand{\opp}{\mathrm{op}}
\newcommand{\ring}{\mathrm{o}}
\newcommand{\point}{\ast}

\newcommand{\Vect}{\cat{Vect}}
\newcommand{\VCat}{\cat{VCat}}
\newcommand{\VSt}{\cat{VSt}}
\newcommand{\CAT}{\cat{Cat}}
\newcommand{\Set}{\cat{Set}}
\newcommand{\Man}{\cat{Man}}

\newcommand{\CartSp}{\cat{CartSp}}
\newcommand{\SmSt}{\cat{SmSt}}

\newcommand{\Rep}{\cat{Rep}}
\renewcommand{\Vec}{\cat{Vec}}
\newcommand{\VecInf}{\cat{Vec}^\infty}
\newcommand{\sVec}{\cat{sVec}}
\newcommand{\Semi}{\cat{Semi}}
\newcommand{\bSemi}{\overline{\cat{Semi}}}
\newcommand{\Op}{\cat{Op}}
\newcommand{\St}{\cat{St}}
\newcommand{\Fun}{\cat{Fun}}
\newcommand{\Shv}{\cat{Shv}}
\newcommand{\Mod}{\cat{Mod}}
\newcommand{\Desc}{\cat{Desc}}
\newcommand{\Cinf}{{C^\infty}}
\newcommand{\CinfOf}[1]{{C^\infty_{#1}}}
\newcommand{\ShC}{\cat{Shv_{\Cinf}}}
\newcommand{\Sky}{\cat{Sky}}
\newcommand{\stC}{{\underline{\cat{C}}}}
\newcommand{\shO}{\sh{O}}   % structure sheaf

\newcommand{\IO}{\mathbb{O}}
\newcommand{\IA}{\mathbb{A}}
\newcommand{\IB}{\mathbb{B}}
\newcommand{\TY}{\cat{T}\cat{Y}}
\newcommand{\DZ}{\mscr{Z}}

\newcommand{\B}{\mathscr{B}}
\renewcommand{\H}{\mathrm{H}}
\newcommand{\ev}{\mathrm{ev}}
\newcommand{\coev}{\mathrm{coev}}
\newcommand{\Hom}{\mathrm{Hom}}
\newcommand{\sHom}{\underline{\mathrm{Hom}}} % sheafy Hom - might want to remove this!
\newcommand{\End}{\mathrm{End}}

\newcommand{\Aut}{\mathrm{Aut}}
\newcommand{\ob}{\operatorname{ob}}
\newcommand{\Irr}{\operatorname{Irr}}
\newcommand{\supp}{\operatorname{supp}}
\newcommand{\Ind}{\operatorname{Ind}}
\newcommand{\K}{\mathrm{K}} % Grothendieck group
\newcommand{\colim}{\operatorname{colim}}
\newcommand{\DiffSt}{\categoryname{DiffSt}}
\newcommand{\LinSt}{\categoryname{LinSt}}
\newcommand{\EtSt}{\categoryname{EtSt}}

\newcommand{\Cat}{\categoryname{Cat}}
\newcommand{\Gpd}{\categoryname{Gpd}}

\newcommand{\ii}{\mathbf{i}}

\newcommand{\im}{\mathrm{im}}

\newcommand{\SU}{\mathrm{SU}}
\newcommand{\SO}{\mathrm{SO}}
\newcommand{\IV}{\mathbb{V}} % Klein four group

\newcommand{\Et}{\mathrm{Et}}
\newcommand{\StackEt}{\widetilde{\mathrm{Et}}}

\newcommand{\shff}{\#} %sheafification/stackification
\newcommand{\Kar}{\mathrm{Kar}}
\newcommand{\swap}{\mathrm{swap}}

\newcommand{\Gr}{\mathrm{Gr}}         % graph
\renewcommand{\restriction}{\mathord{\upharpoonright}}
\newcommand{\restr}{\restriction}
\newcommand{\vdim}{\operatorname{vdim}}
\newcommand{\Bilin}{\mathrm{Bilin}}
\newcommand{\Quad}{\mathrm{Quad}}
\newcommand{\Mat}{\mathrm{Mat}}

\title{Manifold Tensor Categories}
\author{Christoph Weis}

\begin{document}

\maketitle

\begin{abstract}
    We introduce Manifold tensor categories,
    which make precise the notion of a 
    \emph{tensor category with a manifold of simple objects}.
    A basic example is the category of vector spaces graded 
    by a Lie group.
    Unlike classic tensor category theory, our setup 
    keeps track of the smooth (and topological) structure of 
    the manifold of simple objects.
    We set down the necessary definitions for Manifold tensor 
    categories and a generalisation we term 
    Orbifold tensor categories.
    We also construct a number of examples, most notably 
    two families of examples we call 
    Interpolated Tambara-Yamagami categories and 
    Interpolated Quantum Group categories.
    Finally, we show conditions under which pointwise duality 
    data in an Orbifold tensor category automatically assembles 
    into smooth duality data. Our proof uses the classic
    Implicit function theorem from differential geometry.
    This document is an edited excerpt of the author's thesis.
\end{abstract}

\tableofcontents

\section*{Introduction}
\label{sec:IntroManifoldOrbifoldTensorCats}
\addcontentsline{toc}{section}{\nameref{sec:IntroManifoldOrbifoldTensorCats}}
A fusion category (we work over $\IC$ throughout) is a linear autonomous monoidal 
semisimple category with finitely many simple objects 
(see Section~\ref{sec:tensorCategories}).
Up to equivalence, the data of a fusion category consists of a finite set
of simple objects $\{X_i\}$ together with the data of a monoidal structure,
often encoded as a pair of fusion relations
\begin{equation*}
    X_i \cdot X_j = \sum_k N_{ij}^k X_k
\end{equation*}
and associators
\begin{equation*}
    \alpha_{i,j,k}: (X_i \tensor X_j) \tensor X_k \isoto X_i \tensor (X_j \tensor X_k).
\end{equation*}
The theory of fusion categories has links to
Conformal Field Theory~\cite{bartels2009conformal, henriques2017bicommutant},
Topological Quantum Field Theory~\cite{kirillov2010turaev, douglas2020dualizable, henriques2017chern},
and subfactor theory of von Neumann algebras~\cite{fuchs2008fusion, calegari2011cyclotomic, grossman2012quantum, jones2014classification}.
A detailed account of the theory is given in~\cite{etingof2016tensor}.

One may build a tensor category
$\Vec^\omega[G]$ associated to any group $G$ and
group cocycle $\omega \in \mathrm{Z}^3(\B G, \IC^\times)$:
The category has simple objects $\{\IC_g\}$, and monoidal structure linearly extended from
\begin{gather*}
    \IC_g \tensor \IC_h = \IC_{gh}                                 \\
    \alpha_{g,h,k}      = \omega(g,h,k) \in \End(\IC_{ghk}) = \IC.
\end{gather*}
The pentagon equation is satisfied precisely because $\omega$ is a cocycle.
Every object has left and right duals, in particular ${(\IC_g)}^\vee = \IC_{g^{-1}}$.
We refer to $\Vec^\omega[G]$ as the categorified $\omega$-twisted group ring of $G$.
If $G$ is finite, $\Vec^\omega[G]$ is a fusion category.

If $G$ is a Lie group, $\Vec^\omega[G]$ inherits the smoothness of $G$.
We encode this smooth structure
by considering $\Vec^\omega[G]$ as a stack over the site $\Man$ of 
manifolds.
These categories have a genuine manifold of simple objects: the Lie group $G$.
For $G=T$ a torus, such tensor categories were already considered
in~\cite{freed2010topological}, where they played a central role in the
study of toral Chern-Simons theory.

The tensor categories introduced above are all group-like: every simple object
$X \in \cat{C}$ admits an inverse $Y$ such that $X \tensor Y \iso Y \tensor X \iso \ONE$,
the unit object of the category.
Many interesting fusion categories are not of this kind, for example the
representation categories $\Rep(H)$ of finite (non-abelian) groups,
the quantum group categories $\cat{C}(\mathfrak{g},k)$
(Example~\ref{ex:RepkLG})
and the Tambara-Yamagami categories $\TY(A,\chi,\tau)$
(Example~\ref{ex:TambaraYamagamiFusionCategory}).

Manifold tensor categories are a generalisation of fusion categories that
make precise the notion of a
\emph{tensor category with a manifold of simple objects}.
They should not be viewed as exotic objects. Indeed, already the finite-dimensional
(complex) representations of the group $\IZ$ form such a category:
its simple objects assemble into the manifold $\IC^\times$. The unitary
representations form a tensor subcategory with manifold of simple objects $S^1$.
Manifold tensor categories also appear in the study of CFTs at 
central charge $c=1$:
It is shown in~\cite{thorngren2021fusion,chang2021lorentzian} 
that the quantum symmetries of such CFTs assemble into topological spaces 
(the authors 
work without the smooth structure present on these spaces).

We further allow ourselves to talk about \emph{orbifolds of simple objects}.
Natural examples
may again be found in representation theory. Let $\Gamma$ be a virtually
abelian group, ie.\ a group fitting into a short exact sequence
\[\IZ^n \to \Gamma \to H\]
where $H$ is a finite group. The category of (unitary)
representations of $\Gamma$ is an orbifold tensor category.
The simple objects of $\Rep^u(\Gamma)$ form the moduli space of points
of the orbifold $[T^n/H]$, where $T^n=\widehat{\IZ^n}$
is the $n$-torus of representations of $\IZ^n$. The $H$-action is
induced by the conjugation action of $H$ on $\Gamma$.
See Figure~\ref{fig:introExamplesGroupReps} for a depiction
of the moduli space of simple objects in a few examples.
%See below for a depiction of the moduli
%spaces of simple objects for the categories
%$\Rep^u(\IZ)$, $\Rep^u(\Sigma_2 \ltimes \IZ)$ and
%$\Rep^u(\Sigma_3 \ltimes \IZ^2)$, respectively.
%\begin{center}
    %\begin{tikzpicture}[scale=1.5]
        %\begin{scope}[shift={(0,0)}]
            %\draw[ultra thick] (0,0) circle (1);
        %\end{scope}
        %\begin{scope}[shift={(2.7,0)}]
            %\draw[ultra thick] (0,-1) arc (-90:90:1);
            %\filldraw[black] [yshift=1.5pt](-0.1,1) circle (1pt);
            %\filldraw[black] [yshift=-1.5pt](-0.1,1) circle (1pt);
            %\filldraw[black] [yshift=1.5pt](-0.1,-1) circle (1pt);
            %\filldraw[black] [yshift=-1.5pt](-0.1,-1) circle (1pt);
        %\end{scope}
        %\begin{scope}[shift={(5,-0.2)}]
            %\coordinate (A) at (0,0);
            %\coordinate (B) at (1.7,1);
            %\coordinate (C) at (2.5,-0.5);
            %\filldraw[black] ([yshift=2.5pt]A) circle (1pt);
            %\filldraw[black] (A) circle (1pt);
            %\filldraw[black] ([yshift=-3pt]A) circle (1.5pt);
            %\filldraw[black] ([yshift=2.5pt]B.north) circle (1pt);
            %\filldraw[black] (B.north) circle (1pt);
            %\filldraw[black] ([yshift=-3pt]B.north) circle (1.5pt);
            %\filldraw[black] ([yshift=2.5pt]C) circle (1pt);
            %\filldraw[black] (C) circle (1pt);
            %\filldraw[black] ([yshift=-3pt]C) circle (1.5pt);
            %\draw[double,double distance=1pt,shorten <=5pt,shorten >=5pt] (A.south east) -- (B.west);
            %\draw[double,double distance=1pt,shorten <=5pt,shorten >=5pt] (A.north east) -- (C.west);
            %\draw[double,double distance=1pt,shorten <=5pt,shorten >=5pt] (B.south) -- (C.north);
            %\fill[opacity=0.3] ([shift={(6pt,1pt)}]A.east) -- ([shift={(-1pt,-3pt)}]B.south) -- ([shift={(-4pt,3pt)}]C.south east) --cycle;
        %\end{scope}
    %\end{tikzpicture}
%\end{center}
\begin{figure}[htp]
    \centering
    \begin{subfigure}[t]{0.3\textwidth}
        \centering
        \begin{minipage}[t]{0.85\textwidth}
    \def\svgwidth{\columnwidth}
    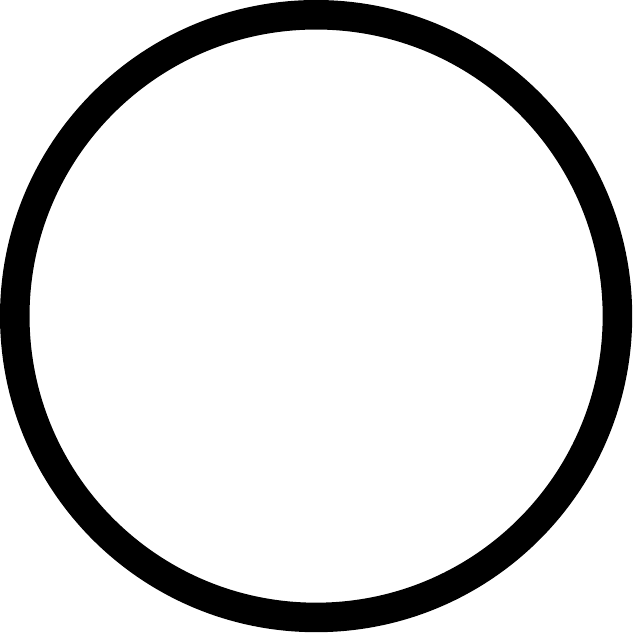

        \end{minipage}
        \caption{$\Rep^u(\IZ)$}
    \end{subfigure}
    \begin{subfigure}[t]{0.3\textwidth}
        \centering
        \begin{minipage}[t]{0.45\textwidth}
    \def\svgwidth{\columnwidth}
    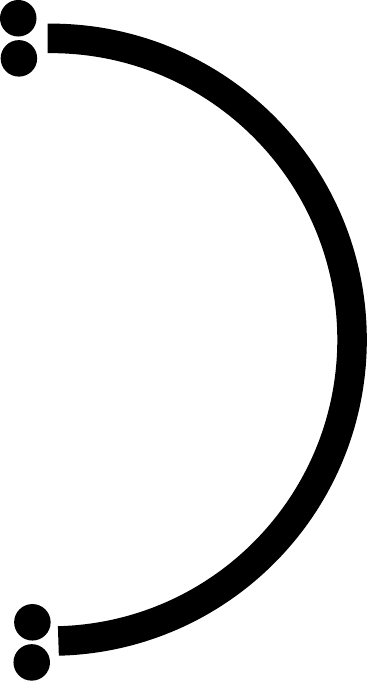

        \end{minipage}
        \caption{$\Rep^u(\Sigma_2 \ltimes\IZ)$}
        \label{fig:RepDinf}
    \end{subfigure}
    \begin{subfigure}[t]{0.3\textwidth}
        \centering
        \begin{minipage}[t]{1\textwidth}
    \def\svgwidth{\columnwidth}
    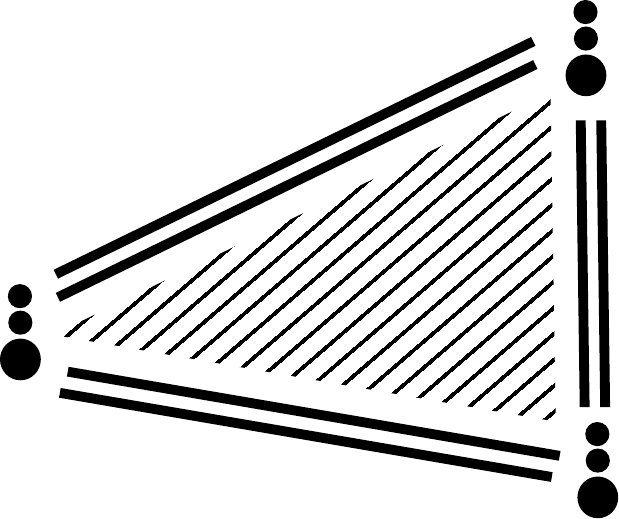

        \end{minipage}
        \caption{$\Rep^u(\Sigma_3 \ltimes \IZ^2)$}
        \label{fig:RepSig3Z2}
    \end{subfigure}
    \caption{Orbifold tensor categories arising in the representation theory of
    virtually abelian groups.}
    \label{fig:introExamplesGroupReps}
\end{figure}

The space of simple objects in Figures~\ref{fig:RepDinf} and~\ref{fig:RepSig3Z2}
%simple objects of $\Rep^u(\Sigma_2 \ltimes \IZ)$
%and $\Rep^u(\Sigma_3 \ltimes \IZ^2)$ 
does not itself form an orbifold.
In fact, the moduli space of simple objects
in these two examples is not even Hausdorff.
It is the 
\emph{moduli space of simple skyscraper sheaves} on 
the orbifolds $[S^1/\Sigma_2]$ and
$[(S^1 \times S^1)/\Sigma_3]$, respectively.
Orbifold tensor categories (Definition~\ref{def:orbifoldTensorCategory})
are defined such that their simple objects form this moduli space.
This is what we mean by a tensor category with an orbifold of simple
objects.

Any fusion category is an example of a manifold tensor category
whose manifold of simple objects is a finite set
(Section~\ref{sec:fusionCatsAreManifusion}).
The product of two orbifold tensor categories is again an orbifold
tensor category, and so a further source of examples are simple products
of the form $\cat{C} \times \Vec^\omega[G]$, where $\cat{C}$ is
a fusion category. These product examples are
hardly interesting by themselves, but
we leverage them to build further orbifold tensor categories in
Sections~\ref{sec:interpolatedQGroups} and \ref{sec:interpolatedTambaraYamagami}.
Recall that there is a Tambara-Yamagami category $\TY(H,\chi,\tau)$ associated
to a finite abelian group $H$, a symmetric non-degenerate
bicharacter $\chi:H\tensor H \to \IC^\times$, and
a square root $\tau$ of $1/|H|$ in $\IC^\times$
(see Example~\ref{ex:TambaraYamagamiFusionCategory}).
It has set of simple objects $H \coprod \{m\}$,
and fusion rules ($a,b \in H$):
\begin{align*}
    a \tensor b \iso ab                 &  &
    a \tensor m \iso m \iso m \tensor a &  &
    m \tensor m \iso \DirSum_{a \in H}a.
\end{align*}
The expression for the associator involves $\chi$ and $\tau$.
We build families of
\emph{interpolated Tambara-Yamagami categories} whose manifolds
of simples are of the form $G \coprod G/H$, for a Lie group $G$ and
central subgroup $H \subset Z(G)$.
These categories contain a Tambara-Yamagami category
$\TY(H,\chi,\tau)$ as a sub-manifold tensor category (for suitable
$\chi$ and $\tau$).
The fusion rules are direct
extensions of the above ($g,g^\prime \in G$):
\begin{align*}
    g \tensor g^\prime \iso gg^\prime                              &  &
    g \tensor g^\prime H \iso gg^\prime H \iso gH \tensor g^\prime &  &
    gH \tensor g^\prime H \iso \DirSum_{h \in H} gg^\prime h.
\end{align*}
We show an example of the resulting manifold of simple objects in
Figure~\ref{fig:IntroTYcirclesZ3}.
\begin{figure}[htp]
    \centering
    \begin{minipage}[t]{0.7\textwidth}
    \def\svgwidth{\columnwidth}
    %% Creator: Inkscape 1.1.1 (3bf5ae0d25, 2021-09-20), www.inkscape.org
%% Accompanies image file 'TYcirclesZ3.pdf' (pdf, eps, ps)
\begingroup%
  \makeatletter%
  \providecommand\color[2][]{%
    \errmessage{(Inkscape) Color is used for the text in Inkscape, but the package 'color.sty' is not loaded}%
    \renewcommand\color[2][]{}%
  }%
  \providecommand\transparent[1]{%
    \errmessage{(Inkscape) Transparency is used (non-zero) for the text in Inkscape, but the package 'transparent.sty' is not loaded}%
    \renewcommand\transparent[1]{}%
  }%
  \providecommand\rotatebox[2]{#2}%
  \newcommand*\fsize{\dimexpr\f@size pt\relax}%
  \newcommand*\lineheight[1]{\fontsize{\fsize}{#1\fsize}\selectfont}%
  \ifx\svgwidth\undefined%
    \setlength{\unitlength}{365.05078951bp}%
    \ifx\svgscale\undefined%
      \relax%
    \else%
      \setlength{\unitlength}{\unitlength * \real{\svgscale}}%
    \fi%
  \else%
    \setlength{\unitlength}{\svgwidth}%
  \fi%
  \global\let\svgwidth\undefined%
  \global\let\svgscale\undefined%
  \makeatother%
  \begin{picture}(1,0.57239566)%
    \lineheight{1}%
    \setlength\tabcolsep{0pt}%
    \put(0,0){\includegraphics[width=\unitlength,page=1]{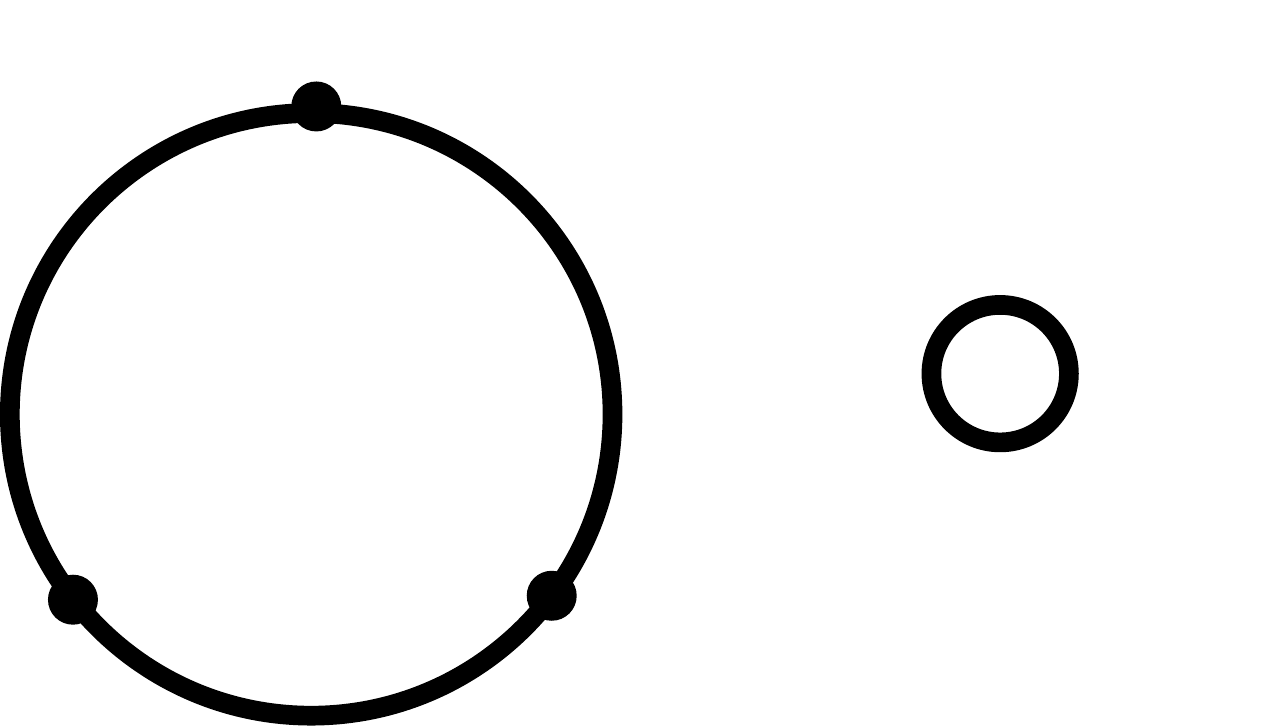}}%
    \put(0.2733972,0.50915552){\color[rgb]{0,0,0}\makebox(0,0)[lt]{\lineheight{1.25}\smash{\begin{tabular}[t]{l}$1$\end{tabular}}}}%
    \put(0.08350114,0.11172278){\color[rgb]{0,0,0}\makebox(0,0)[lt]{\lineheight{1.25}\smash{\begin{tabular}[t]{l}$\omega$\end{tabular}}}}%
    \put(0.46506703,0.11291759){\color[rgb]{0,0,0}\makebox(0,0)[lt]{\lineheight{1.25}\smash{\begin{tabular}[t]{l}$\omega^2$\end{tabular}}}}%
    \put(0.81967173,0.36218652){\color[rgb]{0,0,0}\makebox(0,0)[lt]{\lineheight{1.25}\smash{\begin{tabular}[t]{l}m\end{tabular}}}}%
    \put(0,0){\includegraphics[width=\unitlength,page=2]{TYcirclesZ3.pdf}}%
  \end{picture}%
\endgroup%

    \end{minipage}
    \caption{Interpolated Tambara-Yamagami category for the subgroup
        $\IZ/3 \subset S^1$. The simples of $\TY(\IZ/3,\chi,\tau)$ are
        marked. The size of the right circle indicates that multiplication by
        elements in the left circle is given by the $\times 3$ map. All
        objects on the right circle are non-invertible.}
    \label{fig:IntroTYcirclesZ3}
\end{figure}

In a similar vein, we combine
the quantum group categories $\cat{C}(\mathfrak{g},k)$
with smooth categorified group rings $\Vec^\omega[G]$.\footnote{
    Recall from Example~\ref{ex:RepkLG} that $\cat{C}(\mathfrak{g},k)$
    is equivalent to $\Rep^k LG$ for the simply-connected Lie group
    $G$ integrating $\mathfrak{g}$. 
}
We build examples with underlying manifolds of the form
$G^{\coprod n}$ and $\left(G^{\coprod n}\right) \coprod G/H$,
for compact connected Lie groups $G$ and central subgroups $H \subset Z(G)$.
Figure~\ref{fig:introIntQuantumGps} shows two examples
of interpolated quantum group categories. 
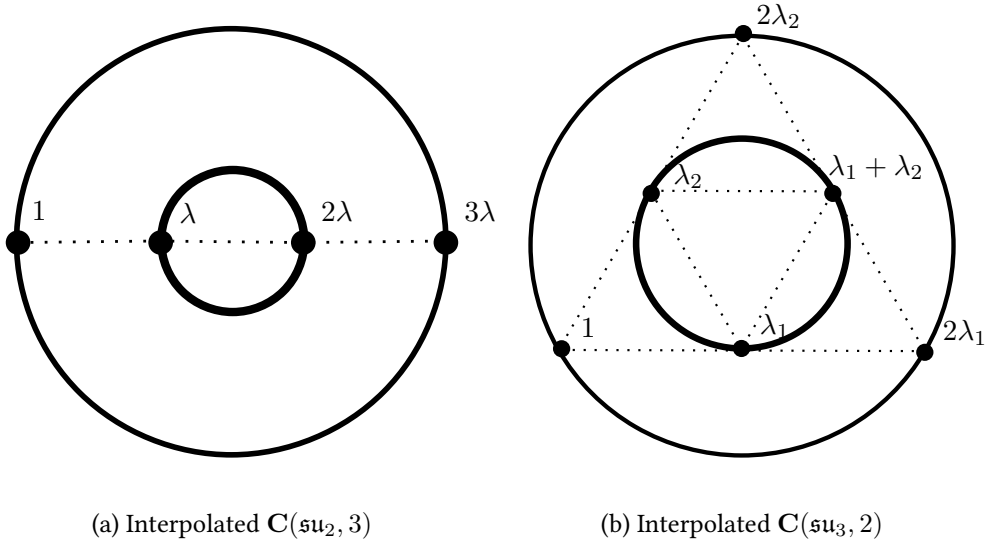
\begin{figure}[htp]
    \centering
    \begin{subfigure}[t]{0.45\textwidth}
        \centering
        \begin{minipage}[t]{0.9\textwidth}
            \centering
    \def\svgwidth{\columnwidth}
    %% Creator: Inkscape inkscape 0.92.5, www.inkscape.org
%% Accompanies image file 'IntQGroupSU2k3.pdf' (pdf, eps, ps)
\begingroup%
  \makeatletter%
  \providecommand\color[2][]{%
    \errmessage{(Inkscape) Color is used for the text in Inkscape, but the package 'color.sty' is not loaded}%
    \renewcommand\color[2][]{}%
  }%
  \providecommand\transparent[1]{%
    \errmessage{(Inkscape) Transparency is used (non-zero) for the text in Inkscape, but the package 'transparent.sty' is not loaded}%
    \renewcommand\transparent[1]{}%
  }%
  \providecommand\rotatebox[2]{#2}%
  \newcommand*\fsize{\dimexpr\f@size pt\relax}%
  \newcommand*\lineheight[1]{\fontsize{\fsize}{#1\fsize}\selectfont}%
  \ifx\svgwidth\undefined%
    \setlength{\unitlength}{240.90113122bp}%
    \ifx\svgscale\undefined%
      \relax%
    \else%
      \setlength{\unitlength}{\unitlength * \real{\svgscale}}%
    \fi%
  \else%
    \setlength{\unitlength}{\svgwidth}%
  \fi%
  \global\let\svgwidth\undefined%
  \global\let\svgscale\undefined%
  \makeatother%
  \begin{picture}(1,0.96143871)%
    \lineheight{1}%
    \setlength\tabcolsep{0pt}%
    \put(0,0){\includegraphics[width=\unitlength,page=1]{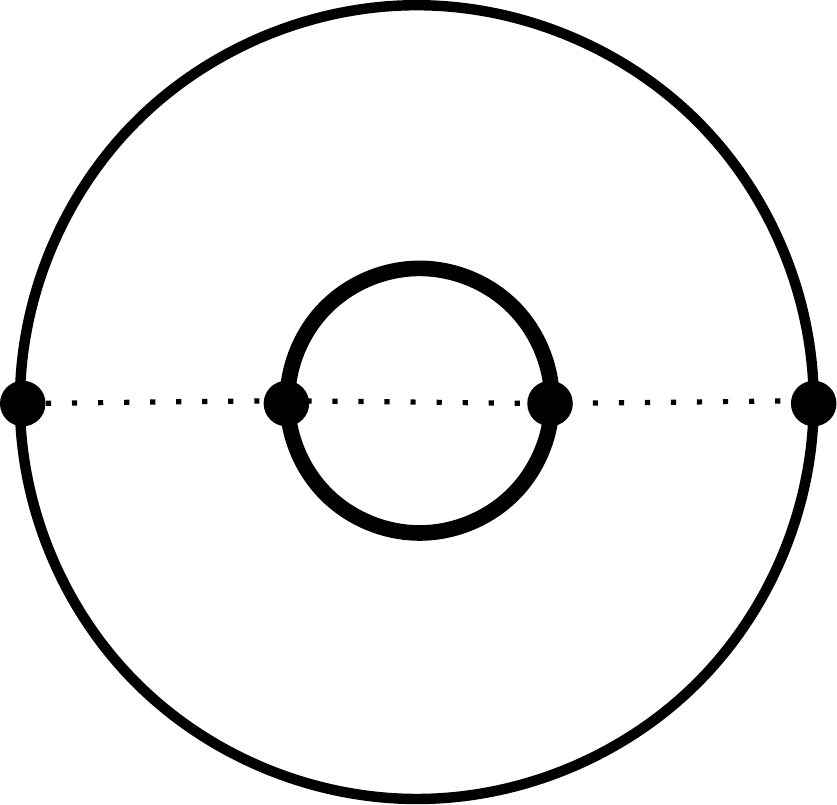}}%
    \put(0.05782315,0.5300391){\color[rgb]{0,0,0}\makebox(0,0)[lt]{\lineheight{1.25}\smash{\begin{tabular}[t]{l}$1$\end{tabular}}}}%
    \put(0.38435379,0.52323631){\color[rgb]{0,0,0}\makebox(0,0)[lt]{\lineheight{1.25}\smash{\begin{tabular}[t]{l}$\lambda$\end{tabular}}}}%
    \put(0.69727889,0.52663775){\color[rgb]{0,0,0}\makebox(0,0)[lt]{\lineheight{1.25}\smash{\begin{tabular}[t]{l}$2\lambda$\end{tabular}}}}%
    \put(1.01360552,0.52663775){\color[rgb]{0,0,0}\makebox(0,0)[lt]{\lineheight{1.25}\smash{\begin{tabular}[t]{l}$3\lambda$\end{tabular}}}}%
  \end{picture}%
\endgroup%

        \end{minipage}
        \caption{Interpolated $\cat{C}(\mathfrak{su}_2,3)$}
    \end{subfigure}
    \begin{subfigure}[t]{0.45\textwidth}
        \centering
        \begin{minipage}[t]{0.85\textwidth}
            \centering
    \def\svgwidth{\columnwidth}
    %% Creator: Inkscape inkscape 0.92.5, www.inkscape.org
%% Accompanies image file 'IntQGroupSU3k2.pdf' (pdf, eps, ps)
\begingroup%
  \makeatletter%
  \providecommand\color[2][]{%
    \errmessage{(Inkscape) Color is used for the text in Inkscape, but the package 'color.sty' is not loaded}%
    \renewcommand\color[2][]{}%
  }%
  \providecommand\transparent[1]{%
    \errmessage{(Inkscape) Transparency is used (non-zero) for the text in Inkscape, but the package 'transparent.sty' is not loaded}%
    \renewcommand\transparent[1]{}%
  }%
  \providecommand\rotatebox[2]{#2}%
  \newcommand*\fsize{\dimexpr\f@size pt\relax}%
  \newcommand*\lineheight[1]{\fontsize{\fsize}{#1\fsize}\selectfont}%
  \ifx\svgwidth\undefined%
    \setlength{\unitlength}{293.06923453bp}%
    \ifx\svgscale\undefined%
      \relax%
    \else%
      \setlength{\unitlength}{\unitlength * \real{\svgscale}}%
    \fi%
  \else%
    \setlength{\unitlength}{\svgwidth}%
  \fi%
  \global\let\svgwidth\undefined%
  \global\let\svgscale\undefined%
  \makeatother%
  \begin{picture}(1,1.02019258)%
    \lineheight{1}%
    \setlength\tabcolsep{0pt}%
    \put(0,0){\includegraphics[width=\unitlength,page=1]{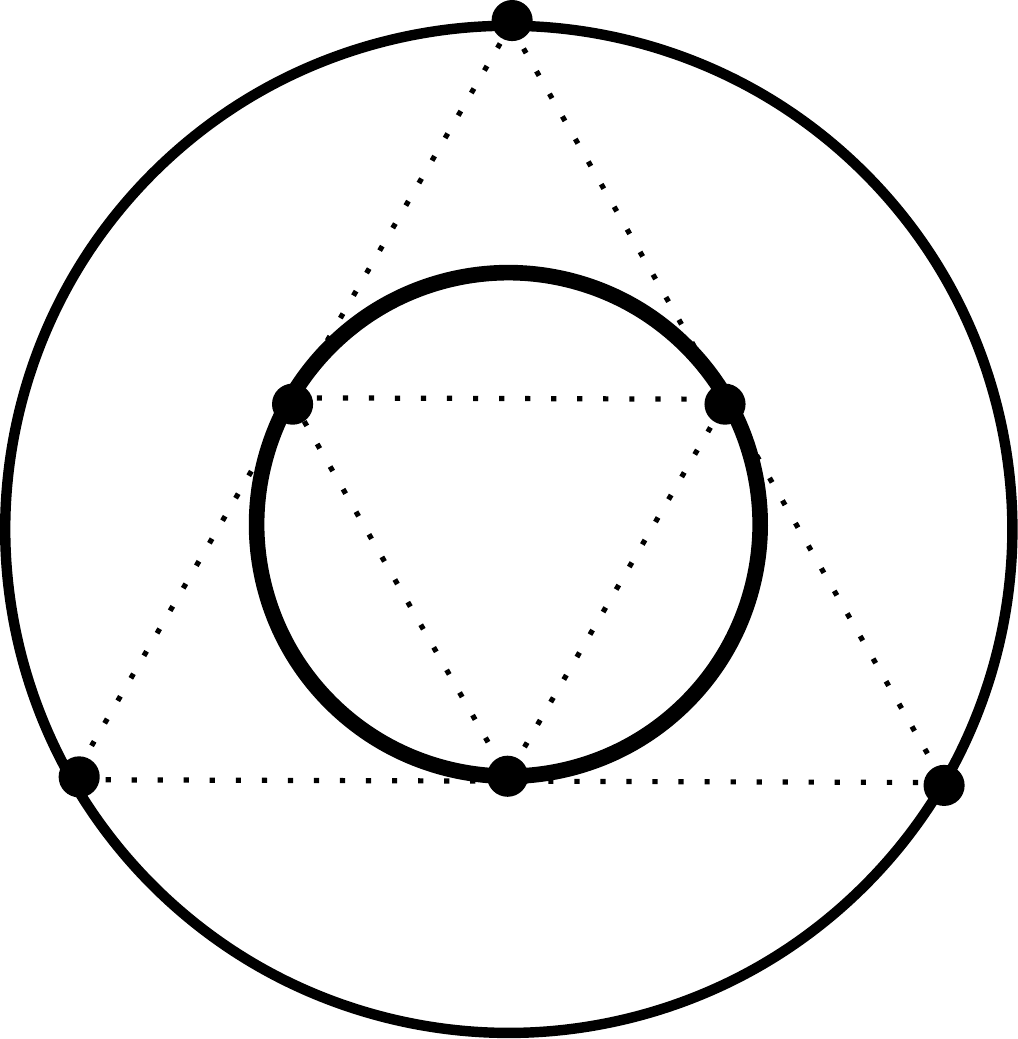}}%
    \put(0.1225538,0.28471591){\color[rgb]{0,0,0}\makebox(0,0)[lt]{\lineheight{1.25}\smash{\begin{tabular}[t]{l}$1$\end{tabular}}}}%
    \put(0.54193839,0.29268201){\color[rgb]{0,0,0}\makebox(0,0)[lt]{\lineheight{1.25}\smash{\begin{tabular}[t]{l}$\lambda_1$\end{tabular}}}}%
    \put(0.96971073,0.27632824){\color[rgb]{0,0,0}\makebox(0,0)[lt]{\lineheight{1.25}\smash{\begin{tabular}[t]{l}$2\lambda_1$\end{tabular}}}}%
    \put(0.33936912,0.64559752){\color[rgb]{0,0,0}\makebox(0,0)[lt]{\lineheight{1.25}\smash{\begin{tabular}[t]{l}$\lambda_2$\end{tabular}}}}%
    \put(0.53355095,1.02283297){\color[rgb]{0,0,0}\makebox(0,0)[lt]{\lineheight{1.25}\smash{\begin{tabular}[t]{l}$2\lambda_2$\end{tabular}}}}%
    \put(0.70188134,0.66691068){\color[rgb]{0,0,0}\makebox(0,0)[lt]{\lineheight{1.25}\smash{\begin{tabular}[t]{l}$\lambda_1 + \lambda_2$\end{tabular}}}}%
  \end{picture}%
\endgroup%

        \end{minipage}
        \caption{Interpolated $\cat{C}(\mathfrak{su}_3,2)$}
    \end{subfigure}
    \caption{Simple objects of interpolated quantum group categories
        associated to subgroups of $S^1$.
        Objects of $\cat{C}(\mathfrak{g},k)$ are
        marked. The dotted lines are a mere visual aid: they emphasise
        the shape of the Weyl alcove of the quantum group category.
        The action of the unit component on the other circle is the usual
        left action of $S^1$ on itself in both categories depicted.}
    \label{fig:introIntQuantumGps}
\end{figure}
An interpolated fusion category $\stC$ built from $\Vec^\omega[G]$ and a 
fusion category $\cat{D}$ comes with a monoidal functor
\[
    \Vec^\omega[G] \boxtimes \cat{D} \to \stC,
\]
and the monoidal structure on $\stC$ may be completely read off from
this functor. 

The construction of the examples above requires a specific choice of
associator $\omega$ on $\Vec^\omega[G]$.
The underlying smooth linear category of $\Vec[G]$ is the moduli stack
of skyscraper sheaves on the manifold $G$.
Not all associators on $\omega \in \H^3(\B G, \IC^\times)$
can be realised as part of a monoidal structure on this stack.
In fact, the transgression homomorphism
\begin{equation*}
    \H^3(\B G, \IC^\times) \to \H^2(G, \IC^\times)
\end{equation*}
sends an associator $\omega$ to the equivalence class of a \emph{gerbe}
$\ger{G}_\omega$ over the manifold $G$.\footnote{
    The transgression map on cohomology arises from the universal $G$-bundle over $\B G$.
    It was lifted to the context of gerbes in~\cite{carey2005bundle}.
}
The associator $\omega$ gives rise to a monoidal structure
on the moduli stack of $\ger{G}_\omega$-twisted skyscraper sheaves
over $G$. We discuss gerbes and twisted sheaves in Section~\ref{sec:GerbesAndTwistedSheaves},
and in more detail in Appendix~\ref{app:twistedSheaves}.

Besides introducing orbifold tensor categories and
constructing examples, we study their general structure.
Most of our results are proved under the assumption that the underlying
orbifold is effective (a notion we recall in Section~\ref{sec:orbifolds}).
We show that if the monoidal unit is supported at a point with trivial
stabiliser, then every point of the orbifold of simple objects has
trivial stabiliser. In other words, manifold tensor categories are 
precisely those effective orbifold tensor categories whose monoidal 
unit has trivial stabiliser (Corollary~\ref{cor:manifoldONEmanifoldALL}).
The main theorem of this note, however, is a categorification
of the following statement in classical Lie theory (recalled as
Lemma~\ref{lem:inverse_automatically_smooth}):
\begin{lem*}
    A smooth monoid all of whose elements are invertible
    is a Lie group: the map $x \mapsto x^{-1}$ assigning inverses is automatically smooth.
\end{lem*}
Making use of the Inverse Function Theorem, we show
that this statement admits the following categorification 
(Theorem~\ref{thm:global_dualisability}):
\begin{thm*}[Smooth Dualisability]
    In an orbifold tensor category with simple monoidal unit,
    pointwise dualisibility implies dualisibility of all smooth families of objects.
\end{thm*}

We begin by introducing twisted sheaves on orbifolds and smooth tensor stacks
in Section~\ref{sec:prelim}. In particular, we show that a gerbe over an orbifold
point with stabiliser $\Gamma$ is essentially the data of a 2-cocycle $\theta \in Z(\B \Gamma,\IC^\times)$,
and that twisted skyscraper sheaves over this orbifold point correspond to $\theta$-twisted
representations of $\Gamma$.
In Section~\ref{sec:orbisimples}, we introduce the notion of a linear category
with an orbifold of simple objects (\emph{orbisimple categories}).
We study the structure of such categories, and show that they
may be thought of as linearisations of gerbes.
We give the definition of orbifold tensor categories in
Section~\ref{sec:orbifoldTensorCategories}. The remainder of that section
is devoted to the construction of explicit examples, in particular the families
of interpolated fusion categories. To this end, we
study central group actions on orbifold tensor categories.
Section~\ref{sec:orbifusion} is concerned with structural results on
orbifold tensor categories: we study the interplay between the smooth
and the monoidal structure.
In Section~\ref{sec:smooth_rig}, we prove the Smooth Dualisability Theorem.
For readers unfamiliar with the notion of (pseudo)monoids internal to
monoidal bicategories, or the theory of sheaves and stacks,
we provide a quick recall and references in Appendix~\ref{ch:background}.
In Appendix~\ref{app:twistedSheaves}, we prove a number of technical
results on twisted sheaves over orbifolds.\footnote{In fact,
we work in the more general setting of \'etale stacks.}
In particular, we establish the classic pullback-pushforward 
adjunction in this context.

\subsection*{Acknowledgements}
First of all, I thank my supervisor Andr\'e Henriques.
Andr\'e, without your support, guidance and intellectual generosity
towards me, this document would 
not exist. In particular, thank you for encouraging me
to study manifold tensor categories.
This note has further benefitted from many, many 
conversations with other mathematicians.
I particularly thank 
Luciana Basualdo Bonatto,
Lukas Brantner,
Arun Debray,
Thibault D\'ecoppet,
Christoph Dorn,
Chris Douglas, 
Dan Freed, 
Nora Ganter,
Dominic Joyce,
Kiran Luecke,
David Reutter,
Jan Steinebrunner,
Ryan Thorngren, 
Konrad Waldorf,
Thomas Wasserman
and
Tom Zielinski
for useful discussions and feedback on earlier drafts of
this document.

\section{Twisted Sheaves on Orbifolds}
\label{sec:prelim}
We will work over $\IC$ throughout.
The category of $\IC$-vector spaces will be denoted by $\VecInf$,
the subcategory of finite-dimensional vector spaces by $\Vec \subset \VecInf$.
We write $\VecInf$-$\Cat$ to denote the bicategory of
$\VecInf$-enriched categories, and $\VCat \subset \VecInf$-$\Cat$ to denote the 
full subbicategory on
those categories that are complete under direct sums.

A reminder on sheaves and stacks is given in Section~\ref{sec:sheavesAndStacks}.
We denote by $\Man$ the site of smooth ($\IR$-)manifolds, equipped with the
surjective submersion Grothendieck topology, and the structure ring
$\Cinf$ of smooth $\IC$-valued functions.
Given a cover $Y \onto M$ in $\Man$, we denote the associated iterated fibre products by
$Y^{[n]} \define Y \times_M Y \times_M \cdots Y$.
We use $\ShC(M)$ to denote the category of $\Cinf$-modules over a manifold $M$,
and freely refer to such $\Cinf$-modules simply as sheaves.
A particular class of $\Cinf$-modules is provided by vector bundles
$E \to M$, which we identify with their sheaves of sections.

\subsection{Twisted Representations}
\label{sec:projectiveReps}
Recall that a finite-dimensional $\IC$-linear representation of a discrete group $\Gamma$
is a vector space $V \in \Vec$, equipped with a homomorphism
$\rho:\Gamma \to \End(V)$.
This may be rephrased by saying a $\IC$-linear representation
is a morphism $\underline{\Gamma} \to \Vec$,
where $\underline{\Gamma}$ denotes the category with a single object whose
endomorphisms form the group $\Gamma$ and $\Vec$ denotes the
category of $\IC$-modules.
The endomorphisms assigned to group elements $g,h,gh \in \Gamma$
satisfy the condition
\begin{equation*}
    \rho(g) \rho(h) = \rho(gh).
\end{equation*}
In a \emph{projective representation},
this equation only holds up to a scalar:
\begin{equation*}
    \rho(g) \rho(h) = \theta(g,h)\rho(gh).
\end{equation*}
In the above, $\theta:\Gamma \times \Gamma \to \IC^\times$ is a
map of sets,
and $\IC^\times$ acts via the usual $\IC$-module structure
on $V \in \Vec$.
Requiring associativity of $\rho$ yields
\begin{equation*}
    \theta(g,h)\theta(gh,k) = \theta(h,k)\theta(g,hk)
\end{equation*}
for every triple $g,h,k \in \Gamma$.
This is the equation of a 2-cocycle $\theta \in Z^2(\B \Gamma, \IC^\times)$.
We call a representation of the above form a
\emph{$\theta$-twisted representation} of $\Gamma$.
Twisted representations for a specific 2-cocycle assemble into a category
$\Rep^\theta(\Gamma)$.

\begin{example}
    \label{ex:twistedKleinFourRep}
    Denote two chosen generators of the Klein four group
    $\IV = \IZ/2 \times \IZ/2$ by $x$ and $p$.
    We pick a 2-cocycle $\theta:\IV \times \IV \to \IC^\times$
    representing the generator of $\H^2(\B\IV,\IC^\times) \simeq \IZ/2$,
    setting $\theta(a,b)=1$ except
    \begin{equation*}
        \theta(p,x)=\theta(xp,xp)=\theta(p,xp)=\theta(xp,x)=-1.
    \end{equation*}

    A 2-dimensional $\theta$-twisted representation $\pi$
    of $\IV$ is furnished by the Pauli matrices
    \begin{align*}
        x \mapsto \sigma_x =
        \begin{pmatrix}
            0 & 1 \\ 1 & 0
        \end{pmatrix}
         &  &
        p \mapsto \sigma_z =
        \begin{pmatrix}
            1 & 0 \\ 0 & -1
        \end{pmatrix}
    \end{align*}
    which satisfy the necessary and sufficient relations
    $\sigma_x^2=\sigma_z^2=\Id$ and $\sigma_x \sigma_z = -\sigma_z \sigma_x$.
\end{example}

The tensor product of a $\theta$-twisted representation with a
$\theta'$-twisted representation is a $\theta\cdot\theta'$-twisted
representation, where $\cdot$ denotes the group operation
in $Z^2(\Gamma,\IC^\times)$.
A trivialisation $\theta'/\theta = \d \kappa$ is equivalently a
1-dimensional $\theta/\theta'$-projective representation
\begin{equation*}
    \kappa(x)\kappa(y) = \frac{\theta(x,y)}{\theta'(x,y)}\kappa(xy).
\end{equation*}
Tensoring with $\kappa$ constitutes an explicit equivalence
\begin{equation*}
    \Rep^{\theta'}(\Gamma) \isoto \Rep^{\theta}(\Gamma).
\end{equation*}
In particular, a trivialisation $\theta = \d \kappa$ gives rise to an
equivalence $\Rep(\Gamma) \isoto \Rep^\theta(\Gamma)$
between $\theta$-twisted representations and untwisted representations.
The discussion above has the following corollary.
\begin{cor}
    \label{cor:no1DrepsforNontrivialTheta}
    The category $\Rep^\theta(\Gamma)$ contains a
    1-dimensional representation if and only if $\theta$
    is trivialisable.
\end{cor}

\begin{rmk}
A 2-cocycle $\theta$ also defines a central extension
\footnote{
   We made the assumption that $\Gamma$ is discrete.
   As a result all $\IC^\times$-extensions
   of $\Gamma$ are split as topological spaces,
   and we may avoid dealing with Segal-Mitchison cocycles
   (see eg.~\cite{schommer2011central} or~\cite{weis2022centre}
   for a review).
}
\begin{equation*}
  \IC^\times \to \widetilde{\Gamma} \to \Gamma,
\end{equation*}
where $\widetilde{\Gamma}$ is the set
$\IC^\times \times \Gamma$ equipped with a multiplication
\begin{equation*}
   (z,\gamma) \cdot_\theta (z^\prime,\gamma^\prime) =
   (z\cdot z^\prime \cdot \theta(\gamma,\gamma^\prime),
   \gamma \cdot \gamma^\prime).
\end{equation*}
Note that the extension built from $\theta$ comes with a distinguished
section of sets $\Gamma \to \widetilde{\Gamma}$ sending
$\gamma \mapsto (1,\gamma)$.
The central extension is classified up to equivalence by the
cohomology class $[\theta] \in \H^2(\B \Gamma, \IC^\times)$.
Every 1-chain $\kappa$ such that $\d \kappa=\theta/\theta'$
defines an explicit isomorphism of the corresponding extensions.
The $\theta$-twisted representations of $\Gamma$ in fact
correspond to representations of $\widetilde{\Gamma}$ where 
the central $\IC^\times \into \widetilde{\Gamma}$ 
acts by the identity character.
We will not need this correspondence.
\end{rmk}

To discuss the restriction-induction adjunction, it is useful to
think about twisted representations using the language of rings and modules.
\begin{defn}
    The \emph{$\theta$-twisted group ring} of $\Gamma$ is
    the (non-commutative) algebra $\IC^\theta[\Gamma]$ with underlying
    vector space the $\IC$-span of the set $\Gamma$, and multiplication
    defined by
    \begin{equation*}
        \underline{\gamma_1} \cdot \underline{\gamma_2} =
        \theta(\gamma_1,\gamma_2) \underline{\gamma_1 \gamma_2}.
    \end{equation*}
\end{defn}

The $\theta$-twisted representations of $\Gamma$ defined above are
precisely the modules over this group ring
\begin{equation*}
    \Rep^\theta(\Gamma) = \Mod_{\IC^\theta[\Gamma]}.
\end{equation*}

The $\theta$-twisted group ring plays the role of the
\emph{regular representation} in ordinary representation theory.
Using twisted character theory (as developed
in~\cite{cheng2015character}),
one may prove that the $\theta$-twisted regular representation
decomposes in the usual way.
Denote by $\Irr^\theta(G)$ a set of representatives for
the $\theta$-twisted irreducible representations of $G$.
\begin{prop}[{\cite[Prop 2.3]{cheng2015character}}]
    \label{prop:decompositionTwistedGroupRing}
    The twisted group ring decomposes as a direct sum of irreducible representations
    and contains each of them with multiplicity their dimension:
    \begin{equation*}
        \IC^\theta[\Gamma] \iso
        \DirSum_{V \in \Irr^\theta(G)} {V^{\dirSum \dim V}}.
    \end{equation*}
\end{prop}

A homomorphism of groups $f:H \to \Gamma$ induces a
pullback 2-cocycle $f^\ast\theta$ for $H$, and
a homomorphism
\begin{equation*}
    f:\IC^{f^\ast\theta}[H] \to \IC^\theta[\Gamma]
\end{equation*}
(we abuse notation to denote it by the same letter).
%In particular, $\widetilde{\Gamma}$ is constructed such that $\theta$
%pulls back to the trivial cocycle 
%$\widetilde{\theta}:(\widetilde{\gamma_1},\widetilde{\gamma_2})\mapsto 1$.

Recall the \emph{restriction} functor $f^\ast:\Mod_{S}\to \Mod_{R}$ associated to a
map of rings $f:R \to S$ is given by
$f^\ast = \Hom_S({}_S S_R, - ) = {}_R S_S \tensor_S -$,
and thus has both a left and a right adjoint, supplied by the
$\tensor$-$\Hom$-adjunction:
\begin{equation*}
    (f_! \dashv f^\ast \dashv f_\ast):
    \Mod_{R}
    \ \substack{\longrightarrow\\\longleftarrow\\\longrightarrow} \
    \Mod_{S}.
\end{equation*}
The left adjoint $f_! = {}_S S_R \tensor -$ is called
\emph{extension of scalars}. When $S$ and $R$ are twisted group
rings, this operation corresponds to \emph{induction} of 
twisted representations.
The right adjoint $f_\ast = \Hom_R({}_R S_S, -)$ is called
\emph{coextension of scalars} and specialises to \emph{coinduction} for
group rings.

\begin{prop}
    \label{prop:inductionIsCoinduction}
    If $f:H \into \Gamma$ is an inclusion of \emph{finite groups},
    induction is equivalent to coinduction:
    \begin{equation*}
        f_! \simeq f_\ast: \Mod_{\IC^{f^\ast\theta}[H]} \to
        \Mod_{\IC^{\theta}[\Gamma]}.
    \end{equation*}
\end{prop}
The proof below is adapted to the $\theta$-twisted case
from~\cite{MSE2012NakayamaIso,MSE2020NakayamaIso2}.
\begin{proof}
    Let $V$ be a left $\IC^{f^\ast\theta}[H]$-module.
    The isomorphism between the induction and coinduction functor is furnished
    by the twisted Nakayama isomorphism
    \begin{align*}
        \mathrm{Nak}:\Hom_{\IC^{f^\ast\theta}[H]}(\IC^\theta[\Gamma],V) & \to
        \IC^{\theta}[\Gamma] \tensor_{\IC^{f^\ast\theta}[H]} V                                                  \\
        \phi                                                          & \mapsto \frac{1}{|H|} \sum_{x \in \Gamma}
        \frac{1}{\theta(x^{-1},x)} \underline{x^{-1}} \tensor \phi(\underline{x}),
    \end{align*}
    which is manifestly natural in $V$.
    Its inverse is given by
    \begin{align*}
        \mathrm{Nak}^{-1}:\IC^{\theta}[\Gamma] \tensor_{\IC^{f^\ast\theta}[H]} V & \to
        \Hom_{\IC^{f^\ast\theta}[H]}(\IC^\theta[\Gamma],V)                                 \\
        \underline{x} \tensor v                                                    & \mapsto
        \theta(x^{-1},x) \sum_{h \in H} \frac{1}{\theta(h,x^{-1})}\underline{hx^{-1}}^\ast \tensor h.v,
    \end{align*}
    where $\underline{x}^\ast \in \Hom_{\IC}(\IC^\theta[G], \IC)$ denotes the
    function that sends $\underline{x} \mapsto 1$ and all other generators to 0.

    It is tedious but straightforward to check that these
    morphisms are indeed well-defined and compatible with the
    $\IC^{\theta}[\Gamma]$-module structure.
    We show why the map $\mathrm{Nak}^{-1}$ is well-defined as an example.
    In $\IC^{\theta}[\Gamma] \tensor_{\IC^{f^\ast\theta}[H]} V$,
    \begin{equation*}
        \underline{x} \tensor v = \underline{xh^{-1}h}\tensor v =
        \frac{1}{\theta(xh^{-1},h)} \underline{xh^{-1}} \tensor h.v
    \end{equation*}
    for all $h \in H$. Under $\mathrm{Nak}^{-1}$, this is sent to
    \begin{align*}
        \frac{1}{\theta(xh^{-1},h)} \underline{xh^{-1}} \tensor h.v & \mapsto
        \frac{\theta(hx^{-1},xh^{-1})}{\theta(xh^{-1},h)} \sum_{h^\prime \in H}
        \frac{1}{\theta(h^\prime,hx^{-1})}\underline{h^\prime hx^{-1}}^\ast \tensor h^\prime.h.v \\ &=
        \sum_{h^\prime \in H} \frac{\theta(hx^{-1},xh^{-1})\theta(h^\prime,h)}{\theta(xh^{-1},h)\theta(h^\prime,hx^{-1})}
        \underline{h^\prime hx^{-1}}^\ast \tensor (h^\prime h).v                                 \\ &=
        \sum_{h'' \in H} \frac{\theta(hx^{-1},xh^{-1})\theta(h''h^{-1},h)}{\theta(xh^{-1},h)\theta(h''h^{-1},hx^{-1})}
        \underline{h''x^{-1}}^\ast \tensor h''.v
    \end{align*}
    The cocycle conditions for the triples $(y,1,z)$ and $(y,y^{-1},y)$ imply that $\theta(y,1)=\theta(1,z)$ and
    $\theta(y,y^{-1})=\theta(y^{-1},y)$ for all $y,z \in \Gamma$.
    Together with these facts, the cocycle conditions for the triples
    $(xh^{-1},h,x^{-1})$ and $(h''h^{-1},h,x^{-1})$ complete the
    above verification:
    \begin{align*}
        \sum_{h'' \in H} \frac{\theta(hx^{-1},xh^{-1})\theta(h''h^{-1},h)}{\theta(xh^{-1},h)\theta(h''h^{-1},hx^{-1})}
        \underline{h''x^{-1}}^\ast \tensor h''.v & =
        \sum_{h'' \in H} \frac{\theta(x^{-1},x)}{\theta(h'',x^{-1})} \underline{h''x^{-1}}^\ast \tensor h''.v
        \\ &= \mathrm{Nak}^{-1}(\underline{x} \tensor v). \qedhere
    \end{align*}
\end{proof}

\begin{defn}[\cite{lauda2006frobenius}]
    An \emph{ambidextrous adjunction} is a
    triple of adjunctions
    \begin{equation*}
        F \dashv G \dashv F': \cat{D}\ \substack{\longrightarrow\\\longleftarrow\\\longrightarrow}\ \cat{C}
    \end{equation*}
    equipped with
    an isomorphism of functors $F \simeq F'$.
\end{defn}
The endofunctor $T=F'G \simeq FG \in \End(\cat{C})$ is
simultaneously a monad because of the adjunction $G \dashv F'$ and a
comonad because $F \dashv G$.
In other words, it is simultaneously an algebra
$(T, \mu, \eta)$ and a coalgebra $(T,\delta,\eps)$ in $\End(\cat{C})$.
The multiplication $\mu$ and comultiplication $\delta$
satisfy the Frobenius axiom
\begin{equation*}
    (\mu \tensor \ONE) \comp (\ONE \tensor \delta) =
    \delta \comp \mu = (\ONE \tensor \mu) \comp (\delta \tensor \ONE)
\end{equation*}
by the interchange law in bicategories (see~\cite{lauda2006frobenius}).
This structure is referred to as a \emph{Frobenius monad}
(first studied in~\cite{street2004frobenius}).

By Proposition~\ref{prop:inductionIsCoinduction}, we obtain an ambidextrous adjunction
\begin{equation*}
    \begin{tikzcd}
        {\Mod_{\IC^{f^\ast\theta}[H]}} &
        {\Mod_{\IC^{\theta}[\Gamma]}}.
        \arrow["f_\ast"{name=0, anchor=north},
        bend right, from=1-1, to=1-2]
        \arrow["f^\ast"{name=1, anchor=south},
        bend right, from=1-2, to=1-1]
        \arrow["{\bot\ \&\ \top}"{anchor=center},
        draw=none, from=0, to=1]
    \end{tikzcd}
\end{equation*}
Denote the associated monad on $\Mod_{\IC^\theta[\Gamma]}$ by
$(T,\mu,\eta,\delta,\eps)$.
For definiteness, we set
\begin{equation*}
    T(V) = %{}_{\IC^\theta[\Gamma]} 
    \IC^\theta[\Gamma] %_{\IC^f^\ast\theta[H]} 
    \tensor_{\IC^{f^\ast\theta}[H]}
    %{}_{\IC^{f^\ast\theta}[H]} 
    \IC^\theta[\Gamma] %_{\IC^\theta[\Gamma]}.
    \tensor_{\IC^\theta[\Gamma]} V.
\end{equation*}
Then the counit $\eps: T \to \Id$ is given by evaluation
\begin{align*}
    \eps: T(V) & \to V                                    \\
    \underline{x} \tensor \underline{x'} \tensor v
               & \mapsto \underline{x}.\underline{x'}.v =
    \theta(x,x') \underline{xx'}.v
\end{align*}
while the unit $\eta: \Id \to T$
uses the Nakayama isomorphism given in the proof of
Proposition~\ref{prop:inductionIsCoinduction}
\begin{align*}
    \eta_V: V & \to T(V) \\
    v         & \mapsto
    \frac{1}{|H|} \sum_{x \in \Gamma}
    \frac{1}{\theta(x^{-1},x)} \underline{x^{-1}} \tensor
    \underline{x} \tensor v.
\end{align*}

The composite $\eps \comp \eta$, an endomorphism of the identity functor,
amounts to multiplication by a scalar:
\begin{align*}
    \eps_V \comp \eta_V: V & \to V   \\
    v                      & \mapsto
    \frac{1}{|H|} \sum_{x \in \Gamma} v = \frac{|\Gamma|}{|H|}v.
\end{align*}

This way, the natural transformation
\begin{equation*}
    e \define \frac{|H|}{|\Gamma|} \eta \comp \eps: T \to T
\end{equation*}
defines an idempotent on $T$, split by
\begin{equation*}
    \begin{tikzcd}
        {T} & \Id.
        \arrow["{\eps}"', shift right, from=1-1, to=1-2]
        \arrow["{\frac{|H|}{|\Gamma|}\eta}"', shift right, from=1-2, to=1-1]
    \end{tikzcd}
\end{equation*}
This identifies the identity functor as a direct summand of $T$,
and in particular $V \in \Mod_{\IC^\theta[\Gamma]}$ as a direct summand of
$T(V)$.
\begin{corollary}
    \label{cor:KaroubiCompletionOfInductionIsAll}
    The category $\Mod_{\IC^\theta[\Gamma]}$ is a
    Karoubi completion of the essential image of
    $\Mod_{\IC^{f^\ast\theta}[H]}$ under the
    induction functor.
    In other words, every $\IC^\theta[\Gamma]$-module
    is a direct summand of a module induced up from
    the subgroup $H$.
\end{corollary}
\begin{proof}
    It suffices to show that every $V \in \Mod_{\IC^\theta[\Gamma]}$
    is a summand of a module of the form $f_\ast W$, where
    $W \in \Mod_{\IC^{f^\ast\theta}[H]}$.
    Setting $W = f^\ast V$, $V$ is a summand of $f_\ast W = T(V)$
    by the discussion above.
\end{proof}

In the special case where $H=\point \into \Gamma$ is the
inclusion of the trivial subgroup, the restriction functor takes
the underlying vector space of a $\theta$-twisted $\Gamma$-representation.
The induction functor sends $V \in \Vec$ to
the representation
\begin{equation*}
    f_\ast V = \IC^\theta[\Gamma] \tensor V.
\end{equation*}

\subsection{Orbifolds as Groupoids}
\label{sec:orbifolds}
Just as a manifold is locally modelled on $\IR^n$,
an orbifold is locally modelled on $\IR^n/\Gamma$, where $\Gamma$
is a finite group acting on $\IR^n$.
As argued in~\cite{moerdijk1997orbifolds,moerdijk2002introduction,lerman2010orbifolds},
orbifolds are best thought of as stacks presented by Lie groupoids.

A \emph{Lie groupoid}
$G_\bullet$ is a pair of manifolds $G_1, G_0$ with two surjective submersions
$s,t:G_1 \rightrightarrows G_0$ and smooth maps
\begin{align*}
    \comp: G_1 \times_{G_0} G_1 \to G_1 &  &
    \id_{-}: G_0 \to G_1                &  &
    (-)^{-1}: G_1 \to G_1,
\end{align*}
giving it the structure of a groupoid internal to $\Man$.
The \emph{stabiliser group} $G_x$ of a point $x \in G_0$ is
the Lie group of arrows $x \to x$.

The Lie groupoid $G_\bullet$ is called \emph{proper} if the map
\begin{equation*}
    G_1 \xrightarrow{(s,t)} G_0 \times G_0
\end{equation*}
is proper.
It is \emph{\'etale} if the maps $s,t$ are local homeomorphisms $G_1 \to G_0$.
In a proper Lie groupoid, all stabiliser groups are compact.
In an \'etale Lie groupoid, they are discrete.
In a proper \'etale Lie groupoid, all stabiliser groups are finite.

\begin{example}
    The Lie groupoid $M \rightrightarrows M$ associated to a manifold $M$
    is a proper \'etale Lie groupoid.
\end{example}

\begin{example}
    The \emph{translation groupoid} $[M/\Gamma]$
    associated to a (right) action $M \curvearrowleft \Gamma$
    is given by
    \begin{equation*}
        \begin{tikzcd}
            M \times \Gamma \arrow[r, "\pi_M", shift right] \arrow[r, "\rho"', shift left] & M.
        \end{tikzcd}
    \end{equation*}
    It is proper iff the group action is proper in the
    traditional sense. It is \'etale iff $\Gamma$ is discrete.
\end{example}

\begin{example}
    The \emph{classifying stack} $[\point/\Gamma]$ is the special case where
    $M$ is a point. It is proper iff $\Gamma$ is compact and \'etale iff
    $\Gamma$ is discrete.
\end{example}

A \emph{smooth functor} of Lie groupoids $G_\bullet \to H_\bullet$
is a pair of smooth maps
$G_0 \to H_0$, $G_1 \to H_1$ compatible with composition and units.
Denote by 
\[
    \Hom(-,G_\bullet): \Man^\opp \to \CAT
\]
the 2-presheaf that sends a manifold $M$
to the category of smooth
functors $(M \rightrightarrows M) \to (G_1 \rightrightarrows G_0)$.
The stack associated to a Lie groupoid is given by the stackification
of this 2-presheaf,
${\Hom(-,G_\bullet)}^\#$.
We call it the stack \emph{represented by} $G_\bullet$.

Recall that a \emph{differentiable stack} is a stack which may be represented by a
Lie groupoid.
\begin{definition}
    A differentiable stack is an \emph{orbifold} if it is representable by a
    proper \'etale Lie groupoid.
\end{definition}

When $G_\bullet$ is an orbifold, the stabiliser group $G_x$ of $x \in G_0$
acts on a small neighbourhood of $x$~\cite[Prop III.5.2]{carchedi2011categorical}.
Recall a group action is called \emph{effective} if the only element
acting trivially is the identity element.
\begin{definition}
    An orbifold is \emph{effective} if all stabilisers $G_x$ act effectively
    on a small neighbourhood of $x$.
\end{definition}

Effectivity can be checked locally. Every
orbifold is locally of the form $[M/\Gamma]$.
Effectivity of the orbifold is then equivalent to the effectivity of the
group action $M \curvearrowleft \Gamma$ for each such local patch.
In particular, effectivity implies that a generic point of the
orbifold has trivial stabiliser.

\begin{example}
    The stack $[\point/H]$ is an orbifold iff the group $H$ is finite.
    If $H$ is nontrivial, $[\point/H]$ is \emph{not} an effective orbifold.
\end{example}

Let $\pi_0: \CAT \to \Set$ denote the functor sending a category to its set of isomorphism classes.
The points of an orbifold $\orb{M}$ form a category $\orb{M}(\point)$. Its set of isomorphism classes
$|\orb{M}| \define \pi_0\orb{M}(\point)$ comes equipped with the following data:
Each functor $f:M \to \orb{M}$ induces a map of sets $|f|:M \to |\orb{M}|$, and $f$ is determined by
$|f|$ up to isomorphism. The collection of maps to $|\orb{M}|$ obtained this way assemble into
a sheaf of sets
\begin{equation*}
    \pi_0 \comp \orb{M}: \Man^\opp \to \CAT \to \Set,
\end{equation*}
whose value on the point is $|\orb{M}|$. In fact, the sheaf
satisfies the condition of being \emph{concrete}~\cite{baez2011convenient}:
\begin{defn}
    A sheaf $\sh{F}:\Man^\opp \to \Set$ is \emph{concrete}
    if for all $U \in \Man$, the map
    \begin{align*}
        \sh{F}(U) & \to [\Hom(\point,U),\sh{F}(\point)] \\
        p         & \mapsto (u \mapsto u^\ast p)
    \end{align*}
    is injective.
\end{defn}
This condition says that sections over $U \in \Man$ are
`allowed families of sections over $\point \in \Man$':
every section is determined by its restrictions over each point $\point \to U$.
It was shown in~\cite{baez2011convenient} that this data is equivalent to
that of a \emph{diffeological space} in the sense
of~\cite{souriau1980groupes,souriau1984groupes,iglesias2013diffeology}.
The sheaf $\pi_0 \comp \orb{M}$ equips the set
$(\pi_0 \comp \orb{M})(\point)=|\orb{M}|$ with a \emph{diffeology}.
The value of the sheaf on $U \in \Man$ is the set of maps of underlying sets
$U \to |\orb{M}|$ that admit a lift to a map of orbifolds.
This gives $|\orb{M}|$ a topology: the
\emph{D-topology}~\cite{iglesias2013diffeology} is the
finest topology for which all these maps are continuous.
For an orbifold represented by a Lie groupoid $G_\bullet=G_1 \rightrightarrows G_0$,
this topology is the topology of the quotient space
$|G_\bullet|=G_0/G_1$.\footnote{
    The quotient space is obtained by quotienting out the underlying topological space of $G_0$ by
    the equivalence relation induced by $G_1$.
}
This is true more generally for differentiable
stacks~\cite{watts2014diffeological}.
\begin{thm}[{\cite[Thm B]{watts2017differential}}]
    The diffeological space $|\orb{M}|$ encodes the
    orbifold $\orb{M}$ up to equivalence.
\end{thm}

\subsection{Sheaves over Orbifolds}
We recall the relationship between an atlas and a representing groupoid for a stack.
Then we introduce a number of sites associated to an \'etale stack, most notably the small
site associated to an atlas.

\begin{defn}
    A map of smooth stacks $F:\st{X} \to \st{Y}$ is \emph{representable} if for
    every map $M \to \st{Y}$ with domain a manifold, the pullback
    $M \times_{\st{Y}} \st{X}$ is representable by a manifold.
    It is \emph{representably submersive} (we also call this a \emph{cover}) 
    if the maps $M \times_{\st{Y}} \st{X} \to M$ are further
    surjective submersions.
    It is \emph{representably \'etale} if the maps
    $M \times_{\st{Y}} \st{X} \to M$ are \'etale maps (ie.\
    local diffeomorphisms).
\end{defn}

The fibre product $\st{X} \times_\st{Z} \st{Y}$ of two maps of stacks
$f:\st{X} \to \st{Z} \ot \st{Y}:g$
is computed object-wise. It assigns to
$U \in \Man$ the category with:
\begin{itemize}
    \item objects: triples $(X \in \st{X}(U),Y \in \st{Y}(U),\phi:f(X) \isoto g(Y) \in \st{Z}(U))$
    \item morphisms $(X,Y,\phi) \to (X^\prime,Y^\prime,\phi^\prime)$: tuples
          $(\alpha:X \to X^\prime,\beta:Y \to Y^\prime)$ such that
          $\phi^\prime \comp f(\alpha) = g(\beta) \comp \phi$.
\end{itemize}

\begin{defn}
    \label{defn:atlas}
    An \emph{atlas} of a differentiable stack $\st{X}$ is a cover
    $X_0 \to \st{X}$ of $\st{X}$ by a manifold $X_0$.
    If the cover is representably \'etale, we call it an
    \'etale atlas.
\end{defn}

\begin{proposition}[{\cite[Prop 4.31]{lerman2010orbifolds}}]
    \label{prop:LieGpdFromAtlas}
    Given a stack $\st{X}$ with an atlas $X_0 \to \st{X}$,
    it is represented by the Lie groupoid
    \begin{equation*}
        X_0 \times_{\st{X}} X_0 \rightrightarrows X_0.
    \end{equation*}
\end{proposition}

\begin{cor}
    A stack is differentiable iff it admits an atlas.
    It is \'etale iff it admits an \'etale atlas.
\end{cor}
\begin{proof}
    The Lie groupoid constructed in Proposition~\ref{prop:LieGpdFromAtlas}
    has source and target morphisms given by the pullbacks of the atlas
    projection along itself.
    A representably \'etale map $X_0 \to \st{X}$ pulls back to \'etale
    source and target maps $X_0 \times_{\st{X}} X_0 \to X_0$, thus
    giving the presented stack the structure of an \'etale stack.
    This covers the `if' direction of the statement.

    If $\st{X}$ is represented by $X_1 \rightrightarrows X_0$,
    there is a natural map $\pi:X_0 \to \st{X}$ induced by
    the identity map on objects.
    It is straightforward to check that $X_1$ represents
    $X_0 \times_{\st{X}} X_0$.
    The map $X_0 \to \st{X}$ is representably \'etale
    if and only if the pullback maps
    $X_0 \times_{\st{X}} X_0$ are:
    as $X_0 \to \st{X}$ is a representable submersion,
    it suffices to check that the pullback atlas
    over any point $\point \to \st{X}$ is \'etale (ie.\ discrete).
    Any point may be factored as $\point \to X_0 \to \st{X}$,
    and the resulting atlas
    $\point \times_\st{X} X_0 \to \point$ is pulled back
    from the \'etale map $X_0 \times_{\st{X}} X_0 \to X_0$.
    It is thus \'etale it itself, which shows $X_0 \times_{\st{X}} X_0$
    is representably \'etale.
\end{proof}

\begin{example}
    The morphism $\xi:\ast \to [\ast/G]$ is representable. For
    any morphism $f:M \to [\ast/G]$, the pullback
    $M \times_{[\ast/G]} \ast$ is the total space of the
    $G$-bundle classified by $f$.
    The map $\xi$ is further \emph{representably \'etale} if $G$ is discrete.
    More generally, morphisms $[\ast/H] \to [\ast/G]$
    are in bijection with homomorphisms $H \to G$.
    They are representable iff the associated homomorphism
    is injective, and representably \'etale iff the quotient $G/H$ is discrete.
\end{example}

\begin{defn}
    \label{def:bigSite}
    The \emph{big site} of a stack $\st{X}$ has underlying category
    $\Man/\st{X}$,
    whose objects are maps $f:M \to \st{X}$, where $M \in \Man$.
    A morphism $(f:M \to \st{X}) \to (g:N \to \st{X})$ is a pair
    $(h:M \to N, \phi: g \comp h \Rightarrow f)$ making
    \begin{equation*}
        \begin{tikzcd}
            M && N \\
            {} \\
            & {\orb{M}}
            \arrow["h", from=1-1, to=1-3]
            \arrow[""{name=0, anchor=center, inner sep=0}, "f"', from=1-1, to=3-2]
            \arrow["g", from=1-3, to=3-2]
            \arrow["\phi", shift left=1, shorten <=14pt, shorten >=14pt, Rightarrow, from=1-3, to=0]
        \end{tikzcd}
    \end{equation*}
    commute.
    The structure as a ringed site is inherited via the forgetful functor
    $\Man/\st{X} \to \Man$.
    We retain the notation $\Cinf$ for the structure sheaf.
\end{defn}
Sheaves over the big site of $\st{X}$ are
to be thought of as families of sheaves over $\Man$
\emph{parameterised by $\st{X}$}. In particular,
sheaves over the big site of the point $\point \in \Man$
are sheaves over the site $\Man/\point \simeq \Man$.
This is markedly different
from the category of sheaves over the site $\Op(\point)$.

For an \'etale stack $\orb{M}$, one can define a site which recovers
the behaviour of the site $\Op(X)$ of open subsets of a space $X$.
(See in particular~\cite{carchedi2010sheaf}.) We discuss
other versions of this site in Appendix~\ref{app:twistedSheaves}.
\begin{defn}
    The \emph{small site} (also called the \emph{\'etale site}) of an \'etale stack $\orb{M}$
    is the full subsite $\Et(\orb{M}) \subset \Man/\st{X}$
    on objects $(f:M \to \orb{M})$ where $f$ is representably \'etale.
\end{defn}

\begin{defn}
    A \emph{sheaf/stack} on an \'etale stack is a
    sheaf/stack on its small site.
\end{defn}
It is not hard to check that for a manifold $M$, the site
$\Op(M)$ is dense in $\Et(M)$, validating the above definition.
We will use $\ShC(\orb{M})$ to denote the category of $C^\infty$-modules over the
small site of $\orb{M}$, and refer to them simply as sheaves over $\orb{M}$.

\begin{example}
    \label{ex:sheavesOnBGamma}
    The small site of $[\point/\Gamma]$ has a dense subsite
    with a single object $\point$ whose endomorphisms form the group $\Gamma$.
    The structure sheaf $\Cinf$ assigns $\IC$ with its
    usual ring structure to $\point$.
    The category of sheaves of $\Cinf$-modules is the category
    \begin{equation*}
        \ShC([\point/\Gamma]) = \Rep(\Gamma)
    \end{equation*}
    of $\IC$-linear representations of $\Gamma$.
\end{example}

Let $\sh{F}$ be a sheaf over an orbifold $\orb{M}$.
For each point $x: \point \to \orb{M}$, one may compute
the pullback $x^\ast\sh{F} \in \ShC(\point)=\Vec$.
The points for which this pullback is non-empty form
a full subcategory $\categoryname{Supp}\ \sh{F} \subset \orb{M}(\point)$.
\begin{defn}
    \label{def:supportOfSheafOverOrbifold}
    The \emph{support} of a sheaf over an orbifold is
    the diffeological subspace
    of the quotient space given by
    \begin{equation*}
        \supp\sh{F} \define \pi_0 \categoryname{Supp}\ \sh{F} \subset 
        \pi_0 \orb{M}(\point) = |\orb{M}|.
    \end{equation*}
\end{defn}

The support plays well with pushforwards and pullbacks.
\begin{lemma}
    \label{lem:behaviourOfSupportUnderPushfwdAndPullback}
    Let $\sh{F} \in \ShC(\orb{M})$, $\sh{F}' \in \ShC(\orb{N})$,
    be sheaves of $\Cinf$-modules over \'etale stacks
    and $f:\orb{M} \to \orb{N}$ a map between the stacks.
    Then
    \begin{align*}
        \supp{f_\ast \sh{F}} = f(\supp \sh{F}) \subset |\orb{N}|
         &  &
        \supp{f^\ast \sh{F}'} = f^{-1}(\supp \sh{F}') \subset |\orb{M}|
    \end{align*}
\end{lemma}
\begin{proof}
    This may be checked pointwise. The stalk of the pushforward/pullback sheaf is
    non-empty precisely when the condition in the statement is satisfied.
\end{proof}

\begin{example}
    Let $E$ be a vector bundle over a manifold $M$ and $f:M \to \orb{N}$ a map to an \'etale stack.
    Then the support of $f_\ast E$ is $\im{f} \subset |\orb{N}|$.
\end{example}

\subsection{Gerbes and Twisted Sheaves}
\label{sec:GerbesAndTwistedSheaves}
We will follow~\cite{schommer2011central} and give a modern
definition of gerbes.
We introduce $\ger{G}$-twisted sheaves (associated to a
$\IC^\times$-gerbe $\ger{G}$) in Definition~\ref{def:GtwistedSheaf}.

Let $G$ be a group object in $\DiffSt$~\cite[Def 41]{schommer2009classification}
and $\st{X}$ a differentiable stack.
Denote by $\DiffSt/\st{X}$ the overcategory of $\st{X}$.
The object $\st{X} \times G$ is then a group object
in $\DiffSt/\st{X}$.
\begin{defn}[{\cite[Def 70]{schommer2011central}}]
    \label{defn:principal2GrpBundle}
    A \emph{principal $G$-bundle over $\st{X}$} is a stack
    $\ger{G} \in \DiffSt/\st{X}$, equipped with
    a right action of $\st{X} \times G$, locally equivalent to the
    trivial principal $G$-bundle:
    There exists a cover $f:\st{Y} \onto \st{X}$ such that the pullback
    $f^\ast \ger{G}=\st{Y} \times_{\st{X}} \ger{G} \in \DiffSt/\st{Y}$ is $G$-equivariantly
    equivalent to $\st{Y} \times G$ with the standard $\st{Y} \times G$-action.
\end{defn}
A principal $G$-bundle is also called a \emph{$G$-torsor}.

\begin{defn}
    \label{defn:ICtimesGerbeAsPrincipalBundle}
    A \emph{$\IC^\times$-gerbe} over $\st{X}$ is a principal
    $[\point/\IC^\times]$-bundle over $\st{X}$.
\end{defn}
We will exclusively work with $\IC^\times$-gerbes,
and just refer to them as gerbes.
The \emph{trivial gerbe} over a stack $\st{X}$ is the gerbe
\begin{equation*}
    \ger{I}_\st{X} \define \st{X} \times [\point/\IC^\times]
\end{equation*}
equipped with the projection morphism $\ger{I}_\st{X} \to \st{X}$.
When the base is clear from context, we omit the subscript and
simply write $\ger{I}$ to denote the trivial gerbe.

We may pull back a gerbe $\ger{G} \to \st{X}$ along a map
$\st{Y} \to \st{X}$ of stacks by forming the fibre product
\begin{equation*}
    \begin{tikzcd}
        f^\ast \ger{G} \ar[r] \ar[d] & \ger{G} \ar[d] \\
        \st{Y} \ar[r] & \st{X}.
        \arrow["\lrcorner"{anchor=center, pos=0.125}, draw=none, from=1-1, to=2-2]
    \end{tikzcd}
\end{equation*}
The trivial gerbe $\ger{I}_\st{X} \to \st{X}$ pulls back to the trivial gerbe
$\ger{I}_\st{Y} \to \st{Y}$ for any map $\st{Y} \to \st{X}$.
\begin{defn}
    A \emph{trivialisation} of a gerbe $\ger{G} \to \st{X}$ is an
    isomorphism $\ger{I}_\st{X} \eqto \ger{G}$ of gerbes over $\st{X}$.
    We call a gerbe \emph{trivialisable} if it admits a trivialisation.
\end{defn}
Every trivialisation gives rise to a section $\st{X} \to \ger{G}$
by composing $\ger{I}_\st{X} \eqto \ger{G}$ with the canonical section
$\st{X} \to \ger{I}_\st{X}$ of the trivial gerbe.
The $[\point/\IC^\times]$-torsor structure makes this assignment an equivalence
between trivialisations and sections.

By definition, every differentiable stack admits a cover by a
manifold, at atlas. Every cover $\st{Y} \to \st{X}$ is refined by
such an atlas. This means we may always replace $\st{Y}$ in
Definition~\ref{defn:principal2GrpBundle} by a manifold.
The local triviality of gerbes allows for a description of gerbes
over differentiable stacks via $\IC^\times$-extensions of Lie groupoids.
\begin{defn}
    A \emph{central $\IC^\times$-extension} of a Lie groupoid $X_\bullet$ is a
    smooth functor of Lie groupoids
    \begin{equation*}
        \begin{tikzcd}
            P & {X_1} \\
            {X_0} & {X_0}
            \arrow["\id",from=2-1, to=2-2]
            \arrow[from=1-1, to=1-2]
            \arrow[shift left=1, from=1-2, to=2-2]
            \arrow[shift right=1, from=1-2, to=2-2]
            \arrow[shift right=1, from=1-1, to=2-1]
            \arrow[shift left=1, from=1-1, to=2-1]
        \end{tikzcd}
    \end{equation*}
    together with a $\IC^\times$-action on $P$ that makes
    $P \to X_1$ a principal $\IC^\times$-bundle.
    The action is required to commute with composition
    of morphisms in $P \rightrightarrows X_0$.
\end{defn}

\begin{example}
    Central Lie groupoid $\IC^\times$-extensions of
    $[\point/G]=G \rightrightarrows \point$ are in bijection
    with central Lie group $\IC^\times$-extensions of $G$.
\end{example}

There is a dictionary between $\IC^\times$-gerbes over differentiable stacks
and central $\IC^\times$-ex\-ten\-sions of Lie
groupoids.
\begin{lemma}[\cite{laurent2009non,behrend2011differentiable}]
    \label{lem:gerbesVsLieGroupoidExtensions}
    Let $X_\bullet=X_1 \rightrightarrows X_0$ be a Lie groupoid presenting
    a stack $\st{X}$.
    There is a one-to-one correspondence between
    equivalence classes of
    $\IC^\times$-gerbes $\ger{G}$ over $\st{X}$
    equipped with a trivialisation over $X_0$,
    and isomorphism classes of central
    $\IC^\times$-extensions $\ger{G}_\bullet \to X_\bullet$.
\end{lemma}

Given a presentation of a gerbe $\ger{G} \to \st{X}$ by
a $\IC^\times$-extension $P \to X_1 \rightrightarrows X_0$, we may extract a
2-cocycle $\theta \in Z^2(\st{X},\IC^\times)$~\cite[App A]{felder2008gerbe}:
If we choose $X_0$ fine enough, we may assume $X_1$ is
a disjoint union of contractible manifolds, and the $\IC^\times$-bundle $P \to X_1$
is equivalent to $X_1 \times \IC^\times \to X_1$.
The class of the $\IC^\times$-bundle in $\H^1(X_1,\IC^\times)$ becomes
a class in the cohomology of $\st{X}$ (see~\cite{behrend2004cohomology}
for a nice review of cohomology theories for stacks).
The class $[\theta] \in \H^2(\st{X},\IC^\times)$
of this 2-cocycle only depends on the equivalence class of the
gerbe $\ger{G} \to \st{X}$. The converse is also true:
\begin{thm}[{\cite{giraud1966cohomologie},\cite[Thm 3.13]{laurent2009non}}]
    \label{thm:gerbesClassifiedByH2}
    $\IC^\times$-gerbes over $\st{X}$ are classified
    up to equivalence by their associated class in $\H^2(\st{X},\IC^\times)$.
\end{thm}

In particular, every gerbe over a stack $\st{X}$ with trivial
$\H^2(\st{X},\IC^\times)$ admits a trivialisation.
\begin{cor}
    \label{cor:classificationOfGerbesOnQuotientOrbifolds}
    Let $U$ be a contractible Euclidean space with a $\Gamma$-action
    (we assume $\Gamma$ is finite).
    Equivalence classes of $\IC^\times$-gerbes over the stack
    $[U/\Gamma]$ are in bijection with $\H^2(\B\Gamma,\IC^\times)$.
\end{cor}
\begin{proof}
    By assumption, $U$ is contractible, so $\H^2(U,\IC^\times)=0$.
    By Theorem~\ref{thm:gerbesClassifiedByH2},
    every gerbe over $[U/\Gamma]$ admits a trivialisation
    over the atlas $U \onto [U/\Gamma]$.
    Lemma~\ref{lem:gerbesVsLieGroupoidExtensions}
    then identifies equivalence classes of gerbes over $[U/\Gamma]$
    with equivalence classes of central $\IC^\times$-extensions
    of the Lie groupoid $U \times \Gamma \rightrightarrows U$.
    Up to equivalence, every $\IC^\times$-bundle $P$ over
    $U \times \Gamma$ is trivial, so the only extra data
    is the composition on the morphism space $P$ of the extension.
    It is easily checked that this encodes above each point $u \in U$
    a central extension of $\Gamma$ by $\IC^\times$.
    This central extension may vary smoothly across $U$, but
    as $\Gamma$ is finite, so is the group
    $\H^2(\B\Gamma, \IC^\times)$ of equivalence classes of central extensions.
    The induced morphism $U \to \H^2(\B\Gamma,\IC^\times)$ is thus constant.
    This establishes a map from $\IC^\times$-gerbes to $\H^2(\B\Gamma,\IC^\times)$.
    Groupoid extensions which induce equivalent central extensions
    of $\Gamma$ by $\IC^\times$ are equivalent.
\end{proof}

\begin{rmk}
    More conceptually, the above statement is a consequence of the fact that
    the stack $[U/\Gamma]$ has the homotopy type of $\B\Gamma$ (in
    the $\infty$-topos of smooth $\infty$-stacks), hence
    $\H^2([U/\Gamma],\IC^\times)=\H^2(\B\Gamma,\IC^\times)$.
\end{rmk}

Every orbifold may be covered by suborbifolds of the form $[U/\Gamma]$,
where $U$ is an open ball in some Cartesian space $\IR^n$, equipped
with a right action by a finite group $\Gamma$.
Pick a point $x_0 \in U$ with maximal stabiliser
$\Aut(x_0) = \Gamma$. Any gerbe $\ger{G}$ over the orbifold pulls back
to a gerbe $[\point/\widetilde{\Gamma}] \to [\point/\Gamma]$.
Let $\IC^\times \into \widetilde{\Gamma} \to \Gamma$ be some central extension
representing this gerbe according to
Corollary~\ref{cor:classificationOfGerbesOnQuotientOrbifolds}.
The class of this central extension in $\H^2(\B\Gamma,\IC^\times)$
represents this gerbe up to equivalence.
We obtain a useful local normal form for $\IC^\times$-gerbes over these
quotient suborbifolds.
We record the outcome of this discussion in the following lemma:
\begin{lemma}
    \label{lem:localNormalFormOfGerbes}
    Let $[U/\Gamma]$ be a quotient suborbifold of an orbifold
    $\orb{M}$ equipped with a gerbe $\ger{G}$.
    Assume further $U$ is contractible, and contains a
    $\Gamma$-fixed point over which $\ger{G}$ is classified
    by the central extension
    \begin{equation*}
        \IC^\times \into \widetilde{\Gamma} \xto{\pi} \Gamma.
    \end{equation*}
    Then the restriction of $\ger{G}$ to $[U/\Gamma]$ admits a
    presentation by the constant $\IC^\times$-extension
    \begin{equation*}
        \begin{tikzcd}
            U \times \widetilde{\Gamma} & {U \times \Gamma} \\
            {U} & {U}
            \arrow["\id",from=2-1, to=2-2]
            \arrow["\id \times \pi", from=1-1, to=1-2]
            \arrow[shift left=1, from=1-2, to=2-2]
            \arrow[shift right=1, from=1-2, to=2-2]
            \arrow[shift right=1, from=1-1, to=2-1]
            \arrow[shift left=1, from=1-1, to=2-1]
        \end{tikzcd}
    \end{equation*}
    of Lie groupoids.
\end{lemma}

A $\IC^\times$-gerbe $\ger{G}$ over an \'etale stack $\orb{M}$ allows a twisting
of the notion of sheaf over said stack. The following material is introduced in more detail in Appendix~\ref{app:twistedSheaves}.
We begin by describing the notion of a twisted sheaf for a trivialisable gerbe.
If $\ger{G}$ is trivialisable over $\orb{M}$, the category of \emph{$\ger{G}$-twisted sheaves} over
it has objects pairs $(\tau: \ger{I} \eqto \ger{G}, \sh{F} \in \ShC(\orb{M}))$, which we write as
\[
    \tau \cdot \sh{F}.
\]
The morphisms $\tau \cdot \sh{F} \to \tau' \cdot \sh{F}'$ are given by morphisms of sheaves ``twisted by the
trivialisations'': the composite of trivialisations $\tau^{-1} \tau'$ is an endomorphism of $\ger{I}$, and
thus defines a $\IC^\times$-bundle over $\orb{M}$. We denote the action of the associated complex line bundle $\mscr{L}_{\tau,\tau'}$
on a sheaf $\sh{F}'$ by
\[
    \tau^{-1} \tau' \cdot \sh{F}' \define {\mscr{L}}_{\tau,\tau'} \tensor_{\Cinf} \sh{F}'.
\]
A morphism $\phi:\tau \cdot \sh{F} \to \tau' \cdot \sh{F}'$ is a morphism of sheaves
\[
    \bar{\phi}:\sh{F} \to \tau^{-1} \tau' \cdot \sh{F}'.
\]
The composite
\[
    \tau\cdot\sh{F} \xto{\phi} \tau'\cdot\sh{F}' \xto{\phi'} \tau'' \cdot \sh{F}''
\]
is given by the morphims of $\Cinf$-modules
\[
    \sh{F} \xto{\bar{\phi}} \tau^{-1} \tau' \cdot \sh{F}' \xto{\tau^{-1}\tau' \cdot \bar{\phi}'}
    \tau^{-1} \tau' \tau'^{-1} \tau'' \cdot \sh{F}'' \simeq \tau^{-1} \tau'' \cdot \sh{F}''.
\]
We denote the resulting category by $\Shv^{\ger{G}}(\orb{M})$. A choice of trivialisation $\tau$ induces a functor
\begin{alignat*}{1}
    \tau \cdot -: \ShC(\orb{M}) & \to \Shv^{\ger{G}}(\orb{M}) \\
    \sh{F}                      & \mapsto \tau \cdot \sh{F},
\end{alignat*}
and this functor is an equivalence (see Lemma~\ref{lem:tauCdotIsEquivalence} and below).
This way, $\ger{G}$-twisted sheaves for a trivialisable gerbe are (non-canonically) identified
with usual sheaves of $\Cinf$-modules.

A map $f: \orb{N} \to \orb{M}$ of \'etale stacks induces a pullback functor
\begin{align*}
    f^\ast: \Shv^{\ger{G}}(\orb{M}) & \to \Shv^{f^\ast\ger{G}}(\orb{N})      \\
    \tau \cdot \sh{F}               & \mapsto f^\ast\tau \cdot f^\ast\sh{F},
\end{align*}
which makes the square
% https://q.uiver.app/?q=WzAsNCxbMCwwLCJcXFNodl57XFxnZXJ7R319KFxcb3Jie019KSJdLFsxLDAsIlxcU2h2XntmXlxcYXN0XFxnZXJ7R319KFxcb3Jie059KSJdLFswLDEsIlxcU2hDKFxcb3Jie019KSJdLFsxLDEsIlxcU2hDKFxcb3Jie059KSJdLFswLDEsImZeXFxhc3QiXSxbMiwzLCJmXlxcYXN0Il0sWzAsMiwiXFx0YXUgXFxjZG90IC0iLDJdLFsxLDMsImZeXFxhc3RcXHRhdSBcXGNkb3QgLSJdXQ==
\[\begin{tikzcd}
        {\Shv^{\ger{G}}(\orb{M})} & {\Shv^{f^\ast\ger{G}}(\orb{N})} \\
        {\ShC(\orb{M})} & {\ShC(\orb{N})}
        \arrow["{f^\ast}", from=1-1, to=1-2]
        \arrow["{f^\ast}", from=2-1, to=2-2]
        \arrow["{\tau \cdot -}"', from=1-1, to=2-1]
        \arrow["{f^\ast\tau \cdot -}", from=1-2, to=2-2]
    \end{tikzcd}\]
commute for any choice of trivialisation $\tau$.

Now we drop the assumption that $\ger{G}$ is trivialisable.
The \'etale site of $\orb{M}$ has a dense subsite $\Et^\ger{G}(\orb{M}) \subset \Et(\orb{M})$ on objects
$u:U \to \orb{M}$ where the pullback gerbe $u^\ast\ger{G}$ admits a trivialisation.
The assignment
\begin{alignat*}{1}
    \Shv^{\ger{G}}:\Et^\ger{G}(\orb{M})^\opp & \to \VCat \subset \VecInf\text{-}\Cat                                            \\
    (u:U \to \orb{M})                        & \mapsto \Shv^{u^\ast\ger{G}}(U)                                                  \\
    \left((f,\phi):v \to u\right)            & \mapsto \left(f^\ast: \Shv^{u^\ast\ger{G}}(U) \to \Shv^{v^\ast\ger{G}}(V)\right)
\end{alignat*}
extends to a $\VCat$-valued stack over $\orb{M}$ (Lemma~\ref{lem:ShvGIsAStack}).

\begin{defn}
    \label{def:GtwistedSheaf}
    A \emph{$\ger{G}$-twisted sheaf} over the \'etale stack $\orb{M}$ is a global section of $\Shv^{\ger{G}}$.
\end{defn}
Given any cover
$\pi:Y \onto \orb{M}$ over which $\ger{G}$ is trivialisable, the category of $\ger{G}$-twisted sheaves over $\orb{M}$
is equivalent to the category of descent data of $\Shv^{\ger{G}}$ over $Y$.
A $\ger{G}$-twisted sheaf is represented by a twisted sheaf $\tau \cdot \sh{F} \in \Shv^{\pi^\ast\ger{G}}(Y)$,
equipped with a descent isomorphism
\[
    \phi: p_1^\ast(\tau \cdot \sh{F}) \eqto p_2^\ast (\tau \cdot \sh{F})
\]
in $\Shv^{p_{\orb{M}}^\ast\ger{G}}(Y^{[2]})$, subject to the condition that
% https://q.uiver.app/?q=WzAsMyxbMCwxLCJwXzFeXFxhc3QgKFxcdGF1IFxcY2RvdCBcXHNoe0Z9KSJdLFsxLDAsInBfMl5cXGFzdCAoXFx0YXUgXFxjZG90IFxcc2h7Rn0pIl0sWzIsMSwicF8zXlxcYXN0IChcXHRhdSBcXGNkb3QgXFxzaHtGfSkiXSxbMCwyLCJwX3sxM31eXFxhc3RcXHBoaSIsMl0sWzAsMSwicF97MTJ9XlxcYXN0XFxwaGkiXSxbMSwyLCJwX3syM31eXFxhc3RcXHBoaSJdXQ==
\[\begin{tikzcd}
        & {p_2^\ast (\tau \cdot \sh{F})} \\
        {p_1^\ast (\tau \cdot \sh{F})} && {p_3^\ast (\tau \cdot \sh{F})}
        \arrow["{p_{13}^\ast\phi}"', from=2-1, to=2-3]
        \arrow["{p_{12}^\ast\phi}", from=2-1, to=1-2]
        \arrow["{p_{23}^\ast\phi}", from=1-2, to=2-3]
    \end{tikzcd}\]
commutes.

Given a specific trivialisation $\tau:\ger{I} \eqto \pi^\ast\ger{G}$ over $Y$ with associated line bundle
${\mscr{L}} = {\mscr{L}}_{p_1^\ast \tau,p_2^\ast \tau}$ over $Y^{[2]}$,
we may use the equivalence $\tau \cdot -: \ShC(Y) \eqto \Shv^{\pi^\ast\ger{G}}(Y)$ to replace the datum $\tau \cdot \sh{F}$
above with a sheaf of $\Cinf$-modules over $Y$.
This identifies the category of descent data with the category of pairs
$(\sh{F} \in \ShC(Y),\bar{\phi}:p_1^\ast\sh{F} \eqto \mscr{L} \tensor p_2^\ast\sh{F})$,
satisfying the cocycle condition
% https://q.uiver.app/?q=WzAsNCxbMCwxLCJwXzFeXFxhc3RcXHNoe0Z9Il0sWzAsMCwicF97MTJ9XlxcYXN0XFxtc2Nye0x9XFx0ZW5zb3IgcF8yXlxcYXN0XFxzaHtGfSJdLFsyLDAsInBfezEyfV5cXGFzdFxcbXNjcntMfVxcdGVuc29yIHBfezIzfV5cXGFzdFxcbXNjcntMfSBcXHRlbnNvciBwXzNeXFxhc3RcXHNoe0Z9Il0sWzIsMSwicF97MTN9XlxcYXN0XFxtc2Nye0x9XFx0ZW5zb3IgcF8zXlxcYXN0XFxzaHtGfS4iXSxbMCwzLCJwX3sxM31eXFxhc3RcXGJhcntcXHBoaX0iXSxbMCwxLCJwX3sxMn1eXFxhc3RcXGJhcntcXHBoaX0iXSxbMSwyLCJwX3sxMn1eXFxhc3RcXG1zY3J7TH1wX3syM31eXFxhc3RcXGJhcntcXHBoaX0iXSxbMiwzLCJcXHNpbWVxIiwzLHsic3R5bGUiOnsiYm9keSI6eyJuYW1lIjoibm9uZSJ9LCJoZWFkIjp7Im5hbWUiOiJub25lIn19fV1d
\[\begin{tikzcd}
        {p_{12}^\ast\mscr{L}\tensor p_2^\ast\sh{F}} && {p_{12}^\ast\mscr{L}\tensor p_{23}^\ast\mscr{L} \tensor p_3^\ast\sh{F}} \\
        {p_1^\ast\sh{F}} && {p_{13}^\ast\mscr{L}\tensor p_3^\ast\sh{F},}
        \arrow["{p_{13}^\ast\bar{\phi}}", from=2-1, to=2-3]
        \arrow["{p_{12}^\ast\bar{\phi}}", from=2-1, to=1-1]
        \arrow["{p_{12}^\ast\mscr{L} \tensor p_{23}^\ast\bar{\phi}}", from=1-1, to=1-3]
        \arrow["\simeq"{marking}, draw=none, from=1-3, to=2-3]
    \end{tikzcd}\]
where the natural isomorphism $p_{12}^\ast\mscr{L} \tensor p_{23}^\ast\mscr{L} \simeq p_{13}^\ast\mscr{L}$
is part of the data of the gerbe.\footnote{In the triangular cocycle condition, this equivalence is hidden
    in the composition operation between morphisms of twisted sheaves.}

\begin{example}
    \label{ex:twistedSheavesOverBGamma}
    Consider the gerbe $\ger{G}$ over $[\point/\Gamma]$ corresponding
    to a group extension $\IC^\times \to \widetilde{\Gamma} \to \Gamma$.
    The gerbe admits a trivialisation over the atlas $Y=\point \to [\point/\Gamma]$,
    with associated line bundle over
    $Y^{[2]} = \point \times_{\point/\Gamma} \point = \Gamma$
    given by $\mscr{L}=\tilde{\Gamma} \to \Gamma$.
    We may pick a trivialisation $\tilde{\Gamma} \iso \IC^\times \times \Gamma$.
    Then the isomorphism of line bundles
    $p_{12}^\ast\mscr{L} \tensor p_{23}^\ast\mscr{L} \eqto p_{13}^\ast\mscr{L}$
    over $Y^{[3]} = \Gamma \times \Gamma$ is the data of a 2-cocycle
    $\theta: \Gamma \times \Gamma \to \IC^\times$.
    By Lemma~\ref{lem:localNormalFormOfGerbes}, all gerbes over $[\point/\Gamma]$
    are of this form (up to equivalence).

    The category of $\ger{G}$-twisted sheaves over $[\point/\Gamma]$
    is then given by a vector space $V$ (the $\Cinf$-module over $\point$),
    equipped with an isomorphism of pullback sheaves over $\Gamma$.
    This isomorphism is the data of a $\IC$-linear isomorphism
    $\phi_{\gamma}: V \to V$ for all $\gamma \in \Gamma$.
    The cocycle condition is twisted by $\theta$, and precisely recovers the definition
    of a $\theta$-twisted representation.
\end{example}

\subsection{The Twisted Equivariantisation Adjunction}
\label{sec:twistedEquivariantisationAdjunction}
Let $\cat{C} \in \VCat$ be a $\VecInf$-enriched category complete under direct sums,
equipped with a group action $\Gamma \to \End(\cat{C})$.
The \emph{equivariantisation of $\Gamma \curvearrowright \cat{C}$} is
a category $\cat{C}^\Gamma$, whose objects
are objects $X \in \cat{C}$ equipped with $\Gamma$-equivariance data.
When $\Gamma$ is finite, the obvious forgetful functor
$u: \cat{C}^\Gamma \to \cat{C}$
admits a simultaneous left and right adjoint, the induction
functor $I$. It sends an object $X \in \cat{C}$ to the orbit
of $X$ under the $\Gamma$-action
$X \mapsto \DirSum_{\gamma \in \Gamma} \gamma.X,$
equipped with the obvious equivariance data.

An example of this in the $\VecInf$-enriched context is the category of
sheaves of $\Cinf$-modules over a quotient orbifold $[M/\Gamma]$:
\[
    \ShC([M/\Gamma]) \simeq {\ShC(M)}^\Gamma.
\]
The data of the ambidextrous adjunction between $u$ and $I$ induces a
Frobenius monad on $\ShC([M/\Gamma])$,
and in particular exhibits every
object of $\ShC([M/\Gamma])$ as a direct summand of an object of the form
$I(X)$.\footnote{
    Ambidextrous adjunctions and the associated Frobenius monad are
    discussed more in Section~\ref{sec:projectiveReps}.
}
In this section, we show that these results also hold for twisted sheaves
in the presence of a non-trivial gerbe over $[M/\Gamma]$.
The statement below is Proposition~\ref{prop:twistedPullbackPushfwdAdjunction}
in Appendix~\ref{app:twistedSheaves}.
\begin{proposition}
    \label{prop:twistedSheavesAdjunction}
    The functor $f^\ast:\Shv^{\ger{G}}(\orb{M}) \to \Shv^{f^\ast\ger{G}}(\orb{N})$ induced
    by a map $f:\orb{N}\to \orb{M}$ of \'etale stacks admits a right adjoint
    \[
        f_\ast:\Shv^{f^\ast\ger{G}}(\orb{N}) \to \Shv^{\ger{G}}(\orb{M}).
    \]
\end{proposition}
We call $f_\ast$ the \emph{pushforward of twisted sheaves}. Recall that a trivialisation of $\ger{G}$
identifies the pullback of twisted sheaves with the pullback of $\Cinf$-modules. By uniqueness of
adjoints, the same is true for the pushforward:
a trivialisation of $\ger{G}$ (over the target of the map $f$) identfies the pushforward of
twisted sheaves along $f$ with the ordinary pushforward.

\begin{example}
    \label{ex:pushforwardTwistedSheafBGamma}
    Let $\ger{G} \to [\point/\Gamma]$ be the gerbe corresponding to the
    2-cocycle $\theta:\Gamma \times \Gamma \to \IC^\times$
    as in Example~\ref{ex:twistedSheavesOverBGamma}.
    A homomorphism of groups $H \into \Gamma$ gives rise to a
    map $h:[\point/H] \to [\point/\Gamma]$.
    The associated adjunction $h^\ast:\Rep^{\alpha\restriction_H}(H) \leftrightarrows \Rep^\alpha(\Gamma):h_\ast$
    is the restriction-induction adjunction of twisted representations given in
    Section~\ref{sec:projectiveReps}.
\end{example}

We now generalise Example~\ref{ex:pushforwardTwistedSheafBGamma} to
the case of a quotient orbifold $[U/\Gamma]$, where $U$ is contractible.
Up to equivalence, any gerbe over $[U/\Gamma]$ is of the form
$[U/\tilde{\Gamma}] \to [U/\Gamma]$, where $\tilde{\Gamma} = \Gamma \times \IC^\times$
with multiplication twisted by a 2-cocycle $\theta$. This is the local normal
form for gerbes established in Lemma~\ref{lem:localNormalFormOfGerbes}.
We pick a gerbe $\ger{G}$ over $[U/\Gamma]$ and assume it is of this form.

We view a $\ger{G}$-twisted sheaf over $[U/\Gamma]$ as
a sheaf $\sh{F}$ over $U$, together with a
descent isomorphism $\phi:p_1^\ast\sh{F} \eqto p_2^\ast\sh{F}$ over
$U \times_{[U/\Gamma]} U = U \times \Gamma$, satisfying the coycle condition.
In the case at hand, this is a set of isomorphisms
$\{\phi_\gamma \define \phi\restriction_{U \times {\gamma}}: \gamma^\ast\sh{F} \eqto \sh{F}\}_{\gamma \in \Gamma},$
such that
\[
    \phi_{\gamma'} \comp \phi_{\gamma} =
    \theta(\gamma',\gamma) \cdot \phi_{\gamma' \cdot \gamma}
\]
(cf. Example~\ref{ex:twistedSheavesOverBGamma}).
Using this description, the pullback associated to the cover
$q:U \onto [U/\Gamma]$ is given by
\begin{align*}
    q^\ast: \Shv^\ger{G}([U/\Gamma]) & \to \ShC(U)     \\
    (\sh{F},\phi)                    & \mapsto \sh{F}.
\end{align*}

\begin{prop}
    \label{prop:twistedEquivariantisationAdjunction}
    The functor $q^\ast$ has a simultaneous left and right adjoint, the induction functor
    \begin{alignat*}{1}
        I: \ShC(U) & \to \Shv^\ger{G}([U/\Gamma])                             \\
        \sh{F}     & \mapsto \left(\DirSum_{\gamma \in \Gamma} \gamma^\ast\sh{F}, \phi\right),
    \end{alignat*}
    where the isomorphism $\phi$ over $U \times \Gamma$ has components
    \[
        \phi_{\gamma_1}:
        \DirSum_{\gamma \in \Gamma} \gamma_1^\ast \gamma^\ast\sh{F} \xto{\simeq}
        \DirSum_{\gamma \in \Gamma} {(\gamma_1 \cdot \gamma)}^\ast\sh{F}
        \xto{\dirSum \theta(\gamma_1,\gamma)}
        \DirSum_{\gamma \in \Gamma} {(\gamma_1 \cdot \gamma)}^\ast\sh{F}
    \]
    given by the 2-cocycle $\theta$.\footnote{
        Note that because $\Gamma$ acts on $U$ from the right, the pullback
        morphisms compose
        as $\gamma_1^\ast \gamma^\ast \simeq {(\gamma_1 \cdot \gamma)}^\ast$.
    }
\end{prop}
Note that the functor $I$ is indeed well-defined: the condition
$\phi_{\gamma'} \comp \phi_{\gamma} =
    \theta(\gamma',\gamma) \cdot \phi_{\gamma' \cdot \gamma}$
follows from the 2-cocycle condition satisfied by $\theta$.
\begin{proof}
    We show the adjunction $q^\ast \dashv I$. The adjunction
    $I \dashv q^\ast$ is proved the same way.
    We do this by exhibiting a natural isomorphism
    \[
        \Hom_{\Shv^\ger{G}([U/\Gamma])}((\sh{F},\phi),I\sh{F}')
        \simeq
        \Hom_{\ShC(U)}(\sh{F},\sh{F'}). 
    \]
    A map $(\sh{F},\phi) \to I\sh{F}'$ is a $\theta$-twisted
    equivariant map $\beta:\sh{F} \to \DirSum_\gamma \gamma^\ast \sh{F}$.
    The natural equivalence of $\Hom$-spaces is given by
    \[
        \left(\beta:\sh{F} \to \DirSum_\gamma \gamma^\ast\sh{F}'\right)
        \mapsto
        \left(\beta_e:\sh{F} \to \sh{F}'\right).
    \]
    Its inverse is the map
    \[
        (\beta_e:\sh{F} \to  \sh{F}') \mapsto (\beta':
        \sh{F} \to \DirSum_\gamma \gamma^\ast\sh{F}'),
    \]
    with component $\beta'_\gamma:\gamma^\ast\sh{F}\to \sh{F}'$ given by
    \[
        \beta'_\gamma = (\gamma^\ast \beta_e) \comp \phi_\gamma.
    \]
\end{proof}

By the uniqueness of adjoints, there is a natural isomorphism of functors
\begin{align*}
    I \simeq q_\ast.
\end{align*}

\begin{corollary}
    \label{cor:twistedSheavesAreLocallySummandsOfMobiles}
    Every twisted sheaf on $[U/\Gamma]$
    is a summand of a twisted sheaf in the image of the functor
    \[
        q_\ast: \ShC(U) \to \Shv^\ger{G}([U/\Gamma]).
    \]
\end{corollary}
\begin{proof}
    The following computation takes place in $\End(\Shv^\ger{G}([U/\Gamma]))$.
    Examining the adjunction data given in the proof of
    Proposition~\ref{prop:twistedEquivariantisationAdjunction},
    the counit and unit compose to the identity natural transformation
    \[
        1: \Id_{\Shv^\ger{G}([U/\Gamma])} \xto{\eta} q_\ast q^\ast \xto{\varepsilon}
        \Id_{\Shv^\ger{G}([U/\Gamma])},
    \]
    hence the composite
    \[
        e:q_\ast q^\ast \xto{\varepsilon} \Id_{\Shv^\ger{G}([U/\Gamma])}
        \xto{\eta} q_\ast q^\ast
    \]
    is an idempotent, split by $\varepsilon$ and $\eta$.
    They identify the identity functor as a summand of $q_\ast q^\ast$.
    In particular, this means every object $V$ is a summand of
    $q_\ast q^\ast V$ which is manifestly in the image of $q_\ast$.
\end{proof}

\subsection{The Stack of Twisted Skyscraper Sheaves}
Let $\ger{G} \to \orb{M}$ be an \'etale stack equipped with a gerbe $\ger{G}$.
We make the two definitions below in analogy with the usual support of a sheaf (Definition~\ref{def:supportOfSheafOverOrbifold}),
and untwisted skyscraper sheaves (Definition~\ref{def:skyscraperSheaf}).
\begin{defn}
    \label{def:twistedSheafSupport}
    The \emph{support} of a twisted sheaf $\sh{F} \in \Shv^\ger{G}(\orb{M})$
    is the subset 
    \[
        \supp \sh{F} \define \{[x] \in |\orb{M}| \text{ s.t. } x^\ast \sh{F} \neq 0\} \subset |\orb{M}|.
    \]
\end{defn}

\begin{defn}
    \label{def:twistedSkyscraperSheaf}
    A \emph{$\ger{G}$-twisted skyscraper sheaf} over $\orb{M}$ is a 
    $\ger{G}$-twisted sheaf over $\orb{M}$ with finite support,
    such that the pullback along any point is finite-dimensional 
    as a $\Cinf$-module.
\end{defn}

\begin{example}
    Every $\ger{G}$-twisted skyscraper sheaf is a direct sum of $\ger{G}$-twisted
    sheaves whose support is (the equivalence class of)
    a single point $x \to \orb{M}$.
    As spelled out in Example~\ref{ex:twistedSheavesOverBGamma}, such sheaves
    correspond (non-canonically!) to twisted representations
    of the stabiliser group $\Gamma_x$.
\end{example}

Let $\ger{G} \to M$ be a gerbe over a manifold.
A choice of trivialisation $\tau_x$ of $x^\ast\ger{G}$ for each point $x: \point \into M$
induces an equivalence of categories between the category of
$\ger{G}$-twisted skyscraper sheaves and untwisted skyscraper sheaves over $M$.
This equivalence is not canonical: different choices of trivialisations
lead to different equivalences. When $\ger{G}$ is non-trivial,
it does not admit a global trivialisation, and the $\tau_x$
cannot be chosen to vary smoothly with $x \in M$.

The difference between twisted and untwisted skyscraper sheaves is detected by families of such
sheaves.
We directly describe the \emph{stack of families of $\ger{G}$-twisted sheaves over $\orb{M}$}.
It assigns
\begin{align*}
    \Shv^\ger{G}_\orb{M}: \Man^\opp & \to \VCat                                               \\
    U                               & \mapsto \Shv^{p_\orb{M}^\ast\ger{G}}(\orb{M} \times U),
\end{align*}
where $p_\orb{M}: \orb{M} \times U \to \orb{M}$ is the usual projection.
A morphisms $f:V \to U \in \Man$ is sent to the pullback morphism
\begin{equation*}
    {(\id_\orb{M} \times f)}^\ast:
    \Shv^{p_\orb{M}^\ast\ger{G}}(\orb{M} \times U) \to
    \Shv^{p_\orb{M}^\ast\ger{G}}(\orb{M} \times V).
\end{equation*}
This assignment uses the isomorphism
\begin{equation*}
    {\left((\id_\orb{M} \times f) \comp p_\orb{M}\right)}^\ast \ger{G}
    \eqto p_\orb{M}^\ast\ger{G}
\end{equation*}
and the associated canonical equivalence
\begin{equation*}
    \Shv^{{\left((\id_\orb{M} \times f) \comp\ p_\orb{M}\right)}^\ast \ger{G}}
    (\orb{M} \times V)
    \eqto \Shv^{p_\orb{M}^\ast\ger{G}}(\orb{M} \times V).
\end{equation*}
\begin{notation}
    To avoid clutter, we drop the pullbacks along $p_\orb{M}$ in the superscript from now on.
\end{notation}

The category assigned to $\point \in \Man$ is the
category $\Shv^\ger{G}(\orb{M})$ of $\ger{G}$-twisted
sheaves over $\orb{M}$. The subcategory
$\Sky^\ger{G}(\orb{M}) \subset \Shv^\ger{G}(\orb{M})$
of twisted skyscraper sheaves is cut out by the condition
that the twisted sheaves have finite support.

A twisted skyscraper sheaf may be thought of as a finite
set of points $x: \point \to \orb{M}$, equipped with
finite-dimensional $\ger{G}_x$-twisted representations of their
stabiliser groups $\Gamma_x$ in $\orb{M}$.
Varying such a skyscraper sheaf smoothly should amount to
smoothly varying the points and representations involved.
The remainder of this section is dedicated to making this notion precise.

Let $U \in \Man$ be a parameterising manifold
and $\pi:Y \onto \orb{M}$ a cover of $\orb{M}$, equipped with
a trivialisation $\tau:\ger{I} \eqto \pi^\ast \ger{G}$ of the pullback gerbe.
A map $f: U \to Y$ defines a graph
\[
    \Gr(f)\define U \xto{\Delta} U \times U \xto{f \times \id_U} Y \times U.
\]
Pushforward along $\Gr(f)$ is a functor
\[\Gr(f)_\ast:\ShC(U) \to \ShC(Y \times U) \overset{\tau}{\simeq}
    \Shv^{\pi^\ast\ger{G}}(Y \times U).\]
A $U$-family of skyscraper sheaves over the point is a $U$-family of
vector spaces (ie.\ a vector bundle over $U$).
\begin{definition}
    A sheaf $\f{S} \in \ShC(Y \times U)$ is a \emph{basic} $U$-family
    of skyscraper sheaves if it decomposes as a finite direct sum of
    pushforwards of vector bundles on $U$ via the graph of a smooth function:
    \[ \f{S} \iso \dirSum_{i=0}^r \Gr(f_i)_\ast E_i \]
    for $f_i:U \to Y$, $E_i \in \Vect(U) \subset \ShC(U)$.
\end{definition}

Let $Y \onto \orb{M}$ be a cover of $\orb{M}$ and
$p:\tilde{U} \to U$ a cover of the parameterising family $U$.
The map $\pi \times p:Y \times \tilde{U} \onto \orb{M} \times U$
and the trivialisation $\tau$ of the gerbe induce an adjunction
\[\ShC(Y \times \tilde{U}) \leftrightarrows
    \Shv^\ger{G}(\orb{M} \times U).\]
We define $U$-families of twisted skyscraper sheaves over an orbifold
$\orb{M}$ in a way that is manifestly local on $\orb{M}$ and on $U$:
\begin{defn}
    \label{defn:familyOfTwistedSkyscraperSheaves}
    A twisted sheaf $\f{S} \in \Shv^\ger{G}_\orb{M}(U)$
    is a \emph{$U$-family of $\ger{G}$-twisted skyscraper sheaves on $\orb{M}$}
    if there exists a cover $Y \onto \orb{M}$ with a trivialisation of the pullback gerbe,
    and a cover $\tilde{U} \onto U$,
    such that the pullback of $\f{S}$ to $Y \times \tilde{U}$ is a basic
    $\tilde{U}$-family.

    We denote the subcategory of $U$-families of skyscraper sheaves by
    \[\Sky^\ger{G}_\orb{M}(U) \subset \Shv^\ger{G}_\orb{M}(U).\]
\end{defn}

\begin{example}
    \label{ex:SkyOfPoint}
    Over the point $\point \in \Man$ (equipped with the trivial gerbe $\ger{I}$),
    the condition of being a
    basic $U$-family of skyscraper sheaves in $\ShC(\point \times U)=\ShC(U)$
    is equivalent to that of being a vector bundle over $U$.
    The stack of skyscraper sheaves over the point is thus
    \begin{equation*}
        \Sky_\point \simeq \Vect,
    \end{equation*}
    the stack of vector bundles over $\Man$.
    The functor $\ShC$ takes coproducts $\coprod_i (\ger{G}_i \to M_i)$ of manifolds equipped
    with gerbes to products of categories, and hence so does $\Sky$.
    In particular, for a finite set $S \in \Man$,
    \begin{equation*}
        \Sky_S \simeq \Vect^{\dirSum S},
    \end{equation*}
    the direct sum of $S$ copies of $\Vect$.
\end{example}

The restriction of a $U$-family of
$\ger{G}$-twisted skyscraper sheaves on $\orb{M}$ along any point $\point \to U$
is a $\ger{G}$-twisted skyscraper sheaf on $\orb{M}$. Up to isomorphism, this is a
finite set of points $x_i:\point \to \orb{M}$, together with finite-dimensional
$\theta_i$-twisted representations of the stabilisers $\Gamma_{x_i}$, where $\theta_i$ denotes
a 2-cocycle representing the restriction of the gerbe $\ger{G}$ to $[\point/\Gamma_{x_i}]$.
The definition of $U$-families in particular implies that the total rank (the sum of the dimensions
of the representations involved) is locally constant across $U$.

\begin{example}
    Some consequences of the definition may already be seen by looking at the support
    of a $U$-family of skyscraper sheaves. As an object
    $\f{S} \in \Shv^{\ger{G}}(U \times \orb{M})$, 
    the support of a $U$-family is a subset of $U \times |\orb{M}|$.
    The support of a pushforward sheaf $\Gr(f)_\ast E$ is given by
    the graph of $|f|:U \to \orb{M}$. This means the support of a family
    of twisted skyscraper sheaves must be locally (in the parameterising 
    family $U$) given by the
    union of finitely many such graphs.
    The reader may like to think of them as graphs of smooth,
    multivalued functions $U \to |\orb{M}|$.
    The following subsets can arise as the support of $\f{S}$
    \begin{center}
        \begin{tikzpicture}
            \begin{scope}[shift={(0,0)}]
                \draw (0,0) -- node[left] {$|\orb{M}|$} (0,3);
                \draw (0,0) -- node[below] {$U$} (3,0);
                \draw (0,2) .. controls (1,-1) and (2,4) .. (3,1);
            \end{scope}
            \begin{scope}[shift={(5,0)}]
                \draw (0,0) -- node[left] {$|\orb{M}|$} (0,3);
                \draw (0,0) -- node[below] {$U$} (3,0);
                \draw (0,1) -- (1,1);
                \draw (1,1) .. controls (2,1) and (1,2.5) .. (3,2.5);
                \draw (0,1) -- (3,1);
            \end{scope}
            \begin{scope}[shift={(10,0)}]
                \draw (0,0) -- node[left] {$|\orb{M}|$} (0,3);
                \draw (0,0) -- node[below] {$U$} (3,0);
                \draw (0,2) .. controls (1,2) and (2,1) .. (3,1);
                \draw (0,0.5) .. controls (1,0.5) and (2, 2.5) .. (3,2.5);
            \end{scope}
        \end{tikzpicture}
    \end{center}
    while the subsets below may not.
    \begin{center}
        \begin{tikzpicture}
            \begin{scope}[shift={(0,0)}]
                \draw (0,0) -- node[left] {$|\orb{M}|$} (0,3);
                \draw (0,0) -- node[below] {$U$} (3,0);
                \draw (0,1.5) .. controls (1,1.5) and (1,1) .. (1.5,1);
                \draw (1.5,2) .. controls (2,2) and (2, 1) .. (3,1);
            \end{scope}
            \begin{scope}[shift={(5,0)}]
                \draw (0,0) -- node[left] {$|\orb{M}|$} (0,3);
                \draw (0,0) -- node[below] {$U$} (3,0);
                \draw (0,2.5) .. controls (3.5,2.5) and (3.5,0.5) .. (0,0.5);
            \end{scope}
            \begin{scope}[shift={(10,0)}]
                \draw (0,0) -- node[left] {$|\orb{M}|$} (0,3);
                \draw (0,0) -- node[below] {$U$} (3,0);
                \draw (0,1.5) .. controls (1,1.5) and (3,3) .. (1.5,1.5);
                \draw (1.5,1.5) .. controls (0,0) and (2,1.5) .. (3,1.5);
            \end{scope}
        \end{tikzpicture}
        \begin{tikzpicture}
            \begin{scope}[shift={(0,0)}]
                \draw (0,0) -- node[left] {$|\orb{M}|$} (0,3);
                \draw (0,0) -- node[below] {$U$} (3,0);
                \draw (0,2) -- (1.5,2);
                \draw (1.5,2) -- (3,1);
            \end{scope}
            \begin{scope}[shift={(5,0)}]
                \draw (0,0) -- node[left] {$|\orb{M}|$} (0,3);
                \draw (0,0) -- node[below] {$U$} (3,0);
                \draw (0,1.5) .. controls (1,1.5) and (1.5,4) .. (1.5,2);
                \draw (1.5,2) -- (1.5,1);
                \draw (1.5,1) .. controls (1.5,-1) and (2,1.5) .. (3,1.5);
            \end{scope}
            \begin{scope}[shift={(10,0)}]
                \draw (0,0) -- node[left] {$|\orb{M}|$} (0,3);
                \draw (0,0) -- node[below] {$U$} (3,0);
                \foreach \y in {1,1.5,2,...,8}{
                        \draw (0,0.5+2/\y) -- (3,0.5+2/\y);
                    }
            \end{scope}
        \end{tikzpicture}
    \end{center}
\end{example}

\begin{defn}
    \label{def:stackOfTwistedSkysheaves}
    The \emph{stack of $\ger{G}$-twisted skyscraper sheaves}
    is the full substack
    \[
        \Sky^\ger{G}_\orb{M} \subset \Shv^\ger{G}_\orb{M}
    \]
    which assigns to $U \in \Man$ the category of
    $U$-families of $\ger{G}$-twisted skyscraper sheaves
    on $\orb{M}$.
\end{defn}

\begin{lemma}
    The condition above indeed cuts out a substack.
\end{lemma}
\begin{proof}
    The condition of being a $U$-family is manifestly local in $U$, so 
    it is preserved under gluing.
    We prove that the condition of being a family of skyscraper sheaves is preserved
    by the pullback morphisms in $\Shv^\ger{G}_\orb{M}$.
    Let $\f{S} \in \Sky^\ger{G}_\orb{M}(U) \subset \Shv^\ger{G}(\orb{M} \times U)$.
    By Definition~\ref{defn:familyOfTwistedSkyscraperSheaves},
    there exists a product cover
    \[
        \pi \times p: Y \times \tilde{U} \onto \orb{M} \times U
    \]
    with a trivialisation $\tau$ of the pullback gerbe,
    such that ${(\pi \times p)}^\ast\f{S}$ decomposes as a finite direct sum
    \[
        {(\pi \times p)}^\ast\f{S} \iso
        \DirSum_{i=0}^r \tau \cdot {\Gr(f_i)}_\ast E_i,
    \]
    where $f_i:\tilde{U} \to Y$ are smooth maps and
    $E_i \onto \tilde{U}$ vector bundles over $\tilde{U}$.

    Consider the pullback morphism ${(\id_{\orb{M}}\times g)}^\ast$
    associated to a map $g:V \to U$.
    The product cover $\pi \times p$ pulls back to a product cover
    $\pi \times g^\ast p:Y \times \tilde{V} \onto \orb{M} \times V$:
    % https://q.uiver.app/?q=WzAsNCxbMSwxLCJcXG9yYntNfSBcXHRpbWVzIFUiXSxbMSwwLCJZIFxcdGltZXNcXHRpbGRle1V9Il0sWzAsMSwiXFxvcmJ7TX1cXHRpbWVzIFYiXSxbMCwwLCJZIFxcdGltZXMgXFx0aWxkZXtWfSJdLFsxLDAsIlxccGkgXFx0aW1lcyBcXHBpJyJdLFszLDIsIlxccGlcXHRpbWVzIGdeXFxhc3RcXHBpJyIsMl0sWzMsMSwiXFxpZF9ZIFxcdGltZXMgXFx0aWxkZXtnfSJdLFsyLDAsIlxcaWRfe1xcb3Jie019fVxcdGltZXMgZyIsMl0sWzMsMCwiIiwxLHsic3R5bGUiOnsibmFtZSI6ImNvcm5lciJ9fV1d
    \[\begin{tikzcd}
            {Y \times \tilde{V}} & {Y \times\tilde{U}} \\
            {\orb{M}\times V} & {\orb{M} \times U}.
            \arrow["{\pi \times p}", from=1-2, to=2-2]
            \arrow["{\pi\times g^\ast p}"', from=1-1, to=2-1]
            \arrow["{\id_Y \times \tilde{g}}", from=1-1, to=1-2]
            \arrow["{\id_{\orb{M}}\times g}"', from=2-1, to=2-2]
            \arrow["\lrcorner"{anchor=center, pos=0.125}, draw=none, from=1-1, to=2-2]
        \end{tikzcd}\]
    Pullbacks compose, so
    \[
        {(\pi \times g^\ast p)}^\ast{(\id_\orb{M} \times g)}^\ast \f{S}
        \simeq
        {(\id_Y \times \tilde{g})}^\ast {(\pi \times p)}^\ast \f{S},
    \]
    which reduces the proof to showing that
    \[
        {(\id_Y \times \tilde{g})}^\ast (\tau \cdot {\Gr(f_i)}_\ast E_i)
    \]
    is a basic $\tilde{V}$-family of skyscraper sheaves.
    This follows from base change (Lemma~\ref{lem:baseChange}) along the pullback diagram
    % https://q.uiver.app/?q=WzAsNCxbMCwwLCJcXHRpbGRle1Z9Il0sWzEsMCwiXFx0aWxkZXtVfSJdLFswLDEsIlkgXFx0aW1lcyBcXHRpbGRle1Z9Il0sWzEsMSwiWVxcdGltZXMgXFx0aWxkZXtVfSJdLFswLDEsIlxcdGlsZGV7Z30iXSxbMCwyLCJ7XFxHcihmX2lcXGNvbXAgXFx0aWxkZXtnfSl9IiwyXSxbMSwzLCJcXEdyKGZfaSkiXSxbMiwzLCJcXGlkX1kgXFx0aW1lcyBcXHRpbGRle2d9IiwyXSxbMCwzLCIiLDEseyJzdHlsZSI6eyJuYW1lIjoiY29ybmVyIn19XV0=
    \[\begin{tikzcd}
            {\tilde{V}} & {\tilde{U}} \\
            {Y \times \tilde{V}} & {Y\times \tilde{U}}.
            \arrow["{\tilde{g}}", from=1-1, to=1-2]
            \arrow["{{\Gr(f_i\comp \tilde{g})}}"', from=1-1, to=2-1]
            \arrow["{\Gr(f_i)}", from=1-2, to=2-2]
            \arrow["{\id_Y \times \tilde{g}}"', from=2-1, to=2-2]
            \arrow["\lrcorner"{anchor=center, pos=0.125}, draw=none, from=1-1, to=2-2]
        \end{tikzcd}\]
    It yields an isomorphism
    \[
        {(\id_Y \times \tilde{g})}^\ast \Gr(f_i)_\ast E_i \simeq
        \Gr{(f_i \comp \tilde{g})}_\ast \tilde{g}^\ast E_i.
    \]
    The right hand side is the pushforward of the vector bundle $\tilde{g}^\ast E_i \to V$
    along the graph of the smooth map $f_i\comp \tilde{g}:V \to Y$,
    and induces a basic $V$-family of skyscraper sheaves
    \[
        ({(\id_Y \times \tilde{g})}^\ast \tau) \cdot {\Gr(f_i \comp \tilde{g})}_\ast \tilde{g}^\ast E_i. \qedhere
    \]
\end{proof}

\subsection{Linear Stacks}
\label{sec:linearStacks}
Recall we write $\VCat$ to denote the bicategory of $\dirSum$-complete $\VecInf$-enriched categories.
It is symmetric monoidal under the tensor product of enriched
$\dirSum$-complete categories
introduced by Kelly in~\cite[Ch 6.5]{kelly1982basic}
and studied in more detail in~\cite{franco2013tensor}.
Let $\cat{A},\cat{B} \in \VCat$, then their tensor product
is obtained by first taking the ordinary product of enriched categories, and then completing
freely under direct sums.
The resulting category $\cat{A} \boxtimes \cat{B}$ has objects formal finite direct sums
$\DirSum_i a_i \boxtimes b_i$ (where $a_i \in \cat{A}, b_i \in \cat{B}$) and the $\Hom$ spaces
are extended linearly from
\begin{equation*}
    \Hom(a \boxtimes b,a' \boxtimes b') \define
    \Hom_{\cat{A}}(a,a') \tensor \Hom_{\cat{B}}(b,b').
\end{equation*}

This tensor product extends the Deligne tensor product
used in the fusion category literature~\cite[Ch 1.11]{etingof2016tensor}.
\begin{thm}[{\cite[Thm 27]{franco2013tensor}}]
    \label{thm:boxproductIsDeligneproductIfAbelian}
    Let $\cat{A},\cat{B}$ be two semisimple abelian categories
    with finite-dimensional $\Hom$-spaces and objects of finite length.
    The product $\cat{A} \boxtimes \cat{B}$ is equivalent to their Deligne
    tensor product.
\end{thm}
In particular, $\cat{A} \boxtimes \cat{B}$ is again semisimple abelian
with finite-dimensional $\Hom$-spaces and objects of finite length.

The symmetric monoidal structure on $\VCat$ induces a symmetric monoidal structure
on the bicategory $\VSt$ of $\VCat$-valued stacks over $\Man$.
The product of $\st{X}, \st{Y} \in \VSt$ is computed by first
forming the pointwise product
\[
    \st{X} \boxtimes^0 \st{Y}: U \mapsto \st{X}(U) \boxtimes \st{Y}(U),
\]
and then stackifying to obtain
\[
    \st{X} \boxtimes \st{Y} \define {(\st{X} \boxtimes^0 \st{Y})}^\shff.
\]
Every category $\cat{C} \in \VCat$ is canonically
a module over $\Vec$: ${\IC}^r \in \Vec$ acts by $X \mapsto X^{\dirSum r}$,
and morphisms act via the $\Vec$-module structure of $\VecInf$.

We may pull back $\Vec$ along $\Man \to \point$ to obtain the
locally constant stack ${\Vec}^\mathrm{const}$,
which assigns the category $\Vec$ to any Cartesian space $U \in \CartSp \subset \Man$
and the identity functor to smooth maps between them.
Any stack in $\VSt/\Man$ is canonically equipped with an action by $\Vec^\mathrm{const}$.

\begin{example}
    \label{ex:VecActionOnSkyscrapers}
    For the stack $\Sky^\ger{G}_\orb{M}$ of $\ger{G}$-twisted skyscraper sheaves
    over an orbifold $\orb{M}$, the action by a vector space $V \in \Vec$
    corresponds to tensoring twisted sheaves $\sh{F} \in \Sky^\ger{G}_\orb{M}(U)
        \subset \Shv^\ger{G}(\orb{M} \times U)$
    with the sheaf of sections $\Cinf(\orb{M} \times U, V)$
    of the trivial $V$-bundle over $\orb{M} \times U$.
    We view this as tensoring the twisted skyscraper sheaf
    over each point $u \in U$ with the vector space $V$.
\end{example}
The $\Vec$-action in Example~\ref{ex:VecActionOnSkyscrapers} admits an immediate extension:
Instead of tensoring the skyscraper sheaf over each point $u \in U$
with the same vector space $V$, we may tensor the $U$-family with
a vector bundle $P$ over $U$.
This equips the stack of skyscraper sheaves with an action by $\Vect$, the
$\Vec$-enriched stack of vector bundles.
The stack $\Vect$ is symmetric monoidal via the usual tensor product
of vector bundles.

\begin{defn}
    \label{def:linearStack}
    A \emph{linear stack} is a $\Vect$-module, ie.\
    a stack $\st{F} \in \VSt/\Man$
    equipped with an action of the stack of vector bundles.
\end{defn}
The \emph{bicategory of linear stacks} $\LinSt$ is the bicategory of $\Vect$-modules.
A 1-mor\-phism of $\Vect$-modules $\st{X},\st{Y}$ is a map
\begin{equation*}
    F: \st{X} \to \st{Y}
\end{equation*}
of underlying $\Vec$-enriched stacks, equipped with an invertible 2-cell
$s$ implementing compatibility with the $\Vect$-module structures $\odot_\st{X}, \odot_\st{Y}$.
For each $U \in \Man$, $P \in \Vect(U)$, $X \in \st{X}(U)$, this 2-cell
is an equivalence
\begin{equation*}
    s_U (X,P):F(X \odot_\st{X} P) \eqto F(X) \odot_\st{Y} P.
\end{equation*}
A 2-morphism between two 1-morphisms $(F, s), (F',s'):\st{X} \to \st{Y}$
is a 2-morphism $a:F \to F'$, satisfying the compatibility condition
% https://q.uiver.app/?q=WzAsNCxbMCwwLCJGKFggXFxvZG90IFApIl0sWzEsMCwiRihYKSBcXG9kb3QgUCJdLFswLDEsIkYnKFggXFxvZG90IFApIl0sWzEsMSwiRicoWCkgXFxvZG90IFAiXSxbMCwxLCJzX1UoWCxQKSJdLFswLDIsImEoWCBcXG9kb3QgUCkiLDJdLFsyLDMsInMnX1UoWCxQKSIsMl0sWzEsMywiYShYKSBcXG9kb3QgUCJdXQ==
\[\begin{tikzcd}
        {F(X \odot_\st{X} P)} & {F(X) \odot_\st{Y} P} \\
        {F'(X \odot_\st{X} P)} & {F'(X) \odot_\st{Y} P}
        \arrow["{s_U(X,P)}", from=1-1, to=1-2]
        \arrow["{a(X \odot P)}"', from=1-1, to=2-1]
        \arrow["{s'_U(X,P)}"', from=2-1, to=2-2]
        \arrow["{a(X) \odot P}", from=1-2, to=2-2]
    \end{tikzcd}\]
for all triples $U, P, X$ as above.

Let $X,Y \in \st{X}(U)$ be two sections of a linear stack $\st{X}$ over
$U \in \Man$.
The restriction functors upgrade $\Hom_{\st{X}(U)}(X,Y)$ to
a presheaf
\begin{alignat*}{1}
    \sHom(X,Y):{(\Man/U)}^\opp & \to \VecInf                       \\
    (f:V \to U)                & \mapsto \Hom(f^\ast X, f^\ast Y).
\end{alignat*}
The descent condition implies that this is actually a sheaf on $\Man/U$.
The $\Vect$-module structure upgrades this sheaf to a $\Cinf$-module:
We use the natural isomorphisms $Y \simeq Y \odot \Cinf$ and
$\Cinf \simeq \sHom_{\Vect}(\Cinf,\Cinf)$, to write down the action map
\[
    \sHom(X,Y) \tensor \Cinf \simeq \sHom(X,Y \odot \Cinf) \tensor
    \sHom(\Cinf,\Cinf) \to \sHom(X,Y).
\]
\begin{lemma}
    \label{lem:vectorBundlesActOnHomSheavesOfLinearStacks}
    The action map extends to an isomorphism of $\Cinf$-modules
    \[
        \sHom(X,Y) \tensor_{\Cinf} P \simeq \sHom(X,Y \odot P)
    \]
    natural in $X,Y \in \st{X}(U)$ and $P \in \Vect(U)$.
\end{lemma}
\begin{proof}
    We use the natural isomorphism $P \simeq \sHom_\Vect(\Cinf,P)$
    to build the map of $\Cinf$-modules
    \[
        \sHom(X,Y \odot \Cinf) \tensor_{\Cinf} \sHom_\Vect(\Cinf,P)
        \to \sHom(X, Y \odot P).
    \]
    It is manifestly natural in $X,Y$ and $P$.
    That this map is an isomorphism may be checked locally.
    Hence, we may assume $P \simeq {(\Cinf)}^{\dirSum r}$,
    and the two sides are isomorphic to
    \[
        {\big(\sHom(X,Y)\big)}^{\dirSum r} \simeq 
        \sHom(X, Y^{\dirSum r}).
        \qedhere
    \]
\end{proof}

The monoidal product of two linear stacks $\st{X},\st{Y}$ is
computed by first computing the pointwise product
$\st{X} \boxtimes^0_\Vect \st{Y}(U)$: Its objects are formal direct sums
$\DirSum_i X_i \boxtimes Y_i$ for $X_i \in \st{X}(U), Y_i \in \st{Y}(U)$,
and the morphism spaces are extended linearly from
\[
    \Hom(X\boxtimes Y,X' \boxtimes Y') \define \Hom(X,X') \tensor_{\Cinf(U)}
    \Hom(Y,Y').
\]
\begin{defn}
    The product of linear stacks $\st{X},\st{Y}$ is the linear stack
    \[
        \st{X} \boxtimes_\Vect \st{Y} \define {(\st{X} \boxtimes_\Vect^0 \st{Y})}^\shff.
    \]
\end{defn}

Note there is a map
$\st{X} \boxtimes \st{Y} \to \st{X} \boxtimes^0_\Vect \st{Y}$
given by the identity on objects, and by the natural quotient map
\begin{align*}
    \Hom(X,X') \tensor_\IC \Hom(Y,Y') \to
    \Hom(X,X') \tensor_{\Cinf(U)} \Hom(Y,Y')
\end{align*}
on $\Hom$-spaces over $U \in \Man$.
Postcomposing with the natural map from a stack to its stackification yields the quotient map
\begin{equation*}
    \st{X} \boxtimes \st{Y} \to \st{X} \boxtimes_\Vect \st{Y}.
\end{equation*}

\begin{lemma}
    \label{lem:specifyVectModuleMapEasily}
    A map of $\Vect$-modules
    \[
        F:\st{X} \boxtimes_\Vect \st{Y} \to \st{Z}
    \]
    is specified (up to equivalence) by a morphism of $\Vec$-modules
    \[
        F^0:\st{X} \boxtimes^0 \st{Y} \to \st{Z},
    \]
    subject to the following condition:
    For every parameterising manifold $U \in \Man$ and
    object $x \boxtimes y \in \st{X}(U) \boxtimes \st{Y}(U)$,
    the map induced by $F^0_U$ on the endomorphisms
    \[
        \End(x \boxtimes y) = \End(x) \tensor \End(y) \to \End(F^0_U(x \boxtimes y))
    \]
    is $\Cinf$-bilinear, ie.\ $F^0_U$ sends both
    $c \cdot \id_x \tensor id_y$ and $\id_x \tensor c \cdot \id_y$ to $c \cdot \id_{F^0_U(x \boxtimes y)}$
    for all smooth functions $c \in \Cinf(U)$.
\end{lemma}
\begin{proof}
    The condition on $F^0$ ensures that the morphism
    $F^0:\st{X} \boxtimes^0 \st{Y} \to \st{Z}$
    factors through the quotient map
    \[
        \st{X} \boxtimes^0 \st{Y} \to \st{X} \boxtimes_\Vect^0 \st{Y}.
    \]
    By the universal property of stackfication, the induced
    map $\st{X} \boxtimes_\Vect^0 \st{Y} \to \st{Z}$
    extends to a map $\st{X} \boxtimes_\Vect \st{Y} \to \st{Z}$,
    unique up to an invertible 2-cell.
\end{proof}

\begin{lemma}
    \label{lem:boxVectPresentsProductOverVect}
    The product $\st{X} \boxtimes_\Vect \st{Y}$ presents the colimit of the diagram
    % https://q.uiver.app/?q=WzAsMyxbMCwwLCIgICAgICAgIFxcc3R7WH0gXFxib3h0aW1lcyBcXFZlY3QgXFxib3h0aW1lcyBcXFZlY3QgXFxib3h0aW1lcyBcXHN0e1l9Il0sWzEsMCwiICAgICAgICBcXHN0e1h9IFxcYm94dGltZXMgXFxWZWN0IFxcYm94dGltZXMgXFxzdHtZfSJdLFsyLDAsIiAgICAgICAgXFxzdHtYfSBcXGJveHRpbWVzIFxcc3R7WX0iXSxbMCwxLCIiLDAseyJvZmZzZXQiOjR9XSxbMCwxLCIiLDIseyJvZmZzZXQiOi00fV0sWzAsMV0sWzEsMiwiIiwxLHsib2Zmc2V0IjoyfV0sWzEsMiwiIiwxLHsib2Zmc2V0IjotMn1dLFsyLDFdLFsxLDAsIiIsMCx7Im9mZnNldCI6Mn1dLFsxLDAsIiIsMCx7Im9mZnNldCI6LTJ9XV0=
    \[\begin{tikzcd}
            {        \st{X} \boxtimes \Vect \boxtimes \Vect \boxtimes \st{Y}} & {        \st{X} \boxtimes \Vect \boxtimes \st{Y}} & {        \st{X} \boxtimes \st{Y}}
            \arrow[shift right=4, from=1-1, to=1-2]
            \arrow[shift left=4, from=1-1, to=1-2]
            \arrow[from=1-1, to=1-2]
            \arrow[shift right=2, from=1-2, to=1-3]
            \arrow[shift left=2, from=1-2, to=1-3]
            \arrow[from=1-3, to=1-2]
            \arrow[shift right=2, from=1-2, to=1-1]
            \arrow[shift left=2, from=1-2, to=1-1]
        \end{tikzcd}\]
    coequalising the action of $\Vect$ on $\st{X}$ and $\st{Y}$.\footnote{
        The splitting maps in the simplicial diagram are
        given by the unit $\Cinf \in \Vect$. We suppress 2-cells.}
\end{lemma}
\begin{proof}
    We show this by a direct computation.
    It is enough to check the statement on 2-presheaves: as a left adjoint, stackification
    preserves colimits.

    Denote the colimit of the diagram, evaluated at $U \in \Man$ by $C(U)$.
    By the universal property of the colimit, a functor
    $F:C(U) \to \cat{D}$ into a category $\cat{D} \in \VCat$
    is the data of a functor
    $F:\st{X} \boxtimes \st{Y} \to \cat{D}$, equipped with an invertible 2-cell
    $\kappa$ whose components
    \[
        \kappa_U(P,X,Y): F\left((X \odot P) \boxtimes Y\right)
        \to \left(X \boxtimes (Y \odot P)\right),
    \]
    are natural in $P \in \Vect(U)$, $X \in \st{X}(U)$, $Y \in \st{Y}(U)$.
    This 2-cell is required to be multiplicative, in that
    % https://q.uiver.app/?q=WzAsNCxbMCwxLCJGXFxsZWZ0KCAoWCBcXG9kb3QgKFAgXFx0ZW5zb3IgUSkpIFxcYm94dGltZXMgWSBcXHJpZ2h0KSJdLFsxLDAsIkZcXGxlZnQoKFggXFxvZG90IFApIFxcYm94dGltZXMgKFkgXFxvZG90IFEpXFxyaWdodCkiXSxbMSwxLCJGXFxsZWZ0KFggXFxib3h0aW1lcyAoWSBcXG9kb3QgKFAgXFx0ZW5zb3IgUSkpXFxyaWdodCkiXSxbMCwwLCJGXFxsZWZ0KChYIFxcb2RvdCBQKSBcXG9kb3QgUSkpIFxcYm94dGltZXMgWVxccmlnaHQpIl0sWzAsMiwiXFxrYXBwYV9VKFgsUFxcdGVuc29yIFEsWSkiLDJdLFsxLDIsIlxca2FwcGFfVShYLFAsWVxcb2RvdCBRKSJdLFszLDEsIlxca2FwcGFfVShYXFxvZG90IFAsUSxZKSJdLFszLDAsIlxcc2ltZXEiLDJdXQ==
    \[\begin{tikzcd}[column sep=huge]
            {F\left(((X \odot P) \odot Q) \boxtimes Y\right)} & {F\left((X \odot P) \boxtimes (Y \odot Q)\right)} \\
            {F\left( (X \odot (P \tensor Q)) \boxtimes Y \right)} & {F\left(X \boxtimes (Y \odot (P \tensor Q))\right)}
            \arrow["{\kappa_U(X,P\tensor Q,Y)}"', from=2-1, to=2-2]
            \arrow["{\kappa_U(X,P,Y\odot Q)}", from=1-2, to=2-2]
            \arrow["{\kappa_U(X\odot P,Q,Y)}", from=1-1, to=1-2]
            \arrow["\simeq"', from=1-1, to=2-1]
        \end{tikzcd}\]
    commutes (the unlabeled morphism is part of the $\Vect$-module structure). Further, it is required to be unital:
    The triangle
    % https://q.uiver.app/?q=WzAsMyxbMSwwLCJGXFxsZWZ0KFggXFxib3h0aW1lcyBZXFxyaWdodCkiXSxbMCwxLCJGXFxsZWZ0KChYIFxcb2RvdCBcXENpbmYpIFxcYm94dGltZXMgWVxccmlnaHQpIl0sWzIsMSwiRlxcbGVmdChYIFxcYm94dGltZXMgKFkgXFxvZG90IFxcQ2luZilcXHJpZ2h0KSJdLFsxLDIsIlxca2FwcGFfXFxDaW5mIl0sWzAsMV0sWzAsMl1d
    \[\begin{tikzcd}
            & {F\left(X \boxtimes Y\right)} \\
            {F\left((X \odot \Cinf) \boxtimes Y\right)} && {F\left(X \boxtimes (Y \odot \Cinf)\right)}
            \arrow["{\kappa_U(X,\Cinf,Y)}"', from=2-1, to=2-3]
            \arrow[from=1-2, to=2-1]
            \arrow[from=1-2, to=2-3]
        \end{tikzcd}\]
    whose top sides are given by the unitors of the $\Vect$-action, commutes.
    A map $\underline{F}:C \to \underline{\cat{D}}$ from the colimit to a
    $\VCat$-valued stack $\underline{\cat{D}}$ carries this structure
    pointwise, and requires the 2-cells to be compatible under the pullback morphisms.

    The unitality and naturality of the 2-cell $\kappa$ imply that the induced functor
    $\underline{F}^0:\st{X} \boxtimes^0 \st{Y} \to \underline{\cat{D}}$ satisfies the
    condition of Lemma~\ref{lem:specifyVectModuleMapEasily}.
    Hence, $\underline{F}$
    factors through a morphism $\st{X} \boxtimes_\Vect \st{Y} \to \underline{\cat{D}}$.

    The construction above provides a natural functor
    \[
        \Hom(C,\underline{\cat{D}}) \to \Hom(\st{X} \boxtimes_\Vect \st{Y},\underline{\cat{D}}).
    \]
    We now construct a natural inverse to this functor.
    Let $\underline{F}': \st{X} \boxtimes_\Vect \st{Y} \to \underline{\cat{D}}$ be a map
    of $\VCat$-valued stacks.
    We need to give the data of a map $\st{X} \boxtimes \st{Y} \to
        \underline{\cat{D}}$ and a 2-cell $\kappa$ as above.
    The former is obtained by precomposing $\underline{F}'$ with the quotient
    morphism $\st{X} \boxtimes \st{Y} \to \st{X} \boxtimes_\Vect \st{Y}$.
    The latter is induced by the natural isomorphism
    \[
        (X \odot P) \boxtimes Y \simeq X \boxtimes (Y \odot P),
    \]
    provided by Lemma~\ref{lem:vectorBundlesActOnHomSheavesOfLinearStacks}:
    \begin{align*}
        \sHom\left(A \boxtimes B, (X \odot P) \boxtimes Y\right) & =
        \sHom(A,X \odot P) \tensor_\Cinf \sHom(B,Y)                                                                               \\
                                                                 & \simeq \sHom(A,X) \tensor_\Cinf P \tensor_\Cinf \sHom(B,Y)     \\
                                                                 & \simeq \sHom\left(A\boxtimes B, X\boxtimes (Y \odot P)\right).
    \end{align*}
    (This is natural in $A,X \in \st{X}(U)$, $B,Y \in \st{Y}(U)$ and $P \in \Vect(U)$.)
    The isomorphism above is one of $\Cinf$-modules, which implies
    compatibility with the pullback functors of the stack.
    The natural isomorphism is also compatible with the tensor product of
    vector bundles, and unital. As such, we obtain a map
    $C \to \underline{\cat{D}}$.

    Thus far, we have constructed natural maps
    $\Hom(C,\underline{\cat{D}}) \leftrightarrows \Hom(\st{X}\boxtimes_\Vect\st{Y},\underline{\cat{D}})$,
    By the co-Yoneda Lemma, these are really maps $C \leftrightarrows \st{X}\boxtimes_\Vect\st{Y}$.
    It is straightforward to check they are mutual inverses.
\end{proof}

Given a category $\cat{D} \in \VCat$,
there is an induced linear stack
\[
    \underline{\cat{D}} \define \cat{D}^\mathrm{const} \boxtimes \Vect.
\]

\begin{prop}
    \label{prop:boxtimesVectTurnsVecModulesIntoVectModules}
    The assignment
    \[
        \cat{D} \mapsto \underline{\cat{D}}
    \]
    extends to a fully faithful, monoidal embedding
    \[
        \VCat \into \LinSt.
    \]
\end{prop}
\begin{proof}
    The map induced on $\Hom$-spaces by the above assignment is
    \begin{alignat*}{1}
        \Hom_{\VCat}(\cat{C},\cat{D}) & \to
        \Hom_{\Mod_\Vect}(\underline{\cat{C}},\underline{\cat{D}})     \\
        F                             & \mapsto F \boxtimes \Id_\Vect.
    \end{alignat*}
    This is an equivalence of bicategories.
    We exhibit an inverse by sending a map of stacks 
    $G: \underline{\cat{C}} \to \underline{\cat{D}}$
    to the functor
    \[
        G_0: \cat{C} \eqto \cat{C} \boxtimes \Vec \xto{G(\point)}
        \cat{D} \boxtimes \Vec \eqot \cat{D}.
    \]
    Here, we make explicit use of the equivalence
    $\underline{\cat{C}}(\point) \simeq \cat{C} \boxtimes \Vec$, and
    the unitor $\cat{C} \simeq \cat{C} \boxtimes \Vec$.\footnote{
        The unitor is part of the canonical $\Vec$-module structure on objects of $\VCat$.
    }
    The composite
    \[
        F \mapsto F \boxtimes \Id_\Vect \mapsto {(F \boxtimes \Id_\Vect)}_0 \simeq F
    \]
    is the identity on the nose.
    In the other direction, $G \mapsto G_0 \mapsto G_0 \boxtimes \Id_\Vect$.
    There is a natural equivalence between $G$ and $G_0 \boxtimes \Id_\Vect$,
    whose component at $c \boxtimes P \in \underline{\cat{C}}(U)$ is
    \begin{alignat*}{2}
        G^0(c) \boxtimes P \simeq \left(G^0(c) \boxtimes \Cinf(U)\right) \odot P
        \simeq G\left(c \boxtimes \Cinf(U)\right) \odot P \simeq G(c \boxtimes P).
    \end{alignat*}
    The outer two equivalences use the $\Vect$-module structure of $G$,
    and the middle equivalence uses compatiblity with pullback along the map
    $U \to \point$.

    The monoidal structure of the 2-functor is constructed via the equivalence
    \[
        (\cat{C}^\mathrm{const} \boxtimes \Vect) \boxtimes_\Vect
        (\cat{D}^\mathrm{const} \boxtimes \Vect)
        \simeq {(\cat{C} \boxtimes \cat{D})}^\mathrm{const} \boxtimes (\Vect \boxtimes_\Vect \Vect),
    \]
    which is natural in $\cat{C}$ and $\cat{D}$.
    By Lemma~\ref{lem:boxVectPresentsProductOverVect},
    \[
        \Vect \boxtimes_\Vect \Vect \simeq \Vect,
    \]
    and we get the required 2-cell
    \[
        \underline{\cat{C}} \boxtimes_\Vect \underline{\cat{D}}
        \eqto \underline{\cat{C} \boxtimes \cat{D}}.
    \]
    The remaining pieces of data for the monoidal 2-functor
    is the map of monoidal units
    \[
        \Vect \simeq \Vec^\mathrm{const} \boxtimes \Vect,
    \]
    and 2-cells relating the associators and unitors
    (see~\cite[Def 3.3]{houston2013linear}).
    The coherence conditions for a functor of monoidal bicategories
    are difficult to check, but in this case vastly simplified by the fact
    that the associators in $\VCat$ and $\LinSt$ can be chosen to be identities.
\end{proof}

\begin{rmk}
    Proposition~\ref{prop:boxtimesVectTurnsVecModulesIntoVectModules} says that
    the functor ${(-)}^{\mathrm{const}} \boxtimes \Vect$ freely turns
    $\Vec$-modules into $\Vect$-modules.
\end{rmk}

\section{Orbisimple Categories}
\label{sec:orbisimples}
A fusion category is a linear monoidal category $\cat{C}$ whose underlying
linear category admits an equivalence
\[
    \cat{C} \overset{\mathrm{lin.}}{\iso} \Vec^{\dirSum S},
\]
where $S$ is a finite set. To define orbifold tensor categories, we
would like to replace the right hand side with
a symbol $\Vec^{\dirSum \orb{M}}$, where $\orb{M}$ is an orbifold.
The corresponding object must keep track of the smooth structure of $\orb{M}$.
Its role in this story is played by the stack
$\Sky^\ger{G}_\orb{M}$ of Definition~\ref{def:stackOfTwistedSkysheaves}.

\begin{defn}
    \label{def:orbisimpleCat}
    An \emph{orbisimple category} is a linear stack over $\Man$ equivalent to
    ${\Sky}_{\orb{M}}^{\ger{G}}$ for some orbifold $\orb{M}$ and gerbe $\ger{G}$ over $\orb{M}$.
    If $\orb{M}$ can be represented by a manifold, we also call such a stack a \emph{manisimple category}.
\end{defn}
We use the adjective \emph{compact} when $\orb{M}$ can be chosen to be compact.

\begin{remark}
    Orbisimple and manisimple are shortenings of the more descriptive but unwieldy terms
    \emph{orbifold/manifold semisimple}.
\end{remark}

In this section, we establish a number of structural results about orbisimple categories.
We begin with the study of the underlying linear semisimple category,
the value of an orbisimple category on the point.
The category of $U$-families of objects is no longer semisimple for
non-trivial $U$, and not all morphisms $\f{A} \to \f{B}$ admit an epi-mono/image factorisation
$\f{A} \onto \f{C} \into \f{B}$.
The central result of this section is the
fact that \emph{locally}, such factorisations still exist
(Proposition~\ref{prop:local_factorisation}).

\subsection{The Category of Points}
Let $\orb{M}$ be an orbifold, and $\ger{G}$ a gerbe over it. We study
the linear stack $\underline{\cat{C}} \define \Sky_\orb{M}^\ger{G}$, and in particular
$\cat{C} \define \underline{\cat{C}}(\point)$, the category of
$\ger{G}$-twisted skyscraper sheaves on $\orb{M}$.
By definition, any orbisimple category admits an equivalence
to such a stack.

\begin{notation}
    We denote by ${\cat{C}}^U \define \stC(U)$ the category of $U$-families.
    An object $\f{S} \in {\cat{C}}^U$ is also suggestively denoted by
    $\f{S}: U \to \cat{C}$.
    Given a point $x:\point \to U$, we denote the pullback by
    $\f{S}(x) \define x^\ast\f{S}$.
    For the pullback along a map $V \to U$, we write $\f{S}\restriction_V$.
    We will use the same notation for morphisms of families and their pullbacks.
\end{notation}

%================ STRUCTURE OVER POINTS ====================================

Let $x:\point \to \st{M}$ be a point of $\st{M}$.
We write $\Gamma_x \define \End(x)$ for its stabiliser group (which is finite
by assumption). We denote the restriction of a gerbe $\ger{G} \to \orb{M}$
along $x$ by $\ger{G}_x = x^\ast \ger{G}$. By
Corollary~\ref{cor:classificationOfGerbesOnQuotientOrbifolds},
the gerbe is locally (around $x$) characterised by a class
$[\theta_x] \in \H^2(\B\Gamma_x,\IC^\times)$. We write $\theta_x$ to denote a
representing 2-cocycle for this cohomology class.

\begin{defn}
    We call $\cat{C}^{(x)} \define (x^\ast\stC)(\point)$ the \emph{category at $x$}.
\end{defn}

The category at $x$ includes into $\cat{C}$ as the full subcategory on
skyscraper sheaves whose support is concentrated at $[x]$.
We abuse notation and denote by $\cat{C}^{(x)}$ its essential image under this
inclusion.
There are no morphisms between sheaves with disjoint support,
so non-isomorphic points $x \not\iso y$ of $\orb{M}$ yield
disjoint full subcategories ${\cat{C}}^{(x)} \cap {\cat{C}}^{(y)} = \{0\}$.
Meanwhile, an isomorphism $\phi:x \eqto x'$ induces an equivalence
$\phi^\ast:x^\ast \eqto x'^\ast$, and by extension ${\cat{C}}^{(x)} \simeq {\cat{C}}^{(x')}$.

\begin{lemma}
    \label{lem:cat_at_point_props}
    Let $x: \point \to \orb{M}$ be a point of $\orb{M}$ with stabiliser $\Gamma_x$,
    and a representing 2-cocycle $\theta_x$ for $\ger{G}_x$.
    Denote by $\IC$ the canonical line bundle over the point.
    Then
    \begin{enumerate}
        \item $\cat{C}^{(x)} \iso \Rep^{\theta_x}(\Gamma_x)$ as linear categories,
              such that $x_\ast{\IC} \iso \IC^{\theta_x}[\Gamma_x]$
        \item  ${\cat{C}}^{(x)} \subset \cat{C}$ is the subcategory
              Karoubi-generated by $x_\ast{\IC}$
        \item Every simple object of $\cat{C}$ is contained in
              ${\cat{C}}^{(x)}$ for some point $x: \point \to \orb{M}$.
    \end{enumerate}
\end{lemma}
\begin{proof}
    Property 1 is discussed in Example~\ref{ex:twistedSheavesOverBGamma},
    and Property 2 is Corollary~\ref{cor:KaroubiCompletionOfInductionIsAll}.
    A skyscraper
    sheaf with support at multiple non-isomorphic points splits into a
    direct sum of skyscraper sheaves with support at a single isomorphism
    class. Hence, a simple skyscraper sheaf must be contained in some
    ${\cat{C}}^{(x)}$, which establishes Property 3.
\end{proof}

Denote by $\Irr{\cat{C}}$ the set of isomorphism classes of simple objects of
$\cat{C}$.
\begin{cor}
    \label{cor:catOfPointsDirSumOfRepCats}
    A choice of representing points $\{x: \point \to \orb{M}\}$ and
    representing 2-cocycles $\theta_x$ for all
    $[x] \in |\orb{M}|$ yields an explicit bijection of sets
    \begin{equation*}
        \Irr{\cat{C}} \isoto \coprod_{[x] \in |\orb{M}|} \Irr{{\cat{C}}^{(x)}}
        = \coprod_{[x] \in |\orb{M}|} \Irr{\Rep^{\theta_x}(\Gamma_x)},
    \end{equation*}
    and in fact,
    \begin{equation*}
        \cat{C} \isoto \DirSum _{[x] \in |\orb{M}|} \Rep^{\theta_x}(\Gamma_x)
    \end{equation*}
    as linear categories.
    In particular, $\cat{C}$ is a direct sum of semisimple abelian categories, and
    hence semisimple and abelian itself.
\end{cor}

A choice of representatives for $|\orb{M}|$ associates to a simple object
$X \in \cat{C}$ a point $x$ and an irreducible $\theta_x$-twisted representation
$\rho$ of its stabiliser $\Gamma_x$.
The dimension of the underlying vector space of $\rho$ and $|\Gamma_x|$ are independent
of the choices made.
\begin{defn}
\label{def:dimensionOfObject}
    With notation as above, the \emph{dimension} of a simple object is
    \begin{equation*}
        \dim{X} \define \frac{\vdim{\rho}}{|\Gamma_x|},
    \end{equation*}
    where $\vdim$ denotes the \emph{vector space dimension} of the
    representation $\rho$.
    This is extended linearly to all objects in $\cat{C}$.
\end{defn}

\begin{example}
    Consider the orbifold $\orb{M}=[\IR/(\IZ/2)]$, equipped with the trivial gerbe.
    Its quotient space is $|\orb{M}| = \IR^+_0$.
    The category $\Sky_{\orb{M}}(\point)$ is generated under direct sum by
    a half-line of objects $\IC_{x>0}$, and a pair of objects
    $\IC^+_0, \IC^-_0 \in \Rep(\IZ/2)$, corresponding to the trivial and sign representation
    of the stabiliser $\IZ/2$ at $0:\point \to [\IR/(\IZ/2)]$.
    They assemble into the moduli space of simple objects as
    \begin{center}
        \begin{tikzpicture}
            \filldraw[black] (-0.07,0.07) circle (1pt);
            \filldraw[black] (-0.07,-0.07) circle (1pt);
            \draw[line width=2pt] (0,0) -- (3,0);
        \end{tikzpicture}
    \end{center}
    and their dimensions are
    \begin{align*}
        \dim{\IC_{x>0}} = 1 &  & \dim{\IC^+_0} = \dim{\IC^-_0} = \tfrac{1}{2}.
    \end{align*}
\end{example}

\begin{defn}
    The \emph{support} $\supp{X}$ of a simple object $X \in \cat{C}$ is the isomorphism
    class $[x] \in |\orb{M}|$ such that $X$ is in the essential image of ${\cat{C}}^{(x)}$.
    It is extended to $\cat{C}$ by sending direct sums to unions of sets.
\end{defn}
Ie.\ the support of an object $Z \in \cat{C}$ is the union
of the supports of the simple objects it decomposes into.
\begin{remark}
    Each object $X \in \cat{C}$ is a summand
    of an object of the form
    \[
        \overline{X}=\DirSum_{i=0}^n x_{i,\ast} E_i,
    \]
    where $n \in \IN$ is a positive integer, the $x_i: \point \to \orb{M}$ are points of the orbifold,
    and the $E_i$ are vector spaces (= vector bundles over the point).
    The support of $X$ is then a subset of the union of the points
    $[x_i]: \point \to |\orb{M}|$.
\end{remark}

\subsection{Morphisms of Families}
\label{sec:local_structure}
This section is the most technical part of this note. We study morphism spaces between
families of twisted skyscraper sheaves.
In an abelian category $\cat{D}$, every morphism $F:A \to B$ admits an
epi-mono factorisation
\[
    F: A \onto \im F \into B.
\]
If $\cat{D}$ is further semisimple, any morphism may be written as a
matrix of maps of simple summands.
In this case, $\im F$ may be identified both as a summand
of $A$ and a summand of $B$.
We established in Corollary~\ref{cor:catOfPointsDirSumOfRepCats} that the
category $\cat{C}$ is semisimple and abelian.
This is not the case for $\cat{C}^U$ unless $U$ is discrete:
\begin{example}
    Consider two simple $\IR$-families $\f{A}$, $\f{B}$ of skyscraper sheaves
    over a manifold $M$.
    Assume further the support $\supp \f{A}(x) = \supp \f{B}(x)$ agrees for
    $x<0 \in \IR$ and is disjoint for $x>0$ (see the example picture below).
    \begin{center}
        \begin{tikzpicture}
            \begin{scope}
                \draw (0,0) -- node[below]{$\IR$} (3,0);
                \draw (0,0) -- node[left]{$M$} (0,3);
                \draw (0,1.5) -- (3,1.5);
            \end{scope}
            \begin{scope}[shift={(5,0)}]
                \draw (0,0) -- node[below]{$\IR$} (3,0);
                \draw (0,0) -- node[left]{$M$} (0,3);
                \draw (0,1.5) .. controls (2,1.5) and (2,1.5) .. (3,2.5);
            \end{scope}
        \end{tikzpicture}
    \end{center}
    Then there exists a morphism
    $F:\f{A} \to \f{B}$, non-zero on $x<0$ and zero for $x \geq 0$.
    The families $\f{A}$ and $\f{B}$ don't share any summands,
    yet there exists a non-zero map between them. Hence $\cat{C}^\IR$ is not
    semisimple.
    The map $F$ also admits neither kernel nor cokernel, hence $\cat{C}^\IR$ is
    not abelian.
\end{example}
However, the epi-mono factorisation via a summand still holds locally in $\cat{C}^U$:
We show in Proposition~\ref{prop:local_factorisation} that
for every morphism $F:\f{A} \to \f{B}$ and $x \in U$, there locally
exists an isomorphism of direct summands of $\f{A}$ and $\f{B}$
which is equal to $F$ at $x$. (The isomorphism does not, in general, equal
$F$ away from $x$ --- not even in a small neighbourhood.)
This local factorisation result has an important implication,
which we record as Corollary~\ref{cor:locally_share_summand}:
The existence of a morphism $F: \f{A} \to \f{B}$ which is non-zero
at $x \in U$ implies that $\f{A}$ and $\f{B}$ share a non-trivial
summand in a neighbourhood of $x$.
Our strategy to prove Proposition~\ref{prop:local_factorisation}
will be to construct a local embedding of the morphism space 
$\Hom_{\cat{C}^U}(\f{A},\f{B})$ into a
space of smooth matrix-valued functions.

Let $\ger{G} \to \orb{M}$ be an orbifold equipped with a gerbe.
We may choose a cover of $\orb{M}$ by a disjoint union $\coprod_i [M_i/\Gamma_i] \onto \orb{M}$
of quotient orbifolds such that all $M_i$ are contractible and all $\Gamma_i$ are finite.
Let $\iota:[M/\Gamma] \into \orb{M}$ be the inclusion of such a quotient suborbifold.
It induces a pair of adjoint maps of stacks
% https://q.uiver.app/?q=WzAsMixbMSwwLCJcXFNodl5cXGdlcntHfV9cXG9yYntNfSJdLFswLDAsIlxcU2h2XntcXGlvdGFeXFxhc3RcXGdlcntHfX1fe1tNL1xcR2FtbWFdfSJdLFsxLDAsIlxcaW90YV9cXGFzdCIsMix7Im9mZnNldCI6MX1dLFswLDEsIlxcaW90YV5cXGFzdCIsMix7Im9mZnNldCI6MX1dXQ==
\[\begin{tikzcd}
        {\Shv^{\iota^\ast\ger{G}}_{[M/\Gamma]}} & {\Shv^\ger{G}_\orb{M}},
        \arrow["{\iota_\ast}"', shift right=1, from=1-1, to=1-2]
        \arrow["{\iota^\ast}"', shift right=1, from=1-2, to=1-1]
    \end{tikzcd}\]
which identifies $\Shv^{\iota^\ast\ger{G}}_{[M/\Gamma]}$ as the full substack of $\Shv^\ger{G}_\orb{M}$
on (families of) sheaves whose support is concentrated in $[M/\Gamma] \subset \orb{M}$.
Any family of skyscraper sheaves $\f{S} \in \Sky^\ger{G}_\orb{M}(U)$ locally decomposes into a direct sum of families
that factor through the inclusion of disjoint suborbifolds of $\orb{M}$.
Hence, we may study the local structure of $\Sky^\ger{G}_\orb{M}$ by
studying $\Sky^{\iota^\ast\ger{G}}_{[M/\Gamma]}$.

Consider a quotient orbifold $[M/\Gamma]$ as above with gerbe $\ger{G}$, equipped with an atlas $M \onto [M/\Gamma]$
and a trivialisation of the pullback gerbe over $M$. The twisted equivariantisation adjunction of
Proposition~\ref{prop:twistedEquivariantisationAdjunction}
gives an adjoint pair of maps of stacks
% https://q.uiver.app/?q=WzAsMixbMCwwLCJcXFNodl9NIl0sWzIsMCwiXFxTaHZee1xcZ2Vye0d9fV97W00vXFxHYW1tYV19Il0sWzAsMSwiSSIsMix7Im9mZnNldCI6MX1dLFsxLDAsInUiLDIseyJvZmZzZXQiOjF9XV0=
\[\begin{tikzcd}
        {\Shv^\ger{I}_M} && {\Shv^{\ger{G}}_{[M/\Gamma]}}.
        \arrow["I=q_\ast"', shift right=1, from=1-1, to=1-3]
        \arrow["u=q^\ast"', shift right=1, from=1-3, to=1-1]
    \end{tikzcd}\]
The components of these maps at $U \in \Man$ are given by the adjoint pair of functors
% https://q.uiver.app/?q=WzAsMixbMCwwLCJcXFNoQyhNXFx0aW1lcyBVKSJdLFsyLDAsIlxcU2h2XntcXGdlcntHfX0oW00vXFxHYW1tYV1cXHRpbWVzIFUpIl0sWzAsMSwiSSBcXGRlZmluZSB7KHEgXFx0aW1lcyBcXGlkX1UpfV9cXGFzdCIsMix7Im9mZnNldCI6MX1dLFsxLDAsInVcXGRlZmluZXsocSBcXHRpbWVzIFxcaWRfVSl9XlxcYXN0IiwyLHsib2Zmc2V0IjoxfV1d
\[\begin{tikzcd}
        {\ShC(M\times U)} && {\Shv^{\ger{G}}([M/\Gamma]\times U)}.
        \arrow["{I_U \define {(q \times \id_U)}_\ast}"', shift right=1, from=1-1, to=1-3]
        \arrow["{u_U \define{(q \times \id_U)}^\ast}"', shift right=1, from=1-3, to=1-1]
    \end{tikzcd}\]
Both $I_U$ and $u_U$ preserve the property of being a family of twisted skyscraper sheaves,
hence this pair of maps restricts to an adjunction
% https://q.uiver.app/?q=WzAsMixbMCwwLCJcXFNreV9NIl0sWzIsMCwiXFxTa3lee1xcZ2Vye0d9fV97W00vXFxHYW1tYV19Il0sWzAsMSwiSSIsMix7Im9mZnNldCI6MX1dLFsxLDAsInUiLDIseyJvZmZzZXQiOjF9XV0=
\[\begin{tikzcd}
        {\Sky^\ger{I}_M} && {\Sky^{\ger{G}}_{[M/\Gamma]}}.
        \arrow["I"', shift right=1, from=1-1, to=1-3]
        \arrow["u"', shift right=1, from=1-3, to=1-1]
    \end{tikzcd}\]

Every $U$-family $\f{T} \in \Sky_M(U) \subset \ShC(M \times U)$ locally decomposes as a direct sum of
pushforwards of vector bundles: There
exists a cover $Y \onto U$ with a trivialisation of the pullback gerbe,
a finite set of functions $f_i:Y \to M$, and vector bundles $E_i$ over $Y$ such
that
\[
    \f{T}\restriction_Y \iso \dirSum_i {\Gr(f_i)}_\ast E_i.
\]
The image of a family ${\Gr(f)}_\ast E$ under the induction functor
$I$ associated to $M \onto [M/\Gamma]$ is the family
${\Gr(q \comp f)}_\ast E$. Any vector bundle $E$ locally splits as a direct sum
of line bundles, so
any family in the image of $I$ locally decomposes as the pushforward of
a line bundle via the graph of a function $g:Y \to [M/\Gamma]$.
Conversely, if a family is locally of this form, it is in the image of an
induction functor.
\begin{defn}
    \label{def:basicMobileAndMobileFamilies}
    A family $\f{S}:U \to \cat{C}$ is a \emph{basic mobile} if it is
    locally equivalent to a family of the form ${\Gr(f)}_\ast \Cinf$,
    where $f$ is a smooth map and $\Cinf$ is the (sheaf of sections of)
    the trivial line bundle over $U$.
    A family is \emph{mobile} if it is locally a direct sum of
    basic mobiles.
\end{defn}
By the discussion above, mobile families are precisely the families which
locally lie in the image of an induction functor.
\begin{example}
    Recall from Lemma~\ref{lem:cat_at_point_props} that
    $\Sky^\ger{G}_{\orb{M}}(\point)$ decomposes as a direct sum
    \[
        \Sky^\ger{G}_{\orb{M}}(\point) \iso \DirSum_{[x] \in |\orb{M}|}
        \Rep^{\theta_x}(\Gamma_x).
    \]
    The mobile objects (mobile $\point$-families) are precisely those objects
    which are isomorphic to direct sums of twisted group rings
    \[
        \IC^{\theta_x}[\Gamma_x] \in \Rep^{\theta_x}(\Gamma_x).
    \]
\end{example}

\begin{definition}
    Let $s: U \to \orb{M}$ be a smooth map and $\tau$
    a trivialisation of the pullback gerbe over $U$.
    The
    \emph{basic mobile induced by $s$ and $\tau$} is the $U$-family
    \[
        \Ind^\tau(s) \define {\Gr(s)}_\ast \tau \cdot \Cinf.
    \]
    Given a set of smooth maps $\{s_i: U \to \orb{M}\}$ and trivialisations $\tau_i$,
    we denote the associated mobile by
    \[
        \Ind^{\{\tau_i\}}\{s_i\} \define \DirSum_i \Ind^{\tau_i} s_i \define
        \DirSum_i\Gr(s_i)_\ast \tau_i \cdot \Cinf,
    \]
    and call it the \emph{mobile induced by $\{s_i, \tau_i\}$}.
\end{definition}
Mobile families play well with pullbacks.
Given a map $f:V \to U$, there is a canonical isomorphism
\begin{equation*}
    f^\ast\Ind^\tau(s) \simeq \Ind^{f^\ast\tau}(s \comp f) \in \cat{C}^V.
\end{equation*}
Pullbacks of basic mobiles yield basic mobiles and thus the same is true for mobile families.
Each family $\f{A}:U \to \cat{C}$ is locally a summand of a mobile family $\overline{\f{A}}$.

\begin{lemma}
    \label{lem:trivialisationsInIndDontMatterLocally}
    Let $U \in \Man$ be contractible and $f: U \to \orb{M}$ a smooth map.
    For any pair $(\tau,\tau')$ of
    trivialisations of the pullback gerbe $f^\ast\ger{G}$,
    the families induced by $(f,\tau)$ and $(f,\tau')$ are
    abstractly isomorphic:
    \[
        \Ind^\tau(f)\restriction_Y \iso \Ind^{\tau'}(f)\restriction_Y.
    \]
\end{lemma}
\begin{proof}
    The different choice of trivialisation amounts to tensoring with the
    line bundle $\mscr{L}_{\tau,\tau'}$ (associated to $\tau^{-1}\tau'$), pushed forward from $U$ via $\Gr(f)$.
    As $U$ is contractible, this line bundle is trivial, and tensoring
    with it does not change the isomorphism type.
\end{proof}
Every manifold admits a cover by contractible manifolds, hence we may
locally ignore the chosen trivialisations when we are only working up to
isomorphism.
\begin{notation}
    \label{notation:IndWithoutTrivialisation}
    Let $f:U \to \orb{M}$ be as in
    Lemma~\ref{lem:trivialisationsInIndDontMatterLocally}.
    We write $\Ind(f)$ to denote the isomorphism type of the
    family $\Ind^\tau(f)$.
\end{notation}

%\begin{definition}
%\label{def:supportFunction}
%    Let $\f{A}:U \to \cat{C}$.
%    The \emph{support function} of $\f{A}$ is the
%    smooth multivalued map
%    \begin{align*}
%        \supp \f{A}: U & \to |\st{X}|            \\
%        x              & \mapsto \supp \f{A}(x).
%    \end{align*}
%\end{definition}

\begin{definition}
    A \emph{mobile cover} of a family $\f{A}$ is a mobile family $\overline{\f{A}}$,
    equipped with data $\f{A} \into \overline{\f{A}} \onto \f{A}$ identifying
    $\f{A}$ as a summand,
    such that the support of the two families agrees pointwise:
    \[
        \supp(\overline{\f{A}}) = \supp(\f{A}) \subset |\orb{M}|.
    \]
\end{definition}

\begin{lem}
    \label{lem:locally_summand_of_mobile}
    Each $\f{A}: U \to \cat{C}$ locally admits a mobile cover around any $x \in U$.
    Such a cover can be obtained by applying the equivariantisation monads $Iu$
    associated to a choice of orbifold charts around $\supp \f{A}(x)$.
\end{lem}
\begin{proof}
    Pick pairwise non-isomorphic representatives $y_i:\point \to \orb{M}$ for
    $\supp \f{A}(x) \subset |\orb{M}|$, and pick pairwise disjoint orbifold charts
    $M_i \to [M_i/\Gamma_i] \into \orb{M}$ around the points $y_i$.
    There exists a neighbourhood $V$ of $x \in U$ such that
    $\f{A}\restriction_V$ splits into a direct sum of families with support
    concentrated in one of the orbifold charts.

    Without loss of generality, we may assume that $\f{A}$ has support concentrated
    in a single quotient orbifold $[M/\Gamma]$.
    Let $I \dashv u$ be the twisted equivariantisation adjunction associated to
    the orbifold chart $M \onto [M/\Gamma]$.
    The family $Iu\f{A}$ is mobile because it is in the image of $I$.
    It is easy to check that its support pointwise agrees with that of $\f{A}$.
    Lastly, Corollary~\ref{cor:twistedSheavesAreLocallySummandsOfMobiles} tells us
    that the unit and counit of the adjunction identify $\f{A}$ as a
    direct summand of $Iu\f{A}$.
\end{proof}

The monad $Iu$ is trivial for an orbifold patch where the quotient group
$\Gamma$ is trivial. For an effective orbifold, this is the generic region.
\begin{cor}
    \label{cor:families_in_generic_region_are_mobile}
    Assume $\orb{M}$ is an effective orbifold.
    Let $\f{A}: U \to \cat{C}$ be such that $\supp \f{A}(x)$ is
    contained in the generic region of the orbifold.
    Then $\f{A}$ is mobile in a neighbourhood of $x$.
\end{cor}

Let $q: M \to \orb{M}$ be an \'etale map and $\tau$
a trivialisation of the pullback gerbe $q^\ast\ger{G}$.
Let $\f{A}: U \to \cat{C}$ be a $U$-family of $\ger{G}$-twisted
skyscraper sheaves with support contained in the image of $q$.
\begin{defn}
    \label{def:supportLift}
    A \emph{set of support lifts for $\f{A}$} with respect to
    $(q:M \to \orb{M},\tau:\ger{I}\eqto q^\ast\ger{G})$ is a collection
    $\{{\tilde{s}}_i: U \to M\}_{i \in I}$, such that the induced mobile family
    \[
        \Ind^{\{\tilde{s}_i^\ast\tau\}}(q\comp\tilde{s}_i)
        = \DirSum_{i \in I} \Ind^{\tilde{s}_i^\ast\tau}(q\comp \tilde{s}_i)
    \]
    is a mobile cover of $\f{A}$.
\end{defn}

Given that we can always find mobile covers locally
(Lemma~\ref{lem:locally_summand_of_mobile}), we can always locally find
support lifts:
\begin{lem}
    \label{lem:local_support_lifts_exist}
    For any family $\f{A}: U \to \cat{C}$, and $x \in U$, there exists an open
    neighbourhood $V \ni x$, a chart $M \to \orb{M}$ and support lifts
    $\{{\tilde{s}}_i: V \to M\}$ for $\f{A}$.
\end{lem}

There are families that indeed only admit these lifts locally.
\begin{example}
    Denote by $S^1$ the standard circle.
    There is a multivalued smooth map $S^1 \to S^1$, sending
    each point $x$ to a pair of opposite points $\pm x/2$, rotating at half speed.
    There are $S^1$-families of skyscraper sheaves on $S^1$ with this support
    function, but they do not admit a global support lift as there
    exists no map $S^1 \to S^1$ of degree $1/2$.
\end{example}

\begin{proposition}
    \label{prop:locallyMorphismsAreFamiliesOfMatrices}
    Let $\f{A},\f{B}:U \to \cat{C}$ be two $U$-families.
    There exists a cover $Y \onto U$, and a faithful embedding
    \[
        \Hom_{\cat{C}^Y}(\f{A}\restriction_Y,\f{B}\restriction_Y) \into
        \Cinf(Y, \Mat_{m \times n}(\IC)) = \Mat_{m \times n}(\CinfOf{Y})
    \]
    into a space of smooth functions valued in complex $(m \times n)$-matrices
    for some $m,n \in \IN$.
\end{proposition}
\begin{proof}
    We prove the statement in a neighbourhood $V$ of a point $x \in U$.
    By choosing $V$ small enough, we may assume the family $\f{A}\restriction_V$
    splits as a direct sum $\dirSum_i\f{A}_i$ of families
    whose supports are contained in pairwise disjoint suborbifolds
    $[M_i/\Gamma_i] \into \orb{M}$. By further restricting $V$, we may assume
    the same is true for $\f{B}$.
    Families with disjoint support have no non-zero morphisms between,
    so it suffices to treat the case where $\f{A}$ and $\f{B}$
    factor through the same quotient orbifold $[M/\Gamma] \into \orb{M}$.

    Denote the induced gerbe over $[M/\Gamma]$ by $\ger{G}$, and equip
    $q:M\to [M/\Gamma]$ with a trivialisation of $q^\ast\ger{G}$.
    Now consider the twisted equivariantisation adjunction
    associated to $q:M\onto [M/\Gamma]$.
    We may think of
    $\f{A},\f{B} \in \Sky^\ger{G}_{[M/\Gamma]}(U) \subset
        \Shv^\ger{G}([M/\Gamma]\times U)$ as objects in
    $\Sky_M(U) \subset \ShC(M \times U)$,
    equipped with $\ger{G}$-twisted $\Gamma$-equivariance data.
    The forgetful functor $u$ discards the equivariance data of an equivariant sheaf.
    It faithfully embeds
    \begin{equation*}
        \Hom_{{\Shv^\ger{G}([M/\Gamma] \times U)}}( \f{A} , \f{B} ) \into
        \Hom_{\ShC(M \times V)}(u \f{A}, u \f{B} )
    \end{equation*}
    as the $\Gamma$-equivariant morphisms.

    By possibly restricting $V$ further, we may assume $u\f{A}$ and
    $u\f{B}$ are of the form
    \begin{align*}
        u\f{A} \iso \dirSum_{i=0}^m {\Gr(f_i)}_\ast \Cinf &  &
        u\f{B} \iso \dirSum_{j=0}^n {\Gr(g_j)}_\ast \Cinf,
    \end{align*}
    where $\{f_i:V \to M\}$ and $\{g_j:V \to M\}$ are smooth functions.
    Using the isomorphism of the pullback-pushforward adjunction,
    \begin{align*}
        \Hom_{\ShC(M\times V)}\left(u \f{A},u\f{B}\right) & \iso
        \Hom_{\ShC(M\times V)}\left(\dirSum_{i=0}^m {\Gr(f_i)}_\ast \Cinf,
        \dirSum_{j=0}^m {\Gr(g_j)}_\ast \Cinf\right)                                         \\
                                                          & \iso \dirSum_{i,j}\Hom_{\ShC(V)}
        \left(\Gr(g_j)^\ast\Gr(f_i)_\ast\Cinf,\Cinf\right).
    \end{align*}
    We examine a single summand of the form
    \[
        \Hom_{\ShC(V)}\left(\Gr(g)^\ast\Gr(f)_\ast\Cinf,\Cinf\right).
    \]
    By Lemma~\ref{lem:behaviourOfSupportUnderPushfwdAndPullback},
    we know the sheaf $\Gr(g)^\ast\Gr(f)_\ast\Cinf$ is zero on the complement
    of the subset $V_{f=g} \subset V$ where the functions are equal.
    If a morphism $\Gr(g)^\ast\Gr(f)_\ast\Cinf \to \Cinf$ is non-zero at a
    point $v \in V$, it must stay non-zero in a neighbourhood of $v$.
    Hence, such a morphism can only be non-zero on the interior
    $V_{f=g}^\ring$ of $V_{f=g}$.
    Restricted to $V_{f=g}^\ring$, however, $\Gr(g)^\ast\Gr(f)_\ast\Cinf = \Cinf$,
    so the $\Hom$-space becomes
    \[
        \Hom_{\ShC(V_{f=g}^\ring)}(\Cinf,\Cinf) = \Cinf(V_{f=g}^\ring,\IC).
    \]
    This way, we obtain an embedding
    \begin{align*}
        \Hom_{\ShC(M\times V)}(u \f{A},u\f{B})
         & \into \Hom_{\ShC(V)}(\Cinf,\Cinf)^{\dirSum m\times n} \simeq \Mat_{m \times n}(\CinfOf{V}).
    \end{align*}
    The image of the embedding is the subspace of those matrix-valued functions which
    are $\ger{G}$-twisted $\Gamma$-equivariant\footnote{Recall that the summands of $u\f{A}$ and $u\f{B}$ are acted upon by $\Gamma$.
        The identifications made above translate this into an action on the set of rows
        and the set of columns of the matrices.}
    and whose $(i,j)$-component is zero outside the open set
    $V_{f_i=g_j}^\ring$.
\end{proof}

Let $F:\f{A} \to \f{B}$ be a morphism of $U$-families, and pick an embedding
\[
    \Hom(\f{A}, \f{B}) \into \Mat_{m\times n}(\Cinf)
\]
around $x \in U$. This identifies $F(x)$ with a matrix.
There is an open neighbourhood $V$ of $x$ where this matrix is in the image
of the embedding.
Denote by $\bar{F}:\f{A}\restriction_V \to \f{B}\restriction_V$ the
morphism of $V$-families which is sent to the constant function at $F(x)$
under the above embedding.

\begin{notation}
    Let $\f{S}:U \to \cat{C}$ and $p \in \End(\f{S})$ a split idempotent.
    We denote the summand splitting the idempotent by $(\f{S},p)$.
\end{notation}
\begin{lemma}
    \label{lem:projection_locally_constant}
    Let $\overline{\f{A}}:U \to \cat{C}$ be a $U$-family, and
    $p \in \End(\overline{\f{A}})$ an idempotent.
    Let $\overline{p}$ be a map obtained as above by locally extending
    $p_x$ for some point $x \in U$.
    Then $\overline{p_x}$ is an idempotent.
    Let $\overline{\cat{C}^U}$ be a completion of
    $\cat{C}^U$ where both $p$ and $\overline{p}$ are split.
    Then the summands of
    $\overline{\f{A}}$ picked out by $p$ and $\overline{p}$
    are isomorphic in a neighbourhood of $x$.
\end{lemma}
\begin{proof}
    Idempotency can be checked pointwise,
    thus $\overline{p}$ is an idempotent.
    We will use the idempotents themselves to establish the isomorphism.
    An idempotent matrix has the property that all its eigenvalues are
    either 0 or 1.
    Eigenvalues vary smoothly in a smooth family of matrices, hence
    the rank of an idempotent matrix is locally constant.
    Now assume $p$ and $\overline{p}$ are split by summands
    $(\overline{\f{A}},p)$ and $(\overline{\f{A}},\overline{p})$, respectively.
    The maps $p$ and $\overline{p}$ now define morphisms
    $(\overline{\f{A}},p) \rightleftarrows (\overline{\f{A}},\overline{p})$.
    As $p(x)=\overline{p}(x)$, the maps are isomorphisms at $x$, 
    ie.\ they have full rank.
    This is then also true in an open neighbourhood,
    which proves the desired isomorphism.
\end{proof}
By Lemma~\ref{lem:locally_summand_of_mobile}, every family
$\f{A}:U \to \cat{C}$ is locally a summand of a family of the form
$\overline{\f{A}}=\DirSum_i\Gr(f_i)_\ast E_i$.
This summand is picked out by an idempotent
$p \in \End(\overline{\f{A}})$. 
We show that this idempotent may be diagonalised around $x$.

\begin{cor}[Local Normal Form]
    \label{cor:locally_karoubi_completion_distributes}
    A family $\f{A}:U \to \cat{C}$ locally decomposes as
    \begin{equation*}
        \f{A} \iso \DirSum_i (\Gr(s_i)_\ast \tau_i \cdot E_i, p_i),
    \end{equation*}
    where the $s_i:U \to \orb{M}$ are smooth maps, the $\tau_i$
    are trivialisations of the pullback gerbes,
    the $E_i$ are vector bundles on $U$, and the $p_i$ are idempotents
    on $\Gr(s_i)_\ast \tau_i \cdot E_i$.
    The set $\{s_i\}$ can be chosen to be obtained
    from any set of
    local support lifts $\{{\tilde{s}}_i: U \to M\}$ for $\f{A}$
    to a chart $q:M \to \orb{M}$.
\end{cor}
\begin{proof}
    Near any point $x \in U$, any set of local support lifts corresponds to
    a mobile cover $\widetilde{\f{A}} \iso \DirSum_i \Gr(s_i)_\ast E_i$
    of $\f{A}$, where $s_i = q \comp {\tilde{s}}_i$.
    Around $x$, $\f{A}$ is a summand of $\widetilde{\f{A}}$, and thus picked
    out by an idempotent $p$ on $\f{A}$.
    At $x$, we can decompose the projection as stated above (in such a way
    that it doesn't mix the summands).
    By Lemma~\ref{lem:projection_locally_constant}, we can assume
    idempotents are locally constant, so the decomposition holds
    on an open neighbourhood of $x$.
\end{proof}

\begin{cor}
    \label{cor:CUisKaroubiComplete}
    The category $\cat{C}^U$ of $U$-families is Karoubi complete.
\end{cor}
\begin{proof}
    By definition $\cat{C}^U = \Sky_{\orb{M}}^\ger{G}(U) \subset \Shv^\ger{G}_{\orb{M}}(U)$
    is cut out by the condition that $\f{S} \in \cat{C}^U$ locally decomposes
    as a direct sum of sheaves of the form ${\Gr(f)}_\ast \tau\cdot E$ for some
    smooth map $f$, trivialisation $\tau$ of $f^\ast\ger{G}$ and vector bundle $E$.
    The category $\Shv^\ger{G}_{\orb{M}}$ is manifestly Karoubi-complete,
    so it suffices to check this condition is preserved under the operation
    of taking a summand.

    Let $\f{S}:U \to \cat{C}$ and $p \in \End(\f{S})$ an idempotent.
    This idempotent splits in $\Shv_{\orb{M}}^\ger{G}$, and we denote
    the corresponding summand by $(\f{S},p)$.
    The sheaf $\f{S}$ locally pulls back to an object of the form
    $u\f{S} = \dirSum_i {\Gr(f_i)}_\ast \tau_i \cdot E_i$.
    The splitting of $p$ pulls back to a sheaf
    $(u\f{S},up)$. We must show that this sheaf is again of the basic form above.
    By Corollary~\ref{cor:locally_karoubi_completion_distributes},
    we may assume the idempotent is diagonal in the direct-sum decomposition.
    It hence suffices to consider a
    single summand $\Gr(f)_\ast \tau \cdot E$, equipped with an idempotent $p$.
    As $\Gr(f)$ is an embedding, $p$ is really an idempotent on $E$, which is
    split by the subbundle $(E,p)$.
    Hence $(\Gr(f)_\ast \tau \cdot E,p)$ is represented by the basic
    family $\Gr(f)_\ast \tau \cdot (E,p)$, which finishes the proof.
\end{proof}

Recall the dimension of an object $X \in \cat{C}$ (Definition~\ref{def:dimensionOfObject}).
\begin{lemma}
    \label{lem:dim_locally_constant}
    The dimension of an object in a family $\f{S}$ is locally constant.
\end{lemma}
\begin{proof}
    By the local form of families given above, it is enough to show this for a single
    family of the form $(\Gr(s)_\ast E, p)$, supported in a quotient suborbifold $[M/\Gamma]$
    of $\orb{M}$. Factoring the pushforward through a cover $M \to [M/\Gamma]$,
    we compute the vector space dimension of $\Gr(s)_\ast E$ at $x \in U$ as
    $\vdim \Gr(s(x))_\ast E = \vdim I\Gr(\tilde{s}(x))_\ast E =
        |\Gamma_{s(x)}| \cdot \vdim E_x.$
    (The induction functor sums over the $\Gamma$-orbit of $\Gr(\tilde{s})_\ast$,
    so $E_x$ gets pushed onto $\tilde{s}(x)$ with multiplicity its stabiliser group.)
    Now, $\vdim E_x$ is locally constant as $E$ is a vector bundle, thus so is
    $\dim \Gr(s(x))_\ast E = \vdim \Gr(s(x))_\ast E / |\Gamma_{s(x)}| = \vdim E_x.$
    The endomorphism $p$ is an idempotent, and thus of constant rank.
    It simply multiplies the dimension by a constant factor
    $\rk p/ \operatorname{max} \rk =
        \rk p/\dim \Gr(s)_\ast E$.
\end{proof}
This shows that the dimension is in fact a well-defined invariant
of a family over a connected manifold.

\begin{defn}
    \label{def:etaleimmersivesubmersive}
    We call a family $\f{A}: U \to \cat{C}$
    \emph{\'etale, immersive} or \emph{submersive},
    if for all $x \in U$, there exist support lifts
    on some neighbourhood of $x$
    which are \'etale, immersive or
    submersive, respectively.
\end{defn}

\begin{lem}
    \label{lem:submersive_implies_mobile}
    If $\orb{M}$ is an effective orbifold, every submersive family is mobile.
\end{lem}
\begin{proof}
    Let $\f{S}:U \to \cat{C}$ be a submersive family. By assumption,
    it has a set of submersive local support lifts
    $\{{\tilde{s}}_i:U \to M\}$ around $x \in U$ to some chart $M \to \orb{M}$.
    By Corollary~\ref{cor:locally_karoubi_completion_distributes}, $\f{S}$ locally
    splits as a direct sum corresponding to these lifts, and the idempotents
    picking out these summands can be chosen as locally constant projections
    in any trivialisations of the vector bundles involved.
    Let $\f{T}$ be the summand corresponding to a local support lift
    ${\tilde{s}}_j$. As it is submersive, the preimage of the generic region
    under ${\tilde{s}}_j$ is nonempty in any neighbourhood of $x$.
    By Corollary~\ref{cor:families_in_generic_region_are_mobile},
    the projection in the generic region picks out a sub-vector bundle.
    As we can assume the projections are constant, this is true on
    an open neighbourhood of $x$ in $U$.
    Hence, $\f{T}$ is the pushforward of this sub-vector bundle,
    and thus mobile.
\end{proof}

We have now adequately prepared to prove the central result of this section.
\begin{prop}[Local Factorisation]
    \label{prop:local_factorisation}
    Let $\f{A},\f{B}: U \to \cat{C}$ and $F \in \Hom(\f{A},\f{B})$, then
    $\forall x \in U$, there exists a neighbourhood $V$ of $x$ and
    an isomorphism of direct summands $\f{A}'\subset \f{A}\restr_V$,
    $\f{B}'\subset \f{B}\restr_V$,
    such that the restriction of the composite
    \begin{equation*}
        \begin{tikzcd}
            \f{A}\restr_V \ar[r, two heads] & \f{A}' \ar[r, "\simeq"]  &
            \f{B}' \ar[r, hook] & \f{B}\restr_V,
        \end{tikzcd}
    \end{equation*}
    to the point $x$ is equal to $F(x)$.
\end{prop}
\begin{proof}
    We pick an embedding of
    $\Hom(\f{A},\f{B})$ into smooth families of matrices
    as granted by Proposition~\ref{prop:locallyMorphismsAreFamiliesOfMatrices}.
    Define $\overline{F}$ to be the constant extension of $F(x)$ under
    this embedding (which exists in an open neighbourhood of $x$).
    Now we use the image factorisation of the map $F(x)$:
    We pick a subspace $\f{A}(x)' \subset \f{A}(x)$ orthogonal to $\ker{F(x)}$.
    $F(x)$ is then an isomorphism between $\f{A}'(x)$ and its image
    $\f{B}(x)' \define \im F(x) \subset \f{B}(x)$
    under $F(x)$.
    Denote by $p_A$ the projection to the summand $\f{A}(x) \onto \f{A}'(x)$ and $p_B$
    the projection $\f{B}(x) \onto \f{B}'(x)$.
    We may extend both these projections and their splittings locally.
    Denote by $\f{A}^\prime$, $\f{B}^\prime$ the corresponding summands of
    $\f{A}\restriction_V$ and $\f{B}\restriction_V$.
    Now $\overline{F}$ induces an isomorphism $\f{A}' \to \f{B}'$, because
    it does so pointwise.
\end{proof}

\begin{cor}
    \label{cor:locally_share_summand}
    If $F:\f{A} \to \f{B}$ is non-zero at $x \in U$, then $\f{A}$ and
    $\f{B}$ share a non-zero isomorphic summand in a neighbourhood of $x$.
\end{cor}

\begin{lemma}
    \label{lem:germ_of_support_lift_unique_up_to_Gamma}
    Given a family $\f{S}:U \to \cat{C}$ with a single local support lift
    $\tilde{s}: U \to M$ to the standard chart
    $q:M \to [M/\Gamma]$ covering a quotient suborbifold $[M/\Gamma] \subset \orb{M}$,
    the germ of $\tilde{s}$ is unique up to postcomposition with the $\Gamma$-action on $M$.
    Hence all mobile covers of $\f{S}$ are locally induced by the same
    map $s \define q \comp \tilde{s}:U \to [M/\Gamma]$.
\end{lemma}
\begin{proof}
    Let $\f{A}:U \to \cat{C}$ be a mobile cover of $\f{S}$, obtained as the
    pushforward of a trivial vector bundle on $U$ via $s$.
    Let $\f{A}':U \to \cat{C}$ be another mobile cover of $\f{S}$, and
    $\tilde{s}': U \to M$ an associated support lift.
    $\f{A}$ and $\f{A}'$ share a summand, so projection to the summand
    followed by inclusion gives a nowhere-zero map between them.
    Using the equivariantisation adjunction associated to the chart $U$,
    \begin{equation*}
        \Hom(\f{A},\f{A}') =
        \DirSum_{g \in \Gamma} \Hom(\Gr(g\cdot\tilde{s}')^\ast\Gr(\tilde{s})_\ast E,E'),
    \end{equation*}
    where $E$ and $E'$ are vector bundles on $U$.
    In order for this to have a section which is non-zero at $x$, there must
    be a summand, such that $g\cdot\tilde{s}'=\tilde{s}$ on an
    open neighbourhood of $x$.
\end{proof}

\begin{corollary}
    \label{cor:summands_yield_subset_of_support_lifts}
    Let $\f{S}:U \to \cat{C}$ be a family with local support lifts
    $\{{\tilde{s}}_i:U \to M\}$ around $x \in U$ to some chart $M \to [M/\Gamma]$
    covering a quotient suborbifold of $\orb{M}$,
    and let $\f{T} \subset \f{S}$ be a summand of $\f{S}$, with
    a set of local support lifts $\{{\tilde{t}}_j:U \to M\}$ around $x$.
    Then each ${\tilde{t}}_j$ is a $\Gamma$-translate of the germ of a
    lift ${\tilde{s}}_i$ at $x$.
\end{corollary}
\begin{proof}
    $\f{T}$ decomposes into summands corresponding to the ${\tilde{t}}_j$ on
    an open neighbourhood of $x$, so we can assume it has a single
    local support lift $\tilde{t}:U \to M$.
    Decomposing $\f{S}$ according to the support lifts ${\tilde{s}}_i$, the
    inclusion $\f{T} \into \f{S}$
    must yield a nonzero map to at least one such summand.
    But this is nonzero at $x$, so by Corollary~\ref{cor:locally_share_summand}
    they share a summand. This summand now has two local support lifts
    $\tilde{t}$ and ${\tilde{s}}_i$.
    By Lemma~\ref{lem:germ_of_support_lift_unique_up_to_Gamma}, they are
    locally related by the $\Gamma$-action on $M$.
\end{proof}

\subsection{Grothendieck Group}
For any orbisimple category $\stC$,
the category of points $\cat{C}=\stC(\point)$ is semisimple,
so its \emph{Grothendieck group} $\K(\cat{C})$ is free abelian
with set of generators $\Irr{\cat{C}}$.
By Corollary~\ref{cor:catOfPointsDirSumOfRepCats},
$\K(\cat{C})$ has a $\IZ$-basis given by $\coprod(\Irr{\Rep^{\theta_x}(\Gamma_x)})$.
\begin{defn}
    \label{def:basic_mobile_at_x}
    Given an isomorphism class of points $[x] \in |\orb{M}|$,
    the \emph{basic mobile at $[x]$}  is the isomorphism class
    \begin{equation*}
        {\IO}_x \define [\Ind x] = [x_\ast (\tau \cdot \Cinf)],
    \end{equation*}
    where $x$ is some representative of $[x]$ and $\tau$
    is any trivialisation of the gerbe $x^\ast\ger{G}$.\footnote{
        Of course, $\Cinf(\point)=\IC$ is just
        the canonical line bundle over the point.
    }
    See Lemma~\ref{lem:trivialisationsInIndDontMatterLocally} and
    Notation~\ref{notation:IndWithoutTrivialisation} for why the
    choice of trivialisation $\tau$ is irrelevant.
\end{defn}
The \emph{dimension} defined earlier is an isomorphism invariant,
and is extended
additively from simple objects, so it descends to the Grothendieck ring.
By Lemma~\ref{lem:cat_at_point_props}, ${\IO}_x = [\IC^{\theta_x}[\Gamma_x]]$,
and so
\[
    \dim \IO_x = \frac{\vdim{\IC^{\theta_x}[\Gamma_x]}}{|\Gamma_x|} = 1.
\]
Basic mobile families~\ref{def:basicMobileAndMobileFamilies} are pointwise of the form
of Definition~\ref{def:basic_mobile_at_x}, and mobile families are direct sums thereof.
\begin{cor}
    Mobile families have integer dimension.
\end{cor}

\subsection{Lie Categories and Cauchy Completion}
\label{sec:LieCatsCauchyCompletion}
In this section, we show how the stack of twisted skyscraper
sheaves $\Sky_\orb{M}^\ger{G}$ may be presented as the Cauchy completion
of a Lie category.
Lie categories are defined exactly as Lie groupoids, without the requirement of an
inverse map for morphisms:
\begin{defn}
    A \emph{Lie category} is a category object internal to $\Man$ whose
    source and target maps are submersions.
\end{defn}

\begin{example}
    \label{ex:LieCategoryBR}
    A smooth monoid $R$
    has an associated Lie category $[\point/R]$
    with object space $\point$ and
    morphisms $\End(\point) = R$.
\end{example}

Smooth functors, smooth natural transformations and smooth equivalences are
the internal
versions (we spelled these out for Lie groupoids in Section~\ref{sec:orbifolds}).
The 2-category of Lie groupoids embeds fully faithfully into the 2-category of Lie categories
via the functor that forgets the inverse map.

A Lie category $\cat{C}_\bullet$ defines a 2-presheaf
\begin{align*}
    \Hom(-,\cat{C}_\bullet): \Man^\opp \to \Cat
\end{align*}
sending a manifold $M$ to the category of smooth functors
$(M \rightrightarrows M) \to \cat{C}_\bullet$.
The stack \emph{presented by $\cat{C}_\bullet$} is the stackification of
$\Hom(-,\cat{C}_\bullet)$.
When it is clear from context, we will denote the stack by $\cat{C}_\bullet$.

By~\cite[Sec 8]{roberts2012internal}, the process of stackification is equivalent
to inverting smooth equivalences, and the category of morphisms between two
presented stacks is equivalent to the category of anafunctors between the
presenting Lie categories.

\begin{example}
    \label{ex:StackBR}
    Consider the stack $[\point/R]$ from Example~\ref{ex:LieCategoryBR}. An anafunctor
    $(M \rightrightarrows M) \to (R \rightrightarrows \point)$ is a choice of cover
    $Y \onto M$, and a transition function $Y^{[2]} \to R$ satisfying the 2-cocycle
    condition on $Y^{[3]}$.
    As every morphism in $Y^{[2]} \rightrightarrows Y$ is invertible, the map $Y^{[2]} \to R$
    factors through $R^\times \to R$.
    Every anafunctor thus factors through the functor $[\point/R^\times] \to [\point/R]$,
    and the corresponding functor $[\point/R^\times](M) \to [\point/R](M)$ is
    faithful.
    However, it is not full in general:
    Recall that $[\point/R^\times]$ is the stack of $R^\times$-bundles.
    Let $P,Q \in [\point/R^\times](M)$ be two $R^\times$-bundles over $M$.
    Morphisms $P \to Q$ in $[\point/R^\times](M)$ correspond to
    usual $R^\times$-bundle morphisms, while morphisms
    $P \to Q$ in $[\point/R](M)$ are morphisms of the associated $R$-bundles
    $P \times_{R^\times} R \to Q \times_{R^\times} R$.
\end{example}

Let $\ger{G} \to \orb{M}$ be an orbifold equipped with a gerbe.
We construct a Lie category whose associated stack recovers $\Sky_\orb{M}^\ger{G}$
once suitably completed under direct sums and direct summands.
Pick a $\IC^\times$-extension of Lie groupoids
\begin{equation*}
    \begin{tikzcd}
        P & {X_1} \\
        {X_0} & {X_0},
        \arrow["\id",from=2-1, to=2-2]
        \arrow[from=1-1, to=1-2]
        \arrow[shift left=1, from=1-2, to=2-2]
        \arrow[shift right=1, from=1-2, to=2-2]
        \arrow[shift right=1, from=1-1, to=2-1]
        \arrow[shift left=1, from=1-1, to=2-1]
    \end{tikzcd}
\end{equation*}
associated to a trivialisation of $\ger{G}$ over $X_0 \onto \orb{M}$
(see Section~\ref{sec:GerbesAndTwistedSheaves}).
The map of morphism spaces $P \to X_1$ is a $\IC^\times$-bundle,
compatible with the composition structure on both spaces.
Denote the associated bundle $P_\IC \define P \times_{\IC^\times} \IC \to X_1$
where $\IC^\times$ acts on $\IC$ in the usual way.
We build a Lie category $P_\IC \rightrightarrows X_0$. Its source and target
morphisms are the composites $P_\IC \to X_1 \rightrightarrows X_0$, and the
composition is induced by the composition on $P$ and multiplication on $\IC$:
\[
    P_\IC \times_{X_0} P_\IC \simeq (P \times_{X_0} P) \times_{(\IC^\times \times \IC^\times)} (\IC \times \IC)
    \to
    P \times_{\IC^\times} \IC = P_\IC.
\]

\begin{example}
    Consider the $\IC^\times$-extension
    \begin{equation*}
        \begin{tikzcd}
            U \times \tilde{\Gamma} & {U \times \Gamma} \\
            {U} & {U},
            \arrow["\id",from=2-1, to=2-2]
            \arrow[from=1-1, to=1-2]
            \arrow[shift left=1, from=1-2, to=2-2]
            \arrow[shift right=1, from=1-2, to=2-2]
            \arrow[shift right=1, from=1-1, to=2-1]
            \arrow[shift left=1, from=1-1, to=2-1]
        \end{tikzcd}
    \end{equation*}
    presenting a gerbe over the quotient orbifold $[U/\Gamma]$.
    Then the composition on $P_\IC \define U \times \tilde{\Gamma} \times_{\IC^\times} \IC$
    is given by
    \begin{alignat*}{1}
        P_\IC \times_{U} P_\IC                                               & \to P_\IC \\
        \big((u,\tilde{\gamma},z) , (u \cdot \gamma,\tilde{\gamma}',z')\big) & \mapsto
        (u \cdot \gamma\gamma', \tilde{\gamma}\tilde{\gamma}',zz').
    \end{alignat*}
\end{example}

Denote the stack over $\orb{M}$ presented by the Lie groupoid $P_\IC \rightrightarrows X_0$ by
$\ger{G}_\IC \to \orb{M}$.
The map $\ger{G} \to \ger{G}_\IC$ is the inclusion of the maximal groupoid-valued
substack. It is represented by the map of Lie groupoids
% https://q.uiver.app/?q=WzAsNCxbMCwwLCJQIl0sWzEsMCwiUF9cXElDIl0sWzAsMSwiWCJdLFsxLDEsIlgiXSxbMiwzLCJcXGlkX1giXSxbMCwyLCIiLDAseyJvZmZzZXQiOi0xfV0sWzAsMiwiIiwwLHsib2Zmc2V0IjoxfV0sWzEsMywiIiwyLHsib2Zmc2V0IjoxfV0sWzEsMywiIiwyLHsib2Zmc2V0IjotMX1dLFswLDFdXQ==
\[\begin{tikzcd}
        P & {P_\IC} \\
        X & X,
        \arrow["{\id_X}", from=2-1, to=2-2]
        \arrow[shift left=1, from=1-1, to=2-1]
        \arrow[shift right=1, from=1-1, to=2-1]
        \arrow[shift right=1, from=1-2, to=2-2]
        \arrow[shift left=1, from=1-2, to=2-2]
        \arrow[from=1-1, to=1-2]
    \end{tikzcd}\]
where $P \to P_\IC$ is the map from $P$ to its associated bundle.
An object in $\ger{G}_\IC(U)$ is a map of stacks $U \to \ger{G}_\IC$.
As $U$ is represented by a groupoid, this functor
factors through $\ger{G} \into \ger{G}_\IC$.
The difference between $\ger{G}$ and $\ger{G}_\IC$ is seen on morphisms
between objects (cf. Example~\ref{ex:StackBR}).
The data of an object $\bar{f}:U \to \ger{G}$ is equivalent to a pair
$(f:U \to \orb{M},\tau:\ger{I} \eqto f^\ast\ger{G})$ of a map to $\orb{M}$ and
a trivialisation of the pullback gerbe.\footnote{This follows from the universal property
    of the pullback $f^\ast\ger{G} = U \times_{\orb{M}} \ger{G}$.}
\begin{lemma}
    \label{lem:morphismsInGerC}
    Given a pair of objects $(f,\tau),(g,\kappa) \in \ger{G}_\IC(U)$,
    the space of maps between them is empty if $f \not\simeq g:U \to \orb{M}$.
    A choice of 2-cell $\alpha: f \eqto g$ induces an isomorphism
    \[
        \Hom\left((f,\tau),(g,\kappa)\right) \simeq \Gamma(P_\IC),
    \]
    where $P_\IC$ is the $\IC$-bundle associated to the pair of trivialisations $(\alpha^\ast\tau,\kappa)$.
\end{lemma}
\begin{proof}
    A morphism $(f,\tau) \to (g,\kappa)$
    can only exist if there is a 2-cell $\alpha:f \eqto g$ identifying
    the underlying maps to $\orb{M}$. A particular choice of such a 2-cell $\alpha$
    induces an isomorphism
    \[
        \alpha^\ast: f^\ast\ger{G} \eqto g^\ast\ger{G}.
    \]
    This way, we obtain a pair of trivialisations
    $(\alpha^\ast\tau, \kappa)$ of $g^\ast\ger{G}$.
    A homomorphism $\bar{f} \to \bar{g}$ in $\ger{G}(U)$ is
    the data of a section of the $\IC^\times$-bundle $P$ given by
    $(\alpha^\ast\tau,\kappa)$, while the space of homomorphisms
    in $\ger{G}_{\IC}(U)$ is the space of sections of the associated $\IC$-bundle $P_\IC$.
\end{proof}

\begin{notation}
    We write $(\alpha,\sigma):(f,\tau) \to (g,\kappa)$ to denote
    a representative $(\alpha:f \simeq g,\sigma \in \Gamma(P_\IC))$,
    where $\alpha$ is used in the definition of $P_\IC$.
\end{notation}

Let $(f,\tau) \in \ger{G}_\IC(U)$ and $\sh{F} \in \ShC(U)$.
Recall the map $\Gr(f): U \to \orb{M} \times U$ is the graph of
$f:U \to \orb{M}$.
Denote by
\[
    (f,\tau)\cdot \sh{F} \define \Gr(f)_{\ast}(\tau \cdot \sh{F}) \in
    \Shv^{p_\orb{M}^\ast \ger{G}}(\orb{M} \times U) = \Shv^{\ger{G}}_\orb{M}(U)
\]
the associated $U$-family of $\ger{G}$-twisted sheaves over $\orb{M}$.
%\begin{lemma}
The assignment
\[
    \left((f,\tau), \sh{F}\right) \mapsto
    (f,\tau)\cdot \sh{F}
\]
naturally extends to a map of stacks
\begin{align*}
    \ger{G}_\IC \times \ShC & \to \Shv^\ger{G}_\orb{M}.
\end{align*}
%\end{lemma}
%\begin{proof}
We describe the value of this map on a morphism
\[
    \left((\alpha,\sigma),\beta\right):
    \left((f,\tau),\sh{F}\right) \to \left((g,\kappa),\sh{F}'\right)
\]
A 2-cell $\alpha:f \eqto g$ induces a natural equivalence
$\Gr(\alpha)_\ast: {\Gr(f)}_\ast \eqto {\Gr(g)}_\ast$.
A morphism of twisted sheaves $\alpha^\ast\tau \cdot \sh{F} \to \kappa \cdot \sh{F}'$
is given by a map of sheaves
\[
    \sh{F} \to {(\alpha^\ast\tau)}^{-1} \kappa \cdot \sh{F}' \define
    P_\IC \tensor_{\Cinf} \sh{F'}.
\]
A section $\sigma \in \Gamma(P_\IC)$ defines a morphism of twisted sheaves
$\alpha^\ast\tau \cdot \sh{F} \to \kappa \cdot \sh{F}$
with underlying map
\[
    \sigma \tensor \id_{\sh{F}}: \sh{F} \to P_\IC \tensor_{\Cinf} \sh{F}.
\]
The map $\ger{G}_\IC \times \ShC \to \Shv^\ger{G}_\orb{M}$ sends
the morphism $\left((f,\tau),\beta\right)$ to the composite of
natural transformations
% https://q.uiver.app/?q=WzAsNCxbMCwwLCJ7XFxHcihmKX1fXFxhc3QoXFx0YXUgXFxjZG90IFxcc2h7Rn0pIl0sWzEsMCwiXFxHcihnKV9cXGFzdChcXHRhdVxcY2RvdFxcc2h7Rn0pIl0sWzIsMCwiXFxHcihnKV9cXGFzdChcXGthcHBhXFxjZG90XFxzaHtGfSkiXSxbMywwLCJcXEdyKGcpX1xcYXN0KFxca2FwcGFcXGNkb3RcXHNoe0Z9JykiXSxbMSwyLCJcXHNpZ21hIl0sWzAsMSwiXFxhbHBoYSJdLFsyLDMsIlxcYmV0YSJdXQ==
\[\begin{tikzcd}
        {{\Gr(f)}_\ast(\tau \cdot \sh{F})} & {\Gr(g)_\ast(\tau\cdot\sh{F})} & {\Gr(g)_\ast(\kappa\cdot\sh{F})} & {\Gr(g)_\ast(\kappa\cdot\sh{F}')}.
        \arrow["\sigma", from=1-2, to=1-3]
        \arrow["\alpha", from=1-1, to=1-2]
        \arrow["\beta", from=1-3, to=1-4]
    \end{tikzcd}\]
(We suppress decorations and simply denote the natural transformations by the datum that induces them.)
By naturality, these maps of twisted sheaves are compatible with the pullback functors,
which identifies this as a map of stacks.
%\end{proof}
Applied to the distinguished object
$\Cinf \in \ShC(U)$, this yields a map
\begin{alignat*}{1}
    - \cdot \Cinf:\ger{G}_\IC & \to \Shv^{\ger{G}}_\orb{M}   \\
    (f,\tau)                  & \mapsto (f,\tau) \cdot \Cinf.
\end{alignat*}
\begin{lemma}
    The map $- \cdot \Cinf$ lands in $\Sky^{\ger{G}}_\orb{M} \subset \Shv^{\ger{G}}_\orb{M}$.
\end{lemma}
\begin{proof}
    We must show the condition in Definition~\ref{defn:familyOfTwistedSkyscraperSheaves}
    holds for $(f,\tau) \cdot \Cinf = \Gr(f)_{\ast}(\tau \cdot \Cinf)$.
    Let $Y \onto \orb{M}$ be a cover equipped with a trivialisation $\kappa$ of
    the gerbe. It pulls back to a cover
    % https://q.uiver.app/?q=WzAsNCxbMCwwLCJcXHRpbGRle1V9Il0sWzEsMCwiWSJdLFswLDEsIlUiXSxbMSwxLCJcXG9yYntNfSJdLFsyLDMsImYiLDJdLFswLDEsIlxcdGlsZGV7Zn0iXSxbMSwzLCJcXHBpIl0sWzAsMiwiXFxwaSciLDJdLFswLDMsIiIsMSx7InN0eWxlIjp7Im5hbWUiOiJjb3JuZXIifX1dXQ==
    \[\begin{tikzcd}
            {\tilde{U}} & Y \\
            U & {\orb{M}}
            \arrow["f"', from=2-1, to=2-2]
            \arrow["{\tilde{f}}", from=1-1, to=1-2]
            \arrow["\pi", from=1-2, to=2-2]
            \arrow["{p}"', from=1-1, to=2-1]
            \arrow["\lrcorner"{anchor=center, pos=0.125}, draw=none, from=1-1, to=2-2]
        \end{tikzcd}\]
    of $U$, and the product cover
    $\pi \times p:Y \times \tilde{U} \to \orb{M} \times U$ is equipped with
    the trivialisation $\kappa$ of the gerbe (pulled back from $Y$).

    Base change (Lemma~\ref{lem:baseChange}) along
    % https://q.uiver.app/?q=WzAsNCxbMCwwLCJcXHRpbGRle1V9Il0sWzAsMSwiVSJdLFsxLDAsIlkgXFx0aW1lcyBcXHRpbGRle1V9Il0sWzEsMSwiXFxvcmJ7TX1cXHRpbWVzIFUiXSxbMiwzLCJcXHBpIFxcdGltZXMgXFxwaSciXSxbMCwxLCJcXHBpJyIsMl0sWzEsMywiXFxHcihmKSIsMl0sWzAsMiwiXFxHcihcXHRpbGRle2Z9KSJdLFswLDMsIiIsMSx7InN0eWxlIjp7Im5hbWUiOiJjb3JuZXIifX1dXQ==
    \[\begin{tikzcd}
            {\tilde{U}} & {Y \times \tilde{U}} \\
            U & {\orb{M}\times U}
            \arrow["{\pi \times p}", from=1-2, to=2-2]
            \arrow["{p}"', from=1-1, to=2-1]
            \arrow["{\Gr(f)}"', from=2-1, to=2-2]
            \arrow["{\Gr(\tilde{f})}", from=1-1, to=1-2]
            \arrow["\lrcorner"{anchor=center, pos=0.125}, draw=none, from=1-1, to=2-2]
        \end{tikzcd}\]
    gives an isomorphism
    \[
        {(\pi\times p)}^\ast{\Gr(f)}_\ast(\tau \cdot \Cinf) \simeq
        {\Gr(\tilde{f})}_\ast{p}^\ast(\tau \cdot \Cinf) \simeq
        {\Gr(\tilde{f})}_\ast\big(p^\ast\tau \cdot p^\ast\Cinf\big).
    \]
    The trivialisation $\kappa$ of the gerbe over $Y\times \tilde{U}$
    identifies the pushforward functor ${\Gr(\tilde{f})}_\ast$ with the
    usual pushforward of $\Cinf$-modules:
    \[
        {\Gr(\tilde{f})}_\ast:
        \tau \cdot \sh{F} \mapsto \kappa \cdot
        {\Gr(\tilde{f})}_\ast(\mscr{L} \tensor \sh{F}),
    \]
    where $\mscr{L}$ denotes the $\IC$-bundle associated to
    the pair of trivialisations $({\Gr(\tilde{f})}^\ast\kappa,\tau)$
    of the gerbe over $\tilde{U}$.
    This shows the pullback of $(f,\tau) \cdot \Cinf$ to $Y \times \tilde{U}$
    is equivalent to
    \[
        {\Gr(\tilde{f})}_\ast \big(p^\ast\tau \cdot p^\ast\Cinf\big) \simeq
        \kappa \cdot
        {\Gr(\tilde{f})}_\ast(\mscr{L}\tensor p^\ast\Cinf),
    \]
    the pushforward of the vector bundle
    $\mscr{L} \tensor p^\ast \Cinf$ over $\tilde{U}$ along the
    graph of the map $\tilde{f}:\tilde{U} \to Y$.
    This identifies $(f,\tau)\cdot\Cinf$ as a $U$-family of
    twisted skyscraper sheaves.
\end{proof}

The map
\[
    -\cdot \Cinf:\ger{G}_{\IC} \to \Sky^{\ger{G}}_\orb{M}
\]
is not an equivalence: its component
\[
    \ger{G}_{\IC}(U) \to \Sky^{\ger{G}}_\orb{M}(U)
\]
over $U \in \Man$ is faithful, but neither full nor essentially surjective.

The failure of the inclusion to be full may be blamed on the absence of a
zero object in $\ger{G}_\IC(U)$.
The zero object in $\Sky^{\ger{G}}_\orb{M}(U)$ is given by the zero (twisted)
sheaf. It plays the role of the empty direct sum,
and the existence of zero morphisms makes $\Hom(X,Y)$ a vector space for
all pairs of objects $X,Y \in \Sky^{\ger{G}}_\orb{M}(U)$.
(Cf. Lemma~\ref{lem:morphismsInGerC}, which shows that $\Hom$-spaces
in $\ger{G}_\IC(U)$ may be empty.)
The zero object is preserved under all pullback functors, and so are
zero morphisms between pairs of objects in $\Sky^{\ger{G}}_\orb{M}(U)$.

Denote by $\ger{G}_\IC^+$ the
minimal full substack of $\Sky^{\ger{G}}_\orb{M}$ which
contains the image of $\ger{G}_\IC$ as well as the zero sheaf.
The stack $\ger{G}_\IC^+$ is a full substack of a $\VecInf$-enriched stack
and thus $\VecInf$-enriched itself.
We now Cauchy complete this stack. Cauchy completion is the completion
of a category under absolute
colimits~\cite{lawvere1973metric},
ie.\ colimits which are preserved by any ($\VecInf$-enriched) functor.
As such, this completion may be performed pointwise:
the pullback functors automatically extend to the Cauchy completions
and preserve the added absolute colimits.
For $\VecInf$-enriched categories, Cauchy completion amounts to completing under
direct sums and splitting idempotents.
Denote by $\cat{C}^\dirSum$ the completion of a category $\cat{C}$
under direct sums, and by $\Kar(\cat{C})$ its \emph{Karoubi envelope}:
the universal completion of $\cat{C}$ where all idempotents split.

We denote by $\Kar(\ger{G}_\IC^\dirSum)$ the stack over $\Man$
associated to the prestack
\[
    U \mapsto \Kar\left({\ger{G}_\IC^+(U)}^\dirSum\right)
\]
obtained from $\ger{G}_\IC^+ \subset \Sky_\orb{M}^\ger{G}$ by pointwise Cauchy completion.
\begin{proposition}
    \label{prop:SkyIsCauchyCompletionOfGCPlus}
    The stack of skyscraper sheaves is the Cauchy completion
    \[
        \Sky^\ger{G}_\orb{M} \simeq
        \Kar\left({\ger{G}_\IC^\dirSum}\right).
    \]
    When $\orb{M} = M$ is a manifold, it suffices to complete under direct sums:
    \[
        \Sky^\ger{G}_M \simeq \ger{G}_\IC^\dirSum.
    \]
\end{proposition}
\begin{proof}
    First, note the stack $\Sky^\ger{G}_\orb{M}$ is indeed
    Cauchy complete: completeness under direct sum is obvious, and
    we proved Karoubi completeness in Corollary~\ref{cor:CUisKaroubiComplete}.

    Now we show that $\ger{G}_\IC^+$ indeed Cauchy completes
    to $\Sky^\ger{G}_\orb{M}$.
    This means every family $\f{S} \in \Sky^\ger{G}_\orb{M}(U)$ locally restricts
    to a section of $\Kar\left({\ger{G}_\IC^\dirSum}\right)$:
    a summand of a sheaf
    \[
        \DirSum_{i=0}^r (f_i,\tau_i)\cdot \Cinf.
    \]
    But this follows directly from the normal form for families of twisted
    skyscraper sheaves established in
    Corollary~\ref{cor:locally_karoubi_completion_distributes}.

    Now assume $\orb{M}=M$ is a manifold. Then $M$ is an effective
    orbifold whose generic region (the region with trivial stabiliser) is
    all of $M$. We may now apply
    Corollary~\ref{cor:families_in_generic_region_are_mobile} to conclude
    that every family $\f{A} \in \Sky^\ger{G}_M(U)$ is mobile,
    and hence a direct sum of objects in $\ger{G}_\IC^+(U)$.
\end{proof}

\section{Orbifold Tensor Categories}
\label{sec:orbifoldTensorCategories}
Recall the notation introduced in Section~\ref{sec:orbisimples}:
$\VCat$ denotes the bicategory of $\dirSum$-complete $\VecInf$-enriched categories,
and $\Vect$ the $\VCat$-valued stack of vector bundles over the site $\Man$ 
of smooth manifolds.
We defined a \emph{linear stack} to be a module over $\Vect$ in the category of
$\VCat$-valued stacks.

\begin{defn}
    \label{def:tensorStack}
    A \emph{tensor stack} is a monoid object in the bicategory of linear stacks.
\end{defn}

Recall also Definition~\ref{def:orbisimpleCat}: An orbisimple category
is a linear stack $\underline{\cat{C}}$ which admits an equivalence
$\underline{\cat{C}} \simeq \Sky_{\orb{M}}^\ger{G}$ for some orbifold
$\orb{M}$ equipped with a gerbe $\ger{G}$. If $\orb{M}$ can be chosen
to be a manifold, such a stack is called a manisimple category.

\begin{defn}
    \label{def:orbifoldTensorCategory}
    An \emph{orbifold/manifold tensor category} is a tensor stack
    whose underlying linear stack is an orbisimple/manisimple category.
\end{defn}

We are primarily interested in the category
$\cat{C} \define \underline{\cat{C}}(\point)$
assigned to $\point \in \Man$.
The remaining data keeps track of the smooth structure of $\cat{C}$.
As $\cat{C} \iso {\Sky}_{\orb{M}}^{\ger{G}}(\point)$, it is automatically
semisimple abelian and Karoubi complete (Corollary~\ref{cor:catOfPointsDirSumOfRepCats}).

Recall that a left dual for an object $Y$
in a tensor category is a pair
$(Y^\vee, \ev: Y^\vee \tensor Y \to \ONE)$ such that there exists a map
$(\coev: \ONE \to Y \tensor Y^\vee)$ satisfying the snake equations
(cf. Definition~\ref{def:dualOfObject}).
%the composites
%\begin{equation*}
%\begin{tikzcd}
%Y \arrow[r, "\coev \tensor \id"]      & Y \tensor Y^\vee \tensor Y \arrow[r, "\id \tensor \ev"]        & Y      \\
%Y^\vee \arrow[r, "\id \tensor \coev"] & Y^\vee \tensor Y \tensor Y^\vee \arrow[r, "\coev \tensor \id"] & Y^\vee
%\end{tikzcd}
%\end{equation*}
%are required to be identities on $Y$ and $Y^\vee$, respectively.
%Right duals are defined as above, with the roles of $Y$ and $Y^\vee$ reversed.
An object with both left and right duals is called \emph{dualisable}, and
a tensor category is \emph{autonomous} if all of its objects are dualisable.

\begin{defn}
    \label{def:orbifusion}
    An orbifold/manifold tensor category $\underline{\cat{C}}$
    is \emph{orbifusion}/\-\emph{manifusion}
    if the category $\cat{C} = \underline{\cat{C}}(\point)$
    assigned to the point
    is autonomous and the monoidal unit $\ONE$ is simple.
\end{defn}

In Section~\ref{sec:smooth_rig}, we will prove that this guarantees that the
categories $\underline{\cat{C}}(U)$ an orbifusion category $\underline{\cat{C}}$ assigns to a manifold
$U$ are also autonomous --- so long as the underlying orbifold $\orb{M}$ may be chosen to be
effective.

\begin{example}
    Using Pontryagin duality and Clifford Theory, one may show that
    the categories of finite-dimensional unitary representations of a
    virtually abelian discrete group naturally form an orbifusion category.
    Indeed, an extension of groups $\IZ^n \to G \to H$ with $H$ finite, yields
    \begin{equation*}
        \Rep^u G = {\Rep^u (\IZ^n)}^H \simeq \Sky(T^n)^H \simeq \Sky([T^n/H]),
    \end{equation*}
    where $T^n$ denotes the n-torus.
    The monoidal structure is given by the usual tensor product of representations.
    We depicted the moduli space of simple objects for two such examples
    in Figure~\ref{fig:introExamplesGroupReps}.
\end{example}

We spell out the components of a tensor stack explicitly.
A tensor stack is a sixtuple consisting of
\begin{itemize}
    \item the underlying linear stack
          \begin{equation*}
              \underline{\cat{C}}: \Man^\opp \to \VCat
          \end{equation*}
    \item two morphisms
          \begin{align*}
              \ONE:    & \Vect \to \underline{\cat{C}}       \\
              \tensor: & \underline{\cat{C}} \boxtimes_\Vect
              \underline{\cat{C}} \to \underline{\cat{C}}
          \end{align*}
    \item and three natural isomorphisms
          \begin{align*}
              \alpha: & (- \tensor -) \tensor - \Longrightarrow - \tensor (- \tensor -) \\
              l:      & (\ONE \tensor -) \Longrightarrow \Id                            \\
              r:      & (- \tensor \ONE) \Longrightarrow \Id,
          \end{align*}
\end{itemize}
satisfying the triangle and pentagon equations familiar from the definition
of a monoidal category (see also Section~\ref{sec:monoidalStructures}).

Evaluating the above data at an object $U \in \Man$ yields the data of a
tensor category
$(\cat{C}^U,\tensor^U, \ONE^U, \alpha^U, l^U, r^U)$
equipped with a $\Vect(U)$-module structure (recall $\cat{C}^U=\stC(U)$).
As $\tensor$ is a natural transformation, it comes with naturality 2-isomorphisms for maps $f:V \to U$ in $\Man$:
\begin{equation*}
    \begin{tikzcd}
        (\underline{\cat{C}} \boxtimes \underline{\cat{C}})(U) \arrow[rr, "\tensor^U"] \arrow[dd, "f^\ast \boxtimes f^\ast"]               &  & \underline{\cat{C}}(U) \arrow[dd, "f^\ast"] \\
        &  &                                             \\
        (\underline{\cat{C}} \boxtimes \underline{\cat{C}})(V) \arrow[rr, "\tensor^V"] \arrow[rruu,"\simeq" sloped, shorten=2em, Rightarrow]              &  & \underline{\cat{C}}(V),
    \end{tikzcd}
\end{equation*}
which make the pullback functors monoidal.
The $\Vect$-module structure further equips this with isomorphisms
\begin{align*}
    b_U(X,P,Y):(X \odot P) \tensor Y \eqto X \tensor (Y \odot P),
\end{align*}
natural in $P \in \Vect(U)$ and $X,Y \in \cat{C}^U$.
There are natural isomorphisms
\[
    f^\ast(P \odot X) \simeq f^\ast P \odot f^\ast X,
\]
and the resulting square
% https://q.uiver.app/?q=WzAsNCxbMCwwLCJmXlxcYXN0XFxsZWZ0KChYIFxcb2RvdCBQKSBcXHRlbnNvciBZXFxyaWdodCkiXSxbMCwxLCIoZl5cXGFzdCBYIFxcb2RvdCBmXlxcYXN0IFApIFxcdGVuc29yIGZeXFxhc3QgWSJdLFsyLDAsImZeXFxhc3RcXGxlZnQoWCBcXHRlbnNvciAoWSBcXG9kb3QgUClcXHJpZ2h0KSJdLFsyLDEsImZeXFxhc3QgWCBcXHRlbnNvciAoZl5cXGFzdCBZIFxcb2RvdCBmXlxcYXN0IFApIl0sWzAsMSwiXFxzaW1lcSJdLFswLDIsImZeXFxhc3QgYihYLFAsWSkiXSxbMiwzLCJcXHNpbWVxIl0sWzEsMywiYihmXlxcYXN0IFgsZl5cXGFzdCBQLGZeXFxhc3QgWSkiLDJdXQ==
\[\begin{tikzcd}
        {f^\ast\left((X \odot P) \tensor Y\right)} && {f^\ast\left(X \tensor (Y \odot P)\right)} \\
        {(f^\ast X \odot f^\ast P) \tensor f^\ast Y} && {f^\ast X \tensor (f^\ast Y \odot f^\ast P)}
        \arrow["\simeq", from=1-1, to=2-1]
        \arrow["{f^\ast b(X,P,Y)}", from=1-1, to=1-3]
        \arrow["\simeq", from=1-3, to=2-3]
        \arrow["{b(f^\ast X,f^\ast P,f^\ast Y)}"', from=2-1, to=2-3]
    \end{tikzcd}\]
commutes.

This data can be combined to yield a ``global'' tensor product,
which we also denote by $\tensor$:
let $U,V \in \Man$, then the global tensor product is the functor
\begin{align*}
    \tensor: \underline{\cat{C}}(U) \boxtimes \underline{\cat{C}}(V) & \to \underline{\cat{C}}(U \times V)                 \\
    (X, Y)                                                           & \mapsto p_1^\ast X \tensor^{U \times V} p_2^\ast Y.
\end{align*}
The associator and unitors are ``globalised'' in the same way.
The pointwise tensor product $\tensor^U$ of $\underline{\cat{C}}(U)$ can be recovered
using the diagonal $\Delta_U:U \to U \times U$:
\begin{equation*}
    X \tensor^U Y = \Delta_U^\ast p_1^\ast X \tensor^U \Delta_U^\ast p_2^\ast Y = \Delta_U^\ast(p_1^\ast X \tensor^{U \times U} p_2^\ast Y) = \Delta_U^\ast (X \tensor Y).
\end{equation*}

\subsection{Fusion Categories are Manifusion Categories}
\label{sec:fusionCatsAreManifusion}
Let $\cat{D}$ be a tensor category and denote by $\Irr{\cat{D}}$ its
set of isomorphism classes of simple objects.
Recall the functor
\[
    - \boxtimes \Vect: \VCat \into \LinSt,
\]
which freely turns $\Vec$-modules into
$\Vect$-modules (ie.\ linear stacks).
This functor is monoidal, as shown in
Lemma~\ref{prop:boxtimesVectTurnsVecModulesIntoVectModules}.
\begin{corollary}
    \label{cor:boxtimesVectSendsTensorCatsToTensorStacks}
    The functor $-\boxtimes \Vect:\cat{D} \mapsto \underline{\cat{D}}$ sends
    $\IC$-linear $\dirSum$-complete monoidal categories to tensor stacks.
\end{corollary}

The value of the underlying linear stack at $U \in \Man$ is
\begin{equation*}
    \underline{\cat{D}}: U \mapsto \cat{D} \boxtimes \Vect(U).
\end{equation*}
%   whose objects $\cat{D} \boxtimes \Vect(U)$ are formal finite direct sums
%   \begin{align*}
%   \dirSum_i X_i \boxtimes P_i &  & X_i \in \cat{D}, P_i \in \Vect(U).
%   \end{align*}
%  and the morphisms are linearly extended from
%  \begin{align*}
%  \Hom(X \boxtimes P, Y \boxtimes Q) \define \Hom(X,Y) \tensor_\IC
%  \Hom(P,Q).
%  \end{align*}
%  \begin{rmk}
%  This naturally identifies the $\Vec$-action on $\cat{D}$ with
%  the $\Vec$-action on $\Vect(U)$. 
%  Let $V \in \Vec$, then
%  \begin{align*}
%  \Hom\big(X \boxtimes P, (Y\tensor V) \boxtimes Q\big) & =
%  \Hom(X,Y \tensor V) \tensor \Hom(P,Q) =                                                         \\
%  \Hom(X,Y) \tensor V \tensor \Hom(P,Q) & = \Hom(X,Y) \tensor \Hom(P,Q\tensor V) =       \\
%  & = \Hom\big(X\boxtimes P,Y \boxtimes (Q\tensor V)\big).
%  \end{align*}
%  \end{rmk}
Pick an equivalence of linear categories $\Vec^{\dirSum \Irr{\cat{D}}} \isoto \cat{D}$,
corresponding to a choice of representatives for the set $\Irr{\cat{D}}$.
It induces an equivalence of linear stacks
\begin{equation*}
    \underline{\cat{D}} \isoot
    \Vec^{\dirSum \Irr{\cat{D}}} \boxtimes \Vect \simeq
    \Vect^{\dirSum \Irr{\cat{D}}}.
\end{equation*}
Let $M \in \Man$ be the discrete manifold $\Irr{\cat{D}}$, a disjoint union of $|\Irr{\cat{D}}|$ points.
Then $\Sky_M$ is isomorphic to $\Vect^{\dirSum \Irr{\cat{D}}}$ (see Example~\ref{ex:SkyOfPoint})
as a linear stack. This identifies $\underline{\cat{D}}$ as a
manisimple category (Definition~\ref{def:orbisimpleCat}).
It also carries a monoidal structure induced by the monoidal structure of $\cat{D}$
(see Corollary~\ref{cor:boxtimesVectSendsTensorCatsToTensorStacks}).
\begin{cor}
    The assignment
    \[
        \cat{D} \mapsto \underline{\cat{D}}
    \]
    fully faithfully embeds the bicategory of $\IC$-linear $\dirSum$-complete
    monoidal categories into that of manifold tensor categories.
\end{cor}
By restriction, the same is true for fusion categories: they embed fully
faithfully into manifusion categories.
Hence, we don't distinguish between fusion categories and their
associated manifusion categories.

The category of points of $\underline{\cat{D}}$ is
\begin{equation*}
    \underline{\cat{D}}(\point) = \cat{D} \boxtimes \Vect(\point) \simeq \cat{D}.
\end{equation*}
It thus follows that
\begin{itemize}
    \item $\underline{\cat{D}}$ is a compact manifold tensor category iff $\Irr{\cat{D}}$ is finite and
    \item $\underline{\cat{D}}$ is manifusion iff $\cat{D}$ is fusion.
\end{itemize}

We describe the tensor product on $\underline{\cat{D}}$ explicitly.
Pick a skeletal representative for $\cat{D}$, so that the set of simples $X \in \cat{D}$
is identified with the (discrete) manifold $M=\Irr{\cat{D}}$.
Let $X \in \cat{D}$ be simple and $E$ a vector bundle over $U \in \Man$.
Then we denote by $(X,E)$ the $U$-family given by the
sheaf in $\ShC(M \times U) \iso \ShC(\coprod_r U)$
given by the pushforward of $E$ onto the connected component
corresponding to $X$.
Every $U$-family of skyscraper sheaves can be globally decomposed into direct
summands of this form.
The tensor product $\tensor^U$ is given by
\begin{equation*}
    (X,E) \tensor^U (Y,E') = (X\tensor Y, E \tensor E') =
    \DirSum_{Z \in X \tensor Y} (Z, E \tensor E'),
\end{equation*}
where the sum is over simples $Z$ in $X \tensor Y = \DirSum Z$.

\subsection{Categorified Group Rings}
\label{sec:categorifiedGroupRings}
The motivating examples for the study of manifusion categories are the
categorified group rings $\Vec^\omega[G]$ built from a Lie group $G$ and a
cocycle $\omega$ representing a cohomology class
$[\omega] \in \H^3(\B G, \IC^\times)$.
We will build $\Vec^\omega[G]$ as a manifold tensor category by linearising
the corresponding String 2-group $G_\omega$.

We begin by recalling the construction of $G_\omega$ as a smooth 2-group
after~\cite{schommer2011central}.
The cocycle $\omega$ may be represented by a simplicial cover
$Y_\bullet \to BG$ of the simplicial manifold $BG$, equipped with a triple
\[
    (\lambda,\mu,\omega) \in
    \Cinf(Y_1^{[3]},\IC^\times) \times
    \Cinf(Y_2^{[2]},\IC^\times) \times \Cinf(Y_3,\IC^\times)
\]
(we abuse notation and denote the third element of this triple by the same
letter as the entire cocycle).
This data may be used to build a smooth 2-group $G_\omega$:
The 2-cocycle $\lambda \in \Cinf(Y_1^{[3]},\IC^\times)$ encodes a
$\IC^\times$-gerbe over $G$ with a trivialisation over $Y_1$.
As explained in Section~\ref{sec:LieCatsCauchyCompletion},
we may build a Lie category $Y_1^{[2]} \times \IC \rightrightarrows Y_1$
from this data:
the composition map is given by
\begin{align*}
    Y_1^{[3]} \times \IC \times \IC & \to Y_1^{[2]} \times \IC                            \\
    (y,y',y'',z,z')                 & \mapsto (y,y'',z \cdot z' \cdot \lambda(y,y',y'')).
\end{align*}
Let $\ger{G}$ be the gerbe over $G$ presented by $\lambda$.
We denote the associated stack over $\Man$ by $\ger{G}_\IC$.
Recall from Section~\ref{sec:LieCatsCauchyCompletion} that there is a
natural map
\[
    - \cdot \Cinf: \ger{G}_\IC \into \Sky_G^\ger{G}.
\]
The data $\mu$ and $\omega$ equip this stack with the structure of a monoid
object in exact analogy with $G_\omega$.
We show that this structure extends along the inclusion above.
Recall we denoted by $\ger{G}_\IC^+ \subset \Sky_G^\ger{G}$
the full substack on the essential image of the inclusion above and
the zero object.
\begin{lemma}
    \label{lem:monoidStructureOnGCExtendsToGCPlus}
    The monoid structure on $\ger{G}_\IC$ extends to a
    monoid structure on $\ger{G}_\IC^+$.
\end{lemma}
\begin{proof}
    We begin by showing the monoidal structure extends objectwise.
    Let $U \in \Man$ be a parameterising family. We extend the
    tensor product $\tensor^U$ on
    $\ger{G}_\IC(U)$ to a tensor product $\tensor^{U,+}$ on
    $\ger{G}_\IC^+(U)$.
    On objects, this extension is straightforward:
    let $X,Y \in \ger{G}_\IC^+(U)$ be a pair of objects over
    $U$. If $X$ or $Y$ are the zero object $0$, then $X \tensor^{U,+} Y = 0$.
    Otherwise, both $X$ and $Y$ are objects of $\ger{G}_\IC(U)$, and
    the tensor product is given by $X \tensor^{U,+} Y = X \tensor^U Y$.
    The value of $\tensor^{U,+}$ on morphisms is constructed pointwise:
    Consider a pair of morphisms $(\alpha:X \to X',\beta:Y \to Y')$.
    By restricting $U$ enough, we may assume the objects involved are of the form
    \begin{align*}
        X = (f,\tau) \cdot \Cinf              &  & X' = (f',\tau') \cdot \Cinf   \\
        Y = (g,\kappa) \cdot \Cinf            &  & Y' = (g',\kappa') \cdot \Cinf \\
        X \tensor Y = (h,\lambda) \cdot \Cinf &  &
        X' \tensor Y' = (h',\lambda') \cdot \Cinf.
    \end{align*}
    Note that for $u \in U$, the equalities $f(u)=f'(u)$ and $g(u)=g'(u)$
    imply $h(u)=h'(u)$.
    Proposition~\ref{prop:locallyMorphismsAreFamiliesOfMatrices} allows us
    to embed $\Hom(X,X')$, $\Hom(Y,Y')$ and $\Hom(X \tensor Y, X' \tensor Y')$ into
    the space of smooth $\IC$-valued functions.
    This (non-canonically) identifies $\alpha$ with an endomorphism
    $\bar{\alpha} \in \End(X) \simeq \Cinf(U)$.
    Let $V_{f=f'}^\ring \subset U$ be the interior of the subset of
    $U$ where $f=f'$.
    The endomorphism $\bar{\alpha}$ is 0 on the complement of $V_{f=f'}^\ring$.
    Similarly, $\beta$ is identified with an endomorphism
    $\bar{\beta} \in \End(Y)$ which is 0 on the complement of $V_{g=g'}^\ring$.
    Both $\bar{\alpha}$ and $\bar{\beta}$ are morphisms in $\ger{G}_\IC(U)$,
    so the tensor product
    $\bar{\alpha} \tensor^U \bar{\beta} \in \End(X \tensor Y)$ is well-defined.
    The resulting endomorphism is 0 outside the set
    $V_{f=f'}^\ring \cap V_{g=g'}^\ring$ where $f=f'$ and $g=g'$,
    and thus in particular 0 outside the set $V_{h=h'}^\ring$.

    We may now use the technique above
    in reverse to identify the endomorphism $\bar{\alpha} \tensor^U \bar{\beta}$
    with a morphism
    \[
        \alpha \tensor^{U,+} \beta: X \tensor Y \to X' \tensor Y'.
    \]
    The choices made above imply that this morphism is only canonical up to
    an invertible function $c \in \Cinf(U,\IC^\times)$.
    As $\alpha \tensor^{U,+} \beta$ is 0 outside $W=V_{f=f'}^\ring \cap V_{g=g'}^\ring$,
    we need only fix the ambiguity there. But on $W$, both $\alpha$ and $\beta$
    restrict to morphisms in $\ger{G}_\IC(W)$, and the tensor product is
    already defined. This completes the construction of
    $\tensor^{U,+}$.

    It remains to extend the associators and unitors.
    The associator
    \[
        \alpha^{U,+}(X,Y,Z): (X \tensor Y) \tensor Z \eqto X \tensor (Y \tensor Z)
    \]
    agrees with $\alpha^U(X,Y,Z)$ unless one of $X,Y,Z \in \ger{G}_\IC^+$ is 0,
    in which case it is the zero morphism. The pentagon equation and naturality
    follow from the fact that any morphism commutes with a zero morphism.
    The unitors are dealt with in the same way.

    The monoidal structure described above is manifestly compatible with
    the restriction functors, because the zero object and zero morphisms are
    preserved under restriction.
    This gives $\ger{G}_\IC^+$ the structure of a monoid.
\end{proof}

\begin{proposition}
    \label{prop:monoidStructureOnGCExtendsToSky}
    The monoid structure on $\ger{G}_\IC$ induces on
    $\Sky_G^\ger{G}$ the structure of a manifusion category.
\end{proposition}
\begin{proof}
    The linear stack $\Sky_G^\ger{G}$ is manifestly a manisimple category
    (Definition~\ref{def:orbisimpleCat}), so it suffices to equip it with
    the structure of a monoid in the bicategory of linear stacks.

    Proposition~\ref{prop:SkyIsCauchyCompletionOfGCPlus} says that
    $\Sky^\ger{G}_G$ is the direct sum completion of the stack
    $\ger{G}_\IC^+ \subset \Sky^\ger{G}_G$.
    Hence every family $\f{A}:U \to \Sky^\ger{G}_G$ locally decomposes
    as a direct sum of objects in $\ger{G}_\IC^+$, and we may extend the
    monoid structure obtained from
    Lemma~\ref{lem:monoidStructureOnGCExtendsToGCPlus} linearly.
    This gives $\Sky^\ger{G}_G$ the structure of a monoid
    in the bicategory of $\VCat$-valued stacks.
    To upgrade this to a monoid structure in $\LinSt$,
    we must provide an invertible 2-cell $s$ implementing
    compatibility between the monoid structure and the $\Vect$-module structure
    (see Section~\ref{sec:linearStacks} for details).
    Its component at $U \in \Man$,
    \[
        s_U(X \boxtimes Y,P): (X \tensor Y) \odot P \eqto X \tensor (Y \odot P)
    \]
    must be natural in $P \in \Vect(U)$ and $X, Y \in \Sky^\ger{G}_G(U)$,
    and compatible with the restriction functors, associators, and unitors.
    The objects of $\Sky^\ger{G}_G(U)$ are locally generated under direct sum
    by objects in $\ger{G}_\IC^+(U)$, and the objects of $\Vect(U)$ by
    the object $\Cinf(U)$. Hence, it suffices to define the natural 2-cell
    on these objects. The morphism
    \[
        s_U(X \boxtimes Y,\Cinf): (X \tensor Y) \odot \Cinf(U)
        \eqto X \tensor Y \eqot X \tensor (Y \odot \Cinf(U)),
    \]
    is the composite of unitors for the $\Vect$-action. It is thus
    automatically compatible with restriction functors. By naturality, it also
    commutes with the associators and unitors for the monoid structure.
    The other components of $s$ are obtained by extending linearly, and hence
    share the above properties.
    This way, the monoid structure on $\Sky^\ger{G}_G$ is upgraded to the structure
    of a manifold tensor category.

    The category of points is the category $\Vec^{\omega^\delta}[G^\delta]$
    of $G^\delta$-graded vector spaces, where $G^\delta$ denotes
    $G$ with the discrete topology and $\omega^\delta$ the pullback cocycle.
    Hence, duals exist pointwise (for a simple object $\IC_g$,
    such a dual is provided by $\IC_{g^{-1}}$), and the unit $\IC_e$ is
    a simple object.
    Hence this manifold tensor category is a manifusion category.
\end{proof}

\begin{defn}
    The \emph{categorified group ring} $\Vec^\omega[G]$ associated to a
    Lie group $G$ and a cocycle $\omega \in \H^3(\B G,\IC^\times)$ is
    the manifusion category obtained by linearising $G_\omega$
    via the procedure outlined in
    Proposition~\ref{prop:monoidStructureOnGCExtendsToSky}.
\end{defn}

In particular, $\Vec^\omega[G]$ is built such that the maximal sub-2-group
recovers $G_\omega$.

\begin{rmk}
    Using the language of \emph{multiplicative gerbes}\footnote{
        These multiplicative $\IC^\times$-gerbes are
        usually called multiplicative \emph{bundle} gerbes
        in the literature.
    },
    we may describe the categorified group ring $\Vec^\omega[G]$ more directly.
    The group $\mathrm{H}^3(\B G, \IC^\times)$
    is in bijection with equivalence classes of
    multiplicative gerbes over $G$~\cite{waldorf2010multiplicative,schommer2011central}.
    A multiplicative structure on a gerbe $\ger{G}$ over $G$ is a pair
    of an equivalence of gerbes
    \[
        \mathrm{M}: p_1^\ast \ger{G} \tensor p_2^\ast \ger{G} \to m^\ast \ger{G}
    \]
    over $G \times G$, and an isomorphism $\alpha$
    implementing associativity over $G \times G \times G$,
    satisfying a coherence condition over $G \times G \times G \times G$.
    Evaluated at $(g,h,k,l) \in G^{\times 4}$, this coherence condition
    is the pentagon equation.
    The equivalence $\mathrm{M}$ allows us to pass between
    $p_1^\ast \ger{G} \tensor p_2^\ast \ger{G}$-twisted sheaves and
    $m^\ast \ger{G}$-twisted sheaves (see Appendix~\ref{app:twistedSheaves}).
    We will denote its action on sheaves by a tensor product.

    This structure allows us to equip $\Sky^\ger{G}_G$ (the underlying
    linear stack of $\Vec^\omega[G]$) with a monoid structure as follows:
    The convolution product of two objects $X, Y \in \Vec^\omega[G](\point)$
    is given by
    \begin{equation*}
        X \tensor Y \define m_\ast(\mathrm{M} \tensor p_1^\ast X \tensor p_2^\ast Y),
    \end{equation*}
    where $p_1, p_2, m:G \times G \to G$ are the obvious maps.
    The product is extended to $U$-families by multiplying all pullback und pushforward maps by $\id_U$.
    This is manifestly compatible under pullback of twisted sheaves.

    It remains to discuss the associator.
    We use the letter $q$ to denote projection maps with domain $G^{\times 3}$.
    The following equalities follow from
    monoidality of pullback functors, smooth base change along
    \begin{equation*}
        \begin{tikzcd}
            G \times G \times G \arrow[r, "m_{12} \times \id"] \arrow[d, "q_{12}"] & G \times G \arrow[d, "p_1"] \\
            G \times G \arrow[r, "m"]                                              & G
        \end{tikzcd}
    \end{equation*}
    and the projection formula
    $f_\ast \sh{F} \tensor \sh{G} = f_\ast(\sh{F} \tensor f^\ast \sh{G})$,
    which are proved in Appendix~\ref{app:twistedSheaves}:
    \begin{align*}
        (X \tensor Y) \tensor Z
         & = m_\ast\left(\mathrm{M} \tensor p_1^\ast m_\ast(\mathrm{M} \tensor p_1^\ast X \tensor p_2^\ast Y) \tensor p_2^\ast Z\right)                                \\
         & = m_\ast\left(\mathrm{M} \tensor {(m_{12} \times \id)}_\ast (q_{12}^\ast \mathrm{M} \tensor q_1^\ast X \tensor q_2^\ast Y) \tensor p_2^\ast Z\right)        \\
         & = m_\ast\left(\mathrm{M} \tensor {(m_{12} \times \id)}_\ast (q_{12}^\ast \mathrm{M} \tensor q_1^\ast X \tensor q_2^\ast Y \tensor q_3^\ast Z)\right)        \\
         & = m_{123,\ast} \left({(m_{12} \times \id)}^\ast \mathrm{M} \tensor q_{12}^\ast \mathrm{M} \tensor q_1^\ast X \tensor q_2^\ast Y \tensor q_3^\ast Z\right), \\
        X \tensor (Y \tensor Z)
         & = m_{123,\ast} \left({(\id \times\ m_{23})}^\ast \mathrm{M} \tensor q_{23}^\ast \mathrm{M} \tensor q_1^\ast X \tensor q_2^\ast Y \tensor q_3^\ast Z\right).
    \end{align*}
    The isomorphism
    \begin{equation*}
        \alpha:{(m_{12} \times \id)}^\ast \mathrm{M} \tensor q_{12}^\ast \mathrm{M}
        \isoto {(\id \times\ m_{23})}^\ast \mathrm{M} \tensor q_{23}^\ast \mathrm{M}
    \end{equation*}
    now precisely encodes the associator.
\end{rmk}

\subsection{Group Algebras in Tensor Stacks}
\label{sec:groupAlgebras}
An algebra object $A$ in a tensor category $\cat{C}$ induces a linear category of (right) modules
over $A$, which we denote $\Mod_A(\cat{C})$. 
In this section, we perform this construction in the world of
tensor stacks.
The central result is that when $\stC$ is an orbifold tensor category and
$A$ is a group algebra, then $\Mod_A(\cat{C})$ is naturally an orbisimple category.

Let $A$ be an algebra object in $\cat{C}=\stC(\point)$ (see Section~\ref{sec:algebrasAndModules}).
For each $U \in \Man$, this induces an algebra object $A^U$ in $\cat{C}^U$,
given by the constant family at $A$.
In particular, this implies
$f^\ast A^U \simeq A^V$ for all $f:V \to U$.
We may form the 2-presheaf (valued in $\VCat$) of modules over $A$
\begin{align*}
    \Mod_A(\stC): U \mapsto \Mod_{A^U}(\cat{C}^U),
\end{align*}
which objectwise assigns the category of right modules over the algebra.

\begin{lemma}
    The 2-presheaf $\Mod_A(\stC)$ is a stack.
\end{lemma}
\begin{proof}
    We prove that descent data of $A$-modules may be glued uniquely.
    Let $\pi:Y \onto U$ be a cover.
    A descent datum for $\Mod_A(\stC)$ over $Y$ is a triple
    \begin{align*}
        X \in \cat{C}^Y &  & \phi: p_1^\ast X \eqto p_2^\ast X  \in \cat{C}^{Y^{[2]}}
                        &  & m:X \tensor A^Y \to X,
    \end{align*}
    where $(X,\phi)$ forms a descent datum for $\stC$,
    and $m$ gives $X$ the structure of an $A_Y$-module.
    The isomorphism $\phi$ must be compatible with $m$,
    in that
    % https://q.uiver.app/?q=WzAsNCxbMCwwLCJwXzFeXFxhc3QgWCBcXHRlbnNvciBBX3tZXntbMl19fSAiXSxbMSwwLCJwXzFeXFxhc3QgWCJdLFswLDEsInBfMl5cXGFzdCBYIFxcdGVuc29yIEFfe1lee1syXX19ICJdLFsxLDEsInBfMl5cXGFzdCBYIl0sWzIsMywicF8yXlxcYXN0IG0iXSxbMCwxLCJwXzFeXFxhc3QgbSJdLFsxLDMsIlxccGhpIl0sWzAsMiwiXFxwaGkgXFx0ZW5zb3IgXFxpZF9BIiwyXV0=
    \[\begin{tikzcd}
            {p_1^\ast X \tensor A^{Y^{[2]}} } & {p_1^\ast X} \\
            {p_2^\ast X \tensor A^{Y^{[2]}} } & {p_2^\ast X}
            \arrow["{p_2^\ast m}", from=2-1, to=2-2]
            \arrow["{p_1^\ast m}", from=1-1, to=1-2]
            \arrow["\phi", from=1-2, to=2-2]
            \arrow["{\phi \tensor \id_A}"', from=1-1, to=2-1]
        \end{tikzcd}\]
    commutes.
    By naturality of the tensor product, the pair
    $(X \tensor A^Y, \phi \tensor \id_A)$ forms a descent datum as well,
    and the commutative square above witnesses $m$ as a morphism of descent
    data
    \[
        m:(X \tensor A^Y, \phi \tensor \id_A) \to (X,\phi).
    \]
    As $\stC$ is a stack, there exists an object $\tilde{X}$,
    unique up to unique isomorphism, which glues the descent datum
    $(X,\phi)$.
    The descent datum $(X \tensor A^Y, \phi \tensor \id_A)$ is then glued
    by $\tilde{X} \tensor A^U$, and the morphism $m$ lifts to
    a unique morphism
    \[
        \tilde{m}:\tilde{X} \tensor A^U \to \tilde{X}.
    \]
    As an action map, $m$ is assocative (and compatible with the unit, a
    datum which we suppress throughout). The condition of associativity
    involves a term $X \tensor A^Y \tensor A^Y$. This is part of a descent
    datum which glues to $\tilde{X} \tensor A^U \tensor A^U$.
    The functor sending objects of $\cat{C}^U$ to their descent data over
    $Y$ is an equivalence, so $\tilde{m}$ satisfies the same
    equations as $m$.
    Hence, the pair $(\tilde{X},\tilde{m})$ is a module over $A^U$ in $\cat{C}^U$.
    It glues the descent datum uniquely up to unique isomorphism,
    which shows $\Mod_A(\stC)$ is a stack.
\end{proof}

\begin{rmk}
    The endomorphism
    \begin{align*}
        - \tensor A: \stC \to \stC,
    \end{align*}
    is part of a monad~\cite{benabou1967introduction} on $\stC$ in $\LinSt$.
    Its unit and multiplication morphisms are induced by the inclusion of the unit
    $\ONE \into A$ and the monoid structure on $\stC$.
    Evaluated at $U \in \Man$, this monad is the usual monad associated to
    the algebra object $A^U \in \cat{C}^U$.
    The \emph{Eilenberg-Moore object} of a monad in a
    bicategory~\cite{street1972formal,lack2002formal} is the object of
    algebras over the monad.
    The stack $\Mod_A(\stC)$ is the Eilenberg-Moore object associated to $- \tensor A$.
\end{rmk}

The \emph{free module map}
\[
    - \tensor A: \stC \to \Mod_A(\stC)
\]
sends an object $X \in \cat{C}^U$ to the object
$X \tensor A \in {\Mod_A(\stC)}^U = \Mod_{A^U}(\cat{C}^U)$,
equipped with the free module structure induced by the algebra
structure on $A$.

\begin{lemma}
    \label{lem:ModAKaroubiCompletionOfFreeModules}
    The stack $\Mod_A(\stC)$ is Karoubi-generated by the
    essential image of the free module map:
    The category $\Mod_{A^U}(\cat{C}^U)$ is the Karoubi completion
    of the essential image of
    \[
        - \tensor A^U: \cat{C}^U \to \Mod_{A^U}(\cat{C}^U).
    \]
\end{lemma}
\begin{proof}
    First note that $\Mod_{A^U}(\cat{C}^U)$ is indeed Karoubi complete:
    Let $e \in \End(X,m)$ be an idempotent of a module. Then $e:X \to X$
    splits in $\cat{C}^U$ by Corollary~\ref{cor:CUisKaroubiComplete}.
    The restriction of $m$ to the summand picked out by $X$ gives it
    the structure of a right $A^U$-module.

    It remains to show that every module $(Z,m)$ is a summand of a free module.
    The morphism
    \[
        (Z \tensor A^U, m_{\mathrm{free}}) \xto{m} (Z, m)
    \]
    is split by the unit morphism (a datum we suppress throughout),
    and exhibits $(Z,m)$ as such a summand.
\end{proof}

Now, we specialise to group algebras.
Let $H$ be a finite group, and
assume $\stC$ admits a linear equivalence
$\stC \overset{\mathrm{lin}}{\simeq} \Sky^\ger{G}_\orb{M}$.
\begin{lemma}
    \label{lem:HactionOnSkyInducesHactionOnGandM}
    A map $\Vec[H] \to \stC$ induces an action of $H$ on $\ger{G} \to \orb{M}$.
\end{lemma}
\begin{proof}
    Denote the image of $\IC_h \subset \IC[H] \in \Vec[H]$ by $h$.
    The object $h$ is invertible, so $- \tensor h$ preserves the
    image of $\ger{G}$ under the inclusion $\ger{G} \into \Sky^\ger{G}_\orb{M}$
    introduced in Section~\ref{sec:LieCatsCauchyCompletion}.
    Thus $- \tensor h$ descends to a functor on $\ger{G}$, and
    we obtain an action of $H$ on $\ger{G}$.
    The action on $\ger{G}$ descends to an action on $\orb{M}$:
    The map $\ger{G} \into \Sky^\ger{G}_\orb{M}$ extends to a
    map $\ger{G} \times \ger{I} \into \Sky^\ger{G}_\orb{M} \boxtimes \Vect$,
    and the compatibility data between the $H$-action and the
    $\Vect$-module structure restrict to compatibility data
    between the $H$-action and the $\ger{I}$-action on $\ger{G}$.
    The orbifold $\orb{M}$ is the quotient of $\ger{G}$ by the 
    $\ger{I}$-action, so
    the compatibility data is equivalent to descent data for the $H$-action.
\end{proof}

We introduce notation for the $H$-action described in
Lemma~\ref{lem:HactionOnSkyInducesHactionOnGandM}:
The 1-cells of the $H$-action are given by diffeomorphisms
% https://q.uiver.app/?q=WzAsNCxbMCwwLCJcXGdlcntHfSJdLFsxLDAsIlxcZ2Vye0d9Il0sWzAsMSwiXFxvcmJ7TX0iXSxbMSwxLCJcXG9yYntNfSJdLFswLDJdLFsxLDNdLFswLDEsImhfe1xcZ2Vye0d9fSJdLFsyLDMsImhfe1xcb3Jie019fSIsMl1d
\[\begin{tikzcd}
        {\ger{G}} & {\ger{G}} \\
        {\orb{M}} & {\orb{M}}
        \arrow[from=1-1, to=2-1]
        \arrow[from=1-2, to=2-2]
        \arrow["{h_{\ger{G}}}", from=1-1, to=1-2]
        \arrow["{h_{\orb{M}}}"', from=2-1, to=2-2]
    \end{tikzcd}\]
for $h \in H$. The 2-cells
% https://q.uiver.app/?q=WzAsMyxbMCwxLCJcXGdlcntHfSJdLFsyLDEsIlxcZ2Vye0d9Il0sWzEsMCwiXFxnZXJ7R30iXSxbMCwyLCJoIl0sWzIsMSwiaCciXSxbMCwxLCJoaCciLDJdLFsyLDUsIlxcbXVfe2gsaCd9IiwyLHsic2hvcnRlbiI6eyJ0YXJnZXQiOjIwfX1dXQ==
\[\begin{tikzcd}
        & {\ger{G}} \\
        {\ger{G}} && {\ger{G}}
        \arrow["h", from=2-1, to=1-2]
        \arrow["{h'}", from=1-2, to=2-3]
        \arrow[""{name=0, anchor=center, inner sep=0}, "{hh'}"', from=2-1, to=2-3]
        \arrow["{\mu_{h,h'}}", shorten >=3pt, Rightarrow, from=1-2, to=0]
    \end{tikzcd}\]
implementing multiplicativity of the action are pulled
back from the associator $\alpha$ on $\stC$ along
the map $\ger{G} \to \stC$.

Associated to the $H$-action on $\ger{G} \to \orb{M}$ is the map of quotient stacks
$[\ger{G}/H] \to [\orb{M}/H]$. This fibration of stacks defines a gerbe over $[\orb{M}/H]$
(see Definition~\ref{defn:ICtimesGerbeAsPrincipalBundle}).
Indeed, the $\ger{I}$-action descends along $\ger{G} \to [\ger{G}/H]$
via the compatibility data between the $H$-action and the $\ger{I}$-action.
The local triviality condition also holds:
Let $Y \onto \orb{M}$ be a cover of $\orb{M}$ over which $\ger{G}$ trivialises.
Then $Y \onto \orb{M} \onto [\orb{M}/H]$ provides a cover over which $[\ger{G}/H]$ trivialises.

The category $\Vec[H]$ contains an algebra object $\IC[H]$,
and the map $\Vec[H] \to \stC$ induces the data of an algebra $A$
in $\stC$, the image of $\IC[H]$ under said map.
\begin{lemma}
    \label{lem:mapFromGoverHToModA}
    The map of stacks
    \[
        - \tensor A:\ger{G} \to \stC \to \Mod_A(\stC)
    \]
    factors through the quotient map
    \[
        \ger{G} \to [\ger{G}/H].
    \]
\end{lemma}
\begin{proof}
    We show this by equipping the map $- \tensor A:\ger{G} \to \Mod_A(\stC)$ with
    descent data for the cover $\ger{G} \to [\ger{G}/H]$.
    That is a family of 2-cells
    \[
        \kappa_h: (- \tensor h) \tensor A \eqto - \tensor A
    \]
    satisfying the cocycle condition
    % https://q.uiver.app/?q=WzAsNCxbMCwwLCIoKC1cXG90aW1lcyBoKVxcb3RpbWVzIGgnKSBcXG90aW1lcyBBIl0sWzAsMSwiKC1cXG90aW1lcyAoaFxcb3RpbWVzIGgnKSlcXG90aW1lcyBBIl0sWzEsMCwiKC1cXG90aW1lcyBoKVxcb3RpbWVzIEEiXSxbMSwxLCItXFxvdGltZXMgQSJdLFsxLDMsIlxca2FwcGFfe2hcXG90aW1lcyBoJ30oLSkiLDJdLFsyLDMsIlxca2FwcGFfaCgtKSJdLFswLDIsIlxca2FwcGFfe2gnfSgtXFxvdGltZXMgaCkiXSxbMCwxLCJcXGFscGhhX3stLGgsaCd9XFx0ZW5zb3IgXFxpZF9BIiwyXV0=
\[\begin{tikzcd}
	{((-\otimes h)\otimes h') \otimes A} & {(-\otimes h)\otimes A} \\
	{(-\otimes (h\otimes h'))\otimes A} & {-\otimes A}
	\arrow["{\kappa_{h\otimes h'}(-)}"', from=2-1, to=2-2]
	\arrow["{\kappa_h(-)}", from=1-2, to=2-2]
	\arrow["{\kappa_{h'}(-\otimes h)}", from=1-1, to=1-2]
	\arrow["{\alpha_{-,h,h'}\tensor \id_A}"', from=1-1, to=2-1]
\end{tikzcd}\]
    Such a family of 2-cells is provided by
    \[
        \kappa_h: (- \tensor h) \tensor A \xto{\alpha_{-,h,A}}
        - \tensor (h \tensor A) \xto{\id_{-} \tensor \mu_h} - \tensor A,
    \]
    where $\mu_h$ denotes the $h$-component of the multiplication on
    $A \simeq \DirSum_{h \in H} h$.
    This family indeed satisfies the cocycle condition:
    for any $X \in \ger{G}(U)$, the diagram
    % https://q.uiver.app/?q=WzAsOCxbMCwwLCIoKFggXFxvdGltZXMgaClcXG90aW1lcyBoJylcXG90aW1lcyBBIl0sWzEsMCwiKFggXFxvdGltZXMgaCkgXFxvdGltZXMgKGgnIFxcb3RpbWVzIEEpIl0sWzIsMCwiKFggXFxvdGltZXMgaCkgXFxvdGltZXMgQSJdLFswLDIsIihYIFxcb3RpbWVzIChoXFxvdGltZXMgaCcpKVxcb3RpbWVzIEEiXSxbMSwyLCJYIFxcb3RpbWVzICgoaFxcb3RpbWVzIGgnKVxcb3RpbWVzIEEpIl0sWzIsMiwiWCBcXG90aW1lcyBBIl0sWzEsMSwiWCBcXG90aW1lcyAoaFxcb3RpbWVzIChoJyBcXG90aW1lcyBBKSkiXSxbMiwxLCJYXFxvdGltZXMoaFxcb3RpbWVzIEEpIl0sWzQsNSwiXFxpZF9YIFxcb3RpbWVzIFxcbXVfe2ggXFxvdGltZXMgaCd9IiwyXSxbMiw3LCJcXGFscGhhX3tYLGgsQX0iXSxbMCwzLCJcXGFscGhhX3tYLGgsaCd9IFxcb3RpbWVzIFxcaWRfQSIsMl0sWzAsMSwiXFxhbHBoYV97WFxcb3RpbWVzIGgsaCcsQX0iXSxbMSw2LCJcXGFscGhhX3tYLGgsaCdcXG90aW1lcyBBfSJdLFs0LDYsIlxcaWRfWCBcXG90aW1lcyBcXGFscGhhX3toLGgnLEF9IiwyXSxbMyw0LCJcXGFscGhhX3tYLGhcXG90aW1lcyBoJyxBfSIsMl0sWzEsMiwiKFxcaWRfWCBcXG90aW1lcyBcXGlkX2gpIFxcb3RpbWVzIFxcbXVfe2gnfSJdLFs2LDcsIlxcaWRfWCBcXG90aW1lcyAoXFxpZF9oIFxcb3RpbWVzIFxcbXVfe2gnfSkiLDJdLFs3LDUsIlxcaWRfWCBcXG90aW1lcyBcXG11X2giXV0=&macro_url=https%3A%2F%2Fchristoph-weis.com%2Fassets%2Fmacros.tex
\[\begin{tikzcd}[column sep=large]
	{((X \otimes h)\otimes h')\otimes A} & {(X \otimes h) \otimes (h' \otimes A)} & {(X \otimes h) \otimes A} \\
	& {X \otimes (h\otimes (h' \otimes A))} & {X\otimes(h\otimes A)} \\
	{(X \otimes (h\otimes h'))\otimes A} & {X \otimes ((h\otimes h')\otimes A)} & {X \otimes A}
	\arrow["{\id_X \otimes \mu_{h \otimes h'}}"', from=3-2, to=3-3]
	\arrow["{\alpha_{X,h,A}}", from=1-3, to=2-3]
	\arrow["{\alpha_{X,h,h'} \otimes \id_A}"', from=1-1, to=3-1]
	\arrow["{\alpha_{X\otimes h,h',A}}", from=1-1, to=1-2]
	\arrow["{\alpha_{X,h,h'\otimes A}}", from=1-2, to=2-2]
	\arrow["{\id_X \otimes \alpha_{h,h',A}}"', from=3-2, to=2-2]
	\arrow["{\alpha_{X,h\otimes h',A}}"', from=3-1, to=3-2]
	\arrow["{(\id_X \otimes \id_h) \otimes \mu_{h'}}", from=1-2, to=1-3]
	\arrow["{\id_X \otimes (\id_h \otimes \mu_{h'})}"', from=2-2, to=2-3]
	\arrow["{\id_X \otimes \mu_h}", from=2-3, to=3-3]
\end{tikzcd}\]
    commutes by the pentagon equation for $\alpha$, the naturality of $\tensor$,
    and associativity of $\mu$.
\end{proof}

Lemma~\ref{lem:mapFromGoverHToModA} constructs a map $[\ger{G}/H] \to \Mod_A(\stC)$.
As $\Mod_A(\stC)$ is $\VCat$-valued (and, in particular, has zero objects),
we may extend this to a map ${[\ger{G}/H]}_\IC \to \Mod_A(\stC)$.
Lemma~\ref{lem:ModAKaroubiCompletionOfFreeModules} implies $\Mod_A(\stC)$
is further Karoubi complete, so
we may extend further to a map of linear stacks
$\Sky_{[\orb{M}/H]}^{[\ger{G}/H]} \to \Mod_A(\stC)$.
\begin{prop}
    \label{prop:ModAisOrbisimple}
    The map
    \[
        \Sky_{[\orb{M}/H]}^{[\ger{G}/H]} \to \Mod_A(\stC)
    \]
    is an equivalence of linear stacks.
    Hence $\Mod_A(\stC)$ is orbisimple.
\end{prop}
\begin{proof}
    The essential image of $[\ger{G}/H] \to \Mod_A(\stC)$ constructed in
    Lemma~\ref{lem:mapFromGoverHToModA} contains the essential
    image of $- \tensor A: \ger{G} \to \Mod_A(\stC)$.
    Lemma~\ref{lem:ModAKaroubiCompletionOfFreeModules} identifies
    $\Mod_A(\stC)$ as the Cauchy completion of the essential image of
    the free module map $-\tensor A:\stC \to \Mod_A(\stC)$.
    Meanwhile,
    Proposition~\ref{prop:SkyIsCauchyCompletionOfGCPlus} identifies
    $\Sky_{[\orb{M}/H]}^{[\ger{G}/H]}$ as the Cauchy completion of the
    essential image of ${[\ger{G}/H]}_\IC$.
    This shows the map in the statement is essentially surjective.

    Full faithfulness may be checked on the map
    $[\ger{G}/H]_{\IC} \to \Mod_A(\stC)$ --- the property is
    preserved under Cauchy completion.
    Every object $X \in [\ger{G}/H]_\IC(U)$ locally lifts to an object
    in $\ger{G}_\IC(U)$. If $\tilde{X}$ is such a lift, then so is
    $\tilde{X} \tensor h$.
    The map ${[\ger{G}/H]}_\IC \to \Mod_A(\stC)$ sends
    \[
        X \mapsto \tilde{X} \tensor A.
    \]
    We must prove that the map
    \[
        \Hom_{{[\ger{G}/H]}_\IC(U)}(X,Y) \to
        \Hom_{\Mod_A(\cat{C}^U)}(\tilde{X} \tensor A, \tilde{Y} \tensor A)
    \]
    is an isomorphism for all $X, Y \in [\ger{G}/H]_\IC$ with
    chosen lifts $\tilde{X},\tilde{Y} \in \ger{G}_\IC$.
    The right-hand side simplifies to
    \[
        \Hom_{\Mod_A(\cat{C}^U)}(\tilde{X} \tensor A, \tilde{Y} \tensor A) \simeq
        \Hom_{\cat{C}^U}(\tilde{X}, \tilde{Y} \tensor A) \simeq
        \DirSum_{h \in H}\Hom_{\cat{C}^U}(\tilde{X}, \tilde{Y} \tensor h),
    \]
    exhibiting every morphism in $\Mod_A(\cat{C}^U)$ as a direct sum of morphisms
    of the form
    \[
        \tilde{X} \tensor A \xto{\tilde{f} \tensor \id_A}
        (\tilde{Y} \tensor h) \tensor A \xto{\kappa_h(\tilde{Y})} \tilde{Y} \tensor A.
    \]
    Here, $\tilde{f}:\tilde{X} \to \tilde{Y} \tensor h$ is a morphism in
    $\cat{C}(U)$, and $\kappa_h: (- \tensor h) \tensor A \to - \tensor A$
    the obvious isomorphism (see also the proof of Lemma~\ref{lem:mapFromGoverHToModA}).

    Applying the quotient functor $\ger{G} \to [\ger{G}/H]$ to $\tilde{f}$,
    we obtain a morphism $f:X \to Y$, whose image under
    $[\ger{G}/H] \to \Mod_A(\cat{C}^U)$ is $\tilde{f} \tensor A$.
    This shows fullness of the map.

    To see the map is also faithful, we use the commutative square
    % https://q.uiver.app/?q=WzAsNCxbMCwwLCJcXGdlcntHfV9cXElDKFUpIl0sWzAsMSwiW1xcZ2Vye0d9L0hdX1xcSUMoVSkiXSxbMSwxLCJcXFNreV97W1xcb3Jie019L0hdfV57W1xcZ2Vye0d9L0hdfShVKSJdLFsxLDAsIlxcU2t5X3tcXG9yYntNfX1ee1xcZ2Vye0d9fShVKSJdLFswLDEsInEiLDJdLFszLDIsInsocVxcdGltZXMgXFxpZF9VKX1fXFxhc3QiXSxbMCwzXSxbMSwyXV0=
    \[\begin{tikzcd}
            {\ger{G}_\IC(U)} & {\Sky_{\orb{M}}^{\ger{G}}(U)} \\
            {[\ger{G}/H]_\IC(U)} & {\Sky_{[\orb{M}/H]}^{[\ger{G}/H]}(U)}
            \arrow["q"', from=1-1, to=2-1]
            \arrow["{{(q\times \id_U)}_\ast}", from=1-2, to=2-2]
            \arrow[from=1-1, to=1-2]
            \arrow[from=2-1, to=2-2]
        \end{tikzcd}\]
    to write
    \begin{align*}
        \Hom_{[\ger{G}/H]_\IC(U)}(X,Y) & =
        \Hom_{[\ger{G}/H]_\IC}\left(
        {(q \times \id_U)}_\ast \tilde{X},
        {(q \times \id_U)}_\ast \tilde{Y}
        \right)                                 \\
                                       & \simeq
        \Hom_{\Sky_{\orb{M}}^{\ger{G}}}\left(
        {(q \times \id_U)}^\ast {(q \times \id_U)}_\ast \tilde{X},
        \tilde{Y}
        \right)                                 \\
                                       & \simeq
        \DirSum_{h \in H} \Hom_{\cat{C}^U}(\tilde{X} \tensor h, \tilde{Y}).
    \end{align*}
    This has the same rank (as a $\Cinf$-module) as
    $\Hom_{\Mod_A(\cat{C}^U)}(\tilde{X} \tensor A, \tilde{Y} \tensor A)$,
    so fullness implies faithfulness.
\end{proof}

%The quotient map $\ger{G} \to [\ger{G}/H]$ extends to a map of stacks
%$\ger{G}_\IC \to {[\ger{G}/H]}_\IC$ (see Section~\ref{sec:LieCatsCauchyCompletion}) 
%and by linearisation to a map of linear stacks
%\[
%    \Sky^\ger{G}_\orb{M} \to \Sky^{[\ger{G}/H]}_{[\orb{M}/H]}.
%\]

\begin{example}
    \label{ex:groupalgebrasInCategorifiedGroupRings}
    Let $\Vec^\omega[G]$ be a categorified group ring associated to a
    Lie group $G$ and 3-cocycle $\omega \in \H^3(\B G, \IC^\times)$,
    as constructed in Section~\ref{sec:categorifiedGroupRings}.
    Let $H$ be a finite subgroup $i: H \into G$ and pick a trivialisation of the
    pullback 3-cocycle $i^\ast \omega$ on $H$.
    This data induces a map $\Vec[H] \to \Vec^\omega[G]$.
    Denote by $\theta \in \H^2(G, \IC^\times)$ the transgressed cocycle of
    $\omega$, classifying the gerbe over $G$.
    The trivialisation of $i^\ast \omega$ induces descent data for 
    the gerbe over $G$ along the quotient map $G \onto G/H$.
    This exhibits $\theta$ as the image of a class
    $\theta/H$ under the morphism $\H^2(G/H,\IC^\times) \to \H^2(G, \IC^\times)$.

    The stack of modules $\Mod_{\IC[H]}(\Vec^\omega[G])$ is a manisimple category,
    equivalent to the linear stack $\Vec^{\theta/H}[G/H]$
    (we slightly abuse notation here: this stack does not have a monoidal product).
\end{example}

A particular example is provided by $\SU(2)$ at level 
$k \in \H^3(\B\SU(2),\IC^\times) \simeq \IZ$, and the subgroup
$i:\IZ/2 \into \SU(2)$. The relevant map on cohomology,
\begin{alignat*}{1}
    i^\ast:\H^3(\B \SU(2),\IC^\times) & \to \H^3(\B \IZ/2),
\end{alignat*}
is the surjective map $\IZ \onto \IZ/2$. Hence there is a map of tensor stacks
$\Vec[\IZ/2] \to \Vec^k[\SU(2)]$ precisely when $k$ is even.

Recall the underlying manifold of $\SU(2)$ is $S^3$.
The quotient manifold $S^3/(\IZ/2)$ is $\RP^3$, the underlying manifold of
$\SO(3)$. Gerbes over $\RP^3$ are classified by the integers, and
the pullback map
\begin{alignat*}{1}
    \H^2(\RP^3,\IC^\times) \to \H^2(S^3,\IC^\times)
\end{alignat*}
is the map $\IZ \xto{\cdot 2} \IZ$.
Now assume $k$ is even, and pick a functor
$\Vec[\IZ/2] \to \Vec^k[\SU(2)]$. By abuse of notation, we denote the manisimple category of
$\IZ/2$-modules in $\Vec^{k}[\SU(2)]$ by $\Vec^{k/2}[\RP^3]$ 
(Recall this is the notation we used for categorified group rings, 
but the manisimple category $\Vec^{k/2}[\RP^3]$ does not carry a monoidal structure.)
In the next section, we show that when $k$ is divisible by 4, the category 
$\Vec^{k/2}[\RP^3]$ naturally acquires the structure of a manifusion category 
(and in fact turns into $\Vec^{k/4}[\SO(3)]$).

\subsection{Central Group Actions on Tensor Stacks}
\label{sec:centralGroupActions}
Recall that the Drinfel'd centre of a tensor category $\cat{C}$ is the
braided tensor category $\DZ \cat{C}$ with
\begin{itemize}
    \item objects: tuples $(X \in \cat{C}, \gamma:X \tensor - \to - \tensor X)$, s.t.\ $\gamma$
          satisfies the hexagon equation,
    \item morphisms $(X, \gamma) \to (X^\prime, \gamma^\prime)$: morphisms $f: X \to X^\prime$
          compatible with $\gamma$ and $\gamma^\prime$.
\end{itemize}
A natural isomorphism $\gamma$ as above is called a \emph{half-braiding} for $X$.

A \emph{commutative algebra} in a tensor category $\cat{C}$ is
an algebra $A$ in its Drinfel'd centre whose multiplication
commutes with the braiding $\beta$ in $\DZ \cat{C}$:
% https://q.uiver.app/?q=WzAsMyxbMCwwLCJBIFxcdGVuc29yIEEiXSxbMiwwLCJBIFxcdGVuc29yIEEiXSxbMSwxLCJBIl0sWzAsMSwiXFxiZXRhKEEsQSkiXSxbMCwyLCJtIiwyXSxbMSwyLCJtIl1d
\[\begin{tikzcd}
        {A \tensor A} && {A \tensor A} \\
        & A.
        \arrow["{\beta(A,A)}", from=1-1, to=1-3]
        \arrow["m"', from=1-1, to=2-2]
        \arrow["m", from=1-3, to=2-2]
    \end{tikzcd}\]

In~\cite{pareigis1995braiding}, it was shown that the category of modules over
a commutative algebra $A$ in $\cat{C}$ is monoidal (subject to certain
completeness conditions on $\cat{C}$).
In this section, we extend this result
to the context of orbifold tensor categories in the case of a finite group algebra.

We recall the necessary details from~\cite{pareigis1995braiding}.
Let $(A,\gamma:A \tensor - \to - \tensor A)$ be a commutative algebra in $\cat{C}$.
The half-braiding $\gamma$ may be used to equip a right module $M$ with a left
module structure via
\[
    A \tensor M \xto{\gamma(M)} M \tensor A \to M.
\]
The commutativity condition ensures that the left module structure is compatible
with the right module structure, making $M$ a bimodule.
The tensor product of bimodules $M,N$ is computed by the coequaliser
(whose existence is assumed)
\[
    M \tensor_A N \define
    \colim(M\tensor A \tensor N \rightrightarrows M \tensor N).
\]
The tensor product of free right modules $M = X \tensor A, N = Y \tensor A$ is
given by a free module
\[
    (X \tensor A) \tensor_A (Y \tensor A) \simeq (X \tensor Y) \tensor A.
\]

Now we carry the above story over to our setup.
As explained in~\cite{street2004centre}, the notion of Drinfel'd centre admits
an internalisation to any braided monoidal bicategory.
In the bicategory $\SmSt$ of stacks over $\Man$, this is equivalent to the braided monoid
\begin{equation*}
    \DZ \stC = \End({}_\stC\stC_\stC)
\end{equation*}
of endomorphisms of $\stC$, considered as a $\stC$-$\stC$-bimodule.
It is equipped with a canonical monoidal morphism $\DZ \stC \to \stC$.
The centre is computed by a limit. Hence the functor
$\ev_U: \SmSt \to \VCat$ that evaluates a stack at $U \in \Man$
induces a monoidal comparison functor
$(\DZ\stC)(U) \to \DZ(\cat{C}^U)$.
This functor should be thought of as including those half-braidings that are smooth.

Let $\Vec[H]$ denote the braided fusion category with simple objects $H$,
equipped with the trivial associator and trivial braiding.
\begin{defn}
    A \emph{central group action} of a finite abelian group $H$ on a
    tensor stack $\stC$ is a braided monoidal morphism of linear stacks
    $\Vec[H] \to \DZ \stC$.
\end{defn}

A central group action
of $H$ on $\stC$ induces a commutative algebra $A$ in $\stC$, the image of
the commutative algebra $\IC[H]$ under the composite
$\Vec[H] \to \DZ \stC \to \stC$.
For each $U \in \Man$, we obtain a commutative algebra $A^U \in \cat{C}^U$.
Forgetting the braiding, the algebra $A$ is a group algebra in $\stC$.
In Section~\ref{sec:groupAlgebras}, we constructed the stack of modules over a
group algebra in an orbifold tensor category, and showed it is an orbisimple category
(Proposition~\ref{prop:ModAisOrbisimple}).

\begin{thm}
    \label{thm:quotientByCentralGroupActionIsOrbifoldTensorCategory}
    Let $A$ be the algebra induced by a central group action
    $\Vec[H] \to \DZ \stC$ on an orbifold tensor category.
    The stack $\Mod_A(\stC)$ is naturally an orbifold tensor category.
\end{thm}
\begin{proof}
    Proposition~\ref{prop:ModAisOrbisimple} assures us that $\Mod_A(\stC)$
    is orbisimple. It remains to equip it with a monoidal structure.

    We do this objectwise: the category
    \[
        \left(\Mod_A(\stC)\right)(U) = \Mod_{A^U}(\cat{C}^U)
    \]
    is monoidal with tensor product $\tensor_{A^U}$
    as long as the necessary colimit in the diagram exists.
    The tensor product of free modules in $\Mod_{A^U}(\cat{C}^U)$ exists: it
    is given by
    \[
        (X \tensor A^U) \tensor_{A^U} (Y \tensor A^U) \simeq
        (X \tensor Y) \tensor A^U.
    \]
    By Lemma~\ref{lem:ModAKaroubiCompletionOfFreeModules}, $\Mod_{A^U}(\cat{C}^U)$
    is Karoubi-generated by free modules.
    The monoidal structure on a category uniquely extends to its Karoubi completion:
    given a pair of objects $(X,e), (Y,f)$ picked out by idempotents
    $e:X \to X$ and $f:Y \to Y$, respectively,
    their tensor product is given by
    \[
        (X,e) \tensor (Y,f) \simeq (X \tensor Y,e \tensor f).
    \]
    This extends to a monoidal structure on the stack, because $A$ is
    preserved under pullback along $f:V \to U$, and thus
    \[
        f^\ast(- \tensor_{A^U} -) \simeq f^\ast - \tensor_{A^V} f^\ast -. \qedhere
    \]
\end{proof}

\begin{definition}
    \label{defn:quotientOfOrbifoldTensorCategoryByCentralGroupAction}
    We call $\Mod_A(\stC)$, equipped with the structure of an orbifold
    tensor category constructed in
    Theorem~\ref{thm:quotientByCentralGroupActionIsOrbifoldTensorCategory}
    the \emph{quotient of $\stC$ by $H$},
    and denote it by $\stC^H$.
\end{definition}

\begin{example}
    \label{ex:centralSubgroupInCategorifiedGroupRing}
    Let $\Vec^\omega[G]$ be a categorified group ring associated to a
    Lie group $G$ and $\omega \in \H^3(\B G, \IC^\times)$.
    At certain $\omega$, central subgroups $H \into Z(G) \to G$
    admit a lift to a central group action
    $\Vec[H] \to \DZ \Vec^\omega[G]$. Such a functor induces the data
    of a 3-cocycle $\omega/H \in \H^3(\B (G/H), \IC^\times)$
    which lifts $\omega$ against
    \[
        \H^3(\B (G/H), \IC^\times) \to \H^3(\B G, \IC^\times).
    \]
    The quotient of $\Vec^\omega[G]$ by $H$ is given by
    the categorified group ring
    \[
        {(\Vec^\omega[G])}^H \simeq \Vec^{\omega/H}[G/H].
    \]
\end{example}
As a particular example,
the map $\H^3(\B \SO(3),\IC^\times) \to \H^3(\B \SU(2),\IC^\times)$
is the map $\IZ \xto{\cdot 4} \IZ$.
We show in
Examples~\ref{ex:centreOfSUn}
that the inclusion of the centre $\IZ/2 \into \SU(2)$
admits a unique lift to a central group action
$\Vec[\IZ/2] \into \DZ \Vec^k[\SU(2)]$ precisely when $k \in \IZ$
is a multiple of 4.
The associated quotient is the categorified group ring $\Vec^{k/4}[\SO(3)]$.
Its underlying manisimple category is $\Vec^{k/2}[\RP^3]$, see the end of 
Section~\ref{sec:groupAlgebras}.

\subsection{Constructing Interpolated Fusion Categories}
\label{sec:constructInterpolatedFusionCats}
In this section, we explain how one may
leverage the results of Section~\ref{sec:centralGroupActions}
to construct new families of manifusion categories.
We study the Drinfel'd centres of
categorified group rings $\Vec^\omega[G]$ and fusion categories $\cat{D}$.
The product $\Vec^\omega[G] \boxtimes_\Vect \cat{D}$ has centre
\begin{equation*}
    \DZ(\Vec^\omega[G] \boxtimes_\Vect \underline{\cat{D}}) \simeq
    \DZ\Vec^\omega[G] \boxtimes_\Vect \DZ\underline{\cat{D}}.
\end{equation*}
As fusion categories embed into the bicategory of
manifusion categories, we have an equivalence
\[
    \DZ(\underline{\cat{D}}) \simeq \underline{\DZ(\cat{D})},
\]
so $\DZ(\cat{D})$ may be computed treating $\cat{D}$ as an ordinary
fusion category.
We find braided inclusions $\Vec^{\omega^\prime,\beta}[H] \to \Vec^\omega[G]$
(where $H$ is a finite abelian group, and $\omega^\prime, \beta$ indicate
the associator and braiding, respectively) first in
$\DZ(\cat{D})$, and then look for maps from the \emph{oppositely} braided category
$\Vec^{-\omega^\prime,-\beta}[H]$ to $\DZ(\Vec^\omega[G])$.
The diagonal embedding
\begin{equation*}
    \Vec[H] \into
    \Vec^{-\omega^\prime,-\beta}[H] \boxtimes
    \Vec^{\omega^\prime,\beta}[H] \to
    \DZ(\Vec^\omega[G] \boxtimes \cat{D})
\end{equation*}
now yields a central $H$-action on $\Vec^\omega[G] \boxtimes_\Vect \cat{D}$.
We denote the quotient by
\begin{equation*}
    \Vec^\omega[G] \boxtimes_{\Vec[H]} \cat{D} \define
    \left(\Vec^\omega[G] \boxtimes_\Vect \cat{D}\right)^H.
\end{equation*}
It is an orbifold tensor category whose underlying
orbifold is given by $[G \times \Irr(\cat{D})/H]$.
The \emph{free algebra morphism}
\begin{equation*}
    - \tensor \IC[H]: \Vec^\omega[G] \boxtimes \cat{D} \to
    \Vec^\omega[G] \boxtimes_{\Vec[H]} \cat{D}
\end{equation*}
is given by tensoring with the group algebra of $H$, and equipping
the result with the free algebra structure. It is monoidal,
so the fusion rules and associator of its essential image are
easy to determine from those on the domain.
When the $H$-action is free, this functor is in fact essentially
surjective, and the monoidal structure on the quotient
is completely described by it.

When drawing pictures of the space of simple objects of
interpolated fusion categories
constructed as above, we mark the image of the objects of $\cat{D}$ under the free
algebra morphism. These simple objects then carry the same monoidal structure as
the corresponding simples in $\cat{D}$. The remaining category ``interpolates''
between them. This is why we call the resulting categories
interpolated fusion categories.

We now compute the Drinfel'd centres of categorified group rings.
Let $G$ be a compact connected Lie group. In~\cite{weis2022centre},
we computed the centre of its invertible part $G_\omega = {(\Vec^\omega[G])}^\times$
in the setting of smooth 2-groups.
All of the material in this section is discussed in more detail there.
To construct interpolated fusion categories,
only the invertible part of the centre is of relevance to us.
\begin{lem}
    \label{lem:InvertiblesofDZisDZofInvertibles}
    \begin{equation*}
        {(\DZ \Vec^\omega[G])}^\times = \DZ ({\Vec^\omega[G]}^\times)
    \end{equation*}
\end{lem}
\begin{proof}
    All simple objects of $\Vec^\omega[G]$ are invertible and vice versa.
    A half-braiding $\gamma$ of $\IC_g \in \Vec^\omega[G]$
    determines a half-braiding for $\IC_g \in \Vec^\omega[G]^\times$.
    By naturality, this data determines $\gamma$ up to equivalence.
\end{proof}

The invertible part of the Drinfel'd centre is a braided categorical group.
We recall the characterisation of these by quadratic forms.
Any braided categorical group is equivalent to $\Vec^\omega[H]^\times$
for some finite group $H$ and cocycle
$\omega \in \H^3(\B H, \IC^\times) \iso \H^4(\B H, \IZ)$.
If $H$ is abelian, the tensor product is commutative at the level
of isomorphism classes, so one may ask for braidings.
A braiding $\beta$ equips $H$ with a quadratic form
\begin{align*}
    q: H & \to \IC^\times               \\
    h    & \mapsto \beta_{\IC_h,\IC_h},
\end{align*}
sending each group element to the self-braiding of the associated object,
viewed as a morphism in $\End(\IC_{2h}) \iso \IC$.
A map $q:H \to \IC^\times$ is called a \emph{quadratic form} if $q(nx)=n^2q(x)$ and
the associated polarisation form 
\begin{align*}
    \sigma: H \times H &\to \IC^\times \\
                (h,h') &\mapsto q(h+h') - q(h) - q(h')
\end{align*}
is bilinear.
We denote the abelian group of quadratic forms by $\Quad(H,\IC^\times)$.
By a classical Theorem of Eilenberg and MacLane~\cite{eilenberg1954groups}, this
assignment establishes an isomorphism between equivalence classes
of braided categorical groups that categorify $H$ and $\Quad(H,\IC^\times)$.

\begin{example}
    \label{ex:quadformsonZn}
    All quadratic forms on $\IZ/n$ are of the form
    \begin{equation*}
        k \mapsto q^{k^2}
    \end{equation*}
    for some complex number $q$ of unit length.
    In order for this to be a well defined map, $q$ must satisfy
    $q^{n^2} = 1$. In order for it to be a quadratic form,
    it must satisfy $q^{2n}=1$.
    This establishes a bijection between quadratic forms on $\IZ/n$
    and $n$-th/$2n$-th roots of unity for odd/even $n$.

    In the case $n=2$, this means quadratic forms correspond to
    powers of $\exp(\pi/4)=\ii$. These give the four braided
    monoidal structures on $\IZ/2$:
    $\IZ/2$-graded vector spaces $\Vec_{\IZ/2}$, the semion $\Semi$,
    super vector spaces $\sVec$ and the anti-semion $\bSemi$.
\end{example}

The set of braidings on a particular fusion category
$\Vec^\omega[H]$ is freely and transitively acted upon by
the set of bicharacters $H \times H \to \IC^\times$,
denoted $\Bilin(H,\IC^\times)$.
The action is given by pointwise multiplication.
The braidings form a (possibly empty) torsor for
the group of bicharacters.
In other words, there is an exact sequence beginning in
\begin{equation*}
    0 \to \Bilin(H,\IC^\times) \to \Quad(H,\IC^\times) \to \H^3(\B H,\IC^\times).
\end{equation*}

\begin{example}
    \label{ex:braidingZn}
    The bicharacters of $\IZ/n$ are given by
    \begin{equation*}
        (i,j) \mapsto \tilde{q}^{ij}
    \end{equation*}
    for an $n$-th root of unity $\tilde{q}$.
    Combined with Example~\ref{ex:quadformsonZn}, this tells
    us that for odd $n$, $\Vec^\omega[\IZ/n]$ only admits
    a braiding if $\omega$ is trivial, while for even $n$,
    there are $n$ braidings for $\omega \in \{0,n/2\}$,
    where $n/2$ denotes
    the unique element of order 2 in $\H^4(\B \IZ/n,\IZ)$.
\end{example}

We finish this section by quickly recalling the centres computed 
in~\cite{weis2022centre}.
Let $G$ be a compact simply-connected Lie group and $\omega$ a
cocycle in $\mathrm{H}^3(\B G,\IC^\times)=\mathrm{H}^4(\B G,\IZ)$.
Let $T_G \into G$ be a maximal torus of $G$, then $\omega$
defines a real-valued inner product
$I_\omega:\mathfrak{t} \tensor \mathfrak{t} \to \IR$
on the Lie algebra $\mathfrak{t}$ of $T_G$. This
identification is chosen such that the norm
of every lift of the identity of $G$ is an \emph{even} integer.
Pick a lift $\bar{z}$ of each element $z \in Z(G)$ against the map
$\mathfrak{t} \to T_G \to G$.
\begin{thm}[{\cite[Thm 5.3]{weis2022centre}}]
    \label{thm:centreOfCategorifiedGroupRingInManifusionChapter}
    The Drinfel'd centre of $\Vec^\omega[G]^\times$ is
    the braided categorical group with
    underlying group $Z(G)$ and quadratic form given by
    \begin{align*}
        q: Z(G) & \to \IC^\times                                       \\
        z       & \mapsto \exp(\tfrac{1}{2}I_\omega(\bar{z},\bar{z})),
    \end{align*}
    where $\exp:\IR \to \IC^\times$ is shorthand for
    the exponential map $x \mapsto e^{\ii \pi x}$ with
    kernel $\IZ \into \IR$.
\end{thm}

If $G$ is also \emph{simple}, then $\mathrm{H}^4(\B G, \IZ) = \IZ$,
with preferred generator the class inducing the smallest
positive-definite inner product on its Lie algebra
(and that of its maximal torus).
For such groups, we use $k \in \IZ$ to denote $k$
times that preferred generator in its degree 4 cohomology.
\begin{example}
    \label{ex:centreOfSUn}
    As the level $k$ varies, the Drinfel'd centre of
    $\Vec^k[\SU(n)]^\times$
    displays all braided monoidal structures
    on $Z(\SU(n))=\IZ/n$. The braided
    categorical group corresponds to the quadratic form
    \begin{align*}
        q: \IZ/n & \to \IC^\times                  \\
        1        & \mapsto \exp(k\tfrac{n-1}{2n}).
    \end{align*}

    For $\SU(2)$, the centre cycles through the 4 braided monoidal
    structures on $\IZ/2$:
    \begin{equation*}
        \DZ \Vec^k[\SU(2)]^\times =
        \begin{cases}
            {\Vec[\IZ/2]}^\times              & k \equiv 0 \mod 4  \\
            {\Semi}^\times            & k \equiv 1 \mod 4  \\
            {\sVec}^\times                    & k \equiv 2 \mod 4  \\
            {\bSemi}^\times & k \equiv 3 \mod 4.
        \end{cases}
    \end{equation*}
\end{example}
\begin{example}
    \label{ex:centreOfOddSpin}
    For odd Spin groups $\mathrm{Spin}(2n+1)$, the centre $\DZ \Vec^k[\mathrm{Spin}(2n+1)]^\times$
    alternates between $\Vec[\IZ/2]^\times$ ($k$ even) and $\sVec^\times$ ($k$ odd).
\end{example}

Let $T$ be a compact torus, and $\omega$ a cocycle
in $\mathrm{H}^4(\B T,\IZ)$. Denote by $\mathfrak{t}$ its Lie algebra,
and by $\tau=I_\omega(-,\cdot):\mathfrak{t} \to \mathfrak{t}^\ast$ the homomorphism
induced by the inner product corresponding to $\omega$.
Write $\Lambda=\Hom(T,U(1)) \subset \mathfrak{t}^\ast$ for the group of
characters of $T$ and $\Pi=\Hom(U(1),T) \subset \mathfrak{t}$ for its group of
cocharacters.
\begin{thm}[{\cite[Prop 4.1]{weis2022centre}}]
    The Drinfel'd centre of $\Vec^\omega[T]^\times$ is
    the braided categorical group with underlying group the
    quotient
    \begin{equation*}
        \frac{\mathfrak{t} \dirSum \Lambda}{\Pi},
    \end{equation*}
    along the inclusion $\pi \mapsto (\pi, \tau(\pi))$,
    and quadratic form induced by
    \begin{align*}
        q: \mathfrak{t} \dirSum \Lambda & \to \IC^\times                                       \\
        (x, \lambda)                    & \mapsto \lambda(x) \exp(-\tfrac{1}{2}I_\omega(x,x)).
    \end{align*}
\end{thm}
(See also~\cite{freed2010topological} for a sketch proof of this.)

\begin{example}
    The Drinfel'd centre of $\Vec[S^1]^\times$
    (with trivial associator) has underlying group
    \begin{equation*}
        \pi_0 \DZ \Vec[S^1]^\times = \frac{\mathfrak{t}}{\Pi} \dirSum \Lambda = T \dirSum \Lambda = S^1 \dirSum \IZ
    \end{equation*}
    whose quadratic form simply pairs the character with the element of $T$:
    $q(x,\lambda)=\lambda(x)$. 

    The Drinfel'd centre of $\Vec^k[S^1]^\times$ for $k \neq 0$ splits
    as a direct sum $\IR \dirSum \Lambda/\Pi$. The second summand is a
    finite group isomorphic to $\IZ/2k$. The braiding on it is
    encoded by the quadratic form
    \begin{align*}
        q: \IZ/2k & \to \IC^\times               \\
        1         & \mapsto \exp(\tfrac{1}{4k}).
    \end{align*}
\end{example}

Let $G$ be any compact connected Lie group. It admits a cover $\pi:\tilde{G} \to G$,
where $\tilde{G}=T \times G_i$ is a product of a torus $T$ and
simply-connected groups $\{G_i\}$.
The cover can be chosen such that $Z \define \ker(\tilde{G} \to G)$
is a finite subgroup of the centre of $\tilde{G}$.
Let $\omega$ be a cocycle in $\H^3(\B G,\IC^\times)$.
The cocycle pulled back to $\tilde{G}$, $\pi_\ast \omega$, decomposes as a product of
a cocycle $\omega_T \in \H^3(\B T,\IC^\times)$ and cocycles
$\omega_i \in \H^3(\B G_i,\IC^\times)$.
The Drinfel'd centre of $\Vec^{\pi^\ast \omega}[\tilde{G}]^\times$ is the product
\begin{equation*}
    \DZ \Vec^{\pi^\ast \omega}[\tilde{G}]^\times =
    \left(\boxtimes_i \DZ \Vec^{\omega_i}[G_i]^\times\right) \boxtimes
    \DZ \Vec^{\omega_T}[T]^\times.
\end{equation*}
There is a braided inclusion $Z \into \DZ \Vec^{\pi^\ast \omega}[\tilde{G}]^\times$,
where $Z$ carries the trivial braiding.
We denote by $Z^\prime$ the subcategory on objects that braid trivially with
all of $Z$. This category has a braided $Z$-action.
\begin{thm}[{\cite[Thm 5.7]{weis2022centre}}]
    The Drinfel'd centre of $\Vec^\omega[G]^\times$ is
    the quotient $Z^\prime/Z$ equipped with the quadratic form
    induced by that on $Z^\prime$.
\end{thm}

\begin{example}
    For $\SO(4)$, we may pick the cover by $\mathrm{Spin}(4)=\SU(2) \times \SU(2)$
    with kernel $\IZ/2$.
    The degree 4 cohomology of $\SO(4)$ is generated by the Euler class $\chi$
    and the first Pontryagin class $p_1$, while $\mathrm{H}^4(\B \SU(2),\IZ)$ is generated
    by the second Chern class $c$.
    The pullback map sends $p_1$ to the sum $2 c_l + 2 c_r$ of
    twice the generator in each factor, and $\chi$ to the difference
    $c_l-c_r$.
    For the cohomology class $\omega=a p_1 + b \chi$,
    one finds that the quadratic form on $Z(\SU(2) \times \SU(2))$ describing
    the braiding on $\DZ \Vec^\omega[\SU(2) \times \SU(2)]^\times$ is
    \begin{align*}
        q: \IZ/2 \times \IZ/2 & \to \IC^\times       \\
        (1,0)                 & \mapsto \ii^{2a + b} \\
        (0,1)                 & \mapsto \ii^{2a - b} \\
        (1,1)                 & \mapsto 1.
    \end{align*}
    The value on $(1,1)$ is a mere reflection of the fact that this cohomology class
    on $\SU(2) \times \SU(2)$ comes from a cohomology class on its quotient by
    the diagonal copy of $\IZ/2$.

    Whether the object corresponding to $(1,0)$ braides trivially with $(1,1)$ is
    equivalent to whether $q(1,0)=q(0,1)$. This happens if $b$ is even,
    as then $2a+b \equiv 2a-b \mod 4$.
    As a result, the size of the centre varies with the choice of cohomology class:
    \begin{equation*}
        \DZ \Vec^{a \cdot p_1 + b \cdot \chi}[\SO(4)]^\times =
        \begin{cases}
            \Vec[\IZ/2]^\times & 2a+b  \equiv 0 \mod 4 \\
            \sVec^\times       & 2a+b  \equiv 2 \mod 4 \\
            \Vec^\times        & \text{ else.}
        \end{cases}
    \end{equation*}
\end{example}

Given the characterisation of braided 2-groups via quadratic forms, the
construction of interpolated fusion categories reduces to computing
and comparing quadratic forms on finite groups.
If we find a braided 2-group $(H,q) \in \DZ \cat{D}$ in the centre of
a fusion category $\cat{D}$, and the oppositely braided 2-group
$(H,-q) \in \DZ \Vec^\omega[G]$ in the centre of a categorified group ring,
we can form an interpolated fusion category
$\Vec^\omega[G] \boxtimes_{\Vec[H]} \cat{D}$.

\subsection{Interpolated Quantum Group Categories}
\label{sec:interpolatedQGroups}
Let $G$ be a simple, simply-connected Lie group with
Lie algebra $\mathfrak{g}$, denote by $h^\vee$ its dual
Coxeter number, and $m$ the lacity.
Choose a maximal torus $T \subset G$, with corresponding
Lie algebra $\mathfrak{t} \subset \mathfrak{g}$.
Denote by $\langle\cdot,\cdot\rangle$ the
inner product on $\mathfrak{t}^\ast$, normalised such that
short roots have squared length $2$, and let $k \in \IZ_{\geq 1}$.
Set $\ell = m(h^\vee + k)$ and $q = \exp(1 / (2\ell))$.
Denote by $\cat{C}(\mathfrak{g},k)$ the fusion
category obtained as a semisimple quotient of the category of
twisting modules for the small quantum group $u_q(\mathfrak{g})$.
It is in fact a \emph{braided} fusion category, and
thus embeds into its Drinfel'd centre both via
its braiding and its opposite braiding.
Introductions to the construction and structure of these
categories can be found in~\cite{sawin2006quantum,schopieray2018lie}.

The simple objects of $\cat{C}(\mathfrak{g},k)$
correspond to integral weights in the interior of
the Weyl alcove, a certain convex subset of $\mathfrak{t}$
depending on $k$.
The centre $Z(G)$ injects into the group of isometries
of the Weyl alcove.
This map furnishes a bijection between $Z(G)$ and the set of highest weights sitting
at those corners of the Weyl alcove that are isometric to the corner around
the unit object $0$ (with the exception of $\cat{C}(\mathfrak{e}_8,2)$).

\begin{example}
    For $\SU(n)$, the Weyl alcove has precisely $n$ corners isometric to
    the corner around $0$. The corresponding highest weights are of the
    form $k\lambda_i$, where $\lambda_i$ denotes the $i$-th fundamental
    weight. Let $\lambda_1$ denote a fundamental weight corresponding
    to either end of the Dynkin diagram of $\SU(n)$. Then $k\lambda_1$
    $\tensor$-generates a subcategory of $\cat{C}(\mathfrak{g},k)$
    which categorifies $\IZ/n$.
\end{example}

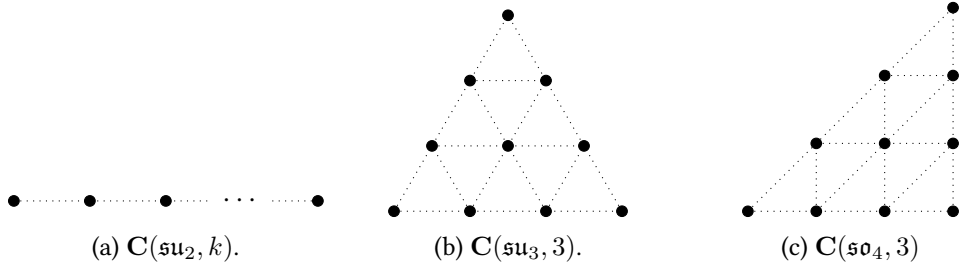
\begin{figure}[ht]
    \centering
    \begin{subfigure}[t]{0.3\textwidth}
        \centering
        \begin{minipage}[t]{0.94\textwidth}
            \centering
            \begin{tikzpicture}
                \filldraw[black] (0,0) circle (2pt); %node[above] {$[0]$};
                \filldraw[black] (1,0) circle (2pt); %node[above] {$[1]$};
                \filldraw[black] (2,0) circle (2pt); %node[above] {$[2]$};
                \filldraw[black] (4,0) circle (2pt); %node[above] {$[k]$};
                \node at (3,0) {$\cdots$};
                \draw[dotted] (0,0) -- (2.6,0);
                \draw[dotted] (3.4,0) -- (4,0);
            \end{tikzpicture}
        \end{minipage}
        \caption{$\cat{C}(\mathfrak{su}_2,k)$.}
        \label{fig:Csu2k}
    \end{subfigure}
    \begin{subfigure}[t]{0.3\textwidth}
        \centering
        \begin{minipage}[t]{0.9\textwidth}
            \centering
            \begin{tikzpicture}
                \tikzmath{\k = 3;}
                \foreach \x in {0,...,\k}{
                        \tikzmath{\ex=(\k+\x)/2; \ey=(\k-\x)*sin(60);
                            \fy=\x*sin(60);}
                        \draw[dotted] (\x,0) -- (\ex,\ey);
                        \draw[dotted] (\k-\x,0) -- (\k-\ex,\ey);
                        \draw[dotted] (\x/2,\fy) -- (\k-\x/2,\fy);
                        \foreach \y in {0,...,\x}{
                                \tikzmath{\px=\x-\y/2;\py=\y*sin(60);}
                                \filldraw (\px,\py) circle (2pt);
                            }}
            \end{tikzpicture}
        \end{minipage}
        \caption{$\cat{C}(\mathfrak{su}_3,3)$.}
        \label{fig:Csu3_3}
    \end{subfigure}
    \begin{subfigure}[t]{0.3\textwidth}
        \centering
        \begin{minipage}[t]{0.9\textwidth}
            \centering
            \begin{tikzpicture}[scale=0.9]
                \tikzmath{\k = 3;}
                \foreach \x in {0,...,\k}{
                        \draw[dotted] (\x,\x) -- (\k,\x);
                        \draw[dotted] (\x,0) -- (\x,\x);
                        \draw[dotted] (\x,0) -- (\k,\k-\x);
                        \foreach \y in {0,...,\x}{
                                \filldraw (\x,\y) circle (2.2pt);}}
            \end{tikzpicture}
        \end{minipage}
        \caption{$\cat{C}(\mathfrak{so}_4,3)$}
    \end{subfigure}
    \caption{Simple objects of Quantum group categories. The dotted lines are visual aids.}
    \label{fig:SimplesOfQGroupCats}
\end{figure}

The ribbon twist on the object $\lambda \in \cat{C}(\mathfrak{g},k)$
is given by
\begin{equation*}
    \theta_\lambda = q^{\langle \lambda, \lambda + 2 \rho \rangle},
\end{equation*}
where $\rho$ denotes the sum of the fundamental weights.
The categories $\cat{C}(\mathfrak{g},k)$ are unitary modular tensor categories.
As a result, the self-braiding of an invertible object is equal to its
ribbon twist.
Denote by $I$ the basic inner product on $\mathfrak{t}$ corresponding to
the generator of $\mathrm{H}^4(\B G, \IZ) \iso \IZ$, with normalisation
such that short coroots have square norm 2.
By~\cite[Prop 10]{henriques2017classification},
the twist on a highest weight $\lambda=k\omega$ can be calculated using
\begin{equation*}
    \langle k\lambda_i,k\lambda_i+2\rho \rangle =
    \tfrac{k}{2} I(\lambda_i^\vee,\lambda_i^\vee).
\end{equation*}

The coweights $\lambda_i^\vee \in \mathfrak{t}$ are
lifts of $z_i \in Z(G)$ against $\mathfrak{t} \onto T \subset G$
(see~\cite{weis2022centre} for more details).
Under the bijection of $Z(G)$ with the highest weights $\lambda_i$, 
this induces the quadratic form
\begin{align*}
    q: Z(G) &\to U(1) \\
    z_i &\mapsto \exp\left(\tfrac{k}{2}I(\bar{z_i},\bar{z_i})\right).
\end{align*}

\begin{example}
    For $\SU(n)$, the self-twists are given by
    \begin{equation*}
        \theta_{k \lambda_j} =
        q^{k^2 \langle \lambda_j, \lambda_j \rangle +
                2k \langle \lambda_j, \rho \rangle}
        =\exp{\tfrac{k \cdot j \cdot (n-j)}{2n}}.
    \end{equation*}
    Evaluating for our chosen generator $\lambda_1$,
    \begin{equation*}
        \theta_{k \lambda_1} =
        \exp{k\tfrac{n-1}{2n}}.
    \end{equation*}
\end{example}

The quadratic form on $Z(G)$ encoding the braiding on
$\cat{C}(\mathfrak{g},k)^\times$
is \emph{equal} to the qua\-dra\-tic form that encodes
$\DZ(\Vec^k[G])^\times$
(see Theorem~\ref{thm:centreOfCategorifiedGroupRingInManifusionChapter}),
and hence the inverse of the quadratic form on $\DZ(\Vec^{-k}[G])^\times$.
There is thus a braided embedding
\begin{equation*}
    \Vec[Z(G)] \into
    \DZ(\Vec^{-k}[G] \boxtimes_\Vect \cat{C}(\mathfrak{g},k)),
\end{equation*}
where $\Vec[Z(G)]$ carries the trivial braiding.
The resulting interpolated fusion category
\begin{equation*}
    \Vec^{-k}[G] \boxtimes_{\Vec[Z(G)]} \cat{C}(\mathfrak{g},k)
\end{equation*}
has set of simples a disjoint union of manifolds of the form
$G/H$, where $H$ is a finite central subgroup of $G$.

\begin{example}
    \label{ex:interpolatedsu2levelk}
    For $\SU(2)$, the quantum group category at level $k$
    has $k+1$ simple objects $[0] \ldots [k]$ which we
    think of as arranged on a line (see Figure~\ref{fig:Csu2k}).
    Tensoring with $Z(\SU(2))=\IZ/2$ induces the flip action on this line.
    The simples of
    $\Vec^{-k}[\SU(2)] \boxtimes \cat{C}(\mathfrak{su}_2,k)$ have underlying
    manifold the disjoint union of $k+1$ copies of $S^3$.
    When $k+1$ is even, the diagonal action of $\IZ/2$ is free on the set
    of connected components. The interpolated fusion category
    $\Vec^{-k}[\SU(2)] \boxtimes_{\Vec[\IZ/2]} \cat{C}(\mathfrak{su}_2,k)$
    then has manifold of simples $\coprod_{k/2} S^3$.
    Antipodal points on each copy of $S^3$ correspond to pairs
    $[i]$-$[k-i]$. If $k+1$ is odd, the object
    $[k/2] \in \cat{C}(\mathfrak{su}_2,k)$ is fixed under the $\IZ/2$-action.
    The set of simples of the interpolated fusion category then has one further
    component whose underlying manifold is $S^3/(\IZ/2)=\RP^3$.
    Any maximal torus $S^1 \into \SU(2)$ contains the centre $\IZ/2$, and
    we can restrict along the embedding to obtain the manifold tensor
    subcategory $\Vec^k[S^1] \boxtimes_{\Vec[\IZ/2]} \cat{C}(\mathfrak{su}_2,k)$.
    We indicate the construction as a quotient by depicting what happens to
    the manifold of simples for the case $k=4$ below.
    At the top, we depict the manifold of simples of
    $\Vec^4[S^1]\boxtimes_\Vect \cat{C}(\mathfrak{su}_2,4)$.
    We label the simples $[i]$ (short for $\ONE \boxtimes [i]$),
    as well as their image $-[i] \define -1 \boxtimes [k-i]$ under the action of
    $-1 \in \IZ/2$.

    \[
        \begin{tikzpicture}
            \tikzmath{\k=4;}
            \foreach \x in {0,...,\k}{
            \draw[ultra thick] (3*\x+1,0) circle (1);% node {$k$};
            \filldraw (3*\x+1,1) circle (2pt) node[above] {$[$\x$]$};
            \tikzmath{\y=int(\k-\x);}
            \filldraw (3*\x+1,-1) circle (2pt) node[below] {$-[$\y$]$};
            }
            \draw[->,thick] (7,-2) -- (7,-3);
            \begin{scope}[shift={(3,-5)}]
                \draw[ultra thick] (1,0) circle (1);% node {$k$};
                \filldraw (1,1) circle (2pt)
                node[above] {$[0]$};
                \filldraw (1,-1) circle (2pt)
                node[below] {$[4]$};
                \draw[ultra thick] (4,0) circle (1);% node {$k$};
                \filldraw (4,1) circle (2pt)
                node[above] {$[1]$};
                \filldraw (4,-1) circle (2pt)
                node[below] {$[3]$};
                \draw[ultra thick] (7,0) circle (0.5);% node {$k/4$};
                \filldraw (7,0.5) circle (2pt)
                node[above] {$[2]$};
            \end{scope}
        \end{tikzpicture}
    \]
\end{example}

The same construction as in Example~\ref{ex:interpolatedsu2levelk}
can be repeated for $\IZ(n) = Z(\SU(n)) \into \SU(n)$ for all $n$.
It can also be restricted to any monoidal subcategory closed under the
action of the centre.

\begin{example}
    The category $\cat{C}(\mathfrak{su}_3,k)$ has simples corresponding to points
    on a triangular lattice in a convex subset of $\IR^2$
    (see Figure~\ref{fig:Csu3_3}).
    We can interpolate using $\SU(3)$ at level $-k$ to obtain the manifusion
    category
    $\Vec^{-k}[\SU(3)] \boxtimes_{\Vec[\IZ/3]} \cat{C}(\mathfrak{su}_3,k)$.
    Below, we draw a subcategory obtained by the restricting along the embedding
    of a circle $S^1 \into \SU(3)$. The simple objects of
    $\cat{C}(\mathfrak{su}_3,k)$ are integral
    points in the weight lattice of $\mathfrak{su}_3$.
    We label them by their coordinates.
    \begin{equation*}
        \begin{tikzpicture}
            \begin{scope}[shift={(-4,0)}]
                \tikzmath{\a0 = 210; \a1=\a0+120; \a2=\a1+120;};
                \draw[ultra thick] (1,0) circle (1);
                \filldraw ({1+cos(\a0)},{sin(\a0)}) circle (2pt)
                %node[left] {$[0]$};
                node[below left] {$[0,0]$};
                \filldraw ({1+cos(\a1)},{sin(\a1)}) circle (2pt)
                %node[below right] {$[3 \omega_1]$};
                node[below right] {$[3,0]$};
                \filldraw ({1+cos(\a2)},{sin(\a2)}) circle (2pt)
                %node[above right] {$[3 \omega_2]$};
                node[above] {$[0,3]$};
            \end{scope}
            \begin{scope}[shift={(-1,0)}]
                \tikzmath{\a0 = 180; \a1=\a0+120; \a2=\a1+120;};
                \draw[ultra thick] (1,0) circle (1);
                \filldraw ({1+cos(\a0)},{sin(\a0)}) circle (2pt)
                %node[left] {$[\omega_2]$};
                node[left] {$[0,1]$};
                \filldraw ({1+cos(\a1)},{sin(\a1)}) circle (2pt)
                %node[below right] {$[2 \omega_1]$};
                node[below right] {$[2,0]$};
                \filldraw ({1+cos(\a2)},{sin(\a2)}) circle (2pt)
                %node[above] {$[\omega_1 + 2 \omega_2]$};
                node[above right] {$[1,2]$};
            \end{scope}
            \begin{scope}[shift={(2,0)}]
                \tikzmath{\a0 = 240; \a1=\a0+120; \a2=\a1+120;};
                \draw[ultra thick] (1,0) circle (1);
                \filldraw ({1+cos(\a0)},{sin(\a0)}) circle (2pt)
                %node[below left] {$[\omega_1]$};
                node[below left] {$[1,0]$};
                \filldraw ({1+cos(\a1)},{sin(\a1)}) circle (2pt)
                %node[right] {$[2 \omega_1 + \omega_2]$};
                node[right] {$[2,1]$};
                \filldraw ({1+cos(\a2)},{sin(\a2)}) circle (2pt)
                %node[above] {$[2 \omega_2]$};
                node[above left] {$[0,2]$};
            \end{scope}
            \begin{scope}[shift={(5,0)}]
                \draw[ultra thick] (1,0) circle (0.5);
                \filldraw (1,0.5) circle (2pt)
                %node[above] {$[\omega_1 + \omega_2]$};
                node[above] {$[1,1]$};
            \end{scope}
        \end{tikzpicture}
    \end{equation*}

\end{example}

\subsection{Interpolated Tambara-Yamagami Categories}
\label{sec:interpolatedTambaraYamagami}
In~\cite{tambara1998tensor}, Tambara and Yamagami obtained a complete
classification of fusion categories which contain only one non-invertible
object.
For convenience, we reproduce here the description of the Tambara-Yamagami
categories (see also Example~\ref{ex:TambaraYamagamiFusionCategory} in 
Appendix~\ref{ch:background}).
The Tambara-Yamagami fusion category $\TY(A, \chi, \tau)$
associated to an abelian group $A$, a symmetric non-degenerate
bicharacter
$\chi: A \tensor A \to \IC^\times$, and a square-root $\tau$ of
$1/|A|$ has set of simple objects $A \coprod \{m\}$.
In what follows, let $a,b,c \in A$.
The fusion rules and non-trivial associators are: \\
\begin{minipage}{.4\textwidth}
    \vspace{-1.3em}
    \begin{align*}
        a \tensor b & = a\!+\!b,            \\
        a \tensor m & = m \tensor a = m     \\
        m \tensor m & = \DirSum_{a \in A} a
    \end{align*}
\end{minipage}%
\begin{minipage}{.6\textwidth}
    \vspace{-1.3em}
    \begin{align*}
        \alpha_{a,m,b} & = \chi(a,b) \id_m: m \to m,     \\
        \alpha_{m,a,m} & = \DirSum_{b,c \in A} \chi(a,b)
        \delta_{b,c} \id_b: \DirSum_{b \in A} b \to
        \DirSum_{c \in A} c                              \\
        \alpha_{m,m,m} & = \DirSum_{b,c \in A}
        \tau \chi(b,c)^{-1} \id_m:
        \DirSum_{b \in A} m \to \DirSum_{c \in A} m.
    \end{align*}
\end{minipage}%

The Drinfel'd centre of a Tambara-Yamagami category
is computed in~\cite{izumi2001structure,gelaki2009centers}
--- see also~\cite{shimizu2011frobenius} for a review.
The centre contains a family of invertible objects $X_{a,\epsilon} = (a, \beta_{a,-})$
parameterised by an underlying object $a \in H$, and a square root $\epsilon$
of $\chi(a,a)$. The half-braiding has components
\begin{align*}
    \beta_{a,b} & = \chi(a,b) \id_{a+b} \\
    \beta_{a,m} & = \epsilon \id_{m}.
\end{align*}
The tensor product among them is given by
$X_{a,\epsilon} \tensor X_{b,\eta} = X_{a+b,\epsilon \eta \chi(a,b)}$.
The associator is inherited from $\TY$, and thus trivial.
For a fixed $a \in H$, there are precisely two possible values of $\epsilon$.
Pick a set of numbers $\{\epsilon_a\}_{a \in H}$ squaring to
$\epsilon_a^2=\chi(a,a)$ such that
$\epsilon_a \epsilon_b \chi(a,b) = \epsilon_{a+b}$.
Then $\epsilon:H \to \IC^\times$ is a \emph{quadratic refinement}
of $\chi$: a quadratic form whose associated bilinear form
is $\chi$.
The choice of such a quadratic refinement is unique up to multiplication by a
homomorphism $s:H \to \IZ/2$.
A quadratic refinement of $\chi$ induces an inclusion of fusion categories
\begin{align*}
    \Vec[H] & \into \DZ(\TY(H,\chi,\tau)) \\
    a       & \mapsto X_{a,\epsilon_a}.
\end{align*}
The induced braiding on $\Vec[H]$ is given by
$\beta_{a,b}=\chi(a,b) \id_{a+b}$, and
so the quadratic form is the square of
$\epsilon$, $q(a)=\chi(a,a)=\epsilon_a^2$.

Now we mix and match quadratic forms to construct
interpolated Tambara-Yamagami categories.
Let $\Vec^{\omega^\prime,\beta}[H] \into \DZ(\Vec^\omega[G])$ be a braided inclusion.
It can be used to interpolate a Tambara-Yamagami category with $\TY(H,\chi,\tau)$
if the quadratic form $q$ induced by $\beta$ on $H$ is the square
of a quadratic form $\epsilon$ whose associated
bilinear form $\chi$ is non-degenerate.

\begin{example}
    Let $H = \IZ/p$ ($p$ an odd prime) be a finite subgroup of $S^1$. It admits
    a monoidal map to the Drinfel'd centre of $\Vec^k[S^1]$ when $k$
    is a multiple of $p$. The induced quadratic form on $\IZ/p$ is given by
    $q:1 \mapsto \exp(\tfrac{k/p}{p})$.
    This form is a square of another quadratic form iff $k$ admits a
    square-root mod $p$. As $p$ is odd, this is always the case.
    The bilinear form associated to the square-root is non-degenerate iff
    $k$ is \emph{not} a multiple of $p^2$.
    Hence, we get that the allowed levels $k \in \H^4(\B S^1,\IZ)=\IZ$ for building the
    interpolated Tambara-Yamagami categories are
    \begin{equation*}
        k \in p\IZ \backslash p^2\IZ.
    \end{equation*}

    Pick such a $k$, and pick a square root $\epsilon:\IZ/p \to \IC^\times$
    of the induced quadratic form $q:n \mapsto \exp(n^2\cdot k/p^2)$ on $\IZ/p$.
    Denote by $\chi$ the bilinear form associated to $\epsilon$.
    Lastly, pick $\tau \in \left\{\tfrac{1}{\sqrt{p}},-\tfrac{1}{\sqrt{p}}\right\}$.
    Then we may form an interpolated Tambara-Yamagami category
    \[
        \Vec^{-k}[S^1] \boxtimes_{\Vec[\IZ/p]} \TY(\IZ/p,\chi,\tau).
    \]
    Its manifold of simples is comprised of two circles, one interpolating the objects
    corresponding to
    $\IZ/p$ in $\TY(\IZ/p,\chi,\tau)$, the other containing the image of the object
    $m \in \TY(\IZ/p,\chi,\tau)$.
    The circles are to be thought of as having different sizes: tensoring with
    $m$ induces the $p$-fold covering map $S^1 \to S^1$ from the unit circle
    to the circle containing $m$.
    An example for $p=5$ is depicted below.
    \begin{equation*}
        \begin{tikzpicture}
            \draw[ultra thick] (0,0) circle (1);% node {$k$};
            \filldraw (0,1) circle (2pt)
            node[above] {$1$};
            \filldraw ({sin(72)},{cos(72)}) circle (2pt)
            node[above right] {$\omega$};
            \filldraw ({sin(2*72)},{cos(2*72)}) circle (2pt)
            node[below right] {$\omega^2$};
            \filldraw ({sin(3*72)},{cos(3*72)}) circle (2pt)
            node[below left] {$\omega^3$};
            \filldraw ({sin(4*72)},{cos(4*72)}) circle (2pt)
            node[above left] {$\omega^4$};
            \draw[ultra thick] (3,0) circle (0.4);% node {$k/5$};
            \filldraw (3,0.4) circle (2pt)
            node[above] {$m$};
        \end{tikzpicture}
    \end{equation*}
\end{example}

\section{Structure of Orbifusion Categories}
\label{sec:orbifusion}
In this section we study the interplay of the smooth structure with
the monoidal structure on an orbifusion category.
We show that tensoring with objects induces a smooth action on the support
of families of skyscraper sheaves.
Further, we show in Corollary~\ref{cor:manifoldONEmanifoldALL} that
if the monoidal unit sits at a point with trivial stabiliser, every 
point on the underlying orbifold has trivial stabiliser.
We also heavily constrain the structure
in a neighbourhood of the monoidal unit:
In Corollary~\ref{cor:gerbeTrivialAtONE}, we show that the gerbe must
admit a trivialisation over a neighbourhood of the support of the monoidal unit.
Lemma~\ref{lem:A_tensor_A}
shows that the tensor product near the unit is encoded by a single smooth
multiplication (up to isomorphism).

\subsection{Tensor Products and Support Functions}
Let $\stC \simeq \Sky_{\orb{M}}^{\ger{G}}$ be an orbifold tensor category, and
$\f{S}:U \to \cat{C}$ a $U$-family.
Recall Definition~\ref{def:supportLift}: a set of support lifts
for $\f{S}$ with respect to an orbifold patch $q:M \to \orb{M}$ with
a trivialisation $\tau$ of $q^\ast\ger{G}$
is a set of smooth maps $\{\tilde{s_i}:U \to M\}$, such that
the induced mobile family
\[
    \bar{\f{S}} \define \Ind^{\tilde{s}_i^\ast \tau}(q \comp \tilde{s}_i)
    = \DirSum_i \Gr(q \comp \tilde{s}_i)_\ast (\tilde{s}_i^\ast\tau \cdot \Cinf)
\]
is a mobile cover for $\f{S}$. This means it contains the family as a direct summand
$\f{S} \subset \bar{\f{S}}$ and induces the same support function
$\supp \f{S} = \supp \bar{\f{S}}:U \to |\orb{M}|$.

The Lemma below is a technical rephrasing of the following claim:
The tensor product plays well with local support lifts, in that
any support lifts of a tensor product $\f{S} \tensor \f{T}$
depend smoothly on any pair of support lifts for $\f{S}$ and $\f{T}$.
\begin{lem}
    \label{lem:tensoring_induces_smooth_function_on_local_support_lifts}
    Let $q:M \to \orb{M}$, $q':M' \to \orb{M}$,
    $q'':M'' \to \orb{M}$ be charts of $\orb{M}$ that
    are disjoint unions of contractible spaces.
    Let $\f{S}:U \to \cat{C}$, $\f{T}:V \to \cat{C}$ be families
    of skyscraper sheaves,
    and fix a pair of points $x \in U$, $y \in V$.
    Assume $\f{S}$ and $\f{T}$ have single-valued local support lifts
    $\tilde{s}:U \to M$ and $\tilde{t}:U \to M'$ around $x$ and $y$,
    respectively.
    Then any set of local support lifts $\{{\tilde{r}}_i:U \to M''\}$ of
    $\f{S} \tensor \f{T}$ around $(x,y)$, depends smoothly on the
    support lifts $\tilde{s}$ and $\tilde{t}$:
    There are smooth functions $f_i:U \times V \to M''$ such that
    ${\tilde{r}}_i = f_i \comp (\tilde{s} \times \tilde{t})$.
\end{lem}
\begin{proof}
    We can use the chart maps $q,q'$ to induce mobile families
    \begin{align*}
        \f{I}  \define {(\Ind q)}^{\dirSum n}:M \to \cat{C} &  &
        \f{I}' \define {(\Ind q')}^{\dirSum n'}:M' \to \cat{C},
    \end{align*}
    where the integers $n, n'$ are chosen large enough such that
    $\f{S}$ and $\f{T}$ have mobile covers given by
    \begin{align*}
        \f{S} \subset \overline{\f{S}} \define
        \tilde{s}^\ast\f{I} &  &
        \f{T} \subset \overline{\f{T}} \define
        \tilde{t}^\ast\f{I}',
    \end{align*}
    and
    \begin{equation*}
        \f{S} \tensor \f{T} \subset
        (\tilde{s} \times \tilde{t})^\ast(\f{I} \tensor \f{I}').
    \end{equation*}
    As we are working up to isomorphism, we need not pick explicit trivialisations
    of the pullback gerbes over $M,M'$ and $M''$
    (see Lemma~\ref{lem:trivialisationsInIndDontMatterLocally} and
    Notation~\ref{notation:IndWithoutTrivialisation}).

    Given local support lifts $\{\tilde{f_i}:M \times M' \to M'''\}$ of
    $\f{I} \tensor \f{I}'$ to some chart $q''': M''' \to \orb{M}$,
    $(\tilde{s} \times \tilde{t})^\ast(\f{I} \tensor \f{I}')$
    has local support lifts
    $\{\tilde{f_i}\comp(\tilde{s} \times \tilde{t}):U \times V \to M'''\}$.
    By Corollary~\ref{cor:summands_yield_subset_of_support_lifts},
    any local support lift of $\f{S} \tensor \f{T}$ to the chart $M'''$
    is a $\Gamma$-translate of one of the above. As their value agrees at $x$,
    they can also be lifted to the chart $M''$. Both $\Gamma$-translation
    and the transition maps between charts are smooth, hence the composition
    of these maps with the maps $\tilde{f_i}$ yield the desired smooth map
    $f_i:M \times M' \to M''$.
\end{proof}

\begin{rmk}
    Corollary~\ref{cor:locally_karoubi_completion_distributes} tells
    us that any family $\f{S}$ locally decomposes into summands with
    single-valued support, so the assumption in the Lemma above
    is not restrictive at all.
\end{rmk}

In particular,
Lemma~\ref{lem:tensoring_induces_smooth_function_on_local_support_lifts}
says that the operations of left- or right-tensoring a family
$\f{S}:U \to \cat{C}$ with an object $X \in \cat{C}$
induce postcomposition with a smooth multivalued function (whose domain is $U$)
on support lifts.
Recall a family $\f{S}:U \to \cat{C}$ was called \'etale\-/submersive\-/immersive
if it locally admits a set of support lifts which have the respective property
(Definition~\ref{def:etaleimmersivesubmersive}).
\begin{lem}
    \label{lem:maximal_rank_preserved_by_tensoring_with_object}
    Let $\f{S}:U \to \cat{C}$ be an \'etale/submersive/immersive
    $U$-family, and $X \in \cat{C}$.
    Then $X \tensor \f{S}$ is \'etale/submersive/immersive.
\end{lem}
\begin{proof}
    It is enough to check this near a point $x \in U$.
    Corollary~\ref{cor:locally_karoubi_completion_distributes} tells us
    that $X \tensor \f{S}$ locally decomposes into summands which admit
    a single local support lift. We will show that each of these
    is \'etale/submersive/immersive.
    Let $\f{T} \subset X \tensor \f{S}$ be a subfamily,
    with a single local support lift $s:U \to M$ to a chart $M \to \orb{M}$.
    Then the inclusion $\f{T} \into X \tensor \f{S}$ is a nowhere-zero morphism in
    \begin{equation*}
        \Hom(X^\vee \tensor \f{T}, \f{S}) \simeq
        \Hom(\f{T}, X \tensor \f{S}),
    \end{equation*}
    so by Corollary~\ref{cor:locally_share_summand}, $X^\vee \tensor \f{T}$ and
    $\f{S}$ share a summand on an open neighbourhood of $x$, which is
    \'etale/submersive/immersive by assumption.
    There must then exist a set of local support lifts $\{s_j':U \to M'\}$ of
    $X^\vee \tensor \f{T}$, which contains an
    \'etale/submersive/immersive map.
    By Lemma~\ref{lem:tensoring_induces_smooth_function_on_local_support_lifts},
    this map is obtained by postcomposing the local support lift $s$
    with a smooth function $M \to M'$.
    As postcomposition by smooth maps cannot increase the rank,
    the map $s$ must already be \'etale/submersive/immersive.
    This proves the assertion.
\end{proof}

\begin{lem}
    \label{lem:mobile_preserved_by_tensoring_with_object}
    Let $\f{S}:U \to \cat{C}$ be a \emph{mobile} family
    in an \emph{effective} orbifold tensor category
    and $X \in \cat{C}$,
    then both $X \tensor \f{S}$ and $\f{S} \tensor X$ are mobile.
\end{lem}
\begin{proof}
    We check this is true at $x \in U$.
    Around $x$, $\f{S}$ decomposes as a direct sum of basic mobiles,
    so it suffices to show the statement if $\f{S}$ is a basic mobile.
    Without loss of generality, $\f{S} \iso \Ind f$
    for a smooth map $f:U \to \orb{M}$ with a local lift $\tilde{f}:U \to M$,
    to a chart $q:M \to \orb{M}$.
    The chart map $q$ is \'etale, and thus the associated basic mobile
    $\Ind q$ is an \'etale family.
    We can compute the tensor product via a pullback:
    \begin{equation*}
        X \tensor \f{S} \iso X \tensor f^\ast\Ind q \iso f^\ast {(X \tensor \Ind q)}.
    \end{equation*}
    Now we combine
    Lemma~\ref{lem:maximal_rank_preserved_by_tensoring_with_object}
    with Lemma~\ref{lem:submersive_implies_mobile}
    (which is only valid in effective orbifold tensor categories) to conclude that
    $X \tensor \Ind q$ is \'etale, thus submersive, thus mobile.
    Being mobile is preserved under pullback, which proves $X \tensor \f{S}$ is mobile.
    The proof for $\f{S} \tensor X$ is identical. 
\end{proof}

We can now heavily constrain the structure of effective 
orbifold tensor categories.
\begin{cor}
\label{cor:manifoldONEmanifoldALL}
    If the unit $\ONE$ in an effective orbifold tensor category $\stC$ 
    is supported at a point with trivial stabiliser $\Gamma=1$,
    then $\stC$ is a manifusion category.
\end{cor}
\begin{proof}
    We prove this by contradiction. Assume $\ONE$ is supported at 
    a point with trivial stabiliser, so $\ONE$ is mobile.
    Then Lemma~\ref{lem:mobile_preserved_by_tensoring_with_object}
    implies any object $X \in \cat{C}$ is mobile:
    A tensor product is mobile if either of the factors is, and 
    $X \simeq \ONE \tensor X$.

    Now assume there exists a point
    $x$ of the underlying orbifold such that $\Gamma_x \neq 1$.
    Then there exists an object $X$ with support $x$ which is not mobile: 
    The category $\cat{C}^{(x)}$ is equivalent to 
    $\Rep^{\theta_x}(\Gamma_x)$ 
    and the basic mobile at $x$ corresponds (up to isomorphism) to the 
    $\theta_x$-twisted group ring (Lemma~\ref{lem:cat_at_point_props}).
    As $\Gamma_x$ is non-trivial, it always has non-trivial summands: 
    This is a direct consequence of the decomposition given in
    Proposition~\ref{prop:decompositionTwistedGroupRing}.
\end{proof}

\subsection{Grothendieck Ring}
\label{sec:grothendieckRing}
The monoidal structure upgrades the Grothendieck group
$\K(\cat{C})$ to a ring.
Its product takes the form
\begin{equation*}
    X_i \cdot X_j = \sum_{k\in I} n_{i,j}^k X_k,
\end{equation*}
where $n_{i,j}^k \in \IZ_+$. The sum is evaluable because,
for fixed $(i,j)$, only finitely
many $n_{i,j}^k$ are nonzero.
In the literature, this structure has been called a $\IZ_+$-based ring
(see~\cite{etingof1995representations,ostrik2003module}).
In our case, the so-called basis $\{X_i\}_{i\in I}$ is usually infinite.
Duality data further equips the ring with an involution
${}^\vee:\K(\cat{C}) \to \K(\cat{C})$.

In an orbifusion category, the unit object $\ONE$ is required to be simple.
We denote its support by $[0] \define \supp(\ONE)$,
and fix a representative $0 \in \orb{M}$
with stabiliser group $\Gamma \define \Gamma_0$,
and a 2-cocycle $\theta_0$ representing the gerbe at the point 0.
\begin{definition}
    The \emph{mobile unit} $\IO$ is the basic mobile at $[0]$,
    (see Definition~\ref{def:basic_mobile_at_x})
    \begin{equation*}
        \IO \define {\IO}_{0}.
    \end{equation*}
\end{definition}

In the following proofs, we perform calculations in the Grothendieck ring
by picking lifts of its elements to the category.
By abuse of notation, we will denote the lifts by the same symbols.
\begin{lem}
    Let $X$ be a simple mobile, and $X^\vee$ its dual.
    $X^\vee \cdot X$ contains precisely
    one copy of $\IO$ and no other objects with support $[0]$:
    \begin{equation*}
        X^\vee \cdot X = \IO + R,
    \end{equation*}
    for some $R \in \K(\cat{C})$ and $[0] \not\in \supp{R}$.
\end{lem}
\begin{proof}
    By Lemma~\ref{lem:mobile_preserved_by_tensoring_with_object},
    $X^\vee \tensor X$ is mobile.
    Hence it splits as a direct sum
    $\IO^{\dirSum n} \dirSum R$, where $n \geq 1$ and
    $R$ has support disjoint from $[0]$.
    However, $\Hom(\IO,\ONE) \neq 0$, so $\IO$ must contain a summand isomorphic to $\ONE$.
    As $X$ is simple, $\Hom(X^\vee \tensor X, \ONE) \iso \Hom(X, X)$
    is one-dimensional.
    It follows that $n=1$ and $\Hom(\IO,\ONE)$ is one-dimensional.
\end{proof}
By Lemma~\ref{lem:cat_at_point_props},
$\IO \iso \Rep^\theta(\Gamma) \iso \DirSum_\chi \chi^{\dirSum \vdim \chi}$,
so
\begin{equation*}
    \dim \Hom(\IO, \ONE) = \vdim \ONE = |\Gamma| \dot \dim \ONE.
\end{equation*}
Thus the above also proves the following:
\begin{equation*}
    \dim{\ONE} = \frac{1}{|\Gamma|}.
\end{equation*}
This means under an equivalence
$\cat{C}^{(0)} \iso \Rep^\theta(\Gamma)$, $\ONE$ is always identified with
a one-dimensional irreducible representation of $\Gamma$.

\begin{cor}
    \label{cor:gerbeTrivialAtONE}
    The gerbe $\ger{G}$ is trivial at the point $0$.
\end{cor}
\begin{proof}
    Corollary~\ref{cor:no1DrepsforNontrivialTheta} says that a 2-cocycle $\theta$
    is trivial if and only if $\Rep^\theta(\Gamma)$ contains a
    1-dimensional representation.
\end{proof}
In other words, we may always take the representing cocycle $\theta_0$
of $\ger{G}_0$ to be trivial
(see the local classification of gerbes over orbifolds in
Corollary~\ref{cor:classificationOfGerbesOnQuotientOrbifolds}).
We will use this in Section~\ref{sec:local_structure} to treat
the neighbourhood of $\ONE$ as being equipped with the trivial gerbe.

\begin{lem}
    \label{lem:calculate_O_times_simple_mobile}
    Let $X$ be a simple object and $p \in \supp{X}$, then
    \begin{equation*}
        \IO \cdot X = |\Gamma| \cdot \dim{X} \cdot {\IO}_p.
    \end{equation*}
\end{lem}
\begin{proof}
    By Lemma~\ref{lem:mobile_preserved_by_tensoring_with_object},
    $\IO \tensor X$ is mobile, so it can be expressed as a sum
    \begin{equation*}
        \IO \tensor X \iso \DirSum_{[q] \in |\orb{M}|} {\IO}_{q}^{\dirSum n_q},
    \end{equation*}
    where we sum over a chosen set of pairwise distinct representatives of the
    equivalence classes of points in $\orb{M}$ (such that $p$ is one of them).
    We use dualisability to write
    \begin{equation*}
        \Hom(\IO \tensor X, \IO_{q}) \iso \Hom(\IO, \IO_{q} \tensor X^\vee).
    \end{equation*}
    The object $\IO_{q} \tensor X^\vee$ is again mobile,
    and thus of the form $\DirSum_{[x] \in |\orb{M}|} \IO_x^{\dirSum m_x}$ .
    The equality 
    \[
        \dim \Hom(\IO, \IO_x) = |\Gamma| \cdot \dim \Hom(\ONE,\IO_x)
    \]
    allows us to compute
    \begin{equation*}
        \dim \Hom(\IO, \IO_{q} \tensor X^\vee) =
        |\Gamma| \cdot \dim \Hom(\ONE, \IO_{q} \tensor X^\vee)
        = |\Gamma| \cdot \dim \Hom(X, \IO_{q}).
    \end{equation*}
    The latter is only non-zero when $q = p$, as otherwise $\IO_{q}$ cannot
    contain $X$. So $n_q=0$ for $q \neq p$.
    The basic mobile ${\IO}_p$ contains $X$ with multiplicity its vector space dimension,
    which is $\dim X \cdot |\Gamma_p|$, and hence
    \begin{equation*}
        \dim \Hom(\IO, \IO_{p} \tensor X^\vee) =
        |\Gamma| \cdot \dim X \cdot |\Gamma_p|.
    \end{equation*}
    Now $n_p$ is given by
    \begin{equation*}
        n_p = \dim \Hom(\IO \tensor X, \IO_p) / \dim \Hom(\IO_p, \IO_p)
        %= |\Gamma| \cdot \dim X \cdot |\Gamma_p| / |\Gamma_p| 
        = |\Gamma| \cdot \dim X. \qedhere
    \end{equation*}
\end{proof}

\begin{example}
\label{ex:O_times_O}
    For a basic mobile $X \in \cat{C}$,
    $X = {\IO}_{\supp X}$ in $\K(\cat{C})$, and
    $\dim{X} = 1$.
    We apply Lemma~\ref{lem:calculate_O_times_simple_mobile} to its simple
    summands to find
    \begin{equation*}
        \IO \cdot X = |\Gamma| \cdot X.
    \end{equation*}
    A particularly useful special case is
    \begin{equation*}
        \IO \cdot \IO = |\Gamma| \cdot \IO.
    \end{equation*}
    As the above plays well with sums, this is true for
    \emph{all} mobile objects.
    In particular, the above is true for all objects $X \in \cat{C}$
    with support in the generic region.
\end{example}

\subsection{Local Model}
In this section, we study properties of the monoidal structure around the 
unit of an effective orbifusion category,
using a (non-canonical) lift of the object $\IO \in \K(\cat{C})$
to the category $\cat{C}$.
We fix an equivalence of linear stacks
$\cat{C} \iso \Sky_\orb{M}^\ger{G}$ (where $\orb{M}$ is effective),
and pick a point $0 \in \orb{M}$ representing
the support of $\ONE \in \cat{C}$.
We call the stabiliser group of this point $\Gamma$.

We pick a chart $q:M \onto [M/\Gamma] \into \orb{M}$ covering
a quotient suborbifold around $0$, with a single preimage
$x_0 \define q^{-1}(0) \in M$.
By Corollary~\ref{cor:gerbeTrivialAtONE}, we may assume the restriction of the
gerbe $\ger{G}$ to $[M/\Gamma]$ is trivial.
We pick once and for all a trivialisation $\tau$ of the pullback gerbe
and use it induce an \'etale family
\begin{equation*}
    \IA \define \Ind^\tau q.
\end{equation*}
By abuse of notation, we denote its pullback over $x_0$ by
\begin{equation*}
    \IO \define x_0^\ast \IA \in \cat{C}.
\end{equation*}
The object $\IO$ as defined above is a lift of the object we called
$\IO \in \K(\cat{C})$ in Section~\ref{sec:grothendieckRing} to $\cat{C}$.
We pick an explicit inclusion
\begin{equation*}
    \ONE \into \IO.
\end{equation*}
If $\Gamma$ is trivial, this map is an isomorphism.

\begin{example}
    When $\orb{M}$ is a manifold, $\IA$ is
    a $U$-family of dimension 1 whose support function $U \to \orb{M}$
    is the embedding of a neighbourhood of $0 \in \orb{M}$.
\end{example}

\begin{example}
    Consider the orbisimple category
    $\Sky_{[\IR/(\IZ/2)]} = {(\Sky_{\IR})}^{\IZ/2}$, equipped
    with the monoidal structure coming from the Lie group structure on $\IR$.
    The family $\IA$ may be represented by the $\IR$-family of skyscraper sheaves
    whose restriction to $x \in \IR$ is the skyscraper sheaf
    $\IC_x \dirSum \IC_{-x}$, with $\IZ/2$-action given by swapping the two
    terms.
    The restriction of this representative along $0: \point \to [\IR/(\IZ/2)]$
    yields the regular representation of $\IZ/2$, a representative for $\IO$.
\end{example}

The family $\IA$ and its tensor product with itself turn out to
carry a lot of information. We briefly introduce some notation.
\begin{defn}
    A family $\f{S}:U \to \cat{C}$ is \emph{simple at $x \in U$} if
    there exists no open neighbourhood $V$ of $x$, such
    that $\f{S}\restr_V$ splits as a direct sum of two families.
    $\f{S}$ is simple if it is simple at all points $x \in U$.
\end{defn}

Given Corollary~\ref{cor:locally_karoubi_completion_distributes},
a family simple at $x$ always admits a local support lift given by
a single smooth function. In particular, this means the support
at $x$ is a single point.
\begin{lem}
    \label{lem:locally_summand_of_pullback_from_A_tensor_X}
    Near $x \in U$, every family $\f{S}:U \to \cat{C}$ simple at $x$ is
    a direct summand of a pullback of $\IA \tensor \f{S}(x)$.
    Ie.\ there exists an open neighbourhood $V$ of $x$ and
    a smooth map $f:V \to M$,
    such that $\f{S}\restriction_V$ admits the structure of a summand of
    $f^\ast\IA \tensor \f{S}(x)$.
\end{lem}
\begin{proof}
    $\IA \tensor \f{S}(x)$ is an \'etale $M$-family whose
    value at $x$ is
    \begin{equation*}
        \IA(x) \tensor \f{S}(x) \iso \IO \tensor \f{S}(x)
        \iso {\left(\f{S}(x)\right)}^{\dirSum |\Gamma|}.
    \end{equation*}
    Thus local support lifts for $\IA \tensor \f{S}(x)$ and
    for $\f{S}$ can be chosen to land in the same chart $q_N:N \to \orb{M}$
    around $\supp \f{S}(x)$.
    Let $\tilde{s}: U \to N$ be such a local support lift for $\f{S}$,
    and $\tilde{t}_i: M \to N$ a local support lift for $\IA \tensor \f{S}(x)$.
    As $\tilde{t}_i$ is \'etale, there exists a map $f:V \to M$ for some open
    neighbourhood $V$ of $x \in U$, that factors $\tilde{s}$
    through $\tilde{t}_i$.
    Hence $f^\ast\IA \tensor \f{S}(x)$ locally contains $\Ind q_N \comp \tilde{s}$
    as a summand. By definition of a local support lift,
    $\Ind q_N \comp \tilde{s}$ contains $\f{S}$ as a summand, which
    proves the assertion.
\end{proof}

We will now study the tensor product of $\IA$ with itself.
%By combining it with the pointwise tensor product of $\cat{C}$,
%this allows us to study the tensor product across all of $\cat{C}$.
We do this by decomposing $\IA \tensor \IA$ into a direct sum of
\'etale families indexed by $\Gamma$.
Their local support lifts may be chosen to be translates of
a multiplication on $M$.
It will be convenient to pretend these lifts exist as maps
$M \times M \to M$, when really they might only be defined on a
smaller domain.
We treat them as partial functions:
their domain is denoted by $M \times M$, despite the functions only
being defined on a subset $D \subset M \times M$. This domain
will always be a neighbourhood of $(x_0,x_0)$.
\begin{lemma}
    \label{lem:A_tensor_A}
    There is a set of local support lifts of $\IA \tensor \IA$ given by
    partial maps 
    \[
        \{\tilde{m}_g:M \times M \to M\}_{g \in \Gamma}
    \] 
    satisfying
    \begin{itemize}
        \item $\tilde{m}_g(x,y) = \tilde{m}_e(x,g\cdot y)$
        \item $\tilde{m}_e(x_0,x) = \tilde{m}_e(x,x_0) = x$.
    \end{itemize}
\end{lemma}
\begin{proof}
    First note that $\IA \tensor \IA$ must admit local support lifts
    with respect to the chart $M \to \orb{M}$. This is because
    $\IA(x_0) \tensor \IA(x_0) \iso \IO \tensor \IO \iso \IO^{\dirSum \Gamma}$ 
    (Example~\ref{ex:O_times_O}),
    whose support is $[0]$, and thus covered by $M$.
    Pick such a set of local support lifts
    $\{\tilde{m}_i: M \times M \to M\}_{i \in I}$.
    By the equation above, these lifts satisfy
    $\tilde{m}_i(x_0,x_0) = x_0$.
    The lifts are further restricted by the condition
    \begin{equation*}
        \IA(x_0) \tensor \IA(x) \iso \IA(x) \tensor \IA(x_0)
        \iso {\IA(x)}^{\dirSum \Gamma},
    \end{equation*}
    which yields pointwise equalities
    \[
        q \comp \tilde{m}_i(x_0,x) = q \comp \tilde{m}_i(x,x_0) = q(x).
    \]
    This forces the maps $\tilde{m}_i(x_0,-):M \to M$
    and $\tilde{m}_i(-,x_0):M \to M$ to be given by translations
    \[
        x \mapsto g\cdot x
    \]
    for some $g \in \Gamma$. (This is true pointwise by the equalities above.
    The element $g$ must be constant because $\Gamma$ is finite.)

    Lifts to $M$ are only defined up to $\Gamma$-action, so we can modify
    all $\tilde{m}_i$ such that
    \begin{equation*}
        \tilde{m}_i(-,x_0) = \id_M.
    \end{equation*}
    We denote by $g_i \in \Gamma$ the element such that
    $\tilde{m}_i(x_0, -) = g_i \cdot -$.
    The derivative of an $\tilde{m}_i$ at $(x_0,x_0)$ is the direct
    sum of the derivatives of $\id_M$ and $g_i\cdot-$.
    These are diffeomorphisms, hence the derivative is surjective,
    and all $\tilde{m}_i$ are submersions in a neighbourhood of
    $(x_0,x_0)$.
    Lemma~\ref{lem:submersive_implies_mobile} implies
    $\IA \tensor \IA$ is mobile in a
    neighbourhood of $(x_0,x_0)$, with each
    $\tilde{m}_i$ contributing a mobile summand .

    Now pick any single support lift $\tilde{m}_i$. We will use invariance
    of $\IA$ under the $\Gamma$-action to show that there is a set of support lifts
    $\{\tilde{m}_i(-, g\cdot-)\}_{g \in \Gamma}$.
    As $\IA \tensor \IA$ is mobile,
    \begin{equation*}
        \IA \tensor \IA \iso
        \DirSum_{i \in I} {\left(\Ind(q \comp \tilde{m}_i)\right)}^{\dirSum n_i}
    \end{equation*}
    for some set of $n_i \in \IN$. By counting dimensions at $(x_0,x_0)$,
    we know $\Sigma_i n_i = |\Gamma|$.
    The $\Gamma$-invariance ${(g\cdot-)}^\ast\IA \iso \IA$ implies
    \begin{equation*}
        {(\id_M \tensor g\cdot-)}^\ast \IA \tensor \IA \iso \IA \tensor \IA,
    \end{equation*}
    and so
    \begin{equation*}
        \DirSum_i {\left(\Ind(q \comp \tilde{m}_i)\right)}^{\dirSum n_i} \iso
        \DirSum_i {\left(\Ind(q \comp \tilde{m}_i(-, g\cdot-))\right)}^{\dirSum n_i}
    \end{equation*}
    for all $g \in \Gamma$.
    In order for there to exist an isomorphism, the maps on the left-hand-side
    and on the right-hand-side must match up in an open neighbourhood of
    $(x_0,x_0)$ by Lemma~\ref{prop:local_factorisation}.
    This proves that all the $n_i$ in the above expression are 1,
    the direct sum is indexed by $\Gamma$, and there is indeed a set of local
    support lifts $\{\tilde{m}_i(-, g\cdot-)\}_{g \in \Gamma}$.
    In particular, this includes a support lift $\tilde{m}_e$ which satisfies
    \[
        \tilde{m}_e(x,x_0)=\tilde{m}_e(x_0,x)=x. \qedhere
    \]
\end{proof}

Define $m_g \define q \comp \tilde{m}_g$, then over some neighbourhood
of $(x_0,x_0) \in M \times M$,
\begin{equation*}
    \IA \tensor \IA \iso \DirSum_{g \in \Gamma} \Ind(m_g).
\end{equation*}

\begin{lemma}
    \label{lem:A_tensor_A_tensor_A}
    Locally around $(x_0,x_0,x_0)$,
    \begin{equation*}
        \tilde{m}_g(\tilde{m}_h(-,-),-) = \tilde{m}_{h}(-,\tilde{m}_{h^{-1}g}(-,-)).
    \end{equation*}
\end{lemma}
\begin{proof}
    The associator $\alpha_{\IA,\IA,\IA}$ is an isomorphism
    $(\IA \tensor \IA) \tensor \IA \isoto \IA \tensor (\IA \tensor \IA)$.
    By Lemma~\ref{lem:A_tensor_A}, there are isomorphisms
    \begin{align*}
        (\IA \tensor \IA) \tensor \IA & \iso
        \DirSum_{h \in \Gamma} \Ind m_h \tensor \IA \iso
        \DirSum_{h \in \Gamma} {(\tilde{m}_h \times \id_M)}^\ast \IA \tensor \IA
        \\ &\iso
        \DirSum_{g,h \in \Gamma} {(\tilde{m}_h \times \id_M)}^\ast \Ind m_g \iso
        \DirSum_{g,h \in \Gamma} \Ind m_g\left(\tilde{m}_h(-,-),-\right)
    \end{align*}
    and
    \begin{equation*}
        \IA \tensor (\IA \tensor \IA) \iso
        \DirSum_{g,h \in \Gamma} \Ind m_g\left(-,\tilde{m}_h(-,-)\right).
    \end{equation*}

    The two expressions witness that $(\IA \tensor \IA) \tensor \IA$ has
    local support lifts $M \times M \times M \to M$ given by
    \begin{align*}
        \{\tilde{m}_{h,g} \define
        \tilde{m}_g\left(\tilde{m}_h(-,-),-\right)\}_{g,h \in \Gamma} &  & \text{or} &  &
        \{\tilde{m}_{h,g}^\prime \define
        \tilde{m}_g\left(-,\tilde{m}_h(-,-)\right)\}_{g,h \in \Gamma}.
    \end{align*}
    We have fixed the ambiguity of the lifts by picking all of them such that
    \begin{equation*}
        \tilde{m}_{h,g}(x,x_0,x_0) = \tilde{m}_{h,g}^\prime (x,x_0,x_0) = x.
    \end{equation*}
    The values on $\{x_0\} \times M \times \{x_0\}$ and
    $\{x_0\} \times \{x_0\} \times M$ are
    \begin{align*}
        \begin{split}
            \tilde{m}_{h,g}(x_0,x,x_0) &= h\cdot x \\
            \tilde{m}_{h,g}(x_0,x_0,x) &= g\cdot x
        \end{split}
         &  &
        \begin{split}
            \tilde{m}_{h,g}^\prime(x_0,x,x_0) &= g\cdot x \\
            \tilde{m}_{h,g}^\prime(x_0,x_0,x) &= gh\cdot x.
        \end{split}
    \end{align*}
    $\{\tilde{m}_{h,g}\}$ and $\{\tilde{m}_{h,g}^\prime\}$
    are thus both sets of $\Gamma \times \Gamma$ distinct functions which don't equal
    each other on any open neighbourhood of $(x_0,x_0,x_0)$.
    Thus, they must match up in pairs under the isomorphism above.
    The value along the second and third factor of $M \times M \times M$ 
    dictates
    \begin{equation*}
        \tilde{m}_{h,g} = \tilde{m}_{h^{-1}g,h}^\prime. \qedhere
    \end{equation*}
\end{proof}
\begin{notation}
    We introduce more compact notation for the support lift
    labelled by the identity:
    \begin{equation*}
        \mu \define \tilde{m}_e.
    \end{equation*}
\end{notation}
We note an immediate Corollary of Lemma~\ref{lem:A_tensor_A_tensor_A}:
\begin{corollary}
    \label{cor:mu_is_monoid_multiplication}
    The partial function $\mu:M \times M \to M$ makes $M$ into
    a partial smooth monoid with unit $x_0$.
\end{corollary}

This multiplication is locally group-like.
\begin{lemma}
    \label{lem:smooth_monoid_locally_group_like}
    Any partial smooth monoid $(M, x_0 \in M, \mu:M \times M \to M)$
    is group-like in a neighbourhood of its identity $x_0$.
    Ie.\ there exists a smooth partial function $s:M \to M$
    (which is defined on some neighbourhood of $x_0$),
    that acts as a left and right inverse.
\end{lemma}
\begin{proof}
    As $\mu$ restricts to the identity on
    $M \times \{x_0\}$ and $\{x_0\} \times M$,
    it is submersive at $(x_0,x_0)$. The Implicit Function Theorem
    then guarantees the existence of two (locally defined) maps
    $s,s':M \to M$, such that
    $\mu(s(x), x) = \mu(x, s'(x)) = x_0$.
    The usual argument yields equality of left and right inverse:
    \begin{equation*}
        s(x) = \mu\left(s(x),\mu(x,s'(x))\right)
        = \mu\left(\mu(s(x),x),s'(x)\right) = s'(x). \qedhere
    \end{equation*}
\end{proof}

\section{Smooth Dualisability}
\label{sec:smooth_rig}
In this section, we prove that pointwise dualisability in effective orbifusion categories
automatically smoothly extends to families:
All categories $\cat{C}^U$ assigned to $U \in \Man$ by an effective orbifusion category
$\underline{\cat{C}}$ are autonomous (Theorem~\ref{thm:global_dualisability}).
This is a categorification of the following Lemma from
classical Lie Theory:
\begin{lemma}
    \label{lem:inverse_automatically_smooth}
    A manifold $M$ equipped with a group structure whose
    multiplication map $m: M \times M \to M$ is smooth is a Lie group.
\end{lemma}
Note that it is \emph{not} necessary to require the inverse to be smooth.
\begin{proof}
    The multiplication $m$ is submersive everywhere, so
    $m^{-1}(e) \subset M \times M$ is a submanifold.
    The restriction of the projection map $M \times M \to M$
    along $m^{-1}(e) \into M \times M$ is of constant rank:
    $p_1:m^{-1}(e) \to M$ is equivariant with respect to
    the conjugation action of $x \in M$, and any two points
    in $m^{-1}(e)$ are related by the action of some $x \in M$.
    The rank at $(e,e) \in M \times M$ is full by
    Lemma~\ref{lem:smooth_monoid_locally_group_like},
    hence $p_1$ has full rank everywhere and is
    a diffeomorphism.
    The inverse map
    $p_2 \comp {(p_1)}^{-1}:M \to M$
    gives $M$ the structure of a Lie group.
\end{proof}

We recall two lemmas from tensor category theory.
\begin{lemma}
    \label{lem:duals_unique_up_to_unique_isomorphism}
    Duality data is unique up to unique isomorphism:
    Let $X$ be an object in a tensor category and
    $(Y, \mathrm{ev}), (Y',\mathrm{ev}')$ two left duals
    for $X$, then there is a unique isomorphism
    $Y \to Y'$ compatible with the duality data.
\end{lemma}
\begin{proof}
    The proof is a categorification of the uniqueness of
    inverses in monoids.
    The isomorphism is furnished by the zig-zags
    $Y \to Y' \tensor X \tensor Y \to Y'$ and
    $Y' \to Y \tensor X \tensor Y' \to Y$.
    This morphism is manifestly compatible with the duality data.
    Uniqueness follows by composing any given isomorphism
    (compatible with the duality data) with the duality equations.
    They exhibit equality between the isomorphism and the
    zig-zag above.
\end{proof}

\begin{lemma}
    \label{lem:summands_of_rigids_are_rigid}
    Dualisability descends to direct summands:
    Let $X$ be a left dualisable object in a
    Karoubi-complete tensor category,
    and $Y \subset X$ a direct summand.
    Then $Y$ is dualisable.
\end{lemma}
\begin{proof}
    On the summand $Y \subset X$, the first duality equation for
    $X$ restricts to
    \begin{equation*}
        Y \to Y \tensor X^\vee \tensor Y \to Y = \id_Y.
    \end{equation*}
    Using this, one shows that the composite
    \begin{equation*}
        X^\vee \to X^\vee \tensor X \tensor X^\vee \to
        X^\vee \tensor Y \tensor X^\vee \to X^\vee \tensor X \tensor X^\vee \to X^\vee
    \end{equation*}
    is an idempotent.
    By Karoubi-completeness, $X^\vee$ contains a summand
    $Y^\vee$, on which this idempotent is the identity.
    It is straightforward to check that the induced maps 
    \begin{align*}
        \ONE \to X \tensor X^\vee \to Y \tensor Y^\vee && \text{ and } &&
        Y^\vee \tensor Y \to X^\vee \tensor X \to \ONE
    \end{align*}
    identify $Y^\vee$ as the left dual of $Y$.
\end{proof}

Recall that we made a choice of chart $q:M \to \orb{X}$ around the
support of the monoidal unit.
Associated to these choices is a mobile family $\IA = \Ind q$, which we only
fixed up to isomorphism so far.
So as to be able to talk about maps into and out of $\IA$, we simply pick
a representative of the isomorphism class of $\IA$.
Its pullback over the point $x_0=q^{-1}(\supp \ONE)$ is denoted $\IO = \IA(x_0)$,
and equipped with the inclusion of a summand $\ONE \into \IO \onto \ONE$.

Recall that $\IA \tensor \IA \iso \DirSum_{g\in \Gamma}
    \Ind(\tilde{m_g})$ (Lemma~\ref{lem:A_tensor_A}).
We pick an explicit isomorphism once and for all.
Restricted to the origin $(x_0,x_0)$, this gives an explicit isomorphism
$\IO \tensor \IO \xto{\iso} \IO^{\dirSum \Gamma}$.
We denote the maps
identifying the summand corresponding to $g \in \Gamma$ by
\begin{equation*}
    \begin{tikzcd}
        \IO \rar[hook]{\iota_g} &
        \IO \tensor \IO \rar[two heads]{p_g} &
        \IO.
    \end{tikzcd}
\end{equation*}
The following technical Lemma will be used in our proof of local dualisability.
\begin{lemma}
    \label{lem:projectionAndInclusionOfUnitSummandNonzero}
    The composites
    \[
        \begin{tikzcd}
            p_e\restriction_{\ONE \tensor \ONE}^\ONE:
            \ONE \tensor \ONE \rar[hook] & \IO \tensor \IO \rar[two heads, "p_e"] &
            \IO \rar[two heads] & \ONE
        \end{tikzcd}
    \]
    and
    \[
        \begin{tikzcd}
            \iota_e\restriction_{\ONE}^{\ONE \tensor \ONE}:
            \ONE \rar[hook] & \IO \rar[hook, "\iota_e"] &
            \IO \tensor \IO \rar[two heads] & \ONE \tensor \ONE
        \end{tikzcd}
    \]
    are nonzero.
\end{lemma}
\begin{proof}
    Recall from Lemma~\ref{lem:A_tensor_A_tensor_A} the two decomposition of
    $(\IA \tensor \IA) \tensor \IA$, related by
    the associator $\alpha_{\IA,\IA,\IA}$.
    This associator decomposes into components
    \begin{equation*}
        \alpha_{\IA,\IA,\IA}^{g,h}:
        \Ind m_h(\tilde{m}_g(-,-),-) \xrightarrow{\quad \iso \quad}
        \Ind m_g(-,\tilde{m}_{g^{-1}h}(-,-)).
    \end{equation*}

    We denote the restriction of
    $\alpha_{\IA,\IA,\IA}^{g,h}$ to the origin $(x_0,x_0,x_0)$ by
    \begin{equation*}
        \alpha^{g,h} \define \alpha_{\IA,\IA,\IA}^{g,h}(x_0,x_0,x_0):
        \IO \xrightarrow{\quad \iso \quad} \IO.
    \end{equation*}
    As indicated above, it is an isomorphism.

    By definition, these components make the following diagram commute:
    \begin{equation*}
        \begin{tikzcd}
            (\IO \tensor \IO) \tensor \IO
            \ar[d, "\alpha_{\IO,\IO,\IO}"]
            \ar[rr, "p_g \tensor \id"]
            &&
            \IO \tensor \IO
            \ar[r, "p_h"]
            &
            \IO
            \ar[d, "\alpha^{g,h}"]
            \\
            \IO \tensor (\IO \tensor \IO)
            \ar[rr, "\id \tensor p_{g^{-1}h}"]
            &&
            \IO \tensor \IO
            \ar[r, "p_g"]
            &
            \IO.
        \end{tikzcd}
    \end{equation*}
    By restricting all maps appropriately, one obtains:
    \begin{equation*}
        \begin{tikzcd}
            (\ONE \tensor \ONE) \tensor \ONE
            \ar[d, "\alpha_{\ONE,\ONE,\ONE}"]
            \ar[rr, "p_g\restriction_{\ONE \tensor \ONE}^\ONE \tensor \id"]
            &&
            \ONE \tensor \ONE
            \ar[r, "p_h\restriction_{\ONE \tensor \ONE}^\ONE"]
            &
            \ONE
            \ar[d, "\alpha^{g,h}\restriction_\ONE^\ONE"]
            \\
            \ONE \tensor (\ONE \tensor \ONE)
            \ar[rr, "\id \tensor p_{g^{-1}h}\restriction_{\ONE \tensor \ONE}^\ONE"]
            &&
            \ONE \tensor \ONE
            \ar[r, "p_g\restriction_{\ONE \tensor \ONE}^\ONE"]
            &
            \ONE.
        \end{tikzcd}
    \end{equation*}
    All objects in the diagram above are isomorphic to $\ONE$, which means
    all maps are multiples of the identity on $\ONE$ and both
    composition and tensoring are given by scalar multiplication.
    The components of the associator are isomorphisms, and thus non-zero,
    so the product
    $p_g\restriction_{\ONE \tensor \ONE}^\ONE \cdot
        p_h\restriction_{\ONE \tensor \ONE}^\ONE$
    coming from the top row is non-zero iff
    $p_{g^{-1}h}\restriction_{\ONE \tensor \ONE}^\ONE \cdot
        p_g\restriction_{\ONE \tensor \ONE}^\ONE$
    (coming from the bottom row) is non-zero.
    As $\sum \iota_g \comp p_g = \id_{\IO \tensor \IO}$, at least one of the maps
    $p_g\restriction_{\ONE \tensor \ONE}^\ONE$ must be non-zero.
    But then
    $p_g\restriction_{\ONE \tensor \ONE}^\ONE \cdot
        p_g\restriction_{\ONE \tensor \ONE}^\ONE$ is also non-zero, which implies
    $p_{e}\restriction_{\ONE \tensor \ONE}^\ONE \cdot
        p_g\restriction_{\ONE \tensor \ONE}^\ONE$ is non-zero.
    We have now proved that
    $p_e\restriction_{\ONE \tensor \ONE}^\ONE$ is not zero.
    Replacing the projections $p_g$ above with their splittings $\iota_g$,
    one proves $\iota_e\restriction_{\ONE}^{\ONE \tensor \ONE}$ is also not zero.
\end{proof}

\begin{lemma}
    \label{lem:A_is_rigid}
    The family $\IA:M \to \cat{C}$ is locally dualisable at $x_0$:
    There exists a
    neighbourhood $V$ of $x_0$, such that $\IA\restriction_V$ is dualisable
    as an object of $\cat{C}^V$.
\end{lemma}
\begin{proof}
    We will use the support lift $\mu = \tilde{m}_e$
    introduced in Lemma~\ref{lem:A_tensor_A_tensor_A}.
    By Corollary~\ref{cor:mu_is_monoid_multiplication},
    $\mu$ equips a neighborhood of $x_0$ with the structure of
    a smooth monoid.
    In Lemma~\ref{lem:smooth_monoid_locally_group_like} we showed
    that there exists a partial function $s:M \to M$
    making $(M, x_0, \mu, s)$ into a smooth partial group.
    Below, we denote by $c_{x_0}$ and $c_0$ the constant maps
    \begin{align*}
        \begin{split}
            c_{x_0}: M &\to M \\
            x &\mapsto x_0
        \end{split}
         &  &
        \begin{split}
            c_0: M &\to [M/\Gamma] \\
            x &\mapsto 0,
        \end{split}
    \end{align*}
    and write $\Gr(s) = (s \times \id_M) \comp \Delta_M$.
    By definition $c_0 = q \comp c_{x_0}$ and
    $c_{x_0} = \mu \comp \Gr(s) =
        \mu \comp \mathrm{swap} \comp \Gr(s)$, where
    $\mathrm{swap}:M \times M \to M \times M$ is the map
    swapping the two arguments.
    Pulling back via $s$ yields a family
    \begin{equation*}
        \IB \define s^\ast\IA:M\to\cat{C}.
    \end{equation*}
    By Lemma~\ref{lem:A_tensor_A},
    $\IA \tensor \IA \iso \DirSum_g \Ind(q \comp \mu \comp (\id \times g\cdot -))$,
    and the $M$-family obtained as the tensor product
    \begin{equation*}
        \IB \tensor^M \IA =
        \Delta_M^\ast (s^\ast \IA \tensor \IA) =
        \Gr(s)^\ast (\IA \tensor \IA)
    \end{equation*}
    contains a summand
    \begin{equation*}
        \Gr(s)^\ast (\IA \tensor \IA) \supset
        \Gr(s)^\ast \Ind q \comp \mu =
        \Ind q \comp \mu \comp \Gr(s) = \Ind c_0.
    \end{equation*}
    $\Ind c_0$ is isomorphic to the constant family at $\IO$,
    which contains the constant
    family at $\ONE$ as a summand. This is the monoidal unit
    $\ONE^M$ of ${\cat{C}}^M$.
    In the same way, there is a summand
    \begin{equation*}
        \ONE^M \subset
        {(\mathrm{swap} \comp \Gr(s))}^\ast \IA \tensor \IA \iso
        \IA \tensor^M \IB
    \end{equation*}
    Now $\IB$ is a candidate for the dual of $\IA$, and the candidate
    duality maps are:
    \begin{align*}
        c: \ONE^M \into \IA \tensor^M \IB &  &
        e: \IB \tensor^M \IA \onto \ONE^M
    \end{align*}
    given by inclusion of and projection to the summand.
    Denote the composites
    \begin{align*}
        \IA \xrightarrow{\iso} \ONE^M \tensor^M \IA
        \xrightarrow{c \tensor^M \id}
        \big(\IA \tensor^M \IB\big) \tensor^M \IA
        \xrightarrow{\iso}
        \IA \tensor^M \big(\IB \tensor^M \IA\big)
        \xrightarrow{\id \tensor^M e}
        \IA \tensor^M \ONE^M \xrightarrow{\iso} \IA \\
        \IB \xrightarrow{\iso} \IB \tensor^M \ONE^M
        \xrightarrow{\id \tensor^M c}
        \IB \tensor^M \big(\IA \tensor^M \IB\big)
        \xrightarrow{\iso}
        \big(\IB \tensor^M \IA\big) \tensor^M \IB
        \xrightarrow{e \tensor^M \id}
        \ONE^M \tensor^M \IB \xrightarrow{\iso} \IB
    \end{align*}
    by $D_{\IA} \in \End(\IA)$ and $D_{\IB} \in \End(\IB)$.
    The duality equations for $\IA$ and $\IB$ can then be written as
    \begin{align*}
        D_{\IA} = id_{\IA} &  & D_{\IB} = \id_{\IB}.
    \end{align*}
    Composition and tensor product are $\Cinf(M)$-linear, thus
    multiplying $e \in \Hom(\IB \tensor^M \IA, \ONE^M)$ by a function
    $f \in \Cinf(M)$ has the effect of multiplying both
    $D_{\IA}$ and $D_{\IB}$ by $f$.
    Thus $\IA$ and $\IB$ are duals if
    \begin{align*}
        f \cdot
        \begin{pmatrix}
            D_{\IA} \\ D_{\IB}
        \end{pmatrix}
        =
        \begin{pmatrix}
            \id_{\IA} \\ \id_{\IB}
        \end{pmatrix}
    \end{align*}
    for some function $f \in \Cinf(M)$.
    At all $x \in M$ where $\IA(x)$ is simple,
    $\Hom(\ONE, \IA(x) \tensor \IB(x))$ and
    $\Hom(\IB(x) \tensor \IA(x), \ONE)$ are one-dimensional.
    $\IA(x)$ and $\IB(x)$ are in $\cat{C}$ and thus dualisable, so
    there are evaluation and coevaluation maps in these $\Hom$-spaces,
    and $c(x)$ and $e(x)$ are automatically non-zero multiples thereof.
    Hence, the equations
    $D_{\IA}(x) = r(x) \cdot \id_{\IA}, D_{\IB}(x) = r(x) \cdot \id_{\IB}$
    hold at these points (for some $r(x) \in \IC$).
    As $\IA$ is an \'etale family, $\IA(x)$ is simple on a dense subset of $M$,
    so this property extends to all of $M$.
    There is then a function $r \in \Hom(\IA,\IA) \iso \Cinf(M)$ satisfying
    \begin{align*}
        \begin{pmatrix}
            D_{\IA} \\ D_{\IB}
        \end{pmatrix}
        =
        r \cdot
        \begin{pmatrix}
            \id_{\IA} \\ \id_{\IB}
        \end{pmatrix}.
    \end{align*}
    It remains to show that $r$ is nowhere-zero in some neighbourhood of
    $x_0 \in M$.
    By smoothness of morphisms, this is equivalent to checking $D_{\IA}(x_0) \neq 0$.
    The map $D_{\IA}(x_0) \in \End(\IO)$ decomposes into matrix components.
    We will check that the component
    \[
        D_{\IA}(x_0)\restr_\ONE^\ONE \in \End(\ONE)
    \]
    is nonzero. It factors as a composition of maps in the one-dimensional
    vector space $\End(\ONE)$:
    \begin{equation*}
        \ONE \xrightarrow{\iso} \ONE \tensor \ONE
        \xrightarrow{c(x_0)\restr^{\ONE \tensor \ONE} \tensor \id}
        \big(\ONE \tensor \ONE\big) \tensor \ONE
        \xrightarrow{\iso}
        \ONE \tensor \big(\ONE \tensor \ONE\big)
        \xrightarrow{\id \tensor e(x_0)\restr_{\ONE \tensor \ONE}}
        \ONE \tensor \ONE \xrightarrow{\iso} \ONE.
    \end{equation*}
    As the isomorphisms are clearly non-zero,
    the proof reduces to showing that
    \begin{equation*}
        c(x_0)\restr^{\ONE \tensor \ONE}=
        \iota_e\restriction_{\ONE}^{\ONE \tensor \ONE}:
        \ONE \to \IO \tensor \IO \onto \ONE \tensor \ONE
    \end{equation*}
    and
    \begin{equation*}
        e(x_0)\restr_{\ONE \tensor \ONE}=
        p_e\restriction_{\ONE \tensor \ONE}^{\ONE}:
        \ONE \tensor \ONE \into \IO \tensor \IO \to \ONE
    \end{equation*}
    are nonzero.
    This was done in Lemma~\ref{lem:projectionAndInclusionOfUnitSummandNonzero}.
\end{proof}

\begin{lemma}
    \label{lem:local_dualisability}
    Every family is locally dualisable. Given a $U$-family $\f{S}$, for any point
    $x \in U$, there exists a neighbourhood $V$ of $x$, such that
    $\f{S}\restriction_V$ is dualisable.
\end{lemma}
\begin{proof}
    Dualisability extends to direct sums, so we can assume $\f{S}$ is simple
    at $x$.
    Then by Lemma~\ref{lem:locally_summand_of_pullback_from_A_tensor_X},
    $\f{S}$ is a summand of a pullback from $\IA \tensor \f{S}(x)$ on some
    neighbourhood $V$ of $x$.
    $\IA \tensor \f{S}(x)$ has a dual ${\f{S}(x)}^\vee \tensor \IA^\vee$.
    They remain dual when pulled back via any smooth map.
    As ${\cat{C}}^U$ is Karoubi-complete,
    Lemma~\ref{lem:summands_of_rigids_are_rigid} shows
    $\f{S}\restriction_V$ is dualisable.
\end{proof}

The following is the central theorem of this section.
\begin{thm}
    \label{thm:global_dualisability}
    In an orbifusion category, ${\cat{C}}^U$ is autonomous for all $U \in \Man$.
\end{thm}
\begin{proof}
    Let $\f{S}:U \to \cat{C}$ be a $U$-family.
    By Lemma~\ref{lem:local_dualisability}, there exists a cover
    $Y \onto U$ such that $\f{S}$ is dualisable when pulled back
    to $Y$. Pick duality data $(\f{T}:Y \to \cat{C}, \mathrm{ev})$
    for $\f{S}\restriction_Y$.
    Pulling back to the double- and triple- overlaps
    $Y^{[2]} \define Y \times_U Y$ and
    $Y^{[3]} \define Y \times_U Y \times_U Y$,
    we obtain duality data for
    $\f{S}\restriction_{Y^{[2]}}$ and
    $\f{S}\restriction_{Y^{[3]}}$.
    By Lemma~\ref{lem:duals_unique_up_to_unique_isomorphism},
    there exist unique isomorphisms between all of them.
    These yield canonical isomorphisms on $Y^{[2]}$ which
    commute on $Y^{[3]}$ by uniqueness.
    Thus local dualisability gives us a valid descent datum
    which can be glued to a global dual of $\f{S}$ on $U$.
\end{proof}

Denote by $\stC^\opp$ the stack obtained from $\stC$ by pointwise replacing
$\cat{C}^U$ by $\cat{C}^{U,\opp}$.
\begin{cor}
    \label{cor:smooth_dualising_functor}
    An orbifusion category admits a smooth dualising functor
    \[
        {(-)}^\vee: \stC \to \stC^\opp.
    \]
\end{cor}
\begin{proof}
    By Theorem~\ref{thm:global_dualisability}, there exist functors
    ${(-)}_U^\vee:\cat{C}^U \to \cat{C}^{U,\opp}$ that assign duals.
    Duals are unique up to unique isomorphism (Lemma~\ref{lem:duals_unique_up_to_unique_isomorphism}),
    hence any two dualising functors are uniquely naturally isomorphic.
    There are thus unique natural isomorphisms
    \[
        \xi_f:f^\ast \comp {(-)}_U^\vee \eqto {(-)}_V^\vee \comp f^\ast
    \]
    for any $f: V \to U$ in $\Man$.
    By uniqueness, these natural isomorphisms satisfy the cocycle condition
    for pairs of composable morphisms $f \comp g: W \to V \to U$,
    so the data $\{{(-)}_U^\vee,\xi_f\}$ indeed assembles into a map of stacks.
\end{proof}

\begin{lemma}
    All simple objects supported on the unit connected component of a
    manifusion category are invertible.
\end{lemma}
\begin{proof}
    Let $X \in \cat{C}$ be simple with support in the unit component
    of the underlying manifold $M$. Pick a path $p: [0,1] \to M$
    from $p(0) = \supp \ONE$ to $p(1)= \supp x$, and form the associated family
    $\Ind p$. There are isomorphisms $\Ind p(0) \iso \ONE$, $\Ind p(1) \iso X$.
    Then, by Lemma~\ref{lem:dim_locally_constant},
    \begin{equation*}
        \dim \ONE = \dim \ONE \tensor \ONE =
        \dim \left( {(\Ind p)}^\vee \tensor \Ind p \right) =
        \dim \left(X^\vee \tensor X\right).
    \end{equation*}
    As a result, $X^\vee \tensor X \iso \ONE$.
\end{proof}

\begin{corollary}
\label{cor:unitComponentIsLieGroup}
    In a manifusion category, the isomorphism classes of objects supported on
    the unit component always form a Lie group.
\end{corollary}
\begin{proof}
    The Grothendieck group underlying the stack encodes a smooth manifold
    $M$, the tensor product an associative smooth multiplication map
    $m:M \times M \to M$, and pointwise invertibility implies there exists a
    smooth inverse $s:M \to M$ by Lemma~\ref{lem:inverse_automatically_smooth}.
\end{proof}

\include{acks}

\appendix

\section{Background Material}
\label{ch:background}
In this section, we review helpful background material.
we do not prove any new statements, and will omit details 
when convenient. We provide suitable references instead.

Manifold and orbifold tensor categories are (2-)monoids
in monoidal bicategories.
We review the relevant notions in 
Section~\ref{sec:monoidalStructures}.
The specific monoidal bicategories we work in will be categories of stacks over $\Man$, 
the site of smooth manifolds. Sheaves and stacks are reviewed in Section~\ref{sec:sheavesAndStacks}.

\subsection{Monoidal Structures}
\label{sec:monoidalStructures}
A unital (1-)monoid is a set equipped with a unit element and an associative multiplication.
This is the beginning of a hierarchy: the notion of a unital $n$-monoid makes sense in a unital 
$(n+1)$-monoid. 
A unital 1-monoid is called commutative if its multiplication commutes with the swap map in the
category of sets.
This notion makes sense in a monoidal category equipped with a braiding 
(a braided unital \emph{2-monoid}).
More generally, there are $n$ levels of ``commutativity'' in the world of $n$-monoids, and  
the notion of a $k$-commutative $n$-monoid makes sense inside a $k$-commutative
$(n+1)$-monoid. See~\cite{baez1998higher} for a precise statement, and~\cite{baez2010lectures}
for a discussion of $n$-categories with varying levels of commutativity.

In this Section, we recall some of the above theory in the context of
\emph{tensor categories} (linear 1-monoids), and in the context of \emph{monoidal bicategories}
(2-monoids).

\subsubsection{Tensor Categories}
\label{sec:tensorCategories}
In this section, we work towards the definition of a fusion category.
We begin by introducing the necessary notions from linear category theory
before we go on to introduce tensor structures and duals. We provide 
examples along the way.
A textbook account of the theory of tensor categories may be found in~\cite{etingof2016tensor}.

We denote by $\VecInf$ the category of vector spaces over $\IC$,
and by $\Vec \subset \VecInf$ the full subcategory on finite-dimensional
vector spaces.
We use the term \emph{linear category} to refer to
a category enriched over $\VecInf$.

Let $\cat{C}$ be a linear category.
\begin{defn}
    A \emph{zero object} in $\cat{C}$ is an object 
    $0 \in \cat{C}$ such that every object admits 
    a unique map to and from $0$.
\end{defn}
The zero object induces a morphism $0:X \to 0 \to Y$ between any
pair of objects. These morphisms are called \emph{zero morphisms}.
\begin{defn}
    A \emph{summand} of an object $X \in \cat{C}$ is an object
    $Y \in \cat{C}$ with a pair of maps
    % https://q.uiver.app/?q=WzAsMyxbMCwwLCJYIl0sWzEsMCwiWSJdLFsyLDAsIlgiXSxbMCwxLCJwIiwwLHsic3R5bGUiOnsiaGVhZCI6eyJuYW1lIjoiZXBpIn19fV0sWzEsMiwiaSIsMCx7InN0eWxlIjp7InRhaWwiOnsibmFtZSI6Imhvb2siLCJzaWRlIjoidG9wIn19fV1d
    \[\begin{tikzcd}
    	X & Y & X
    	\arrow["p", two heads, from=1-1, to=1-2]
    	\arrow["i", hook, from=1-2, to=1-3]
    \end{tikzcd}\]
    satisfying $p \comp i=\id_Y$.
\end{defn}

\begin{defn}
    A \emph{direct sum} of two objects $X,Y \in \cat{C}$ 
    is an object $Z$ equipped with four maps
    % https://q.uiver.app/?q=WzAsMyxbMSwwLCJaIl0sWzAsMCwiWCJdLFsyLDAsIlkiXSxbMSwwLCJpX1giLDJdLFsyLDAsImlfWSJdLFswLDEsInBfWCIsMix7ImN1cnZlIjoxfV0sWzAsMiwicF9ZIiwwLHsiY3VydmUiOi0xfV1d
    \[\begin{tikzcd}
    	X & Z & Y
    	\arrow["{i_X}"', from=1-1, to=1-2]
    	\arrow["{i_Y}", from=1-3, to=1-2]
    	\arrow["{p_X}"', curve={height=6pt}, from=1-2, to=1-1]
    	\arrow["{p_Y}", curve={height=-6pt}, from=1-2, to=1-3]
    \end{tikzcd}\]
    exhibiting $X$ and $Y$ as summands of $Z$ which are further
    orthogonal and exhaust $Z$:
    \begin{align*}
        p_Y \comp i_X = 0 = p_X \comp i_Y && 
        i_X \comp p_X + i_Y \comp p_Y = \id_Z.
    \end{align*}
\end{defn}
Direct sums are unique up to unique isomorphism, so it makes sense 
to speak of ``the'' direct sum, denoted by $X \dirSum Y$.
Direct sums are simultaneously limits and colimits. Their universal
property may be written entirely as a set of equations.
This makes them an example of an 
\emph{absolute limit}: a limit preserved by any functor.

\begin{defn}
    An object $X \in \cat{C}$ is \emph{simple},
    if every monomorphism $Y \into X$ is either an isomorphism
    or a zero morphism.
\end{defn}
The inclusion of a summand is always a monomorphism, so this 
implies summands of a simple morphism are isomorphic to it.

\begin{example}
    A vector space $V \in \VecInf$ is simple iff it is 1-dimensional.
\end{example}

\begin{example}
\label{ex:VecX}
    Let $X$ be a set. We denote by $\Vec[X]$ the linear category
    of $X$-graded vector spaces. Its objects are finite direct sums
    of objects of the form $V_x$, where $V \in \Vec$ is a finite-dimensional 
    vector space and $x \in X$. 
    The morphism spaces in $\Vec[X]$ are generated under direct sum by
    $\Hom(V_x,W_x)=\Hom_{\Vec}(V,W)$ and $\Hom(V_x,W_{x'})=0$ when 
    $x\neq x'$.
    The simple objects of $\Vec[X]$ are those of the form
    $V_x$ where $V$ is 1-dimensional.
\end{example}

\begin{example}
    Let $\Rep(G)$ be the category of finite-dimensional complex representations 
    of a group $G$. The simple objects of $\Rep(G)$ are the 
    irreducible representations.
\end{example}

The classical Schur's Lemma from representation theory admits a 
generalisation to linear categories:
\begin{lemma}[Schur's Lemma]
    For any pair of simple objects $X, Y \in \cat{C}$
    the morphism space $\Hom(X,Y)$ is either 0 or 1-dimensional 
    (in which case $X \iso Y$).
\end{lemma}
In particular, the endomorphism ring of a simple object 
$X \in \cat{C}$ may be identified with the complex numbers 
$\End(X)=\IC$.

\begin{defn}
    A linear category $\cat{C}$ is \emph{semisimple},
    if every object is a finite direct sum of simple objects,
    and every such finite direct sum exists.
\end{defn}
This condition implies, in particular, that all idempotents
\[
    e:X \to X \text{\ s.t.\ } e^2 = e
\]
admit a splitting: a direct summand $Y \subset X$ whose structure maps
% https://q.uiver.app/?q=WzAsMyxbMCwwLCJYIl0sWzEsMCwiWSJdLFsyLDAsIlgiXSxbMCwxLCJwIiwwLHsic3R5bGUiOnsiaGVhZCI6eyJuYW1lIjoiZXBpIn19fV0sWzEsMiwiaSIsMCx7InN0eWxlIjp7InRhaWwiOnsibmFtZSI6Imhvb2siLCJzaWRlIjoidG9wIn19fV1d
\[\begin{tikzcd}
	X & Y & X
	\arrow["p", two heads, from=1-1, to=1-2]
	\arrow["i", hook, from=1-2, to=1-3]
\end{tikzcd}\]
compose to $i \comp p = e$.

The category $\Vec$ is semisimple, and so is $\Vec[X]$.
Semisimplicity of $\Rep(G)$ depends on the group:
when $G$ is finite, $\Rep(G)$ is automatically semisimple.

\begin{defn}
\label{def:tensorCategory}
    A \emph{tensor category} is a linear category $\cat{C}$,
    equipped with a choice of \emph{unit object} 
    $\ONE \in \cat{C}$, a \emph{tensor product} functor
    $\tensor: \cat{C} \times \cat{C} \to \cat{C}$,
    an \emph{associator} natural isomorphism
    \begin{equation*}
        \alpha: - \tensor (- \tensor -) \Rightarrow
        - \tensor (- \tensor -),
    \end{equation*}
    and \emph{unitor} natural isomorphisms
    \begin{align*}
        l: \ONE \tensor - &\Rightarrow - \\
        r: - \tensor \ONE &\Rightarrow -.
    \end{align*}

    These data are subject to the \emph{triangle equations}
    \begin{align*}
        r_X \tensor \id_Y = (\id_X \tensor l_Y) \comp \alpha_{X,\ONE,Y}
    \end{align*}
    for all $X,Y \in \cat{C}$ and the \emph{pentagon equations}
    \begin{equation*}
        (\id_W \tensor \alpha_{X,Y,Z}) 
        \comp \alpha_{W,X \tensor Y,Z} \comp 
        (\alpha_{W,X,Y} \tensor \id_Z)
        = 
        \alpha_{W, X, Y \tensor Z} \comp 
        \alpha_{W \tensor X, Y, Z},
    \end{equation*}
    for all $W,X,Y,Z \in \cat{C}$.
\end{defn}

As explained in~\cite[Rmk 2.2.9]{etingof2016tensor}, the unitors
present in Definition~\ref{def:tensorCategory} may be suppressed:
the data $(\cat{C},\ONE,\tensor,\alpha,l,r)$ is equivalent to 
the data $(\cat{C},\ONE,\tensor,\alpha)$, 
together with an explicit choice of isomorphism $\ONE \tensor \ONE \to \ONE$.
We will usually suppress unitors.

\begin{defn}
    A \emph{tensor functor} of tensor categories is a 
    linear functor $F:\cat{C} \to \cat{D}$, equipped with
    natural transformations
    $\mu: (F-) \tensor (F-) \eqto F(-\tensor -)$ and 
    $\eps:F \ONE_\cat{C} \to \ONE_\cat{D}$
    satisfying associativity and 
    unitality conditions.
\end{defn}

\begin{example}
\label{ex:VecAsAMonoidalCategory}
    The categories $\Vec$ and $\VecInf$ both become
    tensor categories when equipped with 
    the usual tensor product of vector spaces.
    The unitors and associators are the canonical identity 
    morphisms.\footnote{
        In fact, this monoidal structure on $\VecInf$ 
    (considered as an ordinary category) 
    is necessary to even define the notion of $\VecInf$-enriched category. 
    It is straightforward to check that this
    monoidal structure descends to the $\VecInf$-enriched context.}
\end{example}

\begin{example}
\label{ex:VecXAsTensorCategory}
    Recall the linear category $\Vec[X]$ from Example~\ref{ex:VecX}.
    Equip the set $X$ with a unital, associative multiplication $\cdot$.
    Then $\Vec[X]$ inherits a tensor structure given by 
    the convolution product
    \[
        V_x \tensor W_{x'} = (V \tensor W)_{x \cdot x'},
    \]
    with associators and unitors induced by those on $\Vec$.
\end{example}

\begin{example}
\label{ex:VecOmegaX}
    One may ask 
    what other solutions to the pentagon equation there are
    for $\Vec[X]$ equipped with the convolution product.
    It is straightforward to check that the pentagon equation 
    boils down to the equation of a 3-cocycle in the monoid cohomology 
    of $X$ with coefficients in $\IC^\times$. 
    This is the cohomology of the classifying space $\B X$ of $X$.
    Given such a 3-cocycle $\omega \in Z^3(\B X, \IC^\times)$,
    we denote the resulting tensor category by $\Vec^\omega[X]$.
    It has the same structure as $\Vec[X]$, with the exception that 
    the associator has components additively extended from
    \[
        \alpha_{V_x,V'_{y},V''_{z}} = 
        \omega(x,y,z) \cdot \alpha^\Vec_{V,V',V''}.
    \]
    It is straightforward to check that a tensor functor
    $\Vec^\omega[X] \to \Vec^{\omega'}[X]$ whose underlying
    linear functor is the identity 
    $\Vec[X] \to \Vec[X]$ is equivalent to the data of 
    a 2-cocycle $\mu \in Z^2(\B X, \IC^\times)$ whose boundary
    is $d\mu = \omega - \omega'$.
    It is in this sense that tensor structures on $\Vec[X]$
    are classified by the cohomology group $\H^3(\B X, \IC^\times)$.
\end{example}

\begin{example}
\label{ex:VectMAsTensorCategory}
    Let $M$ be a manifold. Then the linear category $\Vect(M)$ of
    vector bundles over $M$ admits a tensor structure given by
    the (pointwise) tensor product of vector bundles.
\end{example}

\begin{example}
\label{ex:RepGAsTensorCategory}
    The category $\Rep(G)$ of representations of a group $G$ 
    is a tensor category under the pointwise tensor product of representations.
\end{example}

\begin{defn}
\label{def:dualOfObject}
    Let $X \in \cat{C}$ be an object in a linear category.
    A \emph{left dual} for $X$ is an object $X^\vee$, such that
    there exist evaluation and coevaluation morphisms
    \begin{align*}
        \ev: X^\vee \tensor X \to \ONE && \coev: \ONE \to X \tensor X^\vee,
    \end{align*}
    satisfying the \emph{snake equations}: The composites
    % https://q.uiver.app/?q=WzAsOCxbMCwwLCJYIl0sWzIsMCwiWCBcXHRlbnNvciAoWF5cXHZlZSBcXHRlbnNvciBYKSJdLFsxLDAsIihYIFxcdGVuc29yIFheXFx2ZWUpIFxcdGVuc29yIFgiXSxbMywwLCJYIl0sWzAsMSwiWF5cXHZlZSJdLFsxLDEsIlheXFx2ZWUgXFx0ZW5zb3IgKFggXFx0ZW5zb3IgWF5cXHZlZSkiXSxbMiwxLCIoWF5cXHZlZSBcXHRlbnNvciBYKSBcXHRlbnNvciBYXlxcdmVlIl0sWzMsMSwiWF5cXHZlZSJdLFswLDIsIlxcY29ldiBcXHRlbnNvciBcXGlkX1giXSxbMiwxLCJcXGFscGhhKFgsWF5cXHZlZSxYKSJdLFsxLDMsIlxcaWRfWCBcXHRlbnNvciBcXGV2Il0sWzQsNSwiXFxpZF97WF5cXHZlZX0gXFx0ZW5zb3IgXFxjb2V2Il0sWzUsNiwiXFxhbHBoYShYXlxcdmVlLFgsWF5cXHZlZSleey0xfSJdLFs2LDcsIlxcZXYgXFx0ZW5zb3IgXFxpZF9YIl1d
\[\begin{tikzcd}[column sep=large]
	X & {(X \tensor X^\vee) \tensor X} & {X \tensor (X^\vee \tensor X)} & X \\
	{X^\vee} & {X^\vee \tensor (X \tensor X^\vee)} & {(X^\vee \tensor X) \tensor X^\vee} & {X^\vee}
	\arrow["{\coev \tensor \id_X}", from=1-1, to=1-2]
	\arrow["{\alpha_{X,X^\vee,X}}", from=1-2, to=1-3]
	\arrow["{\id_X \tensor \ev}", from=1-3, to=1-4]
	\arrow["{\id_{X^\vee} \tensor \coev}", from=2-1, to=2-2]
	\arrow["{{\ \alpha^{-1}_{X^\vee,X,X^\vee}}}", from=2-2, to=2-3]
	\arrow["{\ev \tensor \id_X}", from=2-3, to=2-4]
\end{tikzcd}\]
    %\begin{align*}
    %    \left(\id_X \tensor \ev \right) \comp \alpha_{X,X^\vee,X} 
    %    \comp \left(\coev \tensor \id_X\right) &= \id_X \\
    %    \left(\ev \tensor \id_{X^\vee}\right) \comp \alpha_{X^\vee,X,X^\vee} 
    %    \comp \left(\id_{X^\vee} \tensor \coev\right) &= \id_{X^\vee}.
    %\end{align*}
    must be equal to the respective identity morphisms.
    In that case, $X$ is called a \emph{right dual} of $X^\vee$.
    An object that admits both a right and left dual is called 
    \emph{dualisable}.
    A tensor category $\cat{C}$ is \emph{autonomous} if every object 
    $X \in \cat{C}$ is dualisable.
\end{defn}

\begin{defn}
    An \emph{invertible} object in a monoidal category 
    is an object $X$, such that there exists an object $X^{-1}$ with 
    isomorphisms $X^{-1} \tensor X \iso \ONE \iso X \tensor X^{-1}$.
\end{defn}
We call $X^{-1}$ the inverse of $X$. It is unique up to isomorphism.
It is straightforward to check that it may be equipped with duality data
which identify it as both the left and right dual of $X$.
Being invertible is a property, and 
the invertible objects of a tensor category 
form a subset of its dualisable objects.

\begin{example}
    In $\VecInf$, duality data reduces to the usual 
    notion of duality data familiar from linear algebra.
    The dualisable objects in $\VecInf$ 
    are precisely the finite-dimen\-sio\-nal vector spaces,
    making $\Vec \subset \VecInf$ the maximal autonomous subcategory of 
    $\VecInf$.
    A vector space is invertible precisely when it is 1-dimensional.
\end{example}

\begin{example}
    The dual of a finite-dimensional $G$-representation 
    $(V,\rho) \in \Rep(G)$ is 
    provided by its contragredient representation
    $(V^\vee,\rho^\vee:g \mapsto \rho(g^{-1})^\vee)$.
\end{example}

\begin{example}
\label{ex:DualsInVecX}
    Consider the tensor category $\Vec^\omega[X]$ of $X$-graded vector spaces
    with twisted associator (Example~\ref{ex:VecOmegaX}).
    An object of the form $V_x$ admits a left dual precisely 
    if $x$ admits a left inverse $x^{l}$. A dual is then given by 
    $V^\vee_{x^{l}}$.
    Let $X^\sim \subset X$ be the maximal subgroup of $X$, consisting 
    of those elements $x \in X$ that admit a left and right inverse.
    Then $\Vec^\omega[X^\sim] \subset \Vec[X]$ is the maximal autonomous 
    subcategory.
    An object $V_x \in \Vec^\omega[X]$ is invertible when $x \in X^\sim$ 
    is invertible and $V$ is 1-dimensional. All invertible objects 
    of $\Vec^\omega[X]$ are of this form.
\end{example}

\begin{defn}
    A \emph{fusion category} is a semisimple autonomous tensor category,
    with finitely many isomorphism classes of simple objects,
    and simple monoidal unit $\ONE$.
\end{defn}

\begin{example}
    The category $\Rep(G)$ of representations of a group $G$
    is fusion iff it is semisimple and there are only finitely many
    irreducible representations up to isomorphism.
    This is always the case if $G$ is finite.
\end{example}

\begin{example}
\label{ex:VecGIsFusionCategory}
    As explained in Example~\ref{ex:DualsInVecX},
    the semisimple category $\Vec^\omega[X]$ is autonomous iff $X$ is a group.
    Its isomorphism classes of simple objects are in correspondence with 
    elements of $X$. 
    Let $H$ be a finite group and $\omega \in Z^3(\B H, \IC^\times)$
    a 3-cocycle. Then $\Vec^\omega[H]$ is a fusion category.
\end{example}
As explained in Example~\ref{ex:VecOmegaX}, tensor structures
on $\Vec[H]$ which lift the product on $G$ are classified 
up to equivalence by $\H^3(\B H,\IC^\times)$.
The category $\Vec^\omega[H]$ has the special property that
all of its simple objects are invertible. 
A fusion category with this property is called \emph{pointed}. 
Every pointed fusion category is of the form $\Vec^\omega[H]$.

\begin{example}
\label{ex:FibonacciFusionCategory}
    The smallest example of a non-pointed fusion category is 
    given by the \emph{Fibonacci categories} $\categoryname{Fib}^{\pm}$.
    They have two simple objects $\ONE$ and $X$, with tensor product
    given by 
    \[
        X \tensor X = \ONE \dirSum X.
    \]
    There are two non-equivalent solutions to the pentagon 
    equation~\cite{ostrik2003fusion,moore1989classical},
    both featuring the golden ratio $\phi$.
\end{example}

Tambara and Yamagami obtained in~\cite{tambara1998tensor} a complete
classifications of fusion categories with a single non-invertible object.
These are the Tambara-Yamagami categories.
\begin{example}
    \label{ex:TambaraYamagamiFusionCategory}
    The \emph{Tambara-Yamagami category} associated to a finite abelian
    group $A$, a bicharacter $\chi:A \tensor A \to \IC^\times$,
    and a choice of square root $\tau = \pm \sqrt{{|A|}^{-1}}$
    %\left\{\tfrac{1}{\sqrt{|A|}},-\tfrac{1}{\sqrt{|A|}}\right\}
    has set of simple objects $A \cup \{m\}$.
    Let $a,b,c \in A$.
    The tensor product in $\TY(A,\chi,\tau)$ is given by
    \begin{align*}
        a \tensor b = a\!+\!b && 
        a \tensor m = m \tensor a = m &&
        m \tensor m = \DirSum_{a \in A} a
    \end{align*}
    and the non-trivial components of the associator take the values
    \begin{align*}
        \alpha_{a,m,b} & = \chi(a,b) \id_m: m \to m \\
        \alpha_{m,a,m} & = \DirSum_{b,c \in A} \chi(a,b)
        \delta_{b,c} \id_b: \DirSum_{b \in A} b \to
        \DirSum_{c \in A} c                              \\
        \alpha_{m,m,m} & = \DirSum_{b,c \in A}
        \tau \chi(b,c)^{-1} \id_m:
        \DirSum_{b \in A} m \to \DirSum_{c \in A} m.
    \end{align*}
\end{example}

\subsubsection{Braided Tensor Categories and the Drinfel'd Centre}
\label{sec:braidedTensorCatsAndDrinfeldCentre}
Let $X \in \cat{C}$ be an object in a tensor category. 
\begin{defn}
\label{def:halfBraiding}
    A \emph{half-braiding} on $X$ is a natural isomorphism
    \[
        \gamma: X \tensor - \eqto - \tensor X
    \]
    satisfying, for all $Y,Z \in \cat{C}$, the hexagon equation
    \[
        \alpha_{Y,Z,X} \comp \gamma_{Y \tensor Z} \comp \alpha_{X,Y,Z} =
        \left(\id_Y \tensor \gamma_{Z}\right) \comp 
        \alpha_{Y,X,Z} \comp 
        \left(\gamma_{Y} \tensor \id_Z\right).
    \]
\end{defn}
The hexagon equation says that $\gamma$ is multiplicative 
(ie.\ compatible with the tensor product), and has the consequence that 
there is a unique natural morphism 
$(X \tensor Y) \tensor Z \eqto X \tensor (Y \tensor Z)$.

\begin{defn}
\label{def:braidingsAndSymmetricBraidings}
    A \emph{braiding} on a tensor category $\cat{C}$ is a natural isomorphism
    \[
        \beta: - \tensor \cdot \eqto \cdot \tensor -,
    \]
    such that for all $X \in \cat{C}$,
    $\beta_{X, -}: X \tensor - \eqto - \tensor X$ is a half-braiding for $X \in \cat{C}$
    and $\beta_{-, X}: - \tensor X \eqto X \tensor -$
    is a half-braiding for $X$ in $\cat{C}^{\tensor\opp}$,
    the category equipped with the opposite monoidal structure.
    A tensor category equipped with a braiding is called a \emph{braided
    tensor category}.

    A braiding is \emph{symmetric} if 
    $\beta_{b,a} \comp \beta_{a,b} = \id_{a \tensor b}$ for all 
    $a,b \in \cat{C}$.
    A tensor category equipped with a symmetric braiding is called 
    a \emph{symmetric tensor category}.
\end{defn}
A braided functor of braided categories is a tensor functor whose structure
natural transformations are compatible with the braiding. Ie.\ it is a property for 
a tensor functor to be braided.

\begin{example}
    The usual swap isomorphism of vector spaces 
    $V \tensor W \to W \tensor V$ equips $\Vec$ with a braiding.
    This braiding pulls back to a braiding on 
    $\Rep(G)$ via the forgetful functor $\Rep(G) \to \Vec$.
\end{example}

\begin{example}
    The category of (twisted) $G$-graded vector spaces cannot admit a braiding
    if $G$ is nonabelian, since the existence of a braiding on a category 
    $\cat{C}$ implies $X \tensor Y \iso Y \tensor X$ for all $X,Y \in \cat{C}$.
    Now assume $G=H$ is an abelian group.
    Braided structures on $\Vec[H]$ that lift the group structure on $H$ 
    turn out to be classified by their associated quadratic form 
    \begin{align*}
        q:H &\to \IC^\times \\
        h &\mapsto \beta_{\IC_h,\IC_h}.
    \end{align*}
    This fact is proved in~\cite{eilenberg1954groups}.
    We discuss this for the case of Lie groups
    in~\cite{weis2022centre}.
\end{example}

\begin{example}
\label{ex:fourBraidingsOnVecZ2}
    There are 4 non-equivalent braided tensor structures on the linear 
    category $\Vec[\IZ/2]$ that lift the group structure on $\IZ/2$.
    We write $\Vec[\IZ/2]$ to denote the standard braided tensor 
    structure, given by the trivial associator and trivial braiding 
    inherited from $\Vec$.
    Perhaps better-known is the category $\sVec$ of super vector spaces:
    it also has trivial associator, but the self-braiding of the 
    simple object $\IC_{-1}$ is $\beta_{-1,-1}=-1$.

    The remaining two braided tensor structures are the semion $\Semi$ and 
    anti-semion $\bSemi$. Both of them have twisted associator,
    with $\alpha_{-1,-1,-1}=-1$.
    In $\Semi$, we have $\beta_{-1,-1}=\ii$, and in $\bSemi$, 
    $\beta_{-1,-1}=-\ii$. 
\end{example}

\begin{example}
\label{ex:RepkLG}
    Let $G$ be a compact simple connected Lie group, and 
    $k \in Z^4(\B G, \IZ)$ a $\IZ$-valued 4-cocycle.
    The \emph{loop group} $LG=\Cinf(S^1,G)$ has an associated 
    central extension
    \[
        \IC^\times \to {\widetilde{LG}}^k \to LG.
    \]
    The representations of this central extension are extensively
    studied in~\cite{pressley1985loop}.
    There is a particularly interesting class of representations,
    cut out by the condition of having 
    \emph{positive energy}.
    The subcategory $\Rep^k(LG)$ of such representations 
    is a braided tensor category.
    These categories play an important role in the study of 
    Conformal and Topological Field Theories, in particular
    that of Chern-Simons Theory. 
    When $G$ is simply-connected, the category $\Rep^k(LG)$
    is braided equivalent to the Quantum group category 
    $\cat{C}(\mathfrak{g},k)$ associated to its Lie algebra 
    $\mathfrak{g}$ and the level $k$.
    See also Section~\ref{sec:interpolatedQGroups}.
\end{example}

The notion of the centre of a monoid categorifies to the 
setting of tensor categories.
\begin{defn}
    The \emph{Drinfel'd centre} $\DZ \cat{C}$ of a tensor category $\cat{C}$
    has objects pairs $(X,\gamma)$ where $X \in \cat{C}$ and
    $\gamma$ is a half-braiding for $X$.
\end{defn}
The centre comes equipped with a tensor functor 
$\DZ \cat{C} \to \cat{C}$ which forgets the half-braiding.
It also carries a braiding $\beta$, with components
\[
    \beta_{(X,\gamma),(Y,\gamma')} = \gamma_{Y}.
\]

\begin{example}
    The Drinfel'd centres $\DZ \Vec^\omega[G]$ of pointed fusion categories
    are computed explicitly in~\cite{coste2000finite,gruen2018computing}.
    Let $G=H$ be abelian and $\omega$ trivial.
    Then the simple objects of $\DZ \Vec[H]$ correspond to pairs $(h,\chi)$
    of an element $h \in H$ and a character $\chi \in \widehat{H}$,
    The simple object corresponding to $(h,\chi)$ is 
    $\IC_h \in \Vec[H]$, equipped with the half-braiding
    $\gamma_{\IC_{h'}}=\chi(h')$.
    The resulting braiding on the centre is again given by 
    pairing characters with the group elements.
\end{example}

In~\cite{weis2022centre}, we computed the analogue of 
this centre in the context of smooth 2-groups. We use the
results of this computation to construct \emph{interpolated 
fusion categories}
in Section~\ref{sec:constructInterpolatedFusionCats}.

\subsubsection{Algebras and Modules}
\label{sec:algebrasAndModules}
We recall how one may define algebras and modules internal to (braided) tensor categories.

\begin{defn}
    An \emph{algebra} in a tensor category $(\cat{C},\ONE,\tensor)$
    is an object $A \in \cat{C}$, together with a unit 
    $\eta: \ONE \to A$ and multiplication $\mu:A \tensor A \to A$,
    satisfying unitality
    \begin{align*}
        \mu \comp (\eta \tensor \id_A) = l_A &&
        \mu \comp (\id_A \tensor \eta) = r_A
    \end{align*}
    and associativity
    \begin{equation*}
        \mu \comp (\mu \tensor \id_A) =
        \mu \comp (\id_A \tensor \mu).
    \end{equation*}
\end{defn}

\begin{defn}
    A \emph{commutative algebra} in a braided tensor category 
    $(\cat{C}, \ONE, \tensor, \beta)$
    is an algebra $(A,\mu) \in (\cat{C},\ONE,\tensor)$,
    such that 
    \[
        \mu \comp \beta_{A,A} = \mu: A \tensor A \to A.
    \]
\end{defn}

\begin{defn}
    A (right) \emph{module} over an algebra object $A \in \cat{C}$ 
    is an object $M \in \cat{C}$, together with an action
    $a:M \tensor A \to M$, compatible with the multiplication:
    The two induced morphisms
    $(M \tensor A) \tensor A \to M$ given by
    \begin{equation*}
        a \comp (\id_A \tensor \mu_A) = a \comp (a \tensor \id_A) 
        \comp \alpha_{\cat{C}}
    \end{equation*}
    agree.
\end{defn}
Right modules and morphisms between them (required to commute with the action maps) 
assemble into a category $\Mod_A(\cat{C})$.
Left modules are defined exactly analogously.

When $A$ is commutative, a right module structure on $M \in \cat{C}$
induces a left-module structure by precomposing with the braiding.
This can be done in different ways. In particular, 
$a: M \tensor A \to A$ may
be precomposed with either of 
$\beta_{A,M},\beta_{M,A}^{-1}:A \tensor M \to M \tensor A$.
\begin{defn}
    \label{def:localModule}
    A right module $M \in \Mod_A(\cat{C})$ over a 
    commutative algebra $A$ is \emph{local} if the two 
    induced left-module structures agree.
\end{defn}
In a symmetric tensor category, $\beta_{A,M}=\beta_{M,A}^{-1}$, 
and all modules over commutative algebras are automatically local.

\subsubsection{Symmetric Monoidal Bicategories}
In this section, we briefly recall bicategories, equipped with varying amounts 
of monoidal structure.
The history of the definition of symmetric monoidal bicategories is long: 
A detailed review is given in~\cite[Sec 2.1]{schommer2009classification}.
Given the work of~\cite{gordon1995coherence} on tricategories, it is possible to 
define a monoidal bicategory in a single line.
\begin{definition}
\label{defn:monoidalBicategoryAsTricategory}
    A \emph{monoidal bicategory} is a tricategory with a single object.
\end{definition}
Of course, this definition is only helpful if one is already
familiar with tricategories. 
%Original definitions in context of Gray monoids: \cite{street1997monoidal}
%Coherence theorem (tricategories): \cite{gordon1995coherence}
%Reference explaining the whole monoidal 2-category thing and 
%studying (pseudo)monoids there: \cite{mccrudden2000balanced}.
%Gurski gives more coherence Theorems: 
%\cite{gurski2013coherence} for tricategories,
%\cite{gurski2011loop} for monoidal bicategories and 
%braided monoidal bicategories,
%\cite{gurski2013infinite} for symmetric monoidal bicategories.
%In \cite{stay2016compact}, full definition of symmetric monoidal
%bicategories.
We use Definition~\ref{defn:monoidalBicategoryAsTricategory}
as an excuse to reduce the amount of details we spell out below.
We try to be precise about the data required, but will only give an 
idea of the coherence laws this data is required to satisfy.
The interested reader is invited to consult the 
excellent~\cite{leinster1998basic} for precise definitions 
of bicategories and morphisms between them,
and~\cite{stay2016compact} for a detailed account of 
the definitions of monoidal, braided monoidal, sylleptic monoidal and 
symmetric monoidal bicategories.

\begin{defn}
    A \emph{bicategory} $\twocat{B}$ is given by the following data:
    \begin{itemize}
        \item a collection $\ob{\twocat{B}}$ of objects
        \item a category $\Hom_{\twocat{B}}(a,b)$ of morphisms for each pair 
        $a,b \in \ob{\twocat{B}}$
        \item composition functors 
        $\comp_{a,b,c}:\Hom_{\twocat{B}}(b,c) \times \Hom_{\twocat{B}}(a,b) \to \Hom_{\twocat{B}}(a,c)$
        \item identities $\id_a \in \Hom_{\twocat{B}}(a,a)$.
        \item associator and unitor natural transformations
    \end{itemize}
    The associators identify the 
    two different ways of composing three morphisms,
    and unitors trivialise composing with identity morphism.
    This data is subject to generalisations of the pentagon
    and triangle equations recalled in Section~\ref{sec:tensorCategories}.
\end{defn}
The objects of $\Hom_{\twocat{B}}(a,b)$ in the definition above are called 
1-morphisms (or \emph{1-cells}),
and the morphisms between them in $\Hom_{\twocat{B}}(a,b)$ 
are the 2-morphisms (or \emph{2-cells}) of $\twocat{B}$.

A bicategory is called a \emph{2-category} if its associator and 
unitor natural transformations are trivial. This means composition 
is associative and unital on the nose.

\begin{example}
\label{ex:OneCatasBicat}
    Every 1-category $\cat{C}$ defines a bicategory whose 
    objects and 1-morphisms are those of $\cat{C}$, and 
    all of whose 2-morphisms, associators and unitors are trivial.
\end{example}

\begin{example}
\label{ex:VCatAsBicat}
    Let $\cat{V}$ be a monoidal category. Then
    $\cat{V}$-enriched categories assemble into a bicategory
    $\cat{V}\Cat$, whose 1-cells are $\cat{V}$-functors
    and 2-cells are $\cat{V}$-natural transformations.
\end{example}

\begin{defn}
\label{def:functorOfBicats}
    A \emph{functor} of bicategories $\twocat{A} \to \twocat{B}$ is given by
    \begin{itemize}
        \item a map on objects $F:\ob{\twocat{A}} \to \ob{\twocat{B}}$
        \item functors on morphism categories
        $F_{a,b}:\Hom_{\twocat{A}}(a,b) \to \Hom_{\twocat{B}}(F(a),F(b))$
        \item compositor and identitor natural isomorphisms.
    \end{itemize}
    The compositor natural transformations implement compatibility of 
    the functor with the composition maps in $\twocat{A}$ and $\twocat{B}$.
    The identitors identify the image of identity 1-morphisms in $\twocat{A}$ with 
    identity 1-morphisms in $\twocat{B}$.
    The coherences this data has to satisfy are compatibility conditions 
    of the compositors and identitors with the 
    associator 2-cells and unitor 2-cells of $\twocat{A}$ and 
    $\twocat{B}$.
\end{defn}

\begin{defn}
    \label{def:transformationOfBicatFunctors}
    A \emph{transformation} of functors of bicategories
    $t:F \to G$ between $F,G:\twocat{A} \to \twocat{B}$
    consists of the data of
    \begin{itemize}
        \item a 1-cell $t(A):F(A) \to G(A)$ for each 
        object $A \in \twocat{A}$
        \item natural transformations implementing 
        lax naturality of $t(A)$
    \end{itemize}
    The natural transformations are 2-cells
    $t(f):G(f) \comp t(A) \to t(B) \comp F(f)$ for each 
    1-cell $f:A \to A'$ in $\twocat{A}$.
    The coherences again consist of compatibility conditions with 
    compositors and unitors.
\end{defn}

\begin{defn}
    \label{def:modificationOfBicatFunctors}
    A \emph{modification} of two transformations $m:s \to t$
    consists of 
    \begin{itemize}
        \item a 2-cell $m(A):s(A) \to t(A)$ for each $A \in \twocat{A}$,
    \end{itemize}
    subject to compatibility with the 2-cells
    $s(f),t(f)$ for all 1-cells $f$ in $\twocat{A}$.
\end{defn}

We now turn to the definitions of monoidal structures on bicategories.
As is to be expected, the
coherence conditions involved grow bigger and more numerous.
We will not attempt to describe them --- 
an excellent geometric exposition
is provided in~\cite[Ch 4]{stay2016compact}.
For a complete definition of monoidal bicategories, 
braided monoidal bicategories,
sylleptic monoidal bicategories and symmetric monoidal bicategories
all in one, see~\cite[Def 2.3]{schommer2009classification}.
\begin{defn}
    A \emph{monoidal bicategory} is given by a bicategory
    $\twocat{B}$ equipped with
    \begin{enumerate}
        \item a unit object $1 \in \twocat{B}$
        \item a product functor\footnote{
                  The product $\twocat{A} \times \twocat{B}$ of 
                  two bicategories is given by
                  by taking products of all data.
              }
              \[
                  \boxtimes:\twocat{B} \times \twocat{B} \to \twocat{B}
              \]
        \item invertible associator and unitor transformations
              \begin{align*}
                  (- \boxtimes -) \boxtimes - & \eqto - \boxtimes (- \boxtimes -) \\
                  - \boxtimes 1               & \eqto \Id_\twocat{B}              \\
                  1 \boxtimes -               & \eqto \Id_\twocat{B}
              \end{align*}
        \item
              invertible modifications filling the pentagon and
              triangle coherences (these take the form of equations  
              in Definition~\ref{def:tensorCategory}), 
              as well as an invertible modification
              implementing compatibility of the left and right unitor
    \end{enumerate}
\end{defn}

The coherence theorems in~\cite{gordon1995coherence} allow 
us to work with less structure. Every monoidal bicategory
is equivalent to a \emph{monoidal 2-category}. That is a 
monoidal bicategory whose underlying bicategory is a 2-category.
We follow~\cite{mccrudden2000balanced} and work in this 
generality.

\begin{defn}
    A \emph{braided monoidal bicategory} is a monoidal bicategory 
    equipped with
    \begin{enumerate}
        \item an invertible braiding transformation 
        $b: \boxtimes \eqto \boxtimes \comp \swap$,
        where
        $\swap: \twocat{B} \times \twocat{B} \to 
        \twocat{B} \times \twocat{B}$
        denotes the usual swap functor
        \item invertible modifications filling the two 
        hexagon coherences (which are equations 
        in Definition~\ref{def:braidingsAndSymmetricBraidings}).
    \end{enumerate}
\end{defn}

\begin{defn}
    A \emph{sylleptic monoidal bicategory} is a braided monoidal
    bicategory equipped with invertible modifications
    $\sigma_{x,y}: \Id_{x \tensor y} \to b_{y,x} \tensor b_{x,y}$
    for all $x,y \in \twocat{B}$.
\end{defn}

In a sylleptic monoidal bicategory, there are two modifications 
$b \eqto b \comp b \comp b$ corresponding to 
the two bracketings of the right hand side. (We neglect the 
compositor structure and identify both bracketings with 
$b \comp b \comp b$.)
\begin{defn}
    A \emph{symmetric monoidal bicategory} is a sylleptic monoidal 
    bicategory in which the two canonical modificiations 
    $b \eqto b \comp b \comp b$ agree.
\end{defn}

\subsubsection{Monoids in Monoidal Bicategories}
\label{sec:monoidsInMonoidalBicategories}
The notion of a monoid internal to a monoidal bicategory was introduced
in~\cite{street1997monoidal} in the context 
Gray monoids. The slightly weaker framework of monoidal 2-categories
is chosen in~\cite{mccrudden2000balanced}.
Definitions in the most general context of monoidal bicategories
may be extracted from the definitions of 2-groups and braided 2-groups 
in~\cite[Sec 3]{schommer2011central} by removing the invertibility condition.
All three of these contexts are equivalent by the coherence 
theorems of~\cite{gordon1995coherence}.

\begin{notation}
    We denote by \emph{monoid} the notion
    more commonly referred to by the name 
    \emph{pseudomonoid} in the literature.
\end{notation}

The bicategory with one object, one 1-morphism and one 2-morphism 
is a symmetric monoidal bicategory
in a canonical way. By forgetting structure, it is also canonically
a sylleptic, braided or simply a monoidal bicategory.
Denote this bicategory by $\oneObjBicat$.
As explained in~\cite[Sec 2-4]{mccrudden2000balanced},
one may define 
\begin{itemize}
    \item monoids in a monoidal bicategory 
    \item braided monoids in a braided monoidal bicategory 
    \item symmetric monoids in a sylleptic monoidal bicategory
\end{itemize}
as monoidal/braided monoidal/sylleptic monoidal morphisms 
\[
    \oneObjBicat \to \twocat{B}.
\]
Here, $\twocat{B}$ is monoidal/braided monoidal/sylleptic monoidal,
and $\oneObjBicat$ is equipped with its canonical 
monoidal/braided monoidal/sylleptic monoidal structure. 
We flesh out the resulting structure below.
\begin{defn}
\label{def:monoidInBicategory}
    A \emph{monoid} in a monoidal bicategory $(\twocat{B},\boxtimes,1)$ is an object 
    $\cat{C} \in \twocat{B}$, equipped with
    \begin{enumerate}
        \item a unit $\ONE: 1 \to \cat{C}$,
        \item a tensor morphism 
        $\tensor:\cat{C} \boxtimes \cat{C} \to \cat{C}$,
        \item invertible associator and unitor 2-cells
        \begin{align*}
            \alpha: (- \tensor -) \tensor - &\eqto
            - \tensor (- \tensor -) \\
            l: (\ONE \tensor -) &\eqto \Id_{\cat{C}} \\
            r: (- \tensor \ONE) &\eqto \Id_{\cat{C}},
        \end{align*}
    \end{enumerate}
    satisfying the pentagon equation and triangle equation 
    from Definition~\ref{def:tensorCategory} (phrased internal
    to $\twocat{B}$).
\end{defn}

\begin{defn}
\label{def:braidedMonoidInBicategory}
    A \emph{braided monoid} in a braided monoidal bicategory 
    $(\twocat{B},\boxtimes,1,b)$ is a monoid 
    $(\cat{C},\tensor,\ONE)$ in $\twocat{B}$, 
    equipped with an invertible braiding 2-cell
    \[
        \beta: \tensor \eqto \tensor \comp b.
    \]
    satisfying the hexagon equation from 
    Definition~\ref{def:braidingsAndSymmetricBraidings} (phrased internal
    to $\twocat{B}$).
\end{defn}

The 2-cell $\beta$ above induces a 2-cell
% https://q.uiver.app/?q=WzAsMyxbMCwwLCJcXGJveHRpbWVzIl0sWzEsMCwiXFxib3h0aW1lcyBcXGNvbXAgYiJdLFsyLDAsIlxcYm94dGltZXMgXFxjb21wIGIgXFxjb21wIGIiXSxbMCwxLCJcXGJldGEiXSxbMSwyLCJcXGJldGEgXFxjb21wIGIiXV0=
\[\begin{tikzcd}
	\boxtimes & {\boxtimes \comp b} & {\boxtimes \comp b \comp b}.
	\arrow["\beta", from=1-1, to=1-2]
	\arrow["{\beta \comp b}", from=1-2, to=1-3]
\end{tikzcd}\]
If $\twocat{B}$ is equipped with a sylleptic monoidal structure,
the syllepsis 2-cell $\sigma: \Id \to b \comp b$
induces another 2-morphism 
$\boxtimes \comp \sigma: \boxtimes \to \boxtimes \comp b \comp b$.
\begin{defn}
\label{def:symmetricMonoidInBicategory}
    A \emph{symmetric monoid} in a sylleptic monoidal bicategory
    $(\twocat{B},\boxtimes,1,b,\sigma)$ is a braided monoid
    $(\cat{C},\tensor,\ONE,\beta)$ in $\twocat{B}$, such that 
    the two canonical 2-cells 
    $\boxtimes \to \boxtimes \comp b \comp b$ agree.
\end{defn}

\begin{example}
    The archetypical example is provided by the bicategory $\cat{V}$-$\Cat$ of 
    $\cat{V}$-enriched 1-categories. (Here, $\cat{V}$ is some symmetric monoidal
    1-category.)
    $\cat{V}$-$\Cat$ is symmetric monoidal under Kelly's product of enriched 
    categories~\cite{kelly1982basic}.
    Monoids therein are $\cat{V}$-enriched monoidal categories,
    braided monoids are $\cat{V}$-enriched braided monoidal categories,
    and symmetric monoids are $\cat{V}$-enriched symmetric monoidal categories.
    In particular, for $\cat{V}=\VecInf$, we recover Definition~\ref{def:tensorCategory} 
    of a tensor category.
\end{example}

In analogy with the lower dimensional situation, we may ask for modules over monoid objects.
Let $(\cat{C},\tensor,\ONE)$ be a monoid in a monoidal bicategory $(\twocat{B},\boxtimes,1)$.
\begin{defn}
    A (right) \emph{module} over $\cat{C}$ is an object $\cat{M} \in \twocat{B}$, 
    equipped with 
    \begin{itemize}
        \item an action map 
        \[
            \odot: \cat{M} \boxtimes \cat{C} \to \cat{M}
        \]
        \item an invertible compatibility 2-cell 
        \[
            s: (- \odot -) \odot - \eqto - \odot (- \tensor -),
        \]
    \end{itemize}
    satisfying a coherence relating $s$ and the associator 2-cell of $\cat{C}$,
    as well as a unitality condition.
\end{defn}

\begin{example}
    In the monoidal bicategory $\cat{V}$-$\Cat$, modules over monoids recover the
    usual notion of module categories over monoidal categories. 
    See~\cite{ostrik2003module} for a discussion of the case $\cat{V}=\VecInf$.
\end{example}

\subsection{Sheaves and Stacks}
\label{sec:sheavesAndStacks}
Associated to a subset of a manifold\footnote{We will always talk about 
$\Cinf$-smooth paracompact manifolds.} $U \subset M$ is its ring of functions $\Cinf(U)$.
There is a restriction map $\Cinf(U) \to \Cinf(V)$ associated to any inclusion $V \into U$.
Given a cover $\coprod U_i \to U$ and functions $f_i \in \Cinf(U_i)$ which are compatible on 
overlaps $U_i \cap U_j$, there
is a unique function $f \in \Cinf(U)$ that restricts to the $f_i$.
In other words, the assignment
\[
    U \mapsto \Cinf(U)
\]
satisfies the \emph{descent} or \emph{gluing} condition.
Such an assignment is called a \emph{sheaf}. 

Stepping up the categorical ladder, we may instead consider the assignment 
\[
    U \mapsto \Vect(U),
\]
sending a subset $U \subset M$ to the category of vector bundles over it.
There are restriction functors $\Vect(U) \to \Vect(V)$ associated to $V \into U$.
This assignment also satisfies a descent condition over a cover $\coprod U_i \to U$:
given vector bundles $P_i \in \Vect(U_i)$ and isomorphisms 
$\phi_{ij}:P_i\restriction_{U_{ij}} \eqto P_j\restriction_{U_{ij}}$
which are compatible on triple overlaps $U_{ijk}$,
there is a vector bundle $P \in \Vect(U)$, unique up to unique isomorphism, that 
glues this data.
This makes $\Vect$ an example of a \emph{2-sheaf} or \emph{stack}.

Below, we briefly recall some notions from the theory of sheaves and stacks relevant 
to this document.
Standard references for this topic are~\cite{bourbaki2006theorie,
tennison1975sheaf,kashiwara2005categories,maclane2012sheaves,moerdijk2002introduction,johnstone2002sketches,giraud1966cohomologie}.
We recommend~\cite{vistoli2004notes,metzler2003topological} as thorough introductions.

% Currently makes no sense as we don't introduce differentiable stacks here...
%Relevant to the theory of differentiable stacks, in particular orbifolds, are also 
%~\cite{carchedi2010sheaf,carchedi2011categorical,carchedi2013etale},
%as well as~\cite{moerdijk1997orbifolds,moerdijk2002orbifolds,moerdijk2006classifying}.

\subsubsection{Sheaves}
Recall that a \emph{site} is a category $\cat{C}$ equipped with a \emph{Grothendieck topology}: 
for each object $U \in \cat{C}$, a collection of covers $\{f_i:U_i \to U\}_{i \in I}$.
We assume that all fibre products $U_{ij \cdots k} \define U_i \times_U U_j \times_U \cdots U_k$ exist for each cover. We use suggestive indexing for the projections 
$p_{ij\cdots k}: U_{i'j'\cdots k'} \to U_{ij\cdots k}$ when the domain is clear from context.
Subcategories and overcategories of a site inherit Grothendieck topologies obtained 
by restricting along the forgetful functors. 
A full subcategory $\cat{D}$ of a site $\cat{C}$ is \emph{dense} if each object in $\cat{C}$ admits a 
cover by objects in $\cat{D}$.
Recall that we use $\VecInf$ to denote the category of vector spaces over $\IC$ (possibly infinite-dimensional),
and $\Vec \subset \VecInf$ to denote the full subcategory of finite-dimensional vector spaces.

\begin{example}
    We equip the category $\Man$ of smooth, paracompact manifolds with the structure of a site by specifying covers to be surjective submersions.
    It has a dense subsite $\CartSp$, the full subcategory on Cartesian spaces $\IR^n$.
\end{example}

\begin{defn}
    \label{def:sheaf}
A \emph{sheaf} over a site $\cat{C}$ (valued in vector spaces) is a functor  
$\sh{F}: \cat{C}^\opp \to \VecInf$ that satisfies \emph{descent}: 
For any cover $\mscr{U} = {\left\{U_i \to U\right\}}_{i \in I}$, the map
\begin{equation*}
    \sh{F}(U) \to \sh{F}(\mscr{U}) \define \lim\left( \prod_i \sh{F}(U_i) \rightrightarrows \prod_{i,j} \sh{F}(U_{ij}) \right)
\end{equation*}
is an isomorphism.
We denote the category of sheaves by $\Shv(\cat{C}) \subset \Fun(\cat{C}^\opp,\VecInf)$.
\end{defn}

An element $s \in \sh{F}(U)$ is called a \emph{section of $\sh{F}$ over $U$}. Given a cover 
$\mscr{U} = \{f_i:U_i \to U\}$, we can restrict $s$ to the cover via the 
\emph{restriction maps} $f_i^\ast \define \sh{F}(f_i)$. The resulting sections 
$f_i^\ast s \in \sh{F}(U_i)$ will agree upon further restriction to the overlaps $U_{ij}$. 
The descent condition requires that given sections $s_i \in \sh{F}(U_i)$
which restrict to equal sections in $\sh{F}(U_{ij})$, there is a \emph{unique} section over 
$U$ which restricts to the $s_i$.
We refer to a compatible choice of sections over a cover as a \emph{descent datum}. 
The descent condition can now be stated succinctly: 
\emph{Descent data for a sheaf can be glued uniquely.}

\begin{lemma}[Comparison Lemma]
\label{lem:comparisonLemma}
    The restriction functor $\Shv(\cat{C}) \to \Shv(\cat{D})$ induced by the inclusion of 
    a dense subsite $\cat{D} \into \cat{C}$ is an equivalence of categories.
\end{lemma}

\begin{example}
    A sheaf over $\Man$ is completely determined by its restriction to $\CartSp$.
\end{example}

The category $\Shv(\cat{C})$ directly inherits the additive structure of $\VecInf$.
The inclusion of the full subcategory $\iota: \Shv(\cat{C}) \into \Fun(\cat{C}^\opp,\VecInf)$ admits a left adjoint, the \emph{sheafification functor}
\begin{equation*}
    {(-)}^\shff: \Fun(\cat{C}^\opp, \VecInf) \to \Shv(\cat{C}).
\end{equation*}
The counit of this adjunction is a natural isomorphism ${(\iota-)}^\shff \isoto \Id$, 
which allows us to transport the monoidal structure on $\VecInf$ across to $\Shv(\cat{C})$.
For $\sh{F},\sh{G} \in \Shv(\cat{C})$, 
\begin{equation*}
    \sh{F} \tensor \sh{G} \define {(\iota\sh{F} \tensor \iota\sh{G})}^\shff.
\end{equation*}
The functors above are additive and so $\Shv(\cat{C})$ receives the structure of a symmetric tensor category.

Sheaves on a topological space $X$ are sheaves over the site $\Op(X)$ of open sets of $X$. 
By abuse of notation, we denote the associated sheaf category by $\Shv(X)$.
A map $f:X \to Y$ of topological spaces induces an adjoint pair of tensor functors
\begin{equation*}
    f^{-1}: \Shv(Y) \rightleftarrows \Shv(X) :f_\ast.
\end{equation*}
The \emph{pushforward} is given by
\begin{align*}
    f_\ast\sh{F}(U) \define \sh{F}(f^{-1}(U)),
\end{align*}
and the \emph{pullback} $f^{-1}\sh{F}$ can be computed as the sheafification of the functor
\begin{equation*}
    \mathrm{pre}{f^{-1}\sh{F}}: U \mapsto \colim_{V \supset f(U)} \sh{F}(V).
\end{equation*}

The \emph{stalk} of a sheaf $\sh{F} \in \Shv(M)$ at a point $x: \point \to M$ is the pullback 
$x^{-1}\sh{F} \in \Shv(\point) = \VecInf$.
The \emph{support} of $\sh{F}$ is the subset of points of $M$ with nonzero stalk.

A \emph{ringed site} is a site together with an algebra object in its category of sheaves, the \emph{structure ring} or \emph{structure sheaf} $\shO$.
When working with a ringed site $(\cat{C}, \sh{O})$, we will use $\Shv_\shO(\cat{C})$ to refer to the category of $\sh{O}$-modules.

Let $M$ be a manifold. The sheaf of smooth functions $\CinfOf{M}:{\Op(M)}^\opp \to \VecInf$ equips $\Op(M)$ 
with the structure of a ringed site.
A smooth map $f: M \to N$ of manifolds induces a morphism of structure sheaves $f^\shff:f^{-1}{\CinfOf{N}} \to \CinfOf{M}$, 
and the adjoint functor pair can be upgraded to an adjunction of $\Cinf$-modules.
The pushforward of sheaves automatically descends to a functor 
of module categories 
$f_\ast: \Shv_{\CinfOf{M}}(M) \to \Shv_{\CinfOf{N}}(N)$ 
(and we indeed denote it by the same symbol).
Its left adjoint is given by
\begin{equation*}
    f^\ast \sh{F} = f^{-1}\sh{F} \tensor_{{f^{-1} \CinfOf{N}}} {\CinfOf{M}}.
\end{equation*}

\begin{defn}
\label{def:vectorBundle}
    A \emph{vector bundle} on $M$ is a $\Cinf$-module in $\sh{F} \in \Shv_{\CinfOf{M}}(M)$, for which there exists an open cover $\{f_i: U_i \to M\}$ over which 
    $\sh{F}$ is free of finite rank: for all $i$,
    \begin{equation*}
        f_i^\ast\sh{F} \iso {(\CinfOf{U_i})}^{\dirSum r_i}
    \end{equation*}
    for some $r_i \in \IN$.
\end{defn}

\begin{defn}
\label{def:skyscraperSheaf}
    A \emph{skyscraper sheaf} on a manifold $M$ is a $\Cinf$-module with support a 
    finite subset of $M$, such that pullbacks over any point are finite-dimensional.
\end{defn}
The \emph{category of skyscraper sheaves} on $M$ will play a central role in the definition of 
manifusion categories (Definition~\ref{def:orbifusion}).
The category of skyscraper sheaves on $M$ is equivalent to a direct sum of copies of $\Vec$,
indexed by the underlying set of $M$.
It is in this sense that we view skyscraper sheaves on $M$ as \emph{vector spaces graded by $M$}.
Using the language of stacks, we will also keep track of the smooth structure of $M$.

\subsubsection{Stacks}
\label{sec:stacks}

Let $\cat{C}$ be a 1-category, and $\CAT$ the 2-category of categories. 
We will write $\cat{C}^\opp \to \CAT$ to denote the maximally weak
notion of functor between them: a morphism of underlying bicategories.

\begin{definition}
    A \emph{stack} over a site $\cat{C}$ is a functor $\st{F}: \cat{C}^\opp \to \CAT$ satisfying \emph{descent}:
    The natural functor
    \begin{equation*}
        \st{F}(U) \to \st{F}(\mscr{U}) 
        \define
        \lim \left(\prod_i \st{F}(U_i)
        \substack{\to \\[-1em] \to} 
        \prod_{i,j} \st{F}(U_{ij}) 
        \substack{\to \\[-1em] \to \\[-1em] \to}
        \prod_{i,j,k} \st{F}(U_{ijk}) \right)
    \end{equation*}
    is an equivalence for each cover $\mscr{U}=\{U_i \to U\}_{i \in I}$ of $U$.
\end{definition}

The 2-limit $\st{F}(\mscr{U})$ associated to a cover $\mscr{U}=\{U_i \to U\}$ 
can be computed explicitly as the \emph{category of descent data}.
It has 
\begin{itemize}
    \item objects: collections of sections $\{X_i \in \sh{F}(U_i)\}$, 
    equipped with isomorphisms on overlaps 
    $\{\phi_{ij}: p_i^\ast X_i \eqto p_j^\ast X_j \in \sh{F}(U_{ij})\}$. 
    The isomorphisms are required to be compatible on triple intersections:
        \begin{equation*}
            p_{ij}^\ast\phi_{ij} \comp p_{jk}^\ast\phi_{jk} = p_{ik}^\ast\phi_{ik}
        \end{equation*}
    in $\sh{F}(U_{ijk})$ for all triples $i,j,k$. 
    \item morphisms $\{\{X_i\},\{\phi_{ij}\}\} \to \{\{X_i^\prime\},\{\phi_{ij}^\prime\}\}$: collections $\{f_i: X_i \to X_i^\prime\}$, satisfying
        \begin{equation*}
            \phi_{ij}' \comp p_i^\ast f_i = p_j^\ast f_j \comp \phi_{ij}
        \end{equation*}
        in $\sh{F}(U_{ij})$ for all pairs $i,j$.
\end{itemize}

\begin{rmk}
    The notion of descent datum is visibly a categorification of that for shea\-ves:
    Given a cover $\mscr{U}=\{U_i\}$ of $U$, and sections $\{X_i \in \sh{F}(U_i)\}$ of a stack,
    we can no longer ask them to restrict to \emph{equal} sections over the overlaps $U_{ij}$.
    Rather, a descent datum now further contains isomorphisms of these restricted 
    sections in $\sh{F}(U_{ij})$.
    These isomorphisms are then required to satisfy an equation over triple overlaps.
\end{rmk}

The bicategory $\St/\cat{C}$ of stacks over $\cat{C}$ is the full subcategory of $\Fun(\cat{C}^\opp, \CAT)$ on those functors that satisfy descent.
The data of a stack is a category $\sh{F}(U)$ for each $U \in \cat{C}$, 
a \emph{pullback functor} $f^\ast \define \sh{F}(f)$ for each morphism $f \in \cat{C}$,
as well as invertible natural transformations $\eta_{f,g}: f^\ast \comp g^\ast \to {(g \comp f)}^\ast$, $\epsilon_U: \id_U^\ast \to \Id_{\sh{F}(U)}$ 
which we leave implicit. This data is required to satisfy appropriate coherence conditions.

\begin{prop}\cite{giraud1966cohomologie} 
    The inclusion $\St/\cat{C} \into \Fun(\cat{C}^\opp,\CAT)$ admits a left 2-adjoint, the \emph{stackification functor}
    \begin{equation*}
        {(-)}^\shff:\St/\cat{C} \to \Fun(\cat{C}^\opp,\CAT).
    \end{equation*}
\end{prop}
The universal property of the adjoint implies there is a morphism $\st{F} \to {\st{F}}^\shff$.
\begin{lemma}\cite[\href{https://stacks.math.columbia.edu/tag/02ZN}{Tag 02ZN}]{stacks-project}
\label{lem:stackificationLocallyGlued}
    For any section of the stackification $X \in {\st{F}}^\shff(U)$, there exists a cover
    $\mscr{U}=\{f_i:U_i \to U\}$ such that all $f_i^\ast X \in {\st{F}}^\shff(U_i)$ are in the essential image of
    $\st{F}(U_i) \to {\st{F}}^\shff(U_i)$.
\end{lemma}
In other words, each section of $\st{F}^\shff$
still locally decomposes into sections of $\st{F}$.

\begin{example}
\label{ex:BG}
    The \emph{classifying stack} $\B G = [\point / G]$ of principal $G$-bundles is a stack on $\Man$. 
    It assigns to a manifold $U$ the category of $G$-principal bundles over $U$. 
    $\B G$ can be obtained as the stackification of the prestack which assigns to $U$ its 
    category of trivial $G$-principal bundles.
    In this example, Lemma~\ref{lem:stackificationLocallyGlued} amounts to the statement that
    any principal $G$-bundle is locally trivial.
\end{example}

\begin{example}
    The \emph{representable} stack $\underline{X}$ associated to an object $X \in \cat{C}$ assigns to 
    $Y \in \cat{C}$ the discrete category $\underline{X}(Y) \define \Hom(Y,X)$.
\end{example}

The Yoneda Lemma familiar from presheaves carries over to this context.
\begin{lemma}[2-Yoneda Lemma]
\label{lem:yonedaStacks}
    Let $X \in \cat{C}$, $\st{Y} \in \St/\cat{C}$, then 
    \begin{equation*}
        \Hom_{\St/\cat{C}}(\underline{X},\st{Y}) \simeq \st{Y}(X).
    \end{equation*}
\end{lemma}
We will hence use the notations $X: U \to \st{F}$ and $X \in \st{F}(U)$ interchangeably to refer to 
a section of $\st{F}$ over $U$.

\section{Twisted Sheaf Theory for \'Etale Stacks}
\label{app:twistedSheaves}
A smooth map of manifolds $f:M \to N$ induces an adjunction of $\Cinf$-modules
\[
    f^\ast:\ShC(N) \leftrightarrows \ShC(M):f_\ast.
\]
We study the analogous situation, generalised in two directions:
We work over \emph{\'etale stacks} instead of just manifolds, and we work with
twisted sheaves rather than ordinary sheaves of $\Cinf$-modules.
We generalise the functors $f_\ast$ and $f^\ast$ to this setup,
and show that they are again adjoint functors
(Proposition~\ref{prop:twistedPullbackPushfwdAdjunction}).

\'Etale stacks and their categories of (non-twisted) sheaves were previously studied
in~\cite{carchedi2010sheaf,carchedi2011categorical,carchedi2013etale}.
Treatments of twisted sheaves over manifolds and schemes may be
found in~\cite{caldararu2000derived}
and~\cite{lieblich2007moduli,lieblich2008twisted}, but we could not find a
proof of the pullback\-pushforward adjunction in these references.

The authors of~\cite{bunke2007sheaf} introduce a derived pullback-pushforward
ad\-junc\-tion of shea\-ves over a gerbe. However, their notion of sheaves over a gerbe
differs from that of sheaves twisted by that gerbe.

\begin{notation}
    We denote by $\VecInf$ the category of vector spaces over $\IC$,
    and by $\Vec \subset \VecInf$ its subcategory of finite-dimensional vector spaces.
    We write $\VCat$ for the 2-category of $\VecInf$-enriched categories
    that are complete under finite direct sums.
    We use the word sheaf to refer to a $\VecInf$-valued sheaf.
    The sheaf $\Cinf$ is the sheaf of smooth $\IC$-valued functions over the
    site $\Man$.
    We freely use the term \emph{line bundle} to refer to a sheaf of
    $\Cinf$-modules which is locally equivalent to $\Cinf$.
    In this appendix, we will use the word stack to refer to a stack over the site $\Man$ of
    smooth manifolds, equipped with the surjective submersion topology.
    We will be concerned with both stacks valued in groupoids and stacks valued in $\VCat$.
    Given a map $Y \onto X$, we will denote the iterated fibre products
    $Y \times_X Y \cdots Y$ by $Y^{[n]}$, and the simplicial projection maps by
    $p_{ij \cdots k}:Y^{[m]} \to Y^{[n]}$.
\end{notation}

\subsection{The \'Etale Site of an \'Etale Stack}
\begin{defn}
    A map of stacks $\st{Y} \to \st{X}$ is
    \emph{representable} if for any map
    $U \to \st{X}$ whose codomain is a manifold,
    the pullback $U \times_{\st{X}} \st{Y}$ is
    representable by a manifold.
    We call $\st{Y} \to \st{X}$ representably
    \'etale/sub\-mer\-sive/surjective
    if it is representable and, for any map
    $U \to \st{X}$ with domain a manifold, the pullback map
    $U \times_{\st{X}} \st{Y} \to U$ is
    \'etale/sub\-mer\-sive/surjective as a map of manifolds.\footnote{
        We use the word \emph{\'etale} to mean a local diffeomorphism.
    }
\end{defn}
The pasting law for pullback squares implies that
these properties are preserved under pullback along any
map of stacks $\st{Z} \to \st{X}$.
Using further that the property of being \'etale/\-submersive/\-surjective is preserved
under composition, it is easy to see that the set of representably
\'etale/\-submersive/\-surjective maps are closed under
composition.

\begin{defn}
    An \emph{atlas} for a stack $\st{X}$ is a representably surjective submersion
    $Y \onto \st{X}$ whose codomain $Y$ is a manifold.
    If the map is further representably \'etale, we call it an \emph{\'etale atlas}.
\end{defn}

\begin{defn}
    A \emph{differentiable stack} is a stack which admits an atlas.
    A differentiable stack is an \emph{\'etale stack} if the atlas may be chosen to be \'etale.
\end{defn}

\begin{notation}
    We denote the bicategory of differentiable stacks by $\DiffSt$ and that of \'etale stacks by $\EtSt$.
\end{notation}

\begin{example}
    Every manifold $M$ is an \'etale stack via the identity atlas $M \to M$.
\end{example}

\begin{example}
    Let $X$ be a manifold with an action by a Lie group $G$.
    The quotient stack $[X/G]$ admits an atlas $X \onto [X/G]$,
    and is thus differentiable.
    The map $X \onto [X/G]$ is an \'etale atlas iff $G$ is discrete.
    However, $[X/G]$ may be an \'etale stack even if $G$ is nondiscrete:
    let $X=G$ with the usual right action, then the quotient stack
    is $[G/G] \simeq \point$.
\end{example}

A differentiable stack $\st{X}$ has an associated bicategory $\DiffSt/\st{X}$,
the overcategory of $\st{X}$. Its objects are maps
\[
    \begin{tikzcd}
        \st{Y} \rar[r,"y"] & \st{X},
    \end{tikzcd}
\]
its 1-morphisms $y \to y'$ are pairs $(f, \phi)$ as below
% https://q.uiver.app/?q=WzAsMyxbMCwwLCJcXHN0e1l9Il0sWzIsMCwiXFxzdHtZfSJdLFsxLDEsIlxcc3R7WH0iXSxbMCwyLCJ5IiwyXSxbMSwyLCJ5JyJdLFswLDEsImYiXSxbMSwzLCJcXHBoaSIsMCx7InNob3J0ZW4iOnsic291cmNlIjozMCwidGFyZ2V0IjoyMH19XV0=
\[\begin{tikzcd}
        {\st{Y}} && {\st{Y}'} \\
        & {\st{X}}
        \arrow[""{name=0, anchor=center, inner sep=0}, "y"', from=1-1, to=2-2]
        \arrow["{y'}", from=1-3, to=2-2]
        \arrow["f", from=1-1, to=1-3]
        \arrow["\phi", shorten <=12pt, shorten >=8pt, Rightarrow, from=1-3, to=0]
    \end{tikzcd}\]
and its 2-morphisms $(f, \phi) \to (f', \phi')$ are 2-cells
% https://q.uiver.app/?q=WzAsMixbMCwwLCJcXHN0e1l9Il0sWzIsMCwiXFxzdHtZfSciXSxbMCwxLCJmJyIsMix7ImN1cnZlIjoxfV0sWzAsMSwiZiIsMCx7ImN1cnZlIjotMX1dLFszLDIsIlxcYWxwaGEiLDIseyJzaG9ydGVuIjp7InNvdXJjZSI6MjAsInRhcmdldCI6MjB9fV1d
\[\begin{tikzcd}
        {\st{Y}} && {\st{Y}'.}
        \arrow[""{name=0, anchor=center, inner sep=0}, "{f'}"', curve={height=12pt}, from=1-1, to=1-3]
        \arrow[""{name=1, anchor=center, inner sep=0}, "f", curve={height=-12pt}, from=1-1, to=1-3]
        \arrow["\alpha"', shorten <=2pt, shorten >=2pt, Rightarrow, from=1, to=0]
    \end{tikzcd}\]
such that $\phi = \phi' \comp \alpha$.

We now build the \'etale overcategory of an \'etale stack $\orb{M}$ in two steps.
First, we restrict to the full subcategory of $\DiffSt/\orb{M}$ on objects
$y:\st{Y} \to \orb{M}$ where $y$ is representably \'etale.
Then we form the homotopy 1-category, modding out by 2-cells.
We spell out the result in the following definition.

\begin{defn}
    The \emph{\'etale overcategory} $\StackEt(\orb{M})$  of an \'etale stack $\orb{M}$ has
    \begin{itemize}
        \item objects: representably \'etale maps $y:\st{X} \to \orb{M}$
        \item morphisms $(y \to y')$: pairs $(f, \phi):y \to y'$ as above,
              modulo the equivalence relation generated by 2-cells ---
              $(f, \phi) \sim (f', \phi')$ if there exists a 2-cell $\alpha: f \to f'$
              such that
              \begin{center}
                  \begin{minipage}{0.15\textwidth}
                  \end{minipage}
                  \begin{minipage}{0.3\textwidth}
                      \centering
                      % https://q.uiver.app/?q=WzAsMyxbMCwwLCJcXHN0e1l9Il0sWzIsMCwiXFxzdHtZfSJdLFsxLDEsIlxcc3R7WH0iXSxbMCwyLCJ5IiwyXSxbMSwyLCJ5JyJdLFswLDEsImYiXSxbMSwzLCJcXHBoaSIsMCx7InNob3J0ZW4iOnsic291cmNlIjozMCwidGFyZ2V0IjoyMH19XV0=
                      \[\begin{tikzcd}
                              {\st{Y}} && {\st{Y}} \\
                              & {\st{X}}
                              \arrow[""{name=0, anchor=center, inner sep=0}, "y"', from=1-1, to=2-2]
                              \arrow["{y'}", from=1-3, to=2-2]
                              \arrow["f", from=1-1, to=1-3]
                              \arrow["\phi", shorten <=12pt, shorten >=8pt, Rightarrow, from=1-3, to=0]
                          \end{tikzcd}\]
                  \end{minipage}
                  \begin{minipage}{0.1\textwidth}
                      \centering
                      \[=\]
                  \end{minipage}
                  \begin{minipage}{0.3\textwidth}
                      \centering
                      % https://q.uiver.app/?q=WzAsMyxbMCwwLCJcXHN0e1l9Il0sWzIsMCwiXFxzdHtZfSJdLFsxLDEsIlxcc3R7WH0iXSxbMCwyLCJ5IiwyXSxbMSwyLCJ5JyJdLFswLDEsImYnIiwyXSxbMCwxLCJmIiwwLHsiY3VydmUiOi0zfV0sWzEsMywiXFxwaGknIiwwLHsic2hvcnRlbiI6eyJzb3VyY2UiOjMwLCJ0YXJnZXQiOjIwfX1dLFs2LDUsIlxcYWxwaGEiLDAseyJzaG9ydGVuIjp7InNvdXJjZSI6MjAsInRhcmdldCI6MjB9fV1d
                      \[\begin{tikzcd}
                              {\st{Y}} && {\st{Y}} \\
                              & {\st{X}.}
                              \arrow[""{name=0, anchor=center, inner sep=0}, "y"', from=1-1, to=2-2]
                              \arrow["{y'}", from=1-3, to=2-2]
                              \arrow[""{name=1, anchor=center, inner sep=0}, "{f'}"', from=1-1, to=1-3]
                              \arrow[""{name=2, anchor=center, inner sep=0}, "f", curve={height=-18pt}, from=1-1, to=1-3]
                              \arrow["{\phi'}", shorten <=12pt, shorten >=8pt, Rightarrow, from=1-3, to=0]
                              \arrow["\alpha", shorten <=2pt, shorten >=2pt, Rightarrow, from=2, to=1]
                          \end{tikzcd}\]
                  \end{minipage}
                  \begin{minipage}{0.15\textwidth}
                  \end{minipage}
              \end{center}
    \end{itemize}
\end{defn}
We equip the \'etale overcategory with the structure of a site by declaring covers of $y:\st{Y} \to \orb{M}$
to be collections $\{(f_i,\phi_i):y_i \to y\}_{i \in I}$ such that the map $\coprod f_i$ is a
representably surjective submersion.

The property of being representably \'etale may be checked on an \'etale atlas:
\begin{lemma}
    \label{lem:checkRepresentablyEtaleOnAnyEtaleAtlas}
    A map $f:\st{X} \to \st{Y}$ of \'etale stacks is representably \'etale if
    there exists an \'etale atlas $Y \onto \st{Y}$ such that
    $Y \times_{\st{Y}} \st{X} \to Y$ is an \'etale map of manifolds.
\end{lemma}
\begin{proof}
    Let $g:U \to \st{Y}$ be a map of stacks with codomain a manifold.
    We must show that the map $g^\ast f:U \times_{\st{Y}} \st{X} \to U$
    is \'etale.
    Consider the cube of pullback squares
    % https://q.uiver.app/?q=WzAsOCxbMiwzLCJcXHN0e1l9Il0sWzAsMywiVSJdLFsxLDIsIlUgXFx0aW1lc197XFxzdHtZfX1cXHN0e1h9Il0sWzMsMiwiXFxzdHtYfSJdLFsyLDEsIlkiXSxbMywwLCJYIl0sWzEsMCwiVSBcXHRpbWVzX3tcXHN0e1l9fSBYIl0sWzAsMSwiVVxcdGltZXNfe1xcc3R7WX19WSJdLFszLDAsImYiXSxbMSwwXSxbMiwxXSxbMiwzLCJcXCAiLDFdLFs0LDBdLFs1LDRdLFs1LDNdLFs3LDRdLFs3LDFdLFs2LDddLFs2LDIsIlxcICIsMV0sWzYsNV1d
    \[\begin{tikzcd}[column sep=tiny, row sep=tiny]
            & {U \times_{\st{Y}} X} && X \\
            {U\times_{\st{Y}}Y} && Y \\
            & {U \times_{\st{Y}}\st{X}} && {\st{X}} \\
            U && {\st{Y}.}
            \arrow["{\ }"{description}, from=3-2, to=3-4]
            \arrow["{\ }"{description}, from=1-2, to=3-2]
            \arrow["f", from=3-4, to=4-3]
            \arrow[from=4-1, to=4-3]
            \arrow[from=3-2, to=4-1]
            \arrow[from=2-3, to=4-3]
            \arrow[from=1-4, to=2-3]
            \arrow[from=1-4, to=3-4]
            \arrow[from=2-1, to=2-3]
            \arrow[from=2-1, to=4-1]
            \arrow[from=1-2, to=2-1]
            \arrow[from=1-2, to=1-4]
        \end{tikzcd}\]

    The unlabeled maps in the diagram below
    % https://q.uiver.app/?q=WzAsNCxbMSwxLCJVIl0sWzAsMSwiVSBcXHRpbWVzX3tcXHN0e1l9fVxcc3R7WH0iXSxbMCwwLCJVIFxcdGltZXNfe1xcc3R7WX19IFgiXSxbMSwwLCJVXFx0aW1lc197XFxzdHtZfX1ZIl0sWzMsMF0sWzIsM10sWzIsMV0sWzEsMCwiVSBcXHRpbWVzX3tcXHN0e1l9fSBmIiwyXV0=
    \[\begin{tikzcd}
            {U \times_{\st{Y}} X} & {U\times_{\st{Y}}Y} \\
            {U \times_{\st{Y}}\st{X}} & U
            \arrow[from=1-2, to=2-2]
            \arrow[from=1-1, to=1-2]
            \arrow[from=1-1, to=2-1]
            \arrow["{g^\ast f}"', from=2-1, to=2-2]
        \end{tikzcd}\]
    are \'etale because they are pullbacks of representably \'etale maps.
    The 2-out-of-3 property for \'etale maps of manifolds implies that
    $g^\ast f$ is \'etale, which concludes the proof.
\end{proof}

\begin{lemma}
    \label{lem:allMapsAreRepresentablyEtale}
    Let
    % https://q.uiver.app/?q=WzAsMyxbMCwwLCJcXHN0e1l9Il0sWzIsMCwiXFxzdHtZfSJdLFsxLDEsIlxcc3R7WH0iXSxbMCwyLCJ5IiwyXSxbMSwyLCJ5JyJdLFswLDEsImYiXSxbMSwzLCJcXHBoaSIsMCx7InNob3J0ZW4iOnsic291cmNlIjozMCwidGFyZ2V0IjoyMH19XV0=
    \[\begin{tikzcd}
            {\st{Y}} && {\st{Y}'} \\
            & {\orb{M}}
            \arrow["y"', from=1-1, to=2-2]
            \arrow["y'",from=1-3, to=2-2]
            \arrow["f", from=1-1, to=1-3]
            \arrow["\phi", shorten <=12pt, shorten >=8pt, Rightarrow, from=1-3, to=0]
        \end{tikzcd}\]
    be a morphism in $\StackEt(\orb{M})$.
    Then $f$ is a representably \'etale map of \'etale stacks.
\end{lemma}
\begin{proof}
    An \'etale atlas $M \onto \orb{M}$ pulls back to a representably \'etale
    map $M \times_{\orb{M}} \st{Y} \to \st{Y}$. As $y: \st{Y} \to \orb{M}$
    is representable, $M \times_{\orb{M}} \st{Y}$ is a manifold and thus
    forms an \'etale atlas for $\st{Y}$. Hence $\st{Y}$ is an \'etale stack,
    and the same is true for $\st{Y}'$.
    Pulling the diagram in the statement back along $\pi^\ast:M \onto \orb{M}$,
    we obtain
    \[\begin{tikzcd}
            {M \times_{\orb{M}} \st{Y}} && {M \times_{\orb{M}} \st{Y}'} \\
            & {M}
            \arrow["\pi^\ast y"', from=1-1, to=2-2]
            \arrow["\pi^\ast y'",from=1-3, to=2-2]
            \arrow["\pi^\ast f", from=1-1, to=1-3]
        \end{tikzcd}\]
    Both $\pi^\ast y$ and $\pi^\ast y'$ are \'etale by assumption,
    hence so is $\pi^\ast f$.
    But $\pi^\ast f$ is the pullback of $f$ along the \'etale atlas map
    $M \times_{\orb{M}} \st{Y}' \to \st{Y}'$, so we
    may use Lemma~\ref{lem:checkRepresentablyEtaleOnAnyEtaleAtlas}
    to finish the proof.
\end{proof}

\begin{rmk}
    Let $(x: \st{X} \to \orb{M}) \in \StackEt(\orb{M})$, and consider the overcategory $\StackEt(\orb{M})/x$.
    Lemma~\ref{lem:allMapsAreRepresentablyEtale} implies that forgetting
    the map to $\orb{M}$ is a functor
    \[
        \StackEt(\orb{M})/x \to \StackEt(\st{X}).
    \]
\end{rmk}

\begin{defn}
    The \emph{\'etale site} of an \'etale stack $\orb{M}$ is the subsite
    $\Et(\orb{M}) \subset \StackEt(\orb{M})$
    on those objects $x: X \to \orb{M}$ where $X$ is a manifold.
\end{defn}
Note that there are no 2-cells between two maps $f,f':X \to Y'$, so every
1-morphism in $\Et(\orb{M})$ admits a unique representative $(f,\phi)$,
and there is a forgetful functor $\Et(\orb{M}) \to \Man$.
The notion of covers on $\Et(\orb{M})$ agrees with the one pulled back from $\Man$ via this functor:
A collection $\{(f_i,\phi_i):x_i \to x\}_{i \in I}$ is a cover in $\Et(\orb{M})$ iff the map
$\coprod f_i$ is a surjective submersion. This is equivalent to asking $\{f_i\}_{i \in I}$ to be a cover in $\Man$.
By Lemma~\ref{lem:allMapsAreRepresentablyEtale}, all the maps $f_i$ are automatically
\'etale, so this reduces to the condition that the union be surjective.

\begin{defn}
    Let $Y \onto \orb{M}$ be an \'etale atlas. We denote the \emph{site associated to the atlas $Y$}
    by $\StackEt(Y/\orb{M})$. Its objects are
    \'etale maps $(\tilde{x}:\st{X} \to Y)$ with source an \'etale stack $\st{X}$,
    and morphisms $\tilde{x} \to \tilde{x}'$ are morphisms of the composites
    \[
        x:\st{X} \xto{\tilde{x}} Y \onto \orb{M}
    \]
    in $\StackEt(\orb{M})$.
    We denote by $\Et(Y/\orb{M})$ the full subsite on objects $(\tilde{x}:X \to Y) \in \StackEt(Y/\orb{M})$
    where the source $X$ is a manifold.
\end{defn}
These sites $\Et(Y/\orb{M})$ are built so as to make the functors
$\Et(Y/\orb{M}) \to \Et(\orb{M})$ and $\StackEt(Y/\orb{M}) \to \StackEt(\orb{M})$ fully faithful.
We have introduced 4 sites associated to an \'etale stack, and now show that
for the purposes of studying sheaves and stacks over $\orb{M}$, we can freely
move between these sites.
\begin{lemma}
    \label{lem:everythingIsDense}
    The functors in the diagram below
    % https://q.uiver.app/?q=WzAsNCxbMCwxLCJcXEV0KFkvXFxvcmJ7TX0pIl0sWzEsMCwiXFxTdGFja0V0KFkvXFxvcmJ7TX0pIl0sWzIsMSwiXFxTdGFja0V0KFxcb3Jie019KSJdLFsxLDIsIlxcRXQoXFxvcmJ7TX0pIl0sWzAsMSwiIiwwLHsic3R5bGUiOnsidGFpbCI6eyJuYW1lIjoiaG9vayIsInNpZGUiOiJ0b3AifX19XSxbMSwyLCIiLDAseyJzdHlsZSI6eyJ0YWlsIjp7Im5hbWUiOiJob29rIiwic2lkZSI6InRvcCJ9fX1dLFszLDIsIiIsMix7InN0eWxlIjp7InRhaWwiOnsibmFtZSI6Imhvb2siLCJzaWRlIjoidG9wIn19fV0sWzAsMywiIiwyLHsic3R5bGUiOnsidGFpbCI6eyJuYW1lIjoiaG9vayIsInNpZGUiOiJ0b3AifX19XV0=
    \[\begin{tikzcd}[row sep=tiny,column sep=tiny]
            & {\StackEt(Y/\orb{M})} \\
            {\Et(Y/\orb{M})} && {\StackEt(\orb{M})} \\
            & {\Et(\orb{M})}
            \arrow[hook, from=2-1, to=1-2]
            \arrow[hook, from=1-2, to=2-3]
            \arrow[hook, from=3-2, to=2-3]
            \arrow[hook, from=2-1, to=3-2]
        \end{tikzcd}\]
    are inclusions of dense subsites.
\end{lemma}
\begin{proof}
    As all the functors are fully faithful inclusions of subsites, it suffices to prove $\Et(Y/\orb{M})$
    is dense in $\StackEt(\orb{M})$.
    We need to check that every object $(x: \st{X} \to \orb{M}) \in \StackEt(\orb{M})$
    admits a cover by objects in $\Et(Y/\orb{M})$.
    Form the pullback of $Y \onto \orb{M}$ with $x$. Then the map
    $Y \times_{\orb{M}} \st{X} \onto \st{X}$ is an \'etale atlas
    for $\st{X}$. Indeed, the map is representably \'etale because it is pulled back
    from the representably \'etale map $Y \onto \orb{M}$, and the fibre product is a manifold
    because $x$ is representably \'etale.
    In particular, $Y \times_{\orb{M}} \st{X} \to Y$ is an object in $\Et(Y/\orb{M})$,
    and together with the 2-cell $\phi$
    % https://q.uiver.app/?q=WzAsNCxbMCwwLCJZIFxcdGltZXNfe1xcb3Jie019fSBcXHN0e1h9Il0sWzEsMCwiXFxzdHtYfSJdLFsxLDEsIlxcb3Jie019Il0sWzAsMSwiWSJdLFswLDJdLFsxLDIsIngiXSxbMCwxXSxbMywyXSxbMCwzXSxbMSw0LCJcXHBoaSIsMCx7InNob3J0ZW4iOnsic291cmNlIjozMCwidGFyZ2V0IjoyMH19XSxbMyw0LCIiLDAseyJzaG9ydGVuIjp7InNvdXJjZSI6MjAsInRhcmdldCI6MzB9LCJzdHlsZSI6eyJoZWFkIjp7Im5hbWUiOiJub25lIn19fV1d
    \[\begin{tikzcd}
            {Y \times_{\orb{M}} \st{X}} & {\st{X}} \\
            Y & {\orb{M}}
            \arrow[""{name=0, anchor=center, inner sep=0}, from=1-1, to=2-2]
            \arrow["x", from=1-2, to=2-2]
            \arrow[from=1-1, to=1-2]
            \arrow[from=2-1, to=2-2]
            \arrow[from=1-1, to=2-1]
            \arrow["\phi", shorten <=5pt, shorten >=3pt, Rightarrow, from=1-2, to=0]
            \arrow[shorten <=4pt, shorten >=6pt, Rightarrow, no head, from=2-1, to=0]
        \end{tikzcd}\]
    provided by the universal property of the pullback, it covers
    $x: \st{X} \to \orb{M}$ in $\StackEt(\orb{M})$.
\end{proof}

By the Comparison Lemma (Lemma~\ref{lem:comparisonLemma}), this means that sheaf and stack categories over all these
sites are canonically equivalent.
We henceforth refer to sheaves/stacks over any of these sites as sheaves/stacks over $\orb{M}$.

We equip $\orb{M}$ with the ($\VecInf$-valued) structure sheaf $\Cinf$ of smooth $\IC$-valued functions.
This sheaf is pulled back from $\Man$ via the forgetful functor $\Et(\orb{M}) \to \Man$.
\begin{notation}
    We write $\Shv(\orb{M})$ to denote the category of $\VecInf$-valued sheaves over $\orb{M}$ and
    $\ShC(\orb{M})$ to denote its category of $\Cinf$-modules.
\end{notation}

In what follows, we show that $\StackEt(\orb{M})$ admits certain limits.
To clarify the following section, we explicitly describe the category of maps
$\st{W} \to \st{X} \times_\st{Z} \st{Y}$ from a stack $\st{W}$ into a fibre product of stacks:
An object of $(\st{X} \times_\st{Z} \st{Y})(\st{W})$ is the datum of three maps ${f_{\st{X}}}, {f_{\st{Y}}}, {f_{\st{Z}}}$ and
two 2-cells $\phi_{\st{X}}, \phi_{\st{Y}}$ as below.
% https://q.uiver.app/?q=WzAsMTAsWzAsMCwiXFxzdHtXfSJdLFsxLDAsIlxcc3R7V30iXSxbMiwwLCJcXHN0e1d9Il0sWzAsMSwiXFxzdHtYfSJdLFsxLDEsIlxcc3R7WX0iXSxbMiwxLCJcXHN0e1p9Il0sWzMsMCwiXFxzdHtXfSJdLFszLDEsIlxcc3R7WH0iXSxbNCwxLCJcXHN0e1p9Il0sWzQsMCwiXFxzdHtZfSJdLFswLDMsImZfe1xcc3R7WH19Il0sWzEsNCwiZl97XFxzdHtZfX0iXSxbMiw1LCJmX3tcXHN0e1p9fSJdLFs2LDddLFs3LDhdLFs2LDhdLFs2LDldLFs5LDhdLFs3LDE1LCJcXHBoaV97XFxzdHtYfX0iLDAseyJzaG9ydGVuIjp7InRhcmdldCI6MjB9fV0sWzksMTUsIlxccGhpX3tcXHN0e1l9fSIsMCx7InNob3J0ZW4iOnsidGFyZ2V0IjoyMH19XV0=
\[\begin{tikzcd}
        {\st{W}} & {\st{W}} & {\st{W}} & {\st{W}} & {\st{Y}} \\
        {\st{X}} & {\st{Y}} & {\st{Z}} & {\st{X}} & {\st{Z}}
        \arrow["{f_{\st{X}}}", from=1-1, to=2-1]
        \arrow["{f_{\st{Y}}}", from=1-2, to=2-2]
        \arrow["{f_{\st{Z}}}", from=1-3, to=2-3]
        \arrow[from=1-4, to=2-4]
        \arrow[from=2-4, to=2-5]
        \arrow[""{name=0, anchor=center, inner sep=0}, from=1-4, to=2-5]
        \arrow[from=1-4, to=1-5]
        \arrow[from=1-5, to=2-5]
        \arrow["{\phi_{\st{X}}}", shorten >=2pt, Rightarrow, from=2-4, to=0]
        \arrow["{\phi_{\st{Y}}}", shorten >=2pt, Rightarrow, from=1-5, to=0]
    \end{tikzcd}\]
A morphism between two objects is a triple of 2-cells
$\alpha_{\st{X}/\st{Y}/\st{Z}}:f_{\st{X}/\st{Y}/\st{Z}} \to f_{\st{X}/\st{Y}/\st{Z}}'$
such that both
% https://q.uiver.app/?q=WzAsOCxbMCwwLCJcXHN0e1d9Il0sWzAsMSwiXFxzdHtYfSJdLFsxLDEsIlxcc3R7Wn0iXSxbMiwxXSxbMiwwXSxbMywwLCJcXHN0e1d9Il0sWzMsMSwiXFxzdHtYfSJdLFs0LDEsIlxcc3R7Wn0iXSxbMCwxLCJmX3tcXHN0e1h9fSIsMix7ImN1cnZlIjo0fV0sWzEsMl0sWzAsMiwiZl97XFxzdHtafX0nIl0sWzQsMywiPSIsMSx7InN0eWxlIjp7ImJvZHkiOnsibmFtZSI6Im5vbmUifSwiaGVhZCI6eyJuYW1lIjoibm9uZSJ9fX1dLFs1LDYsImZfe1xcc3R7WH19IiwyXSxbNiw3XSxbNSw3LCJmX3tcXHN0e1p9fSciLDAseyJjdXJ2ZSI6LTN9XSxbNSw3XSxbMCwxXSxbMSwxMCwiXFxwaGlfe1xcc3R7WH19JyIsMix7InNob3J0ZW4iOnsidGFyZ2V0IjoyMH19XSxbMTUsMTQsIlxcYWxwaGFfe1xcc3R7Wn19IiwwLHsic2hvcnRlbiI6eyJzb3VyY2UiOjIwLCJ0YXJnZXQiOjIwfX1dLFs2LDE1LCJcXHBoaV97XFxzdHtYfX0iLDIseyJzaG9ydGVuIjp7InRhcmdldCI6MjB9fV0sWzgsMTYsIlxcYWxwaGFfe1xcc3R7WH19IiwyLHsibGFiZWxfcG9zaXRpb24iOjAsInNob3J0ZW4iOnsic291cmNlIjoyMCwidGFyZ2V0IjoyMH19XV0=&macro_url=file%3A%2F%2F%2FC%3A%2FUsers%2Fchris%2FDownloads%2Fmacros.tex
\[\begin{tikzcd}
        {\st{W}} && {} & {\st{W}} \\
        {\st{X}} & {\st{Z}} & {} & {\st{X}} & {\st{Z}}
        \arrow[""{name=0, anchor=center, inner sep=0}, "{f_{\st{X}}}"', curve={height=24pt}, from=1-1, to=2-1]
        \arrow[from=2-1, to=2-2]
        \arrow[""{name=1, anchor=center, inner sep=0}, "{f_{\st{Z}}'}", from=1-1, to=2-2]
        \arrow["{=}"{description}, draw=none, from=1-3, to=2-3]
        \arrow["{f_{\st{X}}}"', from=1-4, to=2-4]
        \arrow[from=2-4, to=2-5]
        \arrow[""{name=2, anchor=center, inner sep=0}, "{f_{\st{Z}}'}", curve={height=-18pt}, from=1-4, to=2-5]
        \arrow[""{name=3, anchor=center, inner sep=0}, from=1-4, to=2-5]
        \arrow[""{name=4, anchor=center, inner sep=0}, from=1-1, to=2-1]
        \arrow["{\phi_{\st{X}}'}"', shorten >=2pt, Rightarrow, from=2-1, to=1]
        \arrow["{\alpha_{\st{Z}}}", shorten <=3pt, shorten >=3pt, Rightarrow, from=3, to=2]
        \arrow["{\phi_{\st{X}}}"', shorten >=2pt, Rightarrow, from=2-4, to=3]
        \arrow["{\alpha_{\st{X}}}"'{pos=0}, shorten <=5pt, shorten >=5pt, Rightarrow, from=0, to=4]
    \end{tikzcd}\]
and the analogous equation involving $\alpha_\st{Y}, \phi_\st{Y}$ and
$\phi'_\st{Y}$ hold.
The canonical morphisms from $\st{X} \times_{\st{Z}} \st{Y}$ to $\st{X}$, $\st{Y}$ and $\st{Z}$
are the obvious forgetful maps.
\begin{rmk}
    One may also describe the category $\Hom(\st{W},\st{X} \times_\st{Z} \st{Y})$ as
    consisting of pairs $\st{W} \to \st{X}, \st{W} \to \st{Y}$, equipped with
    a single 2-cell relating the two resulting composites $\st{W} \to \st{Z}$.
    There is a map from the stack as we described it above to this one,
    given by forgetting the data $\st{W} \to \st{Z}$ and composing the 2-cells
    $\phi_\st{X}$ and $\phi_\st{Y}$
    to obtain a single 2-cell. It is straightforward to check that this map is an
    equivalence.
\end{rmk}

Let $\st{X} \to \orb{M}$ and $\st{Y} \to \orb{M}$ be objects in
$\StackEt(\orb{M})$. Then $\st{X} \times_\orb{M} \st{Y} \to \orb{M}$ is again
an object of $\StackEt(\orb{M})$, because the property of being representably \'etale
is preserved under pullback and closed under composition.
Similary, given a pair of maps with common target in $\StackEt(\orb{M})$
% https://q.uiver.app/?q=WzAsNCxbMiwyLCJcXG9yYntNfSJdLFsxLDEsIlxcc3R7Wn0iXSxbMCwxLCJcXHN0e1h9Il0sWzEsMCwiXFxzdHtZfSJdLFsxLDBdLFsyLDFdLFszLDFdLFszLDAsIiIsMix7ImN1cnZlIjotMn1dLFsyLDAsIiIsMix7ImN1cnZlIjoyfV0sWzEsOCwiIiwyLHsic2hvcnRlbiI6eyJ0YXJnZXQiOjIwfX1dLFsxLDcsIiIsMix7InNob3J0ZW4iOnsidGFyZ2V0IjoyMH19XV0=
\[\begin{tikzcd}
        & {\st{Y}} \\
        {\st{X}} & {\st{Z}} \\
        && {\orb{M},}
        \arrow[from=2-2, to=3-3]
        \arrow[from=2-1, to=2-2]
        \arrow[from=1-2, to=2-2]
        \arrow[""{name=0, anchor=center, inner sep=0}, curve={height=-12pt}, from=1-2, to=3-3]
        \arrow[""{name=1, anchor=center, inner sep=0}, curve={height=12pt}, from=2-1, to=3-3]
        \arrow[shorten >=2pt, Rightarrow, from=2-2, to=1]
        \arrow[shorten >=4pt, Rightarrow, from=2-2, to=0]
    \end{tikzcd}\]
the pullback $\st{X} \times_\st{Z} \st{Y} \to \orb{M}$ is an object in
$\StackEt(\orb{M})$. The map $\st{X} \times_{\st{Z}} \st{Y} \to \st{X}$ is
representably \'etale because it is pulled back
from $\st{Y} \to \st{Z}$, which is representably \'etale by
Lemma~\ref{lem:allMapsAreRepresentablyEtale}. The map to $\orb{M}$ is
(equivalent to the map) obtained from
it by composing with the representably \'etale map $\st{X} \to \orb{M}$.

\begin{lemma}
    \label{lem:StackEtHasBinaryProdsAndPullbacks}
    The site $\StackEt(\orb{M})$ has binary products and pullbacks, computed by
    % https://q.uiver.app/?q=WzAsOCxbMCwwLCJcXHN0e1h9Il0sWzAsMSwiXFxvcmJ7TX0iXSxbMSwwLCJcXHN0e1l9Il0sWzEsMSwiXFxvcmJ7TX0iXSxbMywwLCJcXHN0e1h9IFxcdGltZXNfe1xcb3Jie019fSBcXHN0e1l9Il0sWzMsMSwiXFxvcmJ7TX0iXSxbMiwwXSxbMiwxXSxbMCwxXSxbMiwzXSxbNCw1XSxbNiw3LCJcXHNpbWVxIiwxLHsic3R5bGUiOnsiYm9keSI6eyJuYW1lIjoibm9uZSJ9LCJoZWFkIjp7Im5hbWUiOiJub25lIn19fV0sWzgsOSwiXFx0aW1lcyIsMyx7InNob3J0ZW4iOnsic291cmNlIjoyMCwidGFyZ2V0IjoyMH0sInN0eWxlIjp7ImJvZHkiOnsibmFtZSI6Im5vbmUifSwiaGVhZCI6eyJuYW1lIjoibm9uZSJ9fX1dXQ==
    \[\begin{tikzcd}
            {\st{X}} & {\st{Y}} & {} & {\st{X} \times_{\orb{M}} \st{Y}} \\
            {\orb{M}} & {\orb{M}} & {} & {\orb{M}}
            \arrow[""{name=0, anchor=center, inner sep=0}, from=1-1, to=2-1]
            \arrow[""{name=1, anchor=center, inner sep=0}, from=1-2, to=2-2]
            \arrow[from=1-4, to=2-4]
            \arrow["\simeq"{description}, draw=none, from=1-3, to=2-3]
            \arrow["\times"{marking}, Rightarrow, draw=none, from=0, to=1]
        \end{tikzcd}\]
    and
    % https://q.uiver.app/?q=WzAsOSxbMywyLCJcXG9yYntNfSJdLFsyLDEsIlxcc3R7Wn0iXSxbMSwxLCJcXHN0e1h9Il0sWzIsMCwiXFxzdHtZfSJdLFswLDEsIlxcbGltIl0sWzQsMSwiXFxzaW1lcSJdLFs4LDFdLFs1LDAsIlxcc3R7WH0gXFx0aW1lc197XFxzdHtafX0gXFxzdHtZfSJdLFs1LDIsIlxcb3Jie019Il0sWzEsMF0sWzIsMSwiZiJdLFszLDEsImciLDJdLFszLDAsIiIsMix7ImN1cnZlIjotMn1dLFsyLDAsIiIsMix7ImN1cnZlIjoyfV0sWzcsOF0sWzEsMTMsIlxcYWxwaGEiLDIseyJzaG9ydGVuIjp7InRhcmdldCI6MjB9fV0sWzEsMTIsIlxcYmV0YSIsMix7InNob3J0ZW4iOnsidGFyZ2V0IjoyMH19XV0=
    \[\begin{tikzcd}
            && {\st{Y}} &&& {\st{X} \times_{\st{Z}} \st{Y}} \\
            \lim & {\st{X}} & {\st{Z}} && \simeq &&&& {} \\
            &&& {\orb{M}} && {\orb{M}}
            \arrow[from=2-3, to=3-4]
            \arrow["f", from=2-2, to=2-3]
            \arrow["g"', from=1-3, to=2-3]
            \arrow[""{name=0, anchor=center, inner sep=0}, curve={height=-12pt}, from=1-3, to=3-4]
            \arrow[""{name=1, anchor=center, inner sep=0}, curve={height=12pt}, from=2-2, to=3-4]
            \arrow[from=1-6, to=3-6]
            \arrow["\alpha"', shorten >=2pt, Rightarrow, from=2-3, to=1]
            \arrow["\beta"', shorten >=4pt, Rightarrow, from=2-3, to=0]
        \end{tikzcd}\]
    respectively.
\end{lemma}
\begin{proof}
    We begin by proving the statement for binary products:
    A morphism $(\st{W} \to \orb{M}) \to (\st{X} \times_{\orb{M}} \st{Y} \to \orb{M})$
    is a pair of a morphism $\st{W} \to \st{X} \times_{\orb{M}} \st{Y}$ and a 2-cell
    $\phi$.
    Each such morphism is canonically isomorphic (via $\phi$ itself)
    to one where the composite
    $\st{W} \to \st{X} \times_{\orb{M}} \st{Y} \to \orb{M}$ is equal to $\st{W} \to \orb{M}$
    and the 2-cell $\phi$ is the trivial 2-cell.
    The remaining data is precisely a pair of maps $\st{W} \to \st{X}$, $\st{W} \to \st{Y}$,
    and a pair of 2-cells between the composites $\st{W} \to \orb{M}$.
    This identifies $(\st{X} \times_{\orb{M}} \st{Y} \to \orb{M})$ as the product.

    A morphism $(\st{W} \to \orb{M}) \to (\st{X} \times_{\st{Z}} \st{Y} \to \orb{M})$
    is equivalent to the data of three maps and three 2-cells as below.
    % https://q.uiver.app/?q=WzAsMTQsWzAsMCwiXFxzdHtXfSJdLFswLDEsIlxcc3R7WH0iXSxbMSwwLCJcXHN0e1d9Il0sWzEsMSwiXFxzdHtZfSJdLFsyLDAsIlxcc3R7V30iXSxbMiwxLCJcXHN0e1p9Il0sWzMsMCwiXFxzdHtXfSJdLFszLDEsIlxcc3R7WH0iXSxbNCwxLCJcXHN0e1p9Il0sWzQsMCwiXFxzdHtZfSJdLFs1LDAsIlxcc3R7V30iXSxbNiwwLCJcXHN0e1p9Il0sWzUsMSwiICJdLFs2LDEsIiAiXSxbMCwxXSxbMiwzXSxbNCw1XSxbNiw3XSxbNyw4XSxbNiw4XSxbNiw5XSxbOSw4XSxbMTAsMTFdLFsxMiwxMywiXFxvcmJ7TX0iLDMseyJzdHlsZSI6eyJib2R5Ijp7Im5hbWUiOiJub25lIn0sImhlYWQiOnsibmFtZSI6Im5vbmUifX19XSxbNywxOSwiXFxwaGlfe1xcc3R7WH19IiwyLHsic2hvcnRlbiI6eyJ0YXJnZXQiOjIwfX1dLFs5LDE5LCJcXHBoaV97XFxzdHtZfX0iLDIseyJzaG9ydGVuIjp7InRhcmdldCI6MjB9fV0sWzEwLDIzLCIiLDAseyJzaG9ydGVuIjp7InRhcmdldCI6NDB9LCJsZXZlbCI6MX1dLFsxMSwyMywiIiwyLHsic2hvcnRlbiI6eyJ0YXJnZXQiOjQwfSwibGV2ZWwiOjF9XSxbMTEsMjYsIlxca2FwcGEiLDIseyJzaG9ydGVuIjp7InRhcmdldCI6MjB9LCJsZXZlbCI6Mn1dXQ==
    \[\begin{tikzcd}
            {\st{W}} & {\st{W}} & {\st{W}} & {\st{W}} & {\st{Y}} & {\st{W}} & {\st{Z}} \\
            {\st{X}} & {\st{Y}} & {\st{Z}} & {\st{X}} & {\st{Z}} & { } & { }
            \arrow[from=1-1, to=2-1]
            \arrow[from=1-2, to=2-2]
            \arrow[from=1-3, to=2-3]
            \arrow[from=1-4, to=2-4]
            \arrow[from=2-4, to=2-5]
            \arrow[""{name=0, anchor=center, inner sep=0}, from=1-4, to=2-5]
            \arrow[from=1-4, to=1-5]
            \arrow[from=1-5, to=2-5]
            \arrow[from=1-6, to=1-7]
            \arrow[""{name=1, anchor=center, inner sep=0}, "{\orb{M}}"{marking}, draw=none, from=2-6, to=2-7]
            \arrow["{\phi_{\st{X}}}"', shorten >=2pt, Rightarrow, from=2-4, to=0]
            \arrow["{\phi_{\st{Y}}}"', shorten >=2pt, Rightarrow, from=1-5, to=0]
            \arrow[""{name=2, anchor=center, inner sep=0}, shorten >=8pt, from=1-6, to=1]
            \arrow[shorten >=8pt, from=1-7, to=1]
            \arrow["\kappa"', shorten >=3pt, Rightarrow, from=1-7, to=2]
        \end{tikzcd}\]
    The universal property of the pullback says that a map into the pullback of
    $(f, \alpha):(\st{X} \to \st{Z})/\orb{M}$ with
    $(g,\beta):(\st{Y} \to \st{Z})/\orb{M}$ is the data of two maps and two 2-cells
    % https://q.uiver.app/?q=WzAsOCxbMCwwLCJcXHN0e1d9Il0sWzAsMSwiXFxzdHtYfSJdLFsxLDAsIlxcc3R7V30iXSxbMSwxLCJcXHN0e1l9Il0sWzIsMCwiXFxzdHtXfSJdLFsyLDEsIlxcc3R7WH0iXSxbMywxLCJcXG9yYntNfSJdLFszLDAsIlxcc3R7WX0iXSxbMCwxXSxbMiwzXSxbNCw1XSxbNSw2XSxbNCw2XSxbNCw3XSxbNyw2XSxbNSwxMiwiXFxsYW1iZGFfe1xcc3R7WH19IiwyLHsic2hvcnRlbiI6eyJ0YXJnZXQiOjIwfX1dLFs3LDEyLCJcXGxhbWJkYV97XFxzdHtZfX0iLDIseyJzaG9ydGVuIjp7InRhcmdldCI6MjB9fV1d
    \[\begin{tikzcd}
            {\st{W}} & {\st{W}} & {\st{W}} & {\st{Y}} \\
            {\st{X}} & {\st{Y}} & {\st{X}} & {\orb{M}}
            \arrow[from=1-1, to=2-1]
            \arrow[from=1-2, to=2-2]
            \arrow[from=1-3, to=2-3]
            \arrow[from=2-3, to=2-4]
            \arrow[""{name=0, anchor=center, inner sep=0}, from=1-3, to=2-4]
            \arrow[from=1-3, to=1-4]
            \arrow[from=1-4, to=2-4]
            \arrow["{\lambda_{\st{X}}}"', shorten >=2pt, Rightarrow, from=2-3, to=0]
            \arrow["{\lambda_{\st{Y}}}"', shorten >=2pt, Rightarrow, from=1-4, to=0]
        \end{tikzcd}\]
    such that
    % https://q.uiver.app/?q=WzAsMTAsWzAsMCwiXFxzdHtXfSJdLFsxLDAsIlxcc3R7WH0iXSxbMiwwLCJcXHN0e1p9Il0sWzEsMSwiXFxvcmJ7TX0iXSxbNCwwLCJcXHN0e1d9Il0sWzUsMCwiXFxzdHtZfSJdLFs2LDAsIlxcc3R7Wn0iXSxbNSwxLCJcXG9yYntNfSJdLFszLDBdLFszLDFdLFs4LDksIlxcc2ltZXEiLDEseyJzdHlsZSI6eyJib2R5Ijp7Im5hbWUiOiJub25lIn0sImhlYWQiOnsibmFtZSI6Im5vbmUifX19XSxbMCwxXSxbMSwyXSxbMCwzXSxbMSwzXSxbMiwzXSxbNCw3XSxbNCw1XSxbNSw2XSxbNiw3XSxbNSw3XSxbNSwxNiwiXFxwaGlfe1xcc3R7WX19IiwyLHsic2hvcnRlbiI6eyJ0YXJnZXQiOjIwfX1dLFsxLDEzLCJcXHBoaV97XFxzdHtYfX0iLDIseyJzaG9ydGVuIjp7InRhcmdldCI6MjB9fV0sWzUsMTksIlxcYmV0YSIsMix7InNob3J0ZW4iOnsidGFyZ2V0IjoyMH19XSxbMSwxNSwiXFxhbHBoYSIsMix7InNob3J0ZW4iOnsidGFyZ2V0IjoyMH19XV0=
    \[\begin{tikzcd}
            {\st{W}} & {\st{X}} & {\st{Z}} & {} & {\st{W}} & {\st{Y}} & {\st{Z}} \\
            & {\orb{M}} && {} && {\orb{M}.}
            \arrow["\simeq"{description}, draw=none, from=1-4, to=2-4]
            \arrow[from=1-1, to=1-2]
            \arrow[from=1-2, to=1-3]
            \arrow[""{name=0, anchor=center, inner sep=0}, from=1-1, to=2-2]
            \arrow[from=1-2, to=2-2]
            \arrow[""{name=1, anchor=center, inner sep=0}, from=1-3, to=2-2]
            \arrow[""{name=2, anchor=center, inner sep=0}, from=1-5, to=2-6]
            \arrow[from=1-5, to=1-6]
            \arrow[from=1-6, to=1-7]
            \arrow[""{name=3, anchor=center, inner sep=0}, from=1-7, to=2-6]
            \arrow[from=1-6, to=2-6]
            \arrow["{\phi_{\st{Y}}}"', shorten >=2pt, Rightarrow, from=1-6, to=2]
            \arrow["{\phi_{\st{X}}}"', shorten >=2pt, Rightarrow, from=1-2, to=0]
            \arrow["\beta"', shorten >=2pt, Rightarrow, from=1-6, to=3]
            \arrow["\alpha"', shorten >=2pt, Rightarrow, from=1-2, to=1]
        \end{tikzcd}\]
    Given the data of a map into $(\st{X} \times_{\st{Z}} \st{Y} \to \orb{M})$ as
    above, one may obtain such data by forgetting the map $\st{W} \to \st{Z}$
    and setting
    % https://q.uiver.app/?q=WzAsMTAsWzQsMCwiXFxzdHtXfSJdLFs1LDEsIlxcc3R7Wn0iXSxbNiwyLCJcXG9yYntNfSJdLFs2LDAsIlxcc3R7WX0iXSxbNCwyLCJcXHN0e1h9Il0sWzMsMSwiXFxkZWZpbmUiXSxbMiwwLCJcXHN0e1l9Il0sWzIsMiwiXFxvcmJ7TX0iXSxbMCwyLCJcXHN0e1h9Il0sWzAsMCwiXFxzdHtXfSJdLFswLDFdLFsxLDJdLFswLDRdLFs0LDJdLFswLDNdLFszLDJdLFs0LDFdLFszLDFdLFs5LDddLFs4LDddLFs2LDddLFs5LDZdLFs5LDhdLFs0LDEwLCJcXHBoaV97XFxzdHtYfX0iLDIseyJzaG9ydGVuIjp7InNvdXJjZSI6MzAsInRhcmdldCI6MjB9fV0sWzMsMTAsIlxccGhpX3tcXHN0e1l9fSIsMCx7InNob3J0ZW4iOnsic291cmNlIjozMCwidGFyZ2V0IjoyMH19XSxbMSwxMywiXFxhbHBoYSIsMCx7InNob3J0ZW4iOnsic291cmNlIjoyMCwidGFyZ2V0IjoyMH19XSxbMSwxNSwiXFxiZXRhIiwwLHsic2hvcnRlbiI6eyJzb3VyY2UiOjIwLCJ0YXJnZXQiOjIwfX1dLFs4LDE4LCJcXGxhbWJkYV97XFxzdHtYfX0iLDIseyJzaG9ydGVuIjp7InRhcmdldCI6MjB9fV0sWzYsMTgsIlxcbGFtYmRhX3tcXHN0e1l9fSIsMix7InNob3J0ZW4iOnsidGFyZ2V0IjoyMH19XV0=
    \[\begin{tikzcd}
            {\st{W}} && {\st{Y}} && {\st{W}} && {\st{Y}} \\
            &&& \define && {\st{Z}} \\
            {\st{X}} && {\orb{M}} && {\st{X}} && {\orb{M}}
            \arrow[""{name=0, anchor=center, inner sep=0}, from=1-5, to=2-6]
            \arrow[from=2-6, to=3-7]
            \arrow[from=1-5, to=3-5]
            \arrow[""{name=1, anchor=center, inner sep=0}, from=3-5, to=3-7]
            \arrow[from=1-5, to=1-7]
            \arrow[""{name=2, anchor=center, inner sep=0}, from=1-7, to=3-7]
            \arrow[from=3-5, to=2-6]
            \arrow[from=1-7, to=2-6]
            \arrow[""{name=3, anchor=center, inner sep=0}, from=1-1, to=3-3]
            \arrow[from=3-1, to=3-3]
            \arrow[from=1-3, to=3-3]
            \arrow[from=1-1, to=1-3]
            \arrow[from=1-1, to=3-1]
            \arrow["{\phi_{\st{X}}}"', shorten <=9pt, shorten >=6pt, Rightarrow, from=3-5, to=0]
            \arrow["{\phi_{\st{Y}}}", shorten <=13pt, shorten >=8pt, Rightarrow, from=1-7, to=0]
            \arrow["\alpha", shorten <=3pt, shorten >=3pt, Rightarrow, from=2-6, to=1]
            \arrow["\beta", shorten <=5pt, shorten >=5pt, Rightarrow, from=2-6, to=2]
            \arrow["{\lambda_{\st{X}}}"', shorten >=6pt, Rightarrow, from=3-1, to=3]
            \arrow["{\lambda_{\st{Y}}}"', shorten >=6pt, Rightarrow, from=1-3, to=3]
        \end{tikzcd}\]
    It is straightforward to check that the 2-cells satisfy the required condition,
    and that this assignment is an isomorphism.
    This identifies $(\st{X} \times_{\st{Z}} \st{Y} \to \orb{M})$ as the pullback.
\end{proof}

\subsection{The Untwisted Pullback-Pushforward Adjunction}
Let $f: \orb{N} \to \orb{M}$ be a map of \'etale stacks.
It defines a functor
\begin{alignat*}{1}
    f^{-1}: \StackEt(\orb{M}) & \to \StackEt(\orb{N})                                  \\
    (\st{X} \to \orb{M})      & \mapsto (\st{X} \times_{\orb{M}} \orb{N} \to \orb{N}).
\end{alignat*}

\begin{lemma}
    \label{lem:fInvPreservesStuff}
    The functor $f^{-1}$ preserves covers, pullbacks and binary products.
\end{lemma}
\begin{proof}
    The functor $f^{-1}$ amounts to computing a pullback in $\DiffSt$.
    Covers are collections of jointly surjective representable submersions,
    a condition which is preserved under pullback.

    By Lemma~\ref{lem:StackEtHasBinaryProdsAndPullbacks}, both pullbacks and
    binary products in $\StackEt(\orb{M})$ may be computed as pullbacks in
    $\DiffSt$. Limits commute with limits, in particular, pullback squares
    pull back to pullback squares along $f$. Hence $f^{-1}$ preserves pullbacks
    and binary products.
\end{proof}

Let $y:\st{Y} \to \orb{N}$ be an object in $\StackEt(\orb{N})$. The arrow category $y \downarrow f^{-1}$ is the
category of pairs $(x: \st{X} \to \orb{M},(g,\phi):y \to f^{-1}(x))$.
It may equivalently be thought of as the category of equivalence classes of diagrams
% https://q.uiver.app/?q=WzAsNCxbMCwwLCJcXHN0e1l9Il0sWzAsMSwiXFxvcmJ7Tn0iXSxbMSwxLCJcXG9yYntNfSJdLFsxLDAsIlxcc3R7WH0iXSxbMCwxLCJ5IiwyXSxbMywyLCJ4Il0sWzEsMiwiZiIsMl0sWzAsMywiZyJdLFszLDEsIlxccGhpIiwwLHsibGV2ZWwiOjJ9XV0=
\[\begin{tikzcd}
        {\st{Y}} & {\st{X}} \\
        {\orb{N}} & {\orb{M},}
        \arrow["y"', from=1-1, to=2-1]
        \arrow["x", from=1-2, to=2-2]
        \arrow["f"', from=2-1, to=2-2]
        \arrow["g", from=1-1, to=1-2]
        \arrow["\phi", Rightarrow, from=1-2, to=2-1]
    \end{tikzcd}\]
where $x$ and $f$ are fixed and the equivalence relation is generated by the obvious action of 2-cells
% https://q.uiver.app/?q=WzAsMixbMCwwLCJcXHN0e1l9Il0sWzEsMCwiXFxzdHtYfSJdLFswLDEsImcnIiwwLHsiY3VydmUiOi0yfV0sWzAsMSwiZyIsMix7ImN1cnZlIjoyfV0sWzIsMywiXFxnYW1tYSIsMCx7InNob3J0ZW4iOnsic291cmNlIjoyMCwidGFyZ2V0IjoyMH19XV0=
\[\begin{tikzcd}
        {\st{Y}} & {\st{X}.}
        \arrow[""{name=0, anchor=center, inner sep=0}, "{g'}", curve={height=-12pt}, from=1-1, to=1-2]
        \arrow[""{name=1, anchor=center, inner sep=0}, "g"', curve={height=12pt}, from=1-1, to=1-2]
        \arrow["\gamma", shorten <=3pt, shorten >=3pt, Rightarrow, from=0, to=1]
    \end{tikzcd}\]

\begin{lemma}
    \label{lem:yOverfInvIsCofiltered}
    The category $y \downarrow f^{-1}$ is cofiltered, ie.\ every
    finite diagram in $y \downarrow f^{-1}$ admits a cone.
\end{lemma}
\begin{proof}
    As any finite diagram may be built out of finite products and equalisers,
    it is enough to show that these admit cones.
    Hence we need only check cones over the empty diagram, binary products,
    and equalisers.

    The category $\StackEt(\orb{M})$ has both binary products and pullbacks by Lemma~\ref{lem:StackEtHasBinaryProdsAndPullbacks}.
    Lemma~\ref{lem:fInvPreservesStuff} says that $f^{-1}$ preserves these.
    Since equalisers may be built out of binary products and pullbacks, the same is true for equalisers.

    Giving a cone over the empty diagram amounts to showing
    $y \downarrow f^{-1}$ is non-empty.
    The category always contains the object
    % https://q.uiver.app/?q=WzAsNCxbMCwwLCJcXHN0e1l9Il0sWzAsMSwiXFxvcmJ7Tn0iXSxbMSwxLCJcXG9yYntNfSJdLFsxLDAsIlxcb3Jie019Il0sWzAsMSwieSIsMl0sWzEsMiwiZiIsMl0sWzMsMiwiXFxpZCJdLFswLDMsImYgXFxjb21wIHkiXSxbMywxLCJcXGlkIiwwLHsic2hvcnRlbiI6eyJzb3VyY2UiOjEwfSwibGV2ZWwiOjJ9XV0=
    \[\begin{tikzcd}
            {\st{Y}} & {\orb{M}} \\
            {\orb{N}} & {\orb{M}.}
            \arrow["y"', from=1-1, to=2-1]
            \arrow["f"', from=2-1, to=2-2]
            \arrow["\id", from=1-2, to=2-2]
            \arrow["{f \comp y}", from=1-1, to=1-2]
            \arrow["\id", shorten <=2pt, Rightarrow, from=1-2, to=2-1]
        \end{tikzcd}\]
    A cone over a pair of objects is given by their binary product as illustrated below
    % https://q.uiver.app/?q=WzAsOCxbMCwwLCJcXHN0e1l9Il0sWzAsMiwiXFxvcmJ7Tn0iXSxbMSwxLCJcXHN0e1h9Il0sWzIsMSwiXFxzdHtYfSciXSxbMSwwXSxbMiwwXSxbMSwyXSxbMiwyXSxbMCwxLCJ5IiwyXSxbMCwyXSxbMCwzLCJcXHF1YWQiLDEseyJsYWJlbF9wb3NpdGlvbiI6NzB9XSxbNCw1LCJcXHN0e1h9XFx0aW1lc197XFxvcmJ7TX19XFxzdHtYfSciLDFdLFs2LDcsIlxcb3Jie019IiwxLHsic3R5bGUiOnsiYm9keSI6eyJuYW1lIjoibm9uZSJ9LCJoZWFkIjp7Im5hbWUiOiJub25lIn19fV0sWzAsMTEsIiIsMCx7InNob3J0ZW4iOnsidGFyZ2V0IjoyMH0sImxldmVsIjoxfV0sWzExLDIsIiIsMSx7InNob3J0ZW4iOnsic291cmNlIjoyMH0sImxldmVsIjoxfV0sWzExLDMsIiIsMSx7InNob3J0ZW4iOnsic291cmNlIjoyMH0sImxldmVsIjoxfV0sWzEsMTIsImYiLDIseyJzaG9ydGVuIjp7InRhcmdldCI6MjB9LCJsZXZlbCI6MX1dLFsyLDEyLCIiLDIseyJzaG9ydGVuIjp7InRhcmdldCI6MjB9LCJsZXZlbCI6MX1dLFszLDEyLCIiLDIseyJzaG9ydGVuIjp7InRhcmdldCI6MjB9LCJsZXZlbCI6MX1dXQ==
    \[\begin{tikzcd}
            {\st{Y}} & {} & {} \\
            & {\st{X}} & {\st{X}'} \\
            {\orb{N}} & {} & {}
            \arrow["y"', from=1-1, to=3-1]
            \arrow[from=1-1, to=2-2]
            \arrow["\quad"{description, pos=0.65}, from=1-1, to=2-3]
            \arrow[""{name=0, anchor=center, inner sep=0}, "{\st{X}\times_{\orb{M}}\st{X}'}"{description,style={font=\small}}, from=1-2, to=1-3]
            \arrow[""{name=1, anchor=center, inner sep=0}, "{\orb{M}.}"{description,style={font=\small}}, draw=none, from=3-2, to=3-3]
            \arrow[shorten >=26pt, from=1-1, to=0]
            \arrow[shorten <=8pt, from=0, to=2-2]
            \arrow[shorten <=8pt, from=0, to=2-3]
            \arrow["f"', shorten >=8pt, from=3-1, to=1]
            \arrow[shorten >=8pt, from=2-2, to=1]
            \arrow[shorten >=8pt, from=2-3, to=1]
        \end{tikzcd}\]
    Similarly, the cone over a parallel pair of morphisms is provided by their equaliser.
\end{proof}

A classical result of topos theory (see~\cite[Rem C.2.3.7]{johnstone2002sketches} or~\cite[Sec 4]{shulman2012exact})
says that a functor $f^{-1}$ which satisfies the conditions proved in Lemma~\ref{lem:fInvPreservesStuff}
and Lemma~\ref{lem:yOverfInvIsCofiltered}
induces a pair of adjoint functors
\begin{equation*}
    f^*_0: \Shv(\orb{M}) \leftrightarrows \Shv(\orb{N}) :f_{\ast,0}
\end{equation*}
between the categories of $\VecInf$-valued sheaves.\footnote{We add the subscript 0 to distinguish these functors from the
    analogous functors induced on the categories of $\Cinf$-modules introduced below.}
The pushforward functor $f_{\ast,0}$ precomposes $\sh{F} \in \Shv(\orb{N})$ with $f^{-1}$
\begin{equation*}
    f_{\ast,0} \sh{F}: (\st{X} \to \orb{M}) \mapsto \sh{F}(\st{X} \times_{\orb{M}} \orb{N})
\end{equation*}
and the pullback $f^\ast_0$ takes $\sh{F}' \in \Shv(\orb{M})$ to the sheafification of the presheaf
\begin{equation*}
    \text{pre}f^\ast_0 \sh{F}: (y:\st{Y} \to \orb{N}) \mapsto
    \colim_{x \in {(y \downarrow f^{-1})}^\opp} \sh{F}(x).
\end{equation*}
By Lemma~\ref{lem:yOverfInvIsCofiltered}, the colimit in the above expression is filtered and thus commutes with
finite limits. Sheafification similarly commutes with finite limits~\cite[Lem A.4.4.6]{johnstone2002sketches}.
As a result, $f^*_0$ not only preserves colimits (as a left adjoint), it also preserves finite limits.

Recall that the category $\Shv(\orb{M})$ is equipped with a ring object $\CinfOf{\orb{M}}$, which sends an object
$(x:\st{X} \to \orb{M}) \in \StackEt(\orb{M})$ to the ring of smooth $\IC$-valued functions on $\st{X}$.
There is a natural map $f_\#: \CinfOf{\orb{M}} \to f_{\ast,0}\CinfOf{\orb{N}}$
with components
\[
    f_\#(\st{X} \to \orb{M}): \Cinf(\st{X}) \to \Cinf(\st{X} \times_{\orb{M}} \orb{N})
\]
given by restriction of functions along the map $\st{X} \times_{\orb{M}} \orb{N} \to \st{X}$.

The data $(f^\ast_0 \dashv f_{\ast,0}, f_\#):(\Shv(\orb{N}),\CinfOf{\orb{N}}) \to (\Shv(\orb{M}), \CinfOf{\orb{M}})$
is known as a \emph{morphism of ringed topoi}.
It is a standard result of topos theory~\cite[\href{https://stacks.math.columbia.edu/tag/03A4}{Chapter 03A4}]{stacks-project})
that a morphism of ringed topoi gives rise to an adjunction of module categories, so
we obtain the following.
\begin{proposition}
    \label{prop:PullbackPushforwardAdjunctionEtaleStacks}
    A map $f: \orb{N} \to \orb{M}$ of \'etale stacks gives rise to an adjoint pair
    of functors
    \[
        f^\ast: \ShC(\orb{M}) \leftrightarrows \ShC(\orb{N}): f_\ast.
    \]
    The left adjoint is equipped with a canonical monoidal structure
    \[
        f^\ast(\sh{F} \tensor_{\CinfOf{\orb{M}}} \sh{F}') \simeq
        f^\ast(\sh{F}) \tensor_{\CinfOf{\orb{N}}} f^\ast(\sh{F}').
    \]
\end{proposition}

These functors are defined as follows.
Using $f_\#$, the image of a $\CinfOf{\orb{N}}$-module $\sh{F}$ under $f_{\ast,0}$ naturally becomes a $\CinfOf{\orb{M}}$-module.
The resulting functor
$f_\ast: \ShC(\orb{N}) \to \ShC(\orb{M})$
has a left adjoint
\begin{align*}
    f^\ast: \ShC(\orb{M}) & \to \ShC(\orb{N})                                                            \\
    \sh{F}                & \mapsto \CinfOf{\orb{N}} \tensor_{f^\ast_0 \CinfOf{\orb{M}}} f^\ast_0\sh{F}.
\end{align*}
The tensor product in the expression above involves the morphism $f^\#:f^\ast_0 \CinfOf{\orb{M}} \to \CinfOf{\orb{N}}$,
which corresponds to $f_\#$ via the adjunction $(f^\ast_0 \dashv f_{\ast,0})$.

For a map between manifolds $f:M \to N$ (viewed as a map of \'etale stacks),
the adjunction $(f^\ast \dashv f_\ast)$ reduces to the usual adjunction of
$\Cinf$-module categories over manifolds.

\begin{lemma}
    \label{lem:PushfwdsAndPullbacksCompose}
    Consider a diagram
    % https://q.uiver.app/?q=WzAsMyxbMCwxLCJcXG9yYntNfSJdLFsyLDEsIlxcb3Jie099Il0sWzEsMCwiXFxvcmJ7Tn0iXSxbMCwyLCJmIl0sWzIsMSwiZyJdLFswLDEsImgiLDJdLFsyLDUsIlxcZ2FtbWEiLDIseyJzaG9ydGVuIjp7InRhcmdldCI6MjB9fV1d
    \[\begin{tikzcd}
            & {\orb{N}} \\
            {\orb{M}} && {\orb{O}}
            \arrow["f", from=2-1, to=1-2]
            \arrow["g", from=1-2, to=2-3]
            \arrow[""{name=0, anchor=center, inner sep=0}, "h"', from=2-1, to=2-3]
            \arrow["\gamma"', shorten >=3pt, Rightarrow, from=1-2, to=0]
        \end{tikzcd}\]
    of \'etale stacks.
    There are natural equivalences
    \begin{align*}
        g_\ast \comp f_\ast \simeq h_\ast &  &
        f^\ast \comp g^\ast \simeq h^\ast.
    \end{align*}
\end{lemma}
\begin{proof}
    As adjoints compose, it is enough to show this for the pushforwards.
    The pushforwards send a sheaf to its postcomposition with $f^{-1}/g^{-1}/h^{-1}$,
    so it suffices to provide a natural transformation
    \[
        \eta_{\gamma}:g^{-1} \comp f^{-1} \eqto h^{-1}.
    \]
    Its component $\eta_{\gamma}(\st{X} \to \orb{M})$ is given by the natural map
    \begin{align*}
        \orb{M} \times_{\orb{N}} \orb{N} \times_{\orb{O}} \st{X} \eqto
        \orb{M} \times_{\orb{O}} \st{X}.
    \end{align*}
    The 2-cell $\gamma$ may be used to build the relevant 2-cells between the 
    composites to $\orb{M}$. This makes the above map a morphism in 
    $\StackEt(\orb{M})$.
\end{proof}

\subsection{The Stack of Twisted Sheaves}
\begin{lemma}
    \label{lem:ShCIsAStack}
    The functor
    \begin{align*}
        \ShC: \Man^\opp & \to \VCat                             \\
        U               & \mapsto \ShC(U)                       \\
        (f:U \to V)     & \mapsto (f^\ast: \ShC(V) \to \ShC(U))
    \end{align*}
    is a stack (valued in linear, $\dirSum$-complete categories).
\end{lemma}
\begin{proof}
    We check the gluing condition. The functor $\ShC$ clearly sends coproducts to
    products, so it is enough to check gluing for a surjective submersion $Y \onto U$.
    We exhibit an inverse to the functor sending a sheaf to its corresponding descent datum
    over $Y$.
    Given a descent datum
    $(\sh{F} \in \ShC(Y), \phi:p_1^\ast\sh{F} \to p_2^\ast\sh{F} \in \ShC(Y^{[2]}))$,
    we construct a sheaf $\bar{\sh{F}} \in \ShC(U)$.
    Consider the full subcategory $\Op(Y/U) \to \Op(U)$ on open sets $V \subset U$
    equipped with a factorisation
    \[
        V \xto{v} Y \to U.
    \]
    It is a dense subsite of $\Op(U)$.
    The sheaf $\bar{\sh{F}}$ assigns to $V \xto{v} Y$ the global sections of the
    pullback sheaf $v^\ast \sh{F}$ over $V$.
    Given an inclusion $\iota: W \into V$, and two chosen lifts $w:W \to Y$, $v:V \to Y$,
    there is an induced map $(w,v \comp \iota):W \to Y^{[2]}$.
    The restriction map $\iota^\ast:\bar{\sh{F}}(v) \to \bar{\sh{F}}(w)$ is given by
    ${(w,v \comp \iota)}^\ast \phi$.
    Functoriality of this assignment is equivalent to the cocycle condition satisfied by
    $\phi$.
\end{proof}

Now let $\orb{M}$ be an \'etale stack.
Recall there is a forgetful functor $\Et(\orb{M}) \to \Man$. We can restrict $\ShC$ along
this functor to obtain the stack $\ShC$ over $\Et(\orb{M})$.
This stack extends to $\StackEt(\orb{M})$ via the dense inclusion $\Et(\orb{M}) \to \StackEt(\orb{M})$.

\begin{rmk}
    We may describe the resulting stack over $\StackEt(\orb{M})$ explicitly.
    For each isomorphism class of morphisms $[(f,\phi)]$, pick a representative $(f,\phi)$ and send it
    to the pullback functor $f^\ast$. Let $(f,\phi)$, $(g,\kappa)$ be composable morphisms as above, and
    denote the chosen representative of the class $[(f,\phi) \comp (g, \kappa)]$ by $(h,\lambda)$.
    Then there exists a unique 2-cell $\gamma: f \comp g \to h$ that identifies
    $(f,\phi) \comp (g,\kappa) \eqto (h,\lambda)$.
    This 2-cell induces a natural transformation of the associated pullback functors.
    That natural transformation is the value of the compositor
    $\eta_{(f,\phi),(g,\kappa)}$ which is part of the data of $\ShC$.
\end{rmk}

Let $\st{X}$ be a differentiable stack and consider
the bicategory $\DiffSt/\st{X}$.
Denote ${\ger{I}}_{\st{X}} \define \st{X} \times [\point/\IC^\times]$.
The category of maps
$(\st{W} \to \st{X}) \to (\ger{I}_{\st{X}} \to \st{X})$
in $\DiffSt/\st{X}$ is the category of $\IC^\times$-bundles over $\st{W}$.
The tensor product of $\IC^\times$-bundles gives $\ger{I}_{\st{X}}$ the structure
of a group object in $\DiffSt$.
The following definition is a special case of~\cite[Def 70]{schommer2011central}:
\begin{defn}
    \label{defn:ICtimesGerbe}
    A \emph{$\IC^\times$-gerbe over $\st{X}$} is a stack
    $\ger{G} \in \DiffSt/\st{X}$, equipped with
    a right action of $\ger{I}_{\st{X}}$,
    locally equivalent to $\ger{I}_{\st{X}} \to \st{X}$:
    There exists a representable submersion $f:\st{Y} \onto \st{X}$ such that the pullback
    $\ger{G} \times_{\st{X}} \st{Y} \in \DiffSt/\st{Y}$ is $[\point/\IC^\times]$-equivariantly
    equivalent to ${\ger{I}}_{\st{Y}}$ with the standard $\ger{I}_{\st{Y}}$-action.

    %    A 1-morphism of gerbes over $\st{X}$ is a morphism of the objects 
    %    $\DiffSt/\st{X}$ equipped with equivariance data for the $\ger{I}_{\st{X}}$-action.
    %    A 2-morphism between such 1-morphisms is a 2-morphism of the underlying 
    %    1-morphism which commutes with the equivariance data.
\end{defn}
From now on, we refer to $\IC^\times$-gerbes simply as gerbes
and suppress the subscript from the trivial gerbe $\ger{I}$ when the base is clear.
We say a gerbe $\ger{G}$ is \emph{trivial} if it admits an equivalence
$\ger{I} \eqto \ger{G}$.

Let $\ger{G} \to \orb{M}$ be a gerbe over an \'etale stack.
The gerbe $\ger{G}$ defines a groupoid-valued stack over the \'etale site of $\orb{M}$. It assigns
to $x:\st{X} \to \orb{M}$ the category of trivialisations of
$x^\ast\ger{G}=\ger{G}\times_{\orb{M}}\st{X}$:
\begin{align*}
    \ger{G}: {\StackEt(\orb{M})}^\opp & \to \Gpd                                    \\
    (x:\st{X} \to \orb{M})            & \mapsto \Hom(x^\ast\ger{I},x^\ast \ger{G}).
\end{align*}
For the trivial gerbe $\ger{I}$, this stack is the stack of $\IC^\times$-bundles,
which is equivalent to the stack of complex line bundles (and invertible
maps between them).
The sheaf of sections of a complex line bundle is a $\Cinf$-module,
so we obtain a map of stacks (valued in categories) $\ger{I} \to \ShC$.
The stack $\ger{I}$ acts on $\ShC$ by composing with the tensor product of
$\Cinf$-modules:
\[
    \ger{I} \times \ShC \to \ShC \times \ShC \to \ShC.
\]
The 2-cells exhibiting this as an action are induced by the associator 2-cells
for the tensor product on $\ShC$.

Every trivialisation $\tau \in \Hom(\ger{I},\ger{G})$ over $\orb{M}$ has an inverse
$\tau^{-1}$. Given another trivialisation $\kappa$ of $\ger{G}$, the composite
$\tau^{-1} \kappa$ is an endomorphism of $\ger{I}$, and thus equivalent to the data
of a line bundle $\mscr{L}_{\tau,\kappa}$ over $\orb{M}$.
\begin{notation}
    We write $\tau^{-1} \kappa \cdot \sh{F}$ to denote the action of the
    line bundle $\mscr{L}_{\tau, \kappa}$ on a $\Cinf$-module sheaf $\sh{F} \in \ShC(\orb{M})$.
\end{notation}
Let $\ger{G}$ be a trivial gerbe over $\orb{M}$.
We define the category $\Shv^{\ger{G}}(\orb{M})$:
Its objects are pairs $(\tau: \ger{I} \to \ger{G}, \sh{F} \in \ShC(\orb{M}))$,
which we denote by $\tau \cdot \sh{F}$.
The $\Hom$-sets are given by
\[
    \Hom_{\Shv^{\ger{G}}(\orb{M})}(\tau \cdot \sh{F}, \tau' \cdot \sh{F}') =
    \Hom_{\ShC(\orb{M})}(\sh{F}, \tau^{-1} \tau' \cdot \sh{F}').
\]
A pair of morphisms
% https://q.uiver.app/?q=WzAsMyxbMCwwLCJcXHRhdSBcXGNkb3QgXFxzaHtGfSJdLFsxLDAsIlxcdGF1J1xcY2RvdCBcXHNoe0Z9JyJdLFsyLDAsIlxcdGF1JycgXFxjZG90IFxcc2h7Rn0nJyJdLFswLDEsIlxccGhpIl0sWzEsMiwiXFx2YXJwaGkiXV0=
\[\begin{tikzcd}
        {\tau \cdot \sh{F}} & {\tau'\cdot \sh{F}'} & {\tau'' \cdot \sh{F}''}
        \arrow["\phi", from=1-1, to=1-2]
        \arrow["\varphi", from=1-2, to=1-3]
    \end{tikzcd}\]
is given by morphisms of sheaves $\bar{\phi}:\sh{F} \to \tau^{-1}\tau \cdot \sh{F}'$
and $\bar{\varphi}: \sh{F}' \to {\tau'}^{-1} \tau'' \cdot \sh{F}''$.
Their composite $\varphi \comp \phi: \tau \cdot \sh{F} \to \tau'' \cdot \sh{F}''$
is given by the morphism of sheaves
% https://q.uiver.app/?q=WzAsNCxbMCwwLCJcXHNoe0Z9Il0sWzEsMCwiXFx0YXVeey0xfVxcdGF1J1xcY2RvdCBcXHNoe0Z9JyJdLFsyLDAsIlxcdGF1XnstMX1cXHRhdSd7XFx0YXUnfV57LTF9XFx0YXUnJyBcXGNkb3QgXFxzaHtGfScnIl0sWzMsMCwiXFx0YXVeey0xfVxcdGF1JydcXGNkb3RcXHNoe0Z9JyciXSxbMCwxLCJcXHBoaSJdLFsxLDIsIlxcdmFycGhpIl0sWzIsMywiXFxzaW1lcSJdXQ==
\[\begin{tikzcd}
        {\sh{F}} & {\tau^{-1}\tau'\cdot \sh{F}'} & {\tau^{-1}\tau'\cdot{\tau'}^{-1}\tau'' \cdot \sh{F}''} & {\tau^{-1}\tau''\cdot\sh{F}''}.
        \arrow["\bar{\phi}", from=1-1, to=1-2]
        \arrow["\bar{\varphi}", from=1-2, to=1-3]
        \arrow["\simeq", from=1-3, to=1-4]
    \end{tikzcd}
\]

Denote by $\lambda: \ger{G}' \eqto \ger{G}$ a morphism of gerbes
(necessarily an isomorphism). It induces an equivalence
\begin{align*}
    \Shv^{\ger{G}}(\orb{M}) & \eqto \Shv^{\ger{G}'}(\orb{M})   \\
    \tau \cdot \sh{F}       & \mapsto \lambda\tau \cdot \sh{F}
\end{align*}
with inverse given by $\lambda^{-1}$.
In particular, a trivialisation $\tau$ of $\ger{G}$ induces an equivalence
\[
    \Shv^{\ger{I}}(\orb{M}) \eqto \Shv^{\ger{G}}(\orb{M}).
\]

\begin{lemma}
    \label{lem:tauCdotIsEquivalence}
    Denote by $\tau_0$ the canonical trivialisation of $\ger{I}$.
    The functor
    \begin{align*}
        \tau_0 \cdot -: \ShC(\orb{M}) & \to \Shv^{\ger{I}}(\orb{M}) & \\
        \sh{F}                        & \mapsto \tau_0 \cdot \sh{F} &
    \end{align*}
    is an equivalence.
\end{lemma}
\begin{proof}
    Full faithfulness uses the natural equivalence $\tau_0^{-1} \tau_0 \simeq \Cinf$:
    \[
        \Hom_{\Shv^\ger{G}(\orb{M})}(\tau_0 \cdot \sh{F}, \tau_0 \cdot \sh{F}') \define
        \Hom_{\ShC(\orb{M})}(\sh{F},\tau_0^{-1}\tau_0 \cdot \sh{F}') \simeq
        \Hom_{\ShC(\orb{M})}(\sh{F},\sh{F}').
    \]
    The functor is also essentially surjective: the object
    $\kappa \cdot \sh{F} \in \Shv^{\ger{G}}(\orb{M})$ is equivalent
    to $\tau_0 \cdot (\tau_0^{-1}\kappa \cdot \sh{F})$
    via the map of sheaves
    \[
        \sh{F} \xto{\id} \sh{F} \eqto \kappa^{-1} \tau_0 \tau_0^{-1} \kappa \cdot
        \sh{F}. \qedhere
    \]
\end{proof}

We can compose the functor $\tau_0 \cdot -$ with the equivalence
$\Shv^{\ger{I}}(\orb{M}) \eqto \Shv^{\ger{G}}(\orb{M})$
induced by a trivialisation $\tau$ to obtain an equivalence
\begin{align*}
    \tau \cdot -: \ShC(\orb{M}) & \eqto \Shv^\ger{G}(\orb{M}) \\
    \sh{F}                      & \mapsto \tau \cdot \sh{F}.
\end{align*}

Let $f:\orb{N} \to \orb{M}$ be a map of \'etale stacks.
A trivialisation $\tau$ of $\ger{G} \to \orb{M}$ pulls back to a trivialisation
$f^\ast\tau$ of $f^\ast \ger{G} \to \orb{N}$.
We define a pullback functor (keeping the same symbol as for $\Cinf$-modules)
\begin{align*}
    f^\ast:\Shv^{\ger{G}}(\orb{M}) & \to \Shv^{f^\ast\ger{G}}(\orb{N})        \\
    \kappa \cdot \sh{F}            & \mapsto f^\ast\kappa \cdot f^\ast\sh{F}.
\end{align*}

We denote this pullback functor by the same functor as for sheaves of $\Cinf$-modules.
For any choice of trivialisation $\tau$, the square
% https://q.uiver.app/?q=WzAsNCxbMCwwLCJcXFNodl57XFxnZXJ7R319KFxcb3Jie019KSJdLFsxLDAsIlxcU2h2XntmXlxcYXN0XFxnZXJ7R319KFxcb3Jie059KSJdLFswLDEsIlxcU2hDKFxcb3Jie019KSJdLFsxLDEsIlxcU2hDKFxcb3Jie059KSJdLFswLDEsImZeXFxhc3QiXSxbMiwzLCJmXlxcYXN0IiwyXSxbMiwwLCJcXHRhdSBcXGNkb3QgLSJdLFszLDEsImZeXFxhc3RcXHRhdSBcXGNkb3QgLSIsMl1d
\[\begin{tikzcd}
        {\Shv^{\ger{G}}(\orb{M})} & {\Shv^{f^\ast\ger{G}}(\orb{N})} \\
        {\ShC(\orb{M})} & {\ShC(\orb{N})}
        \arrow["{f^\ast}", from=1-1, to=1-2]
        \arrow["{f^\ast}"', from=2-1, to=2-2]
        \arrow["{\tau \cdot -}", from=2-1, to=1-1]
        \arrow["{f^\ast\tau \cdot -}"', from=2-2, to=1-2]
    \end{tikzcd}\]
commutes. Now Lemma~\ref{lem:PushfwdsAndPullbacksCompose} implies
that a 2-cell $\gamma:g \comp f \eqto h$ induces a natural equivalence
\[
    f^\ast \comp g^\ast \simeq h^\ast.
\]

Let $\pi:M \onto \orb{M}$ be an atlas such that the
the pullback gerbe $\pi^\ast \ger{G}$ is trivial.
We define a $\VCat$-valued 2-presheaf
\begin{align*}
    \Shv^\ger{G}: {\Et(M/\orb{M})}^\opp & \to \VCat                                   \\
    (x:X \to M)                         & \mapsto \Shv^{x^\ast\ger{G}}(X)             \\
    ((f,\phi):x \to y)                  & \mapsto (f^\ast:\Shv^{y^\ast\ger{G}}(Y) \to
    \Shv^{x^\ast\ger{G}}(X)).
\end{align*}
\begin{rmk}
    We choose $\Et(M/\orb{M})$ because it is a dense subsite of $\StackEt(\orb{M})$
    where every object admits a trivialisation of the pullback gerbe. One could instead
    define this presheaf directly on the (dense) subsite of objects over which the
    pullback gerbe admits a trivialisation.
\end{rmk}

\begin{lemma}
    \label{lem:ShvGIsAStack}
    The functor $\Shv^\ger{G}$ is a stack over $\Et(M/\orb{M})$.
\end{lemma}
\begin{proof}
    Let $\tilde{x}:X \to M$ be an object of $\Et(M/\orb{M})$,
    and $\tilde{y}:Y \to M$ an object covering $\tilde{x}$ via a morphism
    $(f,\phi):\tilde{y} \onto \tilde{x}$. Note in particular, that this identifies
    $f:Y \onto X$ as an atlas of $X$.
    A descent datum for $\Shv^\ger{G}$ over the cover $(f,\phi):\tilde{y} \onto \tilde{x}$
    is given by a pair
    $(\kappa\cdot\sh{F} \in \Shv^{\tilde{y}^\ast\pi^\ast\ger{G}}(Y),
        \varphi:p_1^\ast(\kappa\cdot\sh{F}) \eqto p_2^\ast(\kappa\cdot\sh{F}))$,
    where $p_1,p_2:Y^{[2]} \to Y$ are the two projections.
    The isomorphism $\varphi$ satisfies a cocycle condition
    over $Y^{[3]}$.

    Without loss of generality, we may fix the trivialisation used in the descent datum
    to be the pullback of a chosen trivialisation $\tau$ of $\ger{G}$ over $X$: $\kappa = f^\ast \tau$.
    The isomorphism $\varphi$ then corresponds to an isomorphism of sheaves
    $p_1^\ast \sh{F} \eqto (p_1^\ast f^\ast \tau^{-1} p_2^\ast f^\ast \tau) \cdot p_2^\ast\sh{F}
        \simeq p_2^\ast \sh{F}$.\footnote{
        The last equivalence uses that $\tau^{-1}$ and $\tau$ are pulled back along the same map, and thus cancel.
    }

    We now describe a twisted sheaf
    $\tilde{x}^\ast\tau \cdot \bar{\sh{F}} \in \Shv^{\tilde{x}^\ast\pi^\ast\ger{G}}(X)$
    which is obtained by gluing this descent datum.
    This is analogous to the construction in the proof of Lemma~\ref{lem:ShCIsAStack}.
    We view $\bar{\sh{F}} \in \ShC(X)$ as a sheaf over the site
    $\Et(Y/X)$ associated to the atlas $f:Y \onto X$.
    The sheaf $\bar{\sh{F}}$ assigns to an object
    \begin{align*}
        (u: U \to Y) & \mapsto \sh{F}(u) = u^\ast\sh{F}(U),
    \end{align*}
    the global sections of the pullback sheaf $u^\ast\sh{F}$ over $U$.
    A morphism $v \to u$ in $\Et(Y/X)$ is a map $g:V \to U$ such that the composites
    \[
        f\comp v = f \comp u \comp g: V \to X
    \]
    are equal.
    As a result, the two maps $u \comp g, v:V \to Y$ assemble into a map
    $(u \comp g,v):V \to Y^{[2]}$.
    We pull back $\varphi$ along this map to obtain an isomorphism of sheaves
    \[
        g^\ast {u}^\ast \sh{F} \simeq
        {(u \comp g)}^\ast \sh{F} \eqto
        v^\ast\sh{F}
    \]
    over $V$.
    The sheaf $\bar{\sh{F}}$ assigns to the morphism $g$ the map
    \[
        u^\ast\sh{F}(U) = \sh{F}(u:U \to Y) \xto{\sh{F}(g)} \sh{F}(u \comp g:V \to Y) =
        (g^\ast u^\ast\sh{F})(V) \eqto v^\ast\sh{F}(V).
    \]
    The fact that this assignment is functorial is equivalent to the cocycle condition satisfied
    by $\varphi$ over $Y^{[3]}$.

    We have constructed a functor
    \[
        \Desc(Y/X) \to \Shv^\ger{G}(X),
    \]
    where $\Desc(Y/X)$ denotes the category of descent data of $\Shv^\ger{G}$
    for the cover $(f,\phi):\tilde{y} \onto \tilde{x}$ ---
    recall these were objects $\tilde{y}:Y \to M$ and
    $\tilde{x}:X \to M$.
    It is straightforward to check the composite $\Shv^\ger{G}(X) \to \Desc(Y/X) \to \Shv^\ger{G}(X)$
    is equivalent to the identity functor.
    The composite $\Desc(Y/X) \to \Shv^\ger{G} \to \Desc(Y/X)$ is given by
    \[
        (\sh{F}, \varphi) \mapsto \tau \cdot \bar{\sh{F}} \mapsto (f^\ast\bar{\sh{F}},
        \bar{\varphi}=\id_{p^\ast\bar{\sh{F}}}),
    \]
    where $\bar{\sh{F}}$ is the sheaf constructed above, and $p: Y^{[2]}=Y\times_X Y \to X$
    is the canonical projection.
    The sheaf $f^\ast\bar{\sh{F}}$ sends
    \[
        (u:U \to Y) \mapsto \bar{\sh{F}}(f\comp u:U \to Y \to X)=\sh{F}(u':U \to Y)
    \]
    where $u':U \to Y$ is a choice of factorisation of $f \comp u$. The
    induced map $(u,u'):U \to Y^{[2]}$ may be used to pull back $\varphi$
    and obtain an isomorphism
    \[
        u^\ast \sh{F} \eqto u'^\ast \sh{F}
    \]
    These isomorphisms assemble into an equivalence of descent data
    $(\sh{F},\varphi) \to (f^\ast\bar{\sh{F}},\id )$
    with components
    \[
        \sh{F}(u:U \to Y) = u^\ast \sh{F}(U) \eqto u'^\ast \sh{F}(U) =
        \sh{F}(u':U \to Y)=f^\ast\bar{\sh{F}}(u:U \to Y).
    \]
    The fact that this defines a morphism of descent data is again a
    consequence of the cocycle condition satisfied by $\varphi$.
\end{proof}

\begin{defn}
    Let $\ger{G} \to \orb{M}$ be a gerbe over an \'etale stack.
    We call $\Shv^\ger{G}$ the
    \emph{stack of $\ger{G}$-twisted sheaves over $\orb{M}$}.
    A \emph{$\ger{G}$-twisted sheaf over $\orb{M}$} is a global section
    $\sh{F} \in \Shv^\ger{G}(\id_{\orb{M}})$.
\end{defn}

\begin{rmk}
    The stack $\Shv^\ger{G}$ over $\orb{M}$ is equipped with a map
    \begin{align*}
        \ger{G} \times \ShC & \to \Shv^\ger{G}          \\
        (\tau,\sh{F})       & \mapsto \tau \cdot \sh{F}
    \end{align*}
    and natural equivalences between the two composites
    \[
        \ger{G} \times \ger{I} \times \ShC \rightrightarrows
        \ger{G} \times \ShC \to \Shv^\ger{G},
    \]
    whose components are isomorphisms
    $(\tau \cdot \mscr{L}) \cdot \sh{F} = \tau \cdot (\mscr{L} \cdot \sh{F})$.
    These equivalences satisfy a cocycle condition which exhibits
    $\Shv^\ger{G}$ as the colimit
    % https://q.uiver.app/?q=WzAsNCxbMSwwLCJcXHF1YWQgXFxnZXJ7R30gXFx0aW1lcyBcXGdlcntJfSBcXHRpbWVzIFxcZ2Vye0l9IFxcdGltZXMgXFxTaEMiXSxbMiwwLCJcXGdlcntHfSBcXHRpbWVzIFxcZ2Vye0l9IFxcdGltZXMgXFxTaEMiXSxbMywwLCJcXGdlcntHfSBcXHRpbWVzIFxcU2hDIl0sWzAsMCwiXFxnZXJ7R30gXFx0aW1lc197XFxnZXJ7SX19IFxcU2hDPSJdLFsxLDIsIiIsMCx7Im9mZnNldCI6MX1dLFsxLDIsIiIsMix7Im9mZnNldCI6LTF9XSxbMCwxXSxbMCwxLCIiLDIseyJvZmZzZXQiOjJ9XSxbMCwxLCIiLDIseyJvZmZzZXQiOi0yfV0sWzMsMCwiXFxjb2xpbSIsMyx7InN0eWxlIjp7ImJvZHkiOnsibmFtZSI6Im5vbmUifSwiaGVhZCI6eyJuYW1lIjoibm9uZSJ9fX1dXQ==
    \[\begin{tikzcd}[column sep=small]
            {\ger{G} \times_{\ger{I}} \ShC=} & {\quad \ger{G} \times \ger{I} \times \ger{I} \times \ShC} & {\ger{G} \times \ger{I} \times \ShC} & {\ger{G} \times \ShC.}
            \arrow[shift right=1, from=1-3, to=1-4]
            \arrow[shift left=1, from=1-3, to=1-4]
            \arrow[from=1-2, to=1-3]
            \arrow[shift right=2, from=1-2, to=1-3]
            \arrow[shift left=2, from=1-2, to=1-3]
            \arrow["\quad \colim"{marking}, draw=none, from=1-1, to=1-2]
        \end{tikzcd}\]
    This is akin to the construction of a line bundle $P \times_{\IC^\times} \IC$ associated to a $\IC^\times$-bundle $P$.
\end{rmk}

\subsection{The Twisted Pullback-Pushforward Adjunction}
Let $\ger{G} \to \orb{M}$ be a trivial gerbe over an \'etale stack and
$f:\orb{N} \to \orb{M}$ a map of \'etale stacks.
Pick a trivialisation $\tau$ of $\ger{G}$. We define a functor
\begin{align*}
    f_\ast^\tau:\Shv^{f^\ast\ger{G}}(\orb{N}) & \to \Shv^{\ger{G}}(\orb{M})                                       \\
    \kappa \cdot \sh{F}                       & \mapsto \tau \cdot f_\ast((f^\ast\tau^{-1}) \kappa \cdot \sh{F}).
\end{align*}

\begin{lem}
    The functor $f_\ast^\tau$ is right adjoint to $f^\ast$.
\end{lem}
\begin{proof}
    The equivalences $\tau \cdot -: \ShC(\orb{M}) \to \Shv^\ger{G}(\orb{M})$
    and $f^\ast \tau \cdot -: \ShC(\orb{N}) \to \Shv^{f^\ast \ger{G}}(\orb{N})$
    identify the functors $f^\ast$ and $f_\ast$ with the pullback
    and pushforward functors of Proposition~\ref{prop:PullbackPushforwardAdjunctionEtaleStacks}.
\end{proof}
The adjoint to a functor is unique up to natural isomorphism, so we drop
the superscript $\tau$ from $f_\ast^\tau$.
Below, we transport this adjunction to the case of non-trivial gerbes.
If the gerbe $\ger{G}$ is non-trivial, there is no $\ger{G}$-twisted
sheaf over $\orb{M}$ of the form $\tau \cdot \sh{F}$.
We can instead describe them by descent data. A twisted sheaf is then given
by the data of
a cover $\pi:Y \onto \orb{M}$ with a trivialisation $\tau$ of $\pi^\ast Y$,
a sheaf $\sh{F}$ over $Y$, and an isomorphism of sheaves
$\varphi: p_1^\ast \sh{F} \to p_1^\ast \tau^{-1} p_2^\ast \tau \cdot p_2^\ast\sh{F}$
over $Y^{[2]}$, such that
% https://q.uiver.app/?q=WzAsNCxbMCwwLCJwXzFeKlxcc2h7Rn0iXSxbMCwxLCJwXzFeKlxcdGF1XnstMX1wXzJeXFxhc3RcXHRhdVxcY2RvdCBwXzJeKlxcc2h7Rn0iXSxbMiwwLCJwXzFeKlxcdGF1XnstMX1wXzNeXFxhc3RcXHRhdVxcY2RvdCBwXzNeKlxcc2h7Rn0iXSxbMiwxLCJwXzFeKlxcdGF1XnstMX1wXzJeXFxhc3RcXHRhdVxcY2RvdCBwXzJeXFxhc3RcXHRhdV57LTF9IHBfM15cXGFzdFxcdGF1XFxjZG90IHBfM14qXFxzaHtGfSJdLFswLDEsInBfezEyfV4qXFx2YXJwaGkiLDJdLFsxLDMsInBfMV4qXFx0YXVeey0xfXBfMl5cXGFzdFxcdGF1XFxjZG90IHBfezIzfV4qXFx2YXJwaGkiLDJdLFszLDIsIlxcc2ltZXEiLDMseyJzdHlsZSI6eyJib2R5Ijp7Im5hbWUiOiJub25lIn0sImhlYWQiOnsibmFtZSI6Im5vbmUifX19XSxbMCwyLCJwX3sxM31eXFxhc3RcXHZhcnBoaSJdXQ==
\[\begin{tikzcd}
        {p_1^*\sh{F}} && {p_1^*\tau^{-1}p_3^\ast\tau\cdot p_3^*\sh{F}} \\
        {p_1^*\tau^{-1}p_2^\ast\tau\cdot p_2^*\sh{F}} && {p_1^*\tau^{-1}p_2^\ast\tau\cdot p_2^\ast\tau^{-1} p_3^\ast\tau\cdot p_3^*\sh{F}}
        \arrow["{p_{12}^*\varphi}"', from=1-1, to=2-1]
        \arrow["{p_1^*\tau^{-1}p_2^\ast\tau\cdot p_{23}^*\varphi}"', from=2-1, to=2-3]
        \arrow["\simeq"{marking}, draw=none, from=2-3, to=1-3]
        \arrow["{p_{13}^\ast\varphi}", from=1-1, to=1-3]
    \end{tikzcd}\]
commutes.

By Lemma~\ref{lem:ShvGIsAStack}, $\Shv^{\ger{G}}(\orb{M}) \simeq \Desc(Y/\orb{M})$
where $\pi:Y \onto \orb{M}$ is any cover such that $\pi^\ast\ger{G}$ is trivial.
This cover pulls back to a cover $f^\ast Y = Y \times_{\orb{M}} \orb{N}
    \onto \orb{N}$ and trivialisations $\tau$ of $\pi^\ast\ger{G}$ over $Y$ pull back to
trivialisations $f^\ast \tau$ of $f^\ast\pi^\ast \ger{G}$ over $f^\ast Y$.\footnote{
    Note that, even when $Y \onto \orb{M}$ is an atlas, this is not necessarily
    true for $f^\ast Y \onto \orb{N}$:
    the pullback $f^\ast Y$ is not representable by a manifold unless
    $f: \orb{M} \to \orb{N}$ is a representable map.
}
Pullbacks commute, so there are equivalences
${(f^\ast Y)}^{[n]} \simeq f^\ast(Y^{[n]})$,
and $f:\orb{N} \to \orb{M}$ extends to a map of simplicial diagrams
% https://q.uiver.app/?q=WzAsOCxbMiwxLCJZIl0sWzMsMSwiXFxvcmJ7TX0iXSxbMywwLCJcXG9yYntOfSJdLFsyLDAsImZeXFxhc3QgWSJdLFsxLDAsImZeXFxhc3QgWV57WzJdfSJdLFsxLDEsIllee1syXX0iXSxbMCwwLCJmXlxcYXN0IFlee1szXX0iXSxbMCwxLCJZXntbM119Il0sWzIsMV0sWzMsMl0sWzAsMV0sWzMsMF0sWzQsMywiIiwyLHsib2Zmc2V0IjotMX1dLFs0LDMsIiIsMix7Im9mZnNldCI6MX1dLFs1LDAsIiIsMix7Im9mZnNldCI6MX1dLFs1LDAsIiIsMix7Im9mZnNldCI6LTF9XSxbNiw0LCIiLDIseyJvZmZzZXQiOi0yfV0sWzYsNCwiIiwyLHsib2Zmc2V0IjoyfV0sWzYsNF0sWzcsNV0sWzcsNSwiIiwyLHsib2Zmc2V0IjoyfV0sWzcsNSwiIiwyLHsib2Zmc2V0IjotMn1dLFs2LDddLFs0LDVdXQ==
\[\begin{tikzcd}
        {f^\ast Y^{[3]}} & {f^\ast Y^{[2]}} & {f^\ast Y} & {\orb{N}} \\
        {Y^{[3]}} & {Y^{[2]}} & Y & {\orb{M}.}
        \arrow[from=1-4, to=2-4]
        \arrow[from=1-3, to=1-4]
        \arrow[from=2-3, to=2-4]
        \arrow[from=1-3, to=2-3]
        \arrow[shift left=1, from=1-2, to=1-3]
        \arrow[shift right=1, from=1-2, to=1-3]
        \arrow[shift right=1, from=2-2, to=2-3]
        \arrow[shift left=1, from=2-2, to=2-3]
        \arrow[shift left=2, from=1-1, to=1-2]
        \arrow[shift right=2, from=1-1, to=1-2]
        \arrow[from=1-1, to=1-2]
        \arrow[from=2-1, to=2-2]
        \arrow[shift right=2, from=2-1, to=2-2]
        \arrow[shift left=2, from=2-1, to=2-2]
        \arrow[from=1-1, to=2-1]
        \arrow[from=1-2, to=2-2]
    \end{tikzcd}\]
In other words, the map $f$ commutes with the simplicial maps
$p_{ij\cdots k}:Y^{[m]} \to Y^{[n]}$, and the squares formed by corresponding
maps in the two rows are pullback squares.

We establish the adjunction
\[
    \Shv^{\ger{G}}(\orb{M}) \leftrightarrows \Shv^{f^\ast\ger{G}}(\orb{N})
\]
through a pair of functors
\[
    f^\ast: \Desc(Y/\orb{M}) \leftrightarrows \Desc(f^\ast Y/\orb{N}) :f_\ast.
\]
Denote by $\mscr{L}$ the line bundle corresponding to
$p_2^\ast\tau^{-1}p_1^\ast\tau$ over $Y^{[2]}$. It is equipped with a
natural equivalence
$p_{12}^\ast\mscr{L} \tensor p_{23}^\ast\mscr{L} \simeq p_{13}^\ast\mscr{L}$ over $Y^{[3]}$.
A descent datum in $\Desc(Y/\orb{M})$ is uniquely
equivalent to a descent datum of the form
\[
    \left(\sh{F} \in \ShC(Y), \varphi:p_1^\ast\sh{F} \eqto \mscr{L} \cdot p_2^\ast\sh{F}
    \in \ShC(Y^{[2]})\right),
\]
and similarly for $\Desc(f^\ast Y/\orb{N})$.

We define the pullback functor by
\begin{alignat*}{1}
    f^\ast: \Desc(Y/\orb{M}) & \to \Desc(f^\ast Y/\orb{N})                     \\
    (\sh{F}, \varphi)        & \mapsto (f^\ast \sh{F}, \tilde{f}^\ast\varphi),
\end{alignat*}
where $\tilde{f}^\ast\varphi$ is the composite
\[
    p_1^\ast f^\ast \sh{F} \simeq f^\ast p_1^\ast \sh{F} \xto{f^\ast \varphi}
    f^\ast(\mscr{L} \cdot p_2^\ast \sh{F}) \eqto
    f^\ast\mscr{L} \cdot f^\ast p_2^\ast \sh{F} \simeq
    f^\ast \mscr{L} \cdot p_2^\ast f^\ast \sh{F}.
\]
(The first and last equivalences come from the fact that $p_i$ commutes with $f$,
the third equivalence is given by the monoidal structure of $f^\ast$.)
Up to composition by natural equivalences, the cocycle condition for
$\tilde{f}^\ast\varphi$ is the image of the coycle condition for $\varphi$ under the
functor $f^\ast:\ShC(Y^{[3]}) \to \ShC(f^\ast Y^{[3]})$, and thus
$(f^\ast\sh{F},\tilde{f}^\ast\varphi)$ is indeed a valid descent datum.

\begin{lemma}[Base change]
    \label{lem:baseChange}
    Let
    % https://q.uiver.app/?q=WzAsNCxbMCwwLCJmXlxcYXN0XFxzdHtZfSJdLFsxLDAsIlxcc3R7WH0iXSxbMCwxLCJcXHN0e1l9Il0sWzEsMSwiXFxzdHtafSJdLFsyLDMsInAiLDIseyJzdHlsZSI6eyJoZWFkIjp7Im5hbWUiOiJlcGkifX19XSxbMSwzLCJmIl0sWzAsMSwicCIsMCx7InN0eWxlIjp7ImhlYWQiOnsibmFtZSI6ImVwaSJ9fX1dLFswLDIsImYiLDJdLFswLDMsIiIsMSx7InN0eWxlIjp7Im5hbWUiOiJjb3JuZXIifX1dLFsxLDIsIiIsMSx7InNob3J0ZW4iOnsic291cmNlIjoyMCwidGFyZ2V0IjoyMH0sImxldmVsIjoyfV1d
    \[\begin{tikzcd}
            {f^\ast\st{Y}} & {\st{X}} \\
            {\st{Y}} & {\st{Z}}
            \arrow["p"', two heads, from=2-1, to=2-2]
            \arrow["f", from=1-2, to=2-2]
            \arrow["p", two heads, from=1-1, to=1-2]
            \arrow["f"', from=1-1, to=2-1]
            \arrow["\lrcorner"{anchor=center, pos=0.125}, draw=none, from=1-1, to=2-2]
            \arrow[shorten <=4pt, shorten >=4pt, Rightarrow, from=1-2, to=2-1]
        \end{tikzcd}\]
    be a pullback square of \'etale stacks where $p$ is representably \'etale.   
    Then the natural transformation
    \[
        p^\ast f_\ast \to p^\ast f_\ast p_\ast p^\ast \simeq p^\ast p_\ast
        f_\ast p^\ast \to f_\ast p^\ast
    \]
    is an equivalence of functors $\ShC(\st{X}) \to \ShC(\st{Y})$.
\end{lemma}
\begin{proof}
    The component of this natural transformation on a sheaf $\sh{F} \in \ShC(\st{X})$
    is given by a morphism $p^\ast f_\ast \sh{F} \to f_\ast p^\ast \sh{F}$ in
    $\ShC(\st{Y})$.
    It is enough to exhibit this as an isomorphism on an \'etale atlas of $\st{Y}$.

    Pick an etale atlas $Z \onto \st{Z}$ and an \'etale atlas $X \onto \st{X}$
    fine enough such that the map $f:\st{X} \to \st{Z}$ is represented by a map
    $\bar{f}:X \to Z$. This yields atlases $Y \define p^\ast Z \onto \st{Y}$
    and $f^\ast Y = f^\ast p^\ast Z \onto f^\ast\st{Y}$.
    The resulting pullback square of manifolds
    \[\begin{tikzcd}
            {f^\ast Y} & {X} \\
            {Y} & {Z}
            \arrow["\bar{p}"', two heads, from=2-1, to=2-2]
            \arrow["\bar{f}", from=1-2, to=2-2]
            \arrow["\bar{p}", two heads, from=1-1, to=1-2]
            \arrow["\bar{f}"', from=1-1, to=2-1]
            \arrow["\lrcorner"{anchor=center, pos=0.125}, draw=none, from=1-1, to=2-2]
        \end{tikzcd}\]
    maps to the pull back square in the statement via representably \'etale maps.
    It also satisfies the property that $\bar{p}$ is \'etale, and thus
    the classical base change theorem implies
    \[
        \bar{p}^\ast \bar{f}_\ast \to \bar{f}_\ast \bar{p}^\ast
    \]
    is a natural equivalence in $\ShC(Y)$.
    But this is exactly the image of the natural transformation
    \[
        p^\ast f_\ast \to f_\ast p^\ast
    \]
    under the conservative functor $\ShC(\st{Y}) \to \ShC(Y)$.
\end{proof}

\begin{lemma}[Projection Formula]
    \label{lem:projectionFormula}
    Let $f:\st{X} \to \st{Y}$ be a map of \'etale stacks, $\sh{F} \in \ShC(\st{X})$ a
    sheaf of $\Cinf$-modules, $\mscr{L} \in \ShC(\st{Y})$ a line bundle.
    Then the natural map
    \[
        \mscr{L} \tensor f_\ast \sh{F} \to f_\ast(f^\ast \mscr{L} \tensor \sh{F})
    \]
    is an isomorphism. This isomorphism is compatible with the monoidal structure:
    given another line bundle $\mscr{L}' \in \ShC(\st{Y})$,
    % https://q.uiver.app/?q=WzAsNCxbMCwxLCJmX1xcYXN0KGZeXFxhc3QoXFxtc2Nye0x9IFxcdGVuc29yIFxcbXNjcntMfScpIFxcdGVuc29yIFxcc2h7Rn0pIl0sWzAsMCwiXFxtc2Nye0x9IFxcdGVuc29yIFxcbXNjcntMfSdcXHRlbnNvciBmX1xcYXN0IFxcc2h7Rn0iXSxbMSwwLCJcXG1zY3J7TH0gXFx0ZW5zb3IgZl9cXGFzdChmXlxcYXN0XFxtc2Nye0x9J1xcdGVuc29yXFxzaHtGfSkiXSxbMSwxLCJmX1xcYXN0KGZeXFxhc3RcXG1zY3J7TH0gXFx0ZW5zb3IgZl5cXGFzdFxcbXNjcntMfScgXFx0ZW5zb3IgXFxzaHtGfSkiXSxbMSwwXSxbMSwyXSxbMiwzXSxbMCwzLCJcXHNpbWVxIiwzLHsic3R5bGUiOnsiYm9keSI6eyJuYW1lIjoibm9uZSJ9LCJoZWFkIjp7Im5hbWUiOiJub25lIn19fV1d
    \[\begin{tikzcd}
            {\mscr{L} \tensor \mscr{L}'\tensor f_\ast \sh{F}} & {\mscr{L} \tensor f_\ast(f^\ast\mscr{L}'\tensor\sh{F})} \\
            {f_\ast(f^\ast(\mscr{L} \tensor \mscr{L}') \tensor \sh{F})} & {f_\ast(f^\ast\mscr{L} \tensor f^\ast\mscr{L}' \tensor \sh{F})}
            \arrow[from=1-1, to=2-1]
            \arrow[from=1-1, to=1-2]
            \arrow[from=1-2, to=2-2]
            \arrow["\simeq"{marking}, draw=none, from=2-1, to=2-2]
        \end{tikzcd}\]
    commutes.
\end{lemma}
\begin{proof}
    The first part of the proof proceeds analogously to the proof of Lemma~\ref{lem:baseChange}:
    pick an atlas $Y \onto \st{Y}$. Then the fact that the natural map is an isomorphism
    follows from the classical projection formula for sheaves over manifolds.

    Consider the following tiling of the diagram in the statement:
    \[\begin{tikzcd}[column sep=tiny]
            {\mscr{L}\tensor\mscr{L}'\tensor f_\ast\sh{F}} & {\mscr{L}\tensor f_\ast f^\ast\mscr{L}' \tensor f_\ast\sh{F}} & {\mscr{L}\tensor f_\ast(f^\ast\mscr{L}'\tensor \sh{F})} \\
            & {f_\ast f^\ast\mscr{L}\tensor f_\ast f^\ast\mscr{L}'\tensor f_\ast\sh{F}} & {f_\ast f^\ast \mscr{L}\tensor f_\ast(f^\ast \mscr{L}'\tensor\sh{F})} \\
            {f_\ast f^\ast(\mscr{L}\tensor\mscr{L}')\tensor f_\ast\sh{F}} & {f_\ast(f^\ast\mscr{L}\tensor f^\ast\mscr{L}')\tensor f_\ast\sh{F}} \\
            \\
            {f_\ast(f^\ast(\mscr{L}\tensor \mscr{L}')\tensor \sh{F})} && {f_\ast(f^\ast\mscr{L}\tensor f^\ast\mscr{L}'\tensor \sh{F}).}
            \arrow[from=1-1, to=1-2]
            \arrow[from=1-2, to=1-3]
            \arrow[from=1-3, to=2-3]
            \arrow[""{name=0, anchor=center, inner sep=0}, from=2-3, to=5-3]
            \arrow["\simeq"', no head, from=5-3, to=5-1]
            \arrow[from=1-2, to=2-2]
            \arrow[from=2-2, to=2-3]
            \arrow["{\text{naturality of } \tensor}"{description}, draw=none, from=2-2, to=1-3]
            \arrow[from=1-1, to=3-1]
            \arrow[""{name=1, anchor=center, inner sep=0}, from=3-1, to=5-1]
            \arrow["\simeq", no head, from=3-1, to=3-2]
            \arrow[""{name=2, anchor=center, inner sep=0}, from=3-2, to=5-3]
            \arrow[from=2-2, to=3-2]
            \arrow["{\text{lax monoidality of } f_\ast}"{description, pos=0.4}, draw=none, from=3-2, to=0]
            \arrow["{\text{naturality of lax monoidal structure}}"{description}, Rightarrow, draw=none, from=1, to=2]
        \end{tikzcd}\]
    The labelled squares commute for the reason indicated, while the empty square
    commutes because the lax monoidal structure of $f_\ast$ is induced by
    the monoidal structure of $f^\ast$.
\end{proof}

The pushforward functor of twisted sheaves is given by
\begin{alignat*}{1}
    f_\ast: \Shv^{f^\ast\ger{G}}(\orb{N}) & \to \Shv^{\ger{G}}(\orb{M}) \\
    (\sh{F},\varphi)                      & \mapsto
    (f_\ast\sh{F},\tilde{f}_\ast\varphi),
\end{alignat*}
where $\tilde{f}_\ast\varphi$ is the composite
\[
    p_1^\ast f_\ast \sh{F} \simeq f_\ast p_1^\ast \sh{F}
    \xto{f_\ast\varphi} f_\ast(f^\ast\mscr{L} \tensor p_2^\ast \sh{F})
    \eqot \mscr{L} \tensor f_\ast p_2^\ast \sh{F} \simeq
    \mscr{L} \tensor p_2^\ast f_\ast\sh{F},
\]
making use of base change (Lemma~\ref{lem:baseChange}) and the projection
formula (Lem\-ma~\ref{lem:projectionFormula}).
The cocycle condition for $\tilde{f}_\ast\varphi$ follows from
the coycle condition for $\varphi$ and the compatibility of the projection formula
with the tensor product (Lemma~\ref{lem:projectionFormula}).

\begin{proposition}
    \label{prop:twistedPullbackPushfwdAdjunction}
    The functors defined above are adjoint
    \[
        f^\ast: \Shv^\ger{G}(\orb{M}) \leftrightarrows
        \Shv^{f^\ast\ger{G}}(\orb{N}): f_\ast.
    \]
\end{proposition}
\begin{proof}
    The functors $f^\ast,f_\ast$ we defined between the categories
    $\Shv^\ger{G}(\orb{M}) \simeq \Desc(Y/\orb{M})$ and
    $\Shv^{f^\ast \ger{G}}(\orb{N}) \simeq \Desc(f^\ast Y/ \orb{N})$ make the diagram
    % https://q.uiver.app/?q=WzAsNCxbMCwwLCJcXERlc2MoWS9cXG9yYntNfSkiXSxbMSwwLCJcXERlc2MoZl5cXGFzdCBZL1xcb3Jie059KSJdLFswLDIsIlxcU2hDKFxcb3Jie019KSJdLFsxLDIsIlxcU2hDKFxcb3Jie059KSJdLFswLDEsImZeXFxhc3QiLDIseyJvZmZzZXQiOjEsImN1cnZlIjoxfV0sWzEsMCwiZl9cXGFzdCIsMix7Im9mZnNldCI6MSwiY3VydmUiOjF9XSxbMCwyXSxbMSwzXSxbMiwzLCJmXlxcYXN0IiwyLHsib2Zmc2V0IjoxLCJjdXJ2ZSI6MX1dLFszLDIsImZfXFxhc3QiLDIseyJvZmZzZXQiOjEsImN1cnZlIjoxfV0sWzgsOSwiIiwyLHsibGV2ZWwiOjEsInN0eWxlIjp7Im5hbWUiOiJhZGp1bmN0aW9uIn19XV0=
    \[\begin{tikzcd}
            {\Desc(Y/\orb{M})} & {\Desc(f^\ast Y/\orb{N})} \\
            \\
            {\ShC(\orb{M})} & {\ShC(\orb{N})}
            \arrow["{f^\ast}"', shift right=1, curve={height=6pt}, from=1-1, to=1-2]
            \arrow["{f_\ast}"', shift right=1, curve={height=6pt}, from=1-2, to=1-1]
            \arrow[from=1-1, to=3-1]
            \arrow[from=1-2, to=3-2]
            \arrow[""{name=0, anchor=center, inner sep=0}, "{f^\ast}"', shift right=1, curve={height=6pt}, from=3-1, to=3-2]
            \arrow[""{name=1, anchor=center, inner sep=0}, "{f_\ast}"', shift right=1, curve={height=6pt}, from=3-2, to=3-1]
            \arrow["\dashv"{anchor=center, rotate=90}, draw=none, from=0, to=1]
        \end{tikzcd}\]
    commute, where the vertical functors discard the isomorphism $\varphi$
    of a descent datum $(\sh{F},\varphi)$.
    These functors are faithful: the map
    \[
        \Hom_{\Desc}((\sh{F},\varphi),(\sh{F}',\varphi')) \into
        \Hom_{\ShC}(\sh{F},\sh{F}')
    \]
    includes those maps of sheaves $\sh{F} \to \sh{F}'$ that make the diagram
    % https://q.uiver.app/?q=WzAsNCxbMCwwLCJwXzFeXFxhc3RcXHNoe0Z9Il0sWzEsMCwicF8xXlxcYXN0IFxcc2h7Rn0nIl0sWzAsMSwicF8yXlxcYXN0IFxcc2h7Rn0iXSxbMSwxLCJwXzJeXFxhc3QgXFxzaHtGfSciXSxbMCwyLCJcXHZhcnBoaSJdLFsxLDMsIlxcdmFycGhpJyJdLFswLDFdLFsyLDNdXQ==
    \[
        \begin{tikzcd}
            {p_1^\ast\sh{F}} & {p_1^\ast \sh{F}'} \\
            {p_2^\ast \sh{F}} & {p_2^\ast \sh{F}'}
            \arrow["\varphi", from=1-1, to=2-1]
            \arrow["{\varphi'}", from=1-2, to=2-2]
            \arrow[from=1-1, to=1-2]
            \arrow[from=2-1, to=2-2]
        \end{tikzcd}
    \]
    commute.
    This condition is manifestly preserved by the functors
    \[
    f^\ast:\ShC(\orb{M}) \leftrightarrows \ShC(\orb{N}):f_\ast,
    \]
    hence the natural isomorphism
    \[
        \Hom_{\ShC(\orb{M})}(f^\ast \sh{F},\sh{F}') \simeq
        \Hom_{\ShC(\orb{N})}(\sh{F},f_\ast \sh{F}')
    \]
    downstairs restricts to a natural isomorphism upstairs, proving the statement.
\end{proof}

\printbibliography

\end{document}